\documentclass[10pt,aps,prb, preprint,reqno]{amsart}
\usepackage{amsmath,a4wide}
\usepackage{amssymb}
\usepackage{yfonts}
\usepackage{hyperref}
\usepackage{mathrsfs}
\usepackage {latexsym}
\usepackage{graphics}
\usepackage{txfonts}
\usepackage{breqn}
\usepackage{enumitem}
\usepackage{cite}
\usepackage[T1]{fontenc}
\usepackage{babel}
\usepackage{color}
\allowdisplaybreaks

\newtheorem{theorem}{Theorem}[section]
\newtheorem{remark}{Remark}[section]
\newtheorem{lemma}[theorem]{Lemma}

\newtheorem{proposition}[theorem]{Proposition}
\newtheorem{corollary}[theorem]{Corollary}
\newtheorem{define}{Definition}[section]

\allowdisplaybreaks 

\begin{document}
\title[Stochastic transport-diffusion equation]{Non-uniqueness in law of transport-diffusion equation forced\\ by random noise}
 
\date{\today}



\author[Koley]{Ujjwal Koley}
\address{Centre For Applicable Mathematics (CAM), Tata Institute of Fundamental Research, PO Box 6503, GKVK Post Office, Bangalore 560065, India}
\email{ujjwal@math.tifrbng.res.in}
\author[Yamazaki]{Kazuo Yamazaki}
\address{Department of Mathematics and Statistics, Texas Tech University, Lubbock, Texas, 79409-1042, U.S.A.}
\email{kyamazak@ttu.edu}

\begin{abstract}
We consider a transport-diffusion equation forced by random noise of three types: additive, linear multiplicative in It$\hat{\mathrm{o}}$'s interpretation, and transport in Stratonovich's interpretation. Via convex integration modified to probabilistic setting, we prove existence of a divergence-free vector field with spatial regularity in Sobolev space and corresponding solution to a transport-diffusion equation with spatial regularity in Lebesgue space, and consequently non-uniqueness in law at the level of probabilistically strong solutions globally in time.
%\textbf{Keywords: convex integration; continuity equation; transport-diffusion equation; non-uniqueness; random noise.}
\end{abstract}
%\footnote{2010MSC : 35A02; 35R60}

\maketitle

\tableofcontents

\section{Introduction}\label{Introduction}

\subsection{Motivation from physics and mathematics}\label{Motivation from physics and mathematics}

A transport equation, also known as a continuity equation, appears in various problems of mathematical physics such as fluid mechanics and kinetic theory. There are physical examples such as the vorticity formulation of two-dimensional Euler equations with solution that is only in $L_{t}^{\infty} L_{x}^{p}$ for $p \in (1,2)$ and thus not even bounded (see \cite{CS15}); hence, following the breakthrough work of DiPerna and Lions \cite{DL89} to be described in detail subsequently, in the deterministic case there has been extensive effort to reduce the necessary regularity condition on a vector field and still retain the uniqueness of the equation transported by such a vector field. In contrast, Flandoli, Gubinelli, and Priola \cite{FGP10}, and Beck, Flandoli, Gubinelli, and Maurelli \cite{BFGM19} demonstrated that a transport noise can actually regularize its solution by proving uniqueness with a relatively rough vector field. Finally, the recent developments of convex integration technique led to various non-uniqueness results for a deterministic transport equation, even with an arbitrary strong diffusion, with a surprisingly smooth vector field (e.g., \cite{MS18, MS19, MS20}). The purpose of this manuscript is to employ convex integration and prove non-uniqueness in law of the transport equation forced by random noise of various types: additive, linear multiplicative, and transport, with our most interest in the last case considering other works on the stochastic transport equation  such as \cite{FGP10}. 

\subsection{Previous works}\label{Previous works}
Let us write ``$dD$'' for ``$d$-dimensional'' for $d \geq 2$. With a spatial variable $x \in \mathbb{T}^{d} = \mathbb{R}^{d} \setminus \mathbb{Z}^{d}$, given $\rho^{\text{in}}: \mathbb{T}^{d} \mapsto \mathbb{R}$ and a vector field $u: \mathbb{R}_{+} \times \mathbb{T}^{d} \mapsto \mathbb{R}^{d}$, density $\rho: \mathbb{R}_{\geq 0}\times \mathbb{T}^{d} \mapsto \mathbb{R}$ is a solution to a Cauchy problem of a transport equation if 
\begin{subequations}\label{transport}
\begin{align}
&\partial_{t} \rho(t,x) + \text{div} (u(t,x)\rho(t,x)) = 0, \hspace{3mm} \nabla\cdot u = 0, \hspace{3mm} t > 0,\\
&\rho(0,x) = \rho^{\text{in}}(x), 
\end{align}
\end{subequations} 
where $\partial_{t} \triangleq \frac{\partial}{\partial t}$. We refer to \eqref{transport} forced by certain random noise as a stochastic transport equation  (see \eqref{stochastic transport}). Hereafter, for any $p \in [1,\infty]$ we denote by $p'$ its H$\ddot{\mathrm{o}}$lder dual; i.e.,
\begin{equation}\label{est 235}
\frac{1}{p} + \frac{1}{p'} = 1. 
\end{equation} 
For any space of functions $X$, we indicate by $\mathring{X}$ an additional mean-zero condition imposed; only for $C^{\infty} (\mathbb{T}^{d})$, we follow the convention and denote by $C_{0}^{\infty} (\mathbb{T}^{d})$ an additional mean-zero condition imposed on $C^{\infty} (\mathbb{T}^{d})$. 

Informally, the pioneering work of DiPerna and Lions (e.g., \cite[Cor. II.1]{DL89}) states that given initial data $\rho^{\text{in}} \in L_{x}^{p}$ for any $p \in [1,\infty]$, not only existence but also uniqueness for the solution $\rho \in L_{t}^{\infty}L_{x}^{p}$ to \eqref{transport} holds provided  
\begin{equation}\label{est 2} 
u \in L^{1}_{t}W_{x}^{1,p'} \hspace{1mm} \text{ and } \hspace{1mm} \nabla\cdot u \in L_{t}^{1} L_{x}^{\infty} 
\end{equation} 
(see \cite{ACF15} for a certain local version). Subsequently, Ambrosio \cite{A04} proved uniqueness of solution $\rho \in L_{t,x}^{\infty}$ to \eqref{transport} under the condition that $u \in L_{t}^{1} BV_{x}$ and $\nabla\cdot u \in L_{t}^{1}L_{x}^{\infty}$ (see also \cite{CL02}). More recently, Bianchini and Bonicatto \cite{BB20} proved the uniqueness in case $u \in L_{t}^{1}BV_{x}$ and $u$ is nearly incompressible (see \cite[Definition 1.1]{BB20} for the definition of $u$ being nearly incompressible). Moreover, Caravenna and Crippa \cite{CC16} proved uniqueness of solution $\rho \in L_{t}^{1}L_{x}^{1}$ rather than $L_{t}^{\infty} L_{x}^{p}$, starting from $\rho^{\text{in}} \in L_{x}^{1}$ under additional assumptions that $u \in L_{t}^{1}W_{x}^{1,q}$ for $q > 1$, $\nabla\cdot u \in L_{t}^{1}L_{x}^{\infty}$, and $u(t,\cdot)$ is continuous for almost every $t \in [0,T]$ with a modulus of continuity on compact sets that is uniform in time.  We refer to \cite{A17} for an excellent survey of related results. In the stochastic case, Flandoli, Gubinelli, and Priola \cite{FGP10} demonstrated that a transport noise in Stratonovich's interpretation can regularize the solution enough to prove uniqueness. Specifically, in $\mathbb{R}^{d}$ rather than $\mathbb{T}^{d}$, the authors in \cite{FGP10} considered \eqref{transport} forced on the right hand side (r.h.s.) by 
\begin{equation}
-\sum_{i=1}^{d} \frac{\partial}{\partial x_{i}} \rho(t,x) \circ dB_{i}(t)
\end{equation} 
where $B = (B_{1},\hdots, B_{d})$ is a standard Brownian motion in $\mathbb{R}^{d}$ such that $B\rvert_{t=0} = 0$ $\mathbb{P}$-almost surely (a.s.), and they proved in \cite[Theorem 20]{FGP10} that under a condition that $u \in L_{t}^{\infty} C_{b}^{\alpha}$ for some $\alpha \in (0,1)$ and $ \nabla\cdot u \in L_{t,x}^{p}$ for some $p \in (2,\infty)$, given any $\rho^{\text{in}} \in L_{x}^{\infty}$, there exists a unique process $\rho \in L^{\infty} (\Omega \times [0,T] \times \mathbb{R}^{d})$ such that for all $\psi \in C_{0}^{\infty} (\mathbb{R}^{d})$, the process $\int_{\mathbb{R}^{d}} \psi(x) \rho(t,x) dx$ has a continuous modification which is a semi-martingale and satisfies \eqref{transport} distributionally (see also \cite[Theorem 21]{FGP10} for some variation). We refer to \cite{CO13} for similar results. Moreover, the authors in \cite{BFGM19} proved, under Ladyzhenskaya-Prodi-Serrin (LPS) condition on $u$, path-by-path uniqueness of weak solutions to \eqref{transport} (see \cite[Theorem 1.2]{BFGM19}). Concerning non-uniqueness, DiPerna and Lions in \cite[Section IV]{DL89} provided a few examples via Lagrangian approach, specifically $u \in W_{\text{loc}}^{1,p} (\mathbb{R}^{2})$ for any $p < \infty$ that is bounded, uniformly continuous and has unbounded divergence, as well as $u \in W_{\text{loc}}^{s,1} (\mathbb{R}^{2})$ for all $s \in [0,1)$ such that $u \in L^{p} (\mathbb{R}^{2}) + L^{\infty} (\mathbb{R}^{2})$ for all $p \in [1,2)$ and $\nabla\cdot u = 0$. Moreover, \cite[Theorem 2]{CLR03} showed that in case $x \in \mathbb{R}^{3}$, there exist $u = u(x)$ that is uniformly bounded and divergence-free, as well as a corresponding non-trivial solution $\rho \in L_{t,x}^{\infty}$ starting from $\rho^{\text{in}} \equiv 0$ (see also \cite[Section 6.1]{FGP10}). Such non-uniqueness results were relatively limited until the recent developments of convex integration which we describe next. 

The breakthrough work of Nash \cite{N54} concerning isometric embeddings led to Gromov establishing convex integration in \cite[Part 2.4]{G86}. Further important works by M$\ddot{\mathrm{u}}$ller and $\check{\mathrm{S}}$ver$\acute{\mathrm{a}}$k on convex integration for Lipschitz mappings and more in \cite{MS98, MS03} led to another breakthrough work \cite{DS09}  by De Lellis and Sz$\acute{\mathrm{e}}$kelyhidi Jr. in which non-zero weak solutions $u \in L^{\infty} (\mathbb{R}_{+} \times \mathbb{R}^{d})$ to $dD$ Euler equations with compact support for $d\geq 2$ were constructed via convex integration technique. After various extensions and improvements (e.g., \cite{DS10, DS13, BDIS15}), particularly making use of Mikado flows in convex integration, Isett \cite{I18} settled the negative direction of Onsager's conjecture \cite{O49}. By an additional ingredient of intermittency, Buckmaster and Vicol \cite{BV19a} proved the non-uniqueness of weak solutions to the 3D Navier-Stokes equations. Other recent applications of convex integration can be found in the following references: \cite{BCV18, CDD18, D19, LT20} concerning fractional Laplacian; \cite{BMS21} on power-law flows; \cite{BBV21, FLS21} on magnetohydrodynamics (MHD) system; \cite{LTZ20} on Boussinesq system (see \cite{BV19b} for an excellent review). 

The far-reaching consequence of convex integration technique has recently made impact in the stochastic community as well. For convenience, let us first recall some definitions.
\begin{define}
(e.g., \cite[Definitions 1.2-1.4]{C03}) For \eqref{transport} forced on the r.h.s. by noise involving Brownian motion $B$, we say that uniqueness in law holds if for any solutions $(\rho, B)$ and $(\tilde{\rho}, \tilde{B})$  which may be defined on different filtered probability spaces, $\mathcal{L}(\rho(t))_{t\geq 0} = \mathcal{L}(\tilde{\rho}(t))_{t\geq 0}$; i.e., they have same probability laws. Moreover, we say that path-wise uniqueness holds if for any solutions $(\rho, B)$ and $(\tilde{\rho}, B)$ defined on same filtered probability space $(\Omega, \mathcal{F}, (\mathcal{F}_{t})_{t\geq 0}, \mathbb{P})$, $\mathbb{P} ( \{ \rho(t) = \tilde{\rho}(t) \hspace{1mm} \forall \hspace{1mm} t \geq 0\}) = 1$. Finally, we say that a solution $(\rho, B)$ is probabilistically strong if $\rho$ is adapted to the completed natural filtration of $B$. 
\end{define} 
First, path-wise non-uniqueness of certain stochastic Euler system forced by linear multiplicative noise in It$\hat{\mathrm{o}}$'s interpretation was proven by Breit, Feireisl, and Hofmanov$\acute{\mathrm{a}}$ \cite{BFH20} and Chiodaroli, Feireisl, and Flandoli \cite{CFF19} with the noise in Stratonovich's interpretation. Subsequently, Hofmanov$\acute{\mathrm{a}}$, Zhu, and Zhu \cite{HZZ19} proved non-uniqueness in law of 3D stochastic Navier-Stokes equations forced by additive and linear multiplicative noise in It$\hat{\mathrm{o}}$'s interpretation (see also \cite{HZZ21a, HZZ21b}). For a subsequent comparison purpose, let us formally state this equation as follows: with $u: \mathbb{R}_{\geq 0} \times \mathbb{T}^{3} \mapsto \mathbb{R}^{3}$ and $\pi: \mathbb{R}_{+} \times \mathbb{T}^{3} \mapsto \mathbb{R}$ respectively representing the velocity and pressure fields, given $u^{\text{in}} (x) \triangleq u(0,x)$, 
\begin{equation}\label{stochastic NS}
du(t,x) + (\text{div} (u(t,x) \otimes u(t,x)) + \nabla \pi(t,x) - \Delta u(t,x)) dt = G(u) dB, \hspace{2mm} \nabla\cdot u = 0, \hspace{3mm} t > 0,
\end{equation} 
where $G(u)dB$ represents the stochastic force. More works followed concerning non-uniqueness in law: \cite{HZZ20} on 3D stochastic Euler equations; \cite{RS21, Y20a, Y20c, Y21c} in case of 2D and 3D stochastic Navier-Stokes equations with fractional Laplacian; \cite{Y21a} on 2D and 3D stochastic Boussinesq system; \cite{Y21d} on 3D stochastic MHD system. Let us emphasize that non-uniqueness in law implies non-uniqueness path-wise due to Yamada-Watanabe theorem while uniqueness in law, together with existence of a probabilistically strong solution, implies path-wise uniqueness due to Cherny's theorem (see \cite[Theorem 3.2]{C03}). A remarkable property of solutions to the stochastic partial differential equations (PDEs) obtained via convex integration is that they are probabilistically strong, and the existence of a probabilistically strong solution to the 3D stochastic Navier-Stokes equations was a long-standing open problem (e.g., \cite[p. 84]{F08}). 

Convex integration technique was applied to transport equation first, to the best of our knowledge, by Crippa, Gusev, Spirito, and Wiedemann who demonstrated, with a proof inspired by \cite{DS09}, that in $\mathbb{R}^{d}$ with $d \geq 2$ there exist infinitely many $\rho$ and divergence-free vector field $u$ which are both bounded and compactly supported in space and time that solves \eqref{transport} distributionally and has prescribed energy (see \cite[Theorem 3.1]{CGSW15}). More recently, Modena and Sz$\acute{\mathrm{e}}$kelyhidi Jr. extended such result significantly; in particular, \cite[Corollary 1.3]{MS18} states that on $\mathbb{T}^{d}$ for $d \geq 3$, if $p\in (1,\infty), \tilde{p} \in [1,\infty)$ satisfy 
\begin{equation}\label{est 1}
\frac{1}{p} + \frac{1}{\tilde{p}} > 1 + \frac{1}{d-1},
\end{equation} 
and $\bar{\rho} \in C_{0}^{\infty} (\mathbb{T}^{d})$, then there exist $u \in C_{t} L_{x}^{p'} \cap C_{t} W_{x}^{1,\tilde{p}} $ that is divergence-free and a corresponding weak solution $\rho \in C_{t} L_{x}^{p}$ such that $\rho \rvert_{t=0} \equiv 0$ but $\rho \rvert_{t=T} \equiv \bar{\rho}$, implying non-uniqueness. Subsequently, the same authors extended this result to the case $p = 1$ so that $\rho \in C_{t}L_{x}^{1}$ and surprisingly, not only $u\in C_{t}W_{x}^{1,\tilde{p}} \cap C_{t}L_{x}^{\infty}$ but in fact  $u\in C_{t}W_{x}^{1,\tilde{p}} \cap C_{t,x}$. These results of \cite{MS18, MS19} required $d \geq 3$; this is related to the inherent nature of Mikado flows and it is not the first time that the case $d = 2$ had to be excluded in its application (e.g., see \cite[p. 877]{I18}). Nevertheless, Modena and Sattig in \cite{MS20} significantly improved \cite{MS18, MS19}; specifically, \cite[Theorem 1.1]{MS20} states that for any $d \geq 2$, if $p \in [1,\infty), \tilde{p} \in [1,\infty)$ satisfy 
\begin{equation}\label{est 7}
\frac{1}{p} + \frac{1}{\tilde{p}} > 1 + \frac{1}{d}
\end{equation}
(cf. \eqref{est 1}), then there are infinitely many divergence-free vector fields $u \in C_{t}L_{x}^{p'} \cap C_{t}W_{x}^{1,\tilde{p}}$ if $p \in (1,\infty)$ while $u \in C_{t,x} \cap C_{t}W_{x}^{1,\tilde{p}}$ if $p  =1$, such that the corresponding Cauchy problem with a weak solution $\rho \in C_{t}L_{x}^{p}$ fails uniqueness. These results in \cite{MS18, MS19, MS20} can be extended to the case of arbitrarily strong diffusion in the expense of a few additional constraints; e.g., \cite[Theorem 1.3]{MS20} considers a diffusive case with $-\Delta \rho$ but only for $d \geq 3$. Finally, Cheskidov and Luo \cite[Theorem 1.3]{CL21} improved \eqref{est 7} to $\frac{1}{p} + \frac{1}{\tilde{p}} > 1$ although it additionally requires $d \geq 3$, $p > 1$, $u \in L_{t}^{\infty} L_{x}^{p'} \cap L_{t}^{1}W_{x}^{1, \tilde{p}}$ rather than $C_{t}L_{x}^{p'} \cap C_{t}W_{x}^{1,\tilde{p}}$, and $\rho \in L_{t}^{1}L_{x}^{p}$ rather than $C_{t}L_{x}^{p}$. 
The purpose of this manuscript is to employ, for the first time to the best of our knowledge, probabilistic convex integration to stochastic transport equation to prove non-uniqueness in law. 

More precisely, the main contributions of this paper are listed below:
\begin{itemize}
\item [(a)] We prove that non-uniqueness in law holds for \eqref{transport} with diffusion perturbed by an additive noise on an arbitrary time interval. This has been achieved by gluing a convex integration solution with a weak solution of  stochastic transport equation.
\item [(b)] We also consider the transport equation \eqref{transport} with diffusion perturbed by a linear multiplicative noise in It$\hat{\mathrm{o}}$'s interpretation, and present its non-uniqueness (in law) results.
\item[(c)] Finally, for the transport equation \eqref{transport} with diffusion perturbed by a transport noise in Stratonovich's interpretation, we exhibit two different proofs of non-uniqueness in law. In one situation, we are able to prescribe initial data and construct convex integration solutions.
\end{itemize}

\section{Statement of main results}\label{Section 2}
Hereafter, we will consider the following stochastic transport-diffusion equation 
\begin{subequations}\label{stochastic transport}
\begin{align}
&d \rho (t,x) + ( \text{div} (u(t,x) \rho(t,x)) - \Delta \rho(t,x)) dt = G(\rho) dB, \hspace{3mm} \nabla\cdot u = 0, \\
&\rho(0,x) = \rho^{\text{in}} (x), 
\end{align} 
\end{subequations} 
where $G(\rho) dB$ represents a stochastic force on the probability space $(\Omega, \mathcal{F}, (\mathcal{F}_{t})_{t \geq 0},\mathbb{P})$. We chose to consider the case of diffusion via Laplacian in comparison to the Navier-Stokes equations; most of our discussions and results go through for the cases of zero diffusion and arbitrarily strong diffusion similarly to \cite{MS18, MS19, MS20}. Let us consider three different types of noise:
\begin{enumerate}
\item additive noise; i.e., $G(\rho)dB =dB$ where $B$ is a certain $GG^{\ast}$-Wiener process for a Hilbert-Schmidt operator $G \in L_{2} (U, \mathring{L}^{2}(\mathbb{T}^{d}))$ on some Hilbert space $U$ (e.g., \cite[Section 4.1]{DZ14}) (considered in \cite{HZZ19, HZZ20, HZZ21a, HZZ21b, RS21, Y20a, Y20c, Y21a, Y21c, Y21d}); 
\item linear multiplicative noise in It$\hat{\mathrm{o}}$'s interpretation; i.e., $G(\rho) dB = \rho dB$ where $B$ is a $\mathbb{R}$-valued Wiener process (considered in \cite{BFH20, HZZ19, Y20a, Y20c, Y21a, Y21c, Y21d}); 
\item transport noise in Stratonovich's interpretation; i.e., $G(\rho) dB = - \sum_{i=1}^{d} \frac{\partial}{\partial x_{i}} \rho \circ dB_{i}$ where $B = (B_{1},\hdots, B_{d})$ is a Brownian motion.
\end{enumerate} 
While the noise type (3) has never been considered in probabilistic convex integration, it was the type investigated in \cite{FGP10}. A common property that is shared by \eqref{stochastic transport} with these types of noise is that they may all be informally transformed to random PDEs when these operations are allowed. In the first case of an additive noise, one may consider a heat equation forced by the same noise 
\begin{equation}\label{stochastic heat} 
dz(t,x) = \Delta z(t,x) dt + dB(t,x), \hspace{3mm} z(0,x) \equiv 0
\end{equation} 
so that we may focus on the following random PDE solved by $\theta(t,x) \triangleq \rho(t,x) - z(t,x)$: 
\begin{equation}\label{est 3}
\partial_{t} \theta (t,x) + \text{div} ( u(t,x) \theta (t,x)) = \Delta \theta(t,x) - \text{div} (u(t,x) z(t,x)), \hspace{3mm} \nabla\cdot u = 0, \hspace{3mm} \theta(0,x) = \rho^{\text{in}}(x). 
\end{equation} 
In the second case of a linear multiplicative noise in It$\hat{\mathrm{o}}$'s interpretation, we may focus on the following random PDE solved by $\theta(t,x) \triangleq \rho(t,x) e^{-B(t)}$: 
\begin{equation}\label{est 4}
\partial_{t} \theta(t,x) + \text{div} (u(t,x) \theta(t,x)) + \frac{1}{2} \theta(t,x) = \Delta \theta (t,x), \hspace{3mm} \nabla\cdot u = 0, \hspace{3mm} \theta(0,x) = \rho^{\text{in}}(x).
\end{equation} 
In the third case of a transport noise in Stratonovich's interpretation, we may focus on the following random PDE solved by $\theta (t,x) \triangleq \rho(t, x+ B(t))$ (e.g., \cite[Theorem 3.3.2 on p. 93]{K90}, also \cite{K82, K84}): 
\begin{equation}\label{est 6}
\partial_{t} \theta (t,x) + \text{div} (u(t, x+ B(t)) \theta (t,x)) = \Delta \theta(t,x), \hspace{3mm} \nabla\cdot u = 0, \hspace{3mm} \theta(0,x) = \rho^{\text{in}}(x). 
\end{equation} 
\begin{remark}
We may also consider a linear multiplicative noise in Stratonovich's interpretation; i.e., $G(\rho) dB = \rho \circ dB$ where $B$ is a $\mathbb{R}$-valued Wiener process (considered in \cite{CFF19}). In this case we may focus on the following random PDE solved by $\theta(t,x) \triangleq \rho(t,x) e^{-B(t)}$: 
\begin{equation}\label{est 5}
\partial_{t} \theta(t,x) + \text{div} (u(t,x) \theta(t,x)) = \Delta \theta(t,x), \hspace{3mm} \nabla\cdot u = 0, \hspace{3mm} \theta(0,x) = \rho^{\text{in}}(x).
\end{equation} 
This equation \eqref{est 5} is equivalent to the deterministic transport-diffusion equation (cf. \eqref{transport}) and hence non-uniqueness results from \cite{MS18, MS19, MS20} directly apply to \eqref{est 5} and therefore to \eqref{stochastic transport} in the case of linear multiplicative noise in Stratonovich's interpretation.
\end{remark}
Let us now describe main results in the cases of additive noise, linear multiplicative noise in It$\hat{\mathrm{o}}$'s interpretation, and transport noise in Stratonovich's interpretation; we describe the case of a linear multiplicative noise in It$\hat{\mathrm{o}}$'s interpretation first for convenience of explaining the difficulty of their proofs. Let us mention first that existence  of solution is standard; e.g., Lemma \ref{Lemma 8.1}, which is a straight-forward generalization of \cite[Proposition II.1]{DL89}, proves the existence of a deterministic solution $\theta \in L_{t}^{\infty} L_{x}^{p}$ to \eqref{est 4} with $u \in L^{1}_{t}L_{x}^{p'}$ for any $p \in (1,\infty)$ starting from $\theta^{\text{in}} \in L_{x}^{p}$ and therefore a solution process $\rho = \theta e^{B} \in L_{t}^{\infty} L_{x}^{p}$ to \eqref{stochastic transport} forced by linear multiplicative noise $G(\rho)dB = \rho dB$ where $B$ is a $\mathbb{R}$-valued Wiener process. The following result proves that uniqueness fails by construction via convex integration.  

\begin{theorem}\label{Theorem 2.1}
(Linear multiplicative noise in It$\hat{\mathrm{o}}$'s interpretation) Suppose that $d \geq 3$ and $B$ is a $\mathbb{R}$-valued Wiener process on $(\Omega, \mathcal{F}, \mathbb{P})$ with $(\mathcal{F}_{t})_{t\geq 0}$ being the normal filtration generated by $B$. Let $T > 0$, $p \in (1, \infty), \tilde{p} \in [1, \infty)$ such that 
\begin{equation}\label{est 8}
\frac{1}{p} + \frac{1}{\tilde{p}} > 1 + \frac{1}{d} \hspace{1mm} \text{ and } \hspace{1mm} p' < d. 
\end{equation} 
Then there exist infinitely many pairs $(\rho, u)$ where 
\begin{equation}
u \in C([0,T]; L^{p'} (\mathbb{T}^{d})) \cap C([0,T]; W^{1,\tilde{p}}(\mathbb{T}^{d}))
\end{equation} 
is a deterministic divergence-free vector field and 
\begin{equation}
\rho \in C([0,T]; L^{r} (\Omega; L^{p}(\mathbb{T}^{d}))) \hspace{1mm} \forall \hspace{1mm} r \in [1,\infty) \text{ that is } (\mathcal{F}_{t})_{t\geq 0}\text{-adapted, }  \rho \rvert_{t=0} \equiv 0, \rho \rvert_{t=T} \not\equiv 0 \hspace{1mm} \mathbb{P}\text{-a.s.,}
\end{equation} 
that satisfy \eqref{stochastic transport} forced by linear multiplicative noise in It$\hat{o}$'s interpretation as follows: for every test function $\varphi \in C_{0}^{\infty} (\mathbb{T}^{d})$, the process $\int_{\mathbb{T}^{d}} \rho(t,x) \varphi(x) dx$ has a continuous modification which is a $(\mathcal{F}_{t})_{t\geq 0}$-semimartingale and satisfies 
\begin{align}
\int_{\mathbb{T}^d} \rho(t,x)\,\varphi(x)\,dx &= \int_{\mathbb{T}^d} \rho^{\text{in}}  (x)\,\varphi(x)\,dx + \int_0^t \int_{\mathbb{T}^d} \rho(s,x)\, u(s,x)\, \cdot \nabla \varphi(x)\,dx\, ds \nonumber  \\
&\quad + \int_0^t \int_{\mathbb{T}^d} \rho(s,x)\,\Delta \varphi(x)\,dx\,ds + \int_0^t \int_{\mathbb{T}^d} \rho(s,x)\,\varphi(x)\,dx\,dB(s,x) \label{solution multiplicative}
\end{align}
$\mathbb{P}$-a.s. and for all $t \in [0, T]$. Consequently, non-uniqueness in law holds for \eqref{stochastic transport} forced by the linear multiplicative noise in It$\hat{\mathrm{o}}$'s interpretation on $[0,\infty)$; moreover, for all $T> 0$, non-uniqueness in law holds on $[0,T]$. 
\end{theorem} 

\begin{remark}
\cite[Theorem 1.3]{HZZ19} proves non-uniqueness in law for the 3D stochastic Navier-Stokes equations \eqref{stochastic NS} forced by a linear multiplicative noise in It$\hat{\mathrm{o}}$'s interpretation defined on a random time interval using stopping time to control the noise, and then \cite[Theorem 1.4]{HZZ19} (which is same as the last sentence in Theorem \ref{Theorem 2.1}) more generally proves non-uniqueness in law on an arbitrary deterministic time interval using techniques from martingale problem of \cite{SV97} (see \cite[Section 5]{HZZ19}) and generalization of Cherny's theorem \cite{C03} (see \cite[Theorem C.1]{HZZ19}). Theorem \ref{Theorem 2.1} is interesting in a way that it implies non-uniqueness in law more immediately. The main reason for this difference is that analogous transformation to \eqref{est 4} for the 3D Navier-Stokes equations \eqref{stochastic NS} forced by linear multiplicative noise in It$\hat{\mathrm{o}}$'s interpretation solved by $u$ returns the following random PDE solved by $v(t,x) = u(t,x) e^{-B(t)}$:
\begin{equation}\label{est 15}
\partial_{t} v(t,x) + e^{B(t)} \text{div} (v(t,x) \otimes v(t,x)) + e^{-B(t)} \nabla \pi(t,x) + \frac{1}{2} v(t,x) = \Delta v(t,x), \hspace{3mm} \nabla\cdot v = 0.  
\end{equation} 
The striking difference between \eqref{est 4} and \eqref{est 15} is that the noise appears in the latter but not the former. We shall elaborate on how the appearance of the noise in the transformed random PDE creates major difficulty in Remark \ref{Remark 2.3}. In short, we are able to essentially directly apply the deterministic convex integration technique from \cite{MS20} to \eqref{est 4} with only a few necessary modifications and immediately deduce Theorem \ref{Theorem 2.1}; for this reason, we leave the proof of Theorem \ref{Theorem 2.1} in Appendix A for completeness. 
\end{remark}

Our next theorem concerns the case of additive noise, $G(\rho) dB = dB$ where $B$ is a certain $GG^{\ast}$-Wiener process. 
\begin{remark}\label{Remark 2.3} 
The first major difficulty is the appearance of $z$ in \eqref{est 3} which is H$\ddot{o}$lder continuous in time of an exponent that is strictly less than $\frac{1}{2}$, inherently from the regularity of Brownian motion. Following the Nash-type convex integration schemes in \cite{MS18, MS19, MS20}, let us consider the following transport-diffusion-defect equation based on \eqref{est 3}:
\begin{subequations}\label{est 29}
\begin{align}
\partial_{t} \theta(t,x) + \text{div} (u(t,x) \theta(t,x)) + \text{div} (u(t,x) z(t,x))  - \Delta \theta(t,x) = - \text{div} R(t,x), \label{est 30} \\
\nabla\cdot u = 0. 
\end{align}
\end{subequations} 
E.g., in the key proposition, specifically \cite[Proposition 2.1]{MS20}, the authors assume that there already exists ``smooth solution $(\rho_{0}, u_{0}, R_{0})$'' and construct another ``smooth solution $(\rho_{1}, u_{1}, R_{1})$'' with the desired properties. Unfortunately, such smoothness is not only in space but also in time. Specifically, the authors in \cite{MS20} define $R_{0}^{j}$ to be the $j$-th component of the vector $R_{0}$, define $a_{j}, b_{j}$ in terms of $R_{0}^{j}$ on \cite[p. 1092]{MS20} and assume an estimate of $a_{j}$ and $b_{j}$ in $C_{t,x}^{k}$-norm for an arbitrary $k \in \mathbb{N}$ (see \cite[Equation (4.16b)]{MS20}); similarly, the proof in \cite{MS18} consists of an estimate of ``$\lVert \partial_{t} \vartheta(t) \rVert_{L^{1}}$'' on \cite[p. 30]{MS18} which involves a temporal derivative of $R_{0}$, and the proof of \cite{MS19} also consists of ``$\partial_{t}R_{0,j}$'' in the definitions of ``$A_{1}^{j}$'' and ``$A_{2}^{j}$'' on \cite[p. 25]{MS19}. These authors can do so because their assumption is that $R_{0}$ is smooth in both space and time; however, if we follow the same approach, our $R$, defined by simply taking an anti-divergence operator (see Definition \ref{Definition 3.1}) in \eqref{est 30}, is only H$\ddot{o}$lder continuous in time with an exponent strictly smaller than $\frac{1}{2}$. This type of difficulty was already observed in previous attempts of probabilistic convex integration (e.g., \cite[Remark 1.2]{Y20c}). We will overcome this difficulty by mollifying $R_{0}$ and replacing $R_{0}$ in the proof of \cite{MS20} by the mollified $R_{0}$ appropriately and carefully estimating our new $a_{j}$ and $b_{j}$ (see Lemma \ref{Lemma 4.4}). While this has clear advantage in terms of differentiability, it has a major cost, that we will describe next. 

Let us describe the second major difficulty, which involves the aforementioned ``cost'' of mollifying $R_{0}$. In short, the authors in \cite{MS20} assume the existence of a smooth solution $(\rho_{0}, u_{0}, R_{0})$ and construct $(\rho_{1}, u_{1}, R_{1}) = (\rho_{0}, u_{0}, R_{0}) + \text{perturbation}$. The key idea is that $R_{0}$ is strategically embedded in this perturbation so that for any $t$ such that $R_{0}(t) \equiv 0$, it follows that $R_{1}(t)\equiv 0$ and $(\rho_{1}, u_{1})(t) \equiv (\rho_{0}, u_{0})(t)$ (see the last sentence of \cite[Proposition 2.1]{MS20}). This leads to the key fact in \cite[Theorem 1.2 (iii)]{MS20}; the authors can choose an arbitrary $\bar{\rho} \in C^{\infty}_{t,x}$ with zero mean, divergence-free vector field $\bar{u} \in C^{\infty}_{t,x}$, and construct a new solution $(\rho, u)$ such that $(\rho, u)(t) \equiv (\bar{\rho}, \bar{u})(t)$ for all $t$ such that $(\bar{\rho}, \bar{u})(t)$ satisfies the transport equation. This immediately allows them to prove the existence of a non-zero solution starting from zero initial data ``Therefore, $\rho \rvert_{t=0} \equiv 0$ and $\rho \rvert_{t=T} = \bar{\rho} \not\equiv 0$'' on \cite[p. 1079]{MS20}, and this is a common punchline in all the proofs of \cite{MS18, MS19, MS20}. The cost of mollifying $R_{0}$ and replacing $R_{0}$ appropriately in the proof of \cite{MS20}, which seems unavoidable, is the following: clearly $R_{0}(t) \equiv 0$ does not imply that the  mollified $R_{0}$ at time $t$ is equivalently zero and therefore our perturbation would not vanish at such $t$. In fact, another unique difficulty in the case of an  additive noise, in comparison to all other cases, is that a zero function is no longer a solution starting from zero initial data for either \eqref{stochastic transport} or \eqref{est 3}. Therefore, we cannot follow the approach of \cite{MS18, MS19, MS20} anyway. Hence, we must find an alternative approach to prove non-uniqueness. The complexity is that the convex integration scheme on the transport equation constructs both $\rho$ and $u$ and hence we cannot ``prescribe $u$'' or even ``know $u$ precisely even after the construction,'' and because formally $u$ is ``given'' and $\rho$ is the ``unknown,'' to prove non-uniqueness, we need to \emph{contradict} a \emph{classical} fact that is valid for ``any'' $u$. With the approach of doing so by constructing a non-zero solution starting from zero initial data out of the picture, we turn to another \emph{classical} fact that is valid for ``any'' $u$, namely the $L^{p}$-inequality which we will describe next. 
\end{remark}
 
In case $p \in [2,\infty)$, we first recall that $dB(s,x) = \sum_{j=1}^{\infty} \sqrt{\eta_{j}} e_{j}(x) d \beta_{j}(s)$, where $\{\sqrt{\eta_{j}}\}_{j\in \mathbb{N}}$ and $\{e_{j}\}_{j\in\mathbb{N}}$ respectively are eigenvalues and eigenvectors of $G$ by the property of a $GG^{\ast}$-Wiener process (e.g., \cite[Definition 4.2, Proposition 4.3]{DZ14}). Thus, because \eqref{stochastic transport} has diffusion, we can employ a standard Galerkin approximation in which we may rely on It$\hat{\mathrm{o}}$'s formula for $L^{p}$-norm (e.g., \cite[Lemma 5.1]{K10}), that is valid for only $p \in [2,\infty)$, 
to deduce 
\begin{align}
\lVert \rho(t) & \rVert_{L_{x}^{p}}^{p} = \lVert \rho(0) \rVert_{L^{p}}^{p} + \int_{0}^{t} [-p \int_{\mathbb{T}^{d}} \lvert \rho(l) \rvert^{p-2} \rho(l) (-\Delta \rho(l) + (u\cdot\nabla) \rho(l)) dx \label{est 337}\\ 
&+ \frac{1}{2} p(p-1) \int_{\mathbb{T}^{d}} \lvert \rho(l) \rvert^{p-2} \lVert \sqrt{\eta_{j}} e_{j} \rVert_{l^{2} (\mathbb{N})}^{2} dx ] dl + p \sum_{j=1}^{\infty} \int_{0}^{t} \int_{\mathbb{T}^{d}} \lvert \rho(l) \rvert^{p-2} \rho(l) \sqrt{\eta_{j}} e_{j} dx d \beta_{j}(l). \nonumber 
\end{align} 
In case $p = 2$, this implies
\begin{equation}\label{est 335}
\mathbb{E}^{\mathbb{P}}[\lVert \rho(t) \rVert_{L_{x}^{2}}^{2}] \leq \lVert \rho(0) \rVert_{L^{2}}^{2} + t Tr (GG^{\ast}).
\end{equation} 
In case $p \in (2,\infty)$, we can use the well-known fact that $-p\int_{\mathbb{T}^{d}} \lvert \rho(l) \rvert^{p-2} \rho(l) (-\Delta \rho(l) + (u\cdot\nabla) \rho(l) d \leq 0$ for all $l \in [0,t]$, apply Young's inequality and Gronwall's inequality on \eqref{est 337} to deduce
\begin{equation}\label{est 228}
\mathbb{E}^{\mathbb{P}} [ \lVert \rho(t) \rVert_{L_{x}^{p}}^{p}] \leq e^{t} [ \lVert \rho(0) \rVert_{L^{p}}^{p} + C(p, Tr((-\Delta)^{\frac{d}{2} + 2 \varsigma}GG^{\ast} ))]
\end{equation} 
where 
\begin{equation}\label{est 33}
C(p, Tr((-\Delta)^{\frac{d}{2} + 2 \varsigma}GG^{\ast})) \triangleq \frac{2}{p} (\frac{p-2}{p})^{\frac{p-2}{2}} [ \frac{1}{2} p(p-2) C_{S}^{1} Tr ((-\Delta)^{\frac{d}{2} + 2 \varsigma} GG^{\ast} )]^{\frac{p}{2}}
\end{equation} 
with $C_{S}^{1}$ being the Sobolev constant for $\dot{H}^{\frac{d}{2}+ 2 \varsigma} (\mathbb{T}^{d}) \hookrightarrow L^{\infty}(\mathbb{T}^{d})$ for mean-zero functions. Application of martingale representation theorem (e.g., \cite[Theorem 8.2]{DZ14}) deduces the existence of an analytically weak solution $\rho \in L_{\omega}^{q} L_{t}^{\infty} L_{x}^{p}$ for all $q \in [1,\infty)$ that preserves the bounds \eqref{est 335} and \eqref{est 228} in case $p = 2, p \in (2,\infty)$, respectively. In case $p \in (1,2)$, standard It$\hat{\mathrm{o}}$'s formula is not available (e.g., \cite{K10} which is only for $p \geq 2$). Nonetheless, we can assume \eqref{est 31} so that $z$ that solves \eqref{stochastic heat} has regularity of $L_{\omega}^{\infty} L_{t}^{\infty} W_{x}^{1,\infty}$ due to \eqref{est 32} up to a stopping time $T_{L}$ in \eqref{est 54}, we can apply Lemma \ref{Lemma 8.1} to obtain the existence of an analytically weak solution $\theta \in L_{\omega}^{\infty} L_{t}^{\infty} L_{x}^{p}$ to \eqref{est 3} in case $p \in (1,2]$. Here, we can apply Lemma \ref{Lemma 8.1} precisely only for $p \in (1,2]$ because we must consider $-(u\cdot\nabla) z$ in \eqref{est 3} as the external force $f$ in \eqref{est 227}; even though $\nabla z \in L_{\omega, t, x}^{\infty}$ by our assumption, \eqref{est 13} still requires $u \in L^{1}_{t}L^{p}_{x}$ so that $-(u\cdot\nabla) z \in L_{t}^{1}L_{x}^{p}$ and this is satisfied only if $p \in (1,2]$ so that $p' \in [2,\infty)$. Therefore, we can rely on the bounds \eqref{est 11} and \eqref{est 32} to deduce for all $t \in [0, T_{L}]$ (see \eqref{est 54}) 
\begin{equation}\label{est 336}
\lVert \rho(t) \rVert_{L_{x}^{p}} \leq \lVert \theta(t) \rVert_{L_{x}^{p}} + \lVert z(t) \rVert_{L_{x}^{\infty}} \leq \lVert \rho^{\text{in}} \rVert_{L^{p}} + L^{\frac{1}{4}} (\int_{0}^{t} \lVert u(s) \rVert_{L_{x}^{p}} ds+ 1). 
\end{equation} 
It also follows that $\rho = \theta + z$ is an analytically weak solution to \eqref{stochastic transport} forced by the additive noise. Therefore, we aim to construct a solution $\rho$ such that for a fixed $T > 0$, on a set $\{T_{L} \geq T\}$  
\begin{equation}
\lVert \rho(T) \rVert_{L_{x}^{p}} > 
\begin{cases}
\lVert \rho^{\text{in}} \rVert_{L^{p}} + L^{\frac{1}{4}} ( \int_{0}^{T} \lVert u(s) \rVert_{L_{x}^{p}} ds + 1) & \text{ if } p \in (1,2), \\
\lVert \rho^{\text{in}} \rVert_{L^{2}} + \sqrt{T Tr (GG^{\ast})} & \text{ if } p = 2, \\
e^{\frac{T}{p}} [ \lVert \rho^{\text{in}} \rVert_{L^{p}} + C(p, Tr((-\Delta)^{\frac{d}{2} + 2 \varsigma} GG^{\ast} ))^{\frac{1}{p}}] & \text{ if } p \in (2,\infty) 
\end{cases}
\end{equation} 
(see \eqref{est 53}). Let us now present our main result in the case of an additive noise.
\begin{theorem}\label{Theorem 2.2} 
(Additive noise) Suppose that $d \geq 3$, $B$ is a $GG^{\ast}$-Wiener process, and  
\begin{equation}\label{est 31}
Tr ((-\Delta)^{\frac{d}{2} + 2 \varsigma} GG^{\ast}) < \infty \hspace{1mm} \text{ for some } \hspace{1mm} \varsigma > 0. 
\end{equation} 
Given $T > 0$, $K > 1$, and $\kappa \in (0,1)$, there exists a $\mathbb{P}$-a.s. strictly positive stopping time $T_{L}$ such that 
\begin{equation}\label{est 140}
\mathbb{P} ( \{ T_{L} \geq T \}) > \kappa 
\end{equation}
and the following is additionally satisfied. Let $p \in (1,\infty), \tilde{p} \in [1,\infty)$ such that \eqref{est 8} holds. Then there exist an $(\mathcal{F}_{t})_{t \geq 0}$-adapted process $u$ that is divergence-free such that 
\begin{equation}\label{est 393}
u \in L^{\infty} (\Omega; C([0,T_{L}]; L^{p'} (\mathbb{T}^{d}))) \cap L^{\infty} (\Omega; C([0,T_{L}]; W^{1, \tilde{p}}(\mathbb{T}^{d}))), 
\end{equation} 
an $(\mathcal{F}_{t})_{t\geq 0}$-adapted process 
\begin{equation}\label{est 363}
\rho \in L^{\infty} (\Omega; C([0,T_{L}]; L^{p} (\mathbb{T}^{d}))),  
\end{equation}
and $\rho^{\text{in}} \in L^{p}(\mathbb{T}^{d})$ that is deterministic such that $\rho$ solves the corresponding \eqref{stochastic transport} forced by additive noise $B$ as follows: for every test function $\varphi \in C_{0}^{\infty} (\mathbb{T}^{d})$, the process $\int_{\mathbb{T}^{d}} \rho(t,x) \varphi(x) dx$ has a continuous modification which is a $(\mathcal{F}_{t})_{t\geq 0}$-semimartingale and satisfies 
\begin{align}
\int_{\mathbb{T}^d} \rho(t,x)\,\varphi(x)\,dx &= \int_{\mathbb{T}^d} \rho^{\text{in}}  (x)\,\varphi(x)\,dx + \int_0^t \int_{\mathbb{T}^d} \rho(s,x)\, u(s,x)\, \cdot \nabla \varphi(x)\,dx\, ds \nonumber \\
&\quad + \int_0^t \int_{\mathbb{T}^d} \rho(s,x)\,\Delta \varphi(x)\,dx\,ds + \int_{\mathbb{T}^d} \int_0^t \varphi(x)\,dB(s,x)\,dx \label{solution add}
\end{align}
$\mathbb{P}$-a.s. and for all $t \in [0, T_{L}]$. Moreover, on the set $\{T_{L} \geq T\}$, 
\begin{equation}\label{est 53} 
\lVert \rho(T) \rVert_{L_{x}^{p}} > 
\begin{cases}
K[\lVert \rho^{\text{in}} \rVert_{L^{p}} + L^{\frac{1}{4}} ( \int_{0}^{T} \lVert u \rVert_{L_{x}^{p}} ds + 1)] & \text{ if } p \in (1,2), \\
K[\lVert \rho^{\text{in}} \rVert_{L^{2}} + \sqrt{T Tr (GG^{\ast})}] & \text{ if } p = 2, \\
Ke^{\frac{T}{p}} [ \lVert \rho^{\text{in}} \rVert_{L^{p}} + C(p, Tr((-\Delta)^{\frac{d}{2} + 2 \varsigma} GG^{\ast} ))^{\frac{1}{p}}] & \text{ if } p \in (2,\infty). 
\end{cases}
\end{equation} 
\end{theorem}

\begin{remark}\label{Remark in additive case}
Theorem \ref{Theorem 2.2} is an analogue of \cite[Theorem 1.1]{HZZ19} and already represents non-uniqueness in law for solutions defined on the random time interval $[0,T_{L}]$ (cf. \eqref{est 335}, \eqref{est 228}, \eqref{est 33}, \eqref{est 336}, and \eqref{est 53}). On the other hand, \cite[Theorems 1.2]{HZZ19} is an extension that shows non-uniqueness in law over a deterministic time interval similarly to the last sentence of our Theorems \ref{Theorem 2.1}, \ref{Theorem 2.4}-\ref{Theorem 2.5}. Such an extension from Theorem \ref{Theorem 2.2} following the proof of \cite[Theorem 1.2]{HZZ19} would require proving the existence of a solution to a martingale problem in the spirit of \cite{SV97} (see \cite[Definitions 3.1-3.2, Theorem 3.1]{HZZ19}) and thereafter its various properties following \cite[Section 3]{HZZ19}. A systematic approach for the former task is given in \cite[Theorem 4.6]{GRZ09}; however, it seems to be only for initial data in a Hilbert space limiting this possibility only to the special case $p = 2$. Other differences from \cite[Theorem 1.1]{HZZ19} and Theorem \ref{Theorem 2.2} include the fact that the solution to the 3D Navier-Stokes equations constructed via convex integration actually has Sobolev regularity $H^{\gamma}(\mathbb{T}^{3})$, although for $\gamma > 0$ arbitrarily small, and it is part of the definition of a solution (see ``(M3)'' in \cite[Definitions 3.1-3.2]{HZZ19}) while $\rho$ in Theorem \ref{Theorem 2.2} has no such higher regularity, unless we work with arbitrarily strong diffusion (see \cite[Theorem 1.4]{MS20}). This creates major obstacle in applying \cite[Theorem 4.6]{GRZ09}.

One last possible approach to obtain non-uniqueness over a deterministic time interval in the case of an additive noise may be to follow the approach of \cite{MS18} in which \cite[Proposition 3.1]{MS18} is similar in spirit to \cite[Proposition 16]{BMS21}; the advantage here is that the proof of non-uniqueness does not rely on the strategy of $R_{0} (t) \equiv 0$ implying $R_{1}(t)\equiv 0$ that we described in Remark \ref{Remark 2.3} and thus will not be affected by the fact that we will have to mollify $R_{0}$. Hence, one idea will be to assign $z\rvert_{t=0} = \rho^{\text{in}}$ so that the solution to the convex integration scheme can have zero initial data similarly to \cite{HZZ21a}, utilize Mikado density and Mikado field from \cite{MS20} to obtain the desired H$\ddot{o}$lder relation \eqref{est 7} rather than \eqref{est 1}, obtain key iteration estimates similarly to \cite[Proposition 3.1]{MS18} and \cite[Proposition 16]{BMS21} up to a stopping time $T_{L}$ in hope to take the final value at this stopping time and repeat following the approach of \cite[Theorem 1.1, Corollary 1.2]{HZZ21a}. Alas, even if this works, as stated in \cite[Theorem 1.2 (c)]{MS18}, for any prescribed function $\bar{\rho}$, we can deduce a pair $(\rho, u)$ that satisfies the transport-diffusion equation forced by an additive noise such that $\rho\rvert_{t=0} = \bar{\rho}\rvert_{t=0}$ and $\rho\rvert_{t=T_{L}} = \bar{\rho}\rvert_{t=T_{L}}$. In the deterministic case, Modena and Sz$\acute{\mathrm{e}}$kelyhidi Jr. \cite{MS18} can conclude now by taking $\bar{\rho}\rvert_{t=0} \equiv 0$ and $\bar{\rho}\rvert_{t=T} \not\equiv 0$ because for any $u$, starting from $\rho\rvert_{t=0} \equiv 0$, $\rho \equiv 0$ is a solution for them; however, this is not the case when forced by an additive noise. At the time of writing this manuscript, it is not clear to us how to prove non-uniqueness in law of the stochastic transport-diffusion equation forced by additive noise over an arbitrary deterministic time interval unconditionally; nonetheless, our next result Theorem \ref{Theorem 2.3} actually shows that we have been able to overcome the aforementioned difficulty, although the temporal continuity and the $(\mathcal{F}_{t})_{t\geq 0}$-adaptedness of the vector field $u$ have been lost.  
\end{remark}

\begin{theorem}\label{Theorem 2.3}
(Additive noise) Suppose that $d \geq 3$, $B$ is a $GG^{\ast}$-Wiener process, and \eqref{est 31} holds. For any $T > 0, K > 1$, and $\kappa \in (0,1)$, there exists a $\mathbb{P}$-a.s. strictly positive stopping time $T_{L}$ such that \eqref{est 140} holds and the following is additionally satisfied. Let $p \in (1,\infty), \tilde{p} \in [1,\infty)$ such that \eqref{est 8} holds. Then there exist a process $u$ that is divergence-free such that 
\begin{equation}
u \in L^{\infty} (\Omega; L^{\infty} ([0, T]; L^{p'} (\mathbb{T}^{d}))) \cap L^{\infty} (\Omega; L^{\infty} ([0,T]; W^{1,\tilde{p}}(\mathbb{T}^{d}))), 
\end{equation} 
an $(\mathcal{F}_{t})_{t\geq 0}$-adapted process 
\begin{equation}
\rho \in C([0, T]; L^{r} (\Omega; L^{p} (\mathbb{T}^{d})))
\end{equation} 
for all $r \in [1,\infty)$ such that $\rho$ solves the corresponding \eqref{stochastic transport} forced by the additive noise $B$; i.e., for every test function $\varphi \in C_{0}^{\infty} (\mathbb{T}^{d})$, the process $\int_{\mathbb{T}^{d}} \rho(t,x) \varphi(x) dx$ has a continuous modification which is a $(\mathcal{F}_{t})_{t\geq 0}$-semimartingale and satisfies \eqref{solution add} $\mathbb{P}$-a.s. and for all $t \in [0,T]$. Moreover, on the set $\{T_{L} \geq T\}$ \eqref{est 53} holds. Consequently, non-uniqueness in law holds for \eqref{stochastic transport} forced by the additive noise $B$ on $[0,\infty)$; moreover, for all $T> 0$, non-uniqueness in law holds on $[0,T]$. 
\end{theorem}

At last, let us discuss the case of a transport noise in Stratonovich's interpretation: $G(\rho) dB = - \sum_{i=1}^{d} \frac{\partial}{\partial x_{i}} \rho \circ dB_{i}$. \begin{remark}
For this case of a transport noise in Stratonovich's interpretation, there is an advantage that a zero function is a solution, in contrast to the case of additive noise. Nonetheless, there is a familiar difficulty and a new difficulty. A familiar difficulty, which is similar to the presence of $z$ in \eqref{est 29} as described in Remark \ref{Remark 2.3}, is that due to the presence of $B(t)$ in \eqref{est 6}, upon employing a Nash-type convex integration scheme similarly to \eqref{est 29}, $R(t,x)$ will not be differentiable in time. Therefore, the same difficulty explained in Remark \ref{Remark 2.3} applies and we must mollify $R$. A new difficulty is that while the case of transport noise in Stratonovich's interpretation does not have an external force ``$-\text{div} (u(t,x) z(t,x))$'' on the r.h.s. of \eqref{est 6}, we now have a ``mismatch'' of variables within the nonlinear term, specifically $u(t,x+B(t)) \cdot \nabla \theta(t,x)$. As we will describe subsequently, this creates a new difficulty in the convex integration scheme (see Remark \ref{Remark on difficulty of transport}).
\end{remark}
Existence of an analytically weak solution to \eqref{stochastic transport} forced by a transport noise in Stratonovich's interpretation with $u \in C_{t}L_{x}^{p'}$ such that $\nabla\cdot u = 0$, starting from $\rho^{\text{in}} \in L_{x}^{p}$ for $p \in (1,\infty)$, has been proven in variations; e.g., \cite[Lemma 2.2]{CO13} in case of zero diffusion (also \cite[Theorem 15]{FGP10} in case $p = \infty$). In short, we can consider the random PDE \eqref{est 6} in which $u(t, x+ B(t)) \in L_{t}^{1}L_{x}^{p'}$ due to $u \in L_{t}^{1}L_{x}^{p'}$, apply Lemma \ref{Lemma 8.1} (or its slight variation by considering test functions $\phi \in C_{x}^{\infty}$) and apply It$\hat{\mathrm{o}}$-Wentzell-Kunita formula (e.g., \cite[Theorem 3.3.2 on p. 93]{K90}, also \cite{K82, K84}) to conclude (see the proofs of \cite[Lemma 2.2]{CO13} and \cite[Theorem 15]{FGP10}). Concerning the $L^{p}$-inequality, we see from \eqref{est 11} that for all $p \in (1,\infty)$, 
\begin{align}
\left( \int_{\mathbb{T}^{d}} \lvert \rho(t,x) \rvert^{p} dx \right)^{\frac{1}{p}} = \lVert \theta(t) \rVert_{L_{x}^{p}} \overset{\eqref{est 11}}{\leq} \lVert \theta^{\text{in}} \rVert_{L_{x}^{p}} = \left( \int_{\mathbb{T}^{d}} \lvert \rho(0, x + B(0)) \rvert^{p} dx \right)^{\frac{1}{p}} = \lVert \rho^{\text{in}} \rVert_{L_{x}^{p}}.
\end{align} 
Therefore we aim to construct solutions such that $\lVert \rho(T) \rVert_{L^{p}} > K \lVert \rho^{\text{in}} \rVert_{L^{p}}$ for $K > 1$. Because $B(t)$ only appears ``within'' the vector field $u(t,x+B(t))$, we are able to obtain the following result over any prescribed deterministic interval $[0,T]$. 

\begin{theorem}\label{Theorem 2.4} 
(Transport noise in Stratonovich's interpretation) Suppose that $d \geq 3$ and $B(t) = (B_{1}, \hdots, B_{d})(t)$ is a standard Brownian motion on $(\Omega, \mathcal{F}, \mathbb{P})$. Let $T > 0$, $K > 1$, $p \in (1,\infty), \tilde{p} \in [1,\infty)$ such that \eqref{est 8} holds. Then there exist an $(\mathcal{F}_{t})_{t \geq 0}$-adapted process $u$ that is divergence-free such that 
\begin{equation}
u \in L^{\infty} (\Omega; C([0, T]; L^{p'} (\mathbb{T}^{d}))) \cap  L^{\infty} (\Omega; C([0, T]; W^{1, \tilde{p}} (\mathbb{T}^{d}))),
\end{equation} 
an $(\mathcal{F}_{t})_{t\geq 0}$-adapted process 
\begin{equation}
\rho \in L^{\infty} (\Omega; C([0, T]; L^{p} (\mathbb{T}^{d}))),
\end{equation}
and $\rho^{\text{in}} \in L^{p}(\mathbb{T}^{d})$ that is deterministic such that $\rho$ solves the corresponding \eqref{stochastic transport} forced by the transport noise in Stratonovich's interpretation as follows: for every test function $\varphi \in C_{0}^{\infty} (\mathbb{T}^{d})$, the process $\int_{\mathbb{T}^{d}} \rho(t,x) \varphi(x) dx$ has a continuous modification which is a $(\mathcal{F}_{t})_{t\geq 0}$-semimartingale and satisfies
\begin{align}
\int_{\mathbb{T}^d} \rho(t,x)\,\varphi(x)\,dx &= \int_{\mathbb{T}^d} \rho^{\text{in}}  (x)\,\varphi(x)\,dx + \int_0^t \int_{\mathbb{T}^d} \rho(s,x)\, u(s,x)\, \cdot \nabla \varphi(x)\,dx\, ds \nonumber \\
&\quad + \int_0^t \int_{\mathbb{T}^d} \rho(s,x)\,\Delta \varphi(x)\,dx\,ds - \sum_{i=1}^{d}\int_0^t \int_{\mathbb{T}^d}  \varphi(x) \frac{\partial}{\partial x_{i}} \rho(s,x)\,dx\,\circ dB_{i}(s,x) \label{solution transport}
\end{align}
$\mathbb{P}$-a.s. and for all $t \in [0,T]$. Moreover, $\rho$ satisfies 
\begin{equation}\label{est 355}
\lVert \rho(T) \rVert_{L_{x}^{p}} > K \lVert \rho^{\text{in}} \rVert_{L^{p}}.
\end{equation}
Consequently, non-uniqueness in law holds for \eqref{stochastic transport} forced by the transport noise in Stratonovich's interpretation on $[0,\infty)$; moreover, for all $T> 0$, non-uniqueness in law holds on $[0,T]$. 
\end{theorem}

We can also prove non-uniqueness in law for the case of a transport noise as follows:
\begin{theorem}\label{Theorem 2.5}
(Transport noise in Stratonovich's interpretation) Suppose that $d \geq 3$ and $B(t) = (B_{1}, \hdots, B_{d})(t)$ is a standard Brownian motion on $(\Omega, \mathcal{F}, \mathbb{P})$. Let $T > 0$, $\xi \in (0,T)$, $p \in (1,\infty), \tilde{p} \in [1,\infty)$ such that \eqref{est 8} holds, and $\upsilon \in (1,p)$. Then there exist infinitely many pairs $(\rho, u)$ such that $u \rvert_{t=0} \equiv 0, \rho \rvert_{t=0} \equiv 0$ $\mathbb{P}$-a.s., 
\begin{subequations}
\begin{align}
& u\in L^{\infty} (\Omega; C((0, T]; L^{p'} (\mathbb{T}^{d}))) \cap L^{\infty} (\Omega; C([0, T]; W^{1,\tilde{p}}(\mathbb{T}^{d}))), \\
& \rho \in L^{\infty} (\Omega; C((0,T]; L^{p} (\mathbb{T}^{d}))) \cap L^{\infty} (\Omega; C([0,T]; L^{\upsilon} (\mathbb{T}^{d}))), 
\end{align}
\end{subequations} 
both $\rho$ and $u$ are $(\mathcal{F}_{t})_{t\geq 0}$-adapted, $\rho$ satisfies the corresponding \eqref{stochastic transport} forced by transport noise in Stratonovich's interpretation as follows: for every test function $\varphi \in C_{0}^{\infty} (\mathbb{T}^{d})$, the process $\int_{\mathbb{T}^{d}} \rho(t,x) \varphi(x) dx$ has a continuous modification which is a $(\mathcal{F}_{t})_{t\geq 0}$-semimartingale and satisfies \eqref{solution transport} $\mathbb{P}$-a.s. and for all $t \in [0,T]$. Moreover, for all the pairs $(\rho, u)$ and all $t \in (\xi,T]$, 
\begin{equation}\label{est 236} 
\int_{\mathbb{T}^{d}} \rho(t, x) u(t,x) dx  
\end{equation}  
are distinct $\mathbb{P}$-a.s. Consequently, non-uniqueness in law holds for \eqref{stochastic transport} forced by the transport noise in Stratonovich's interpretation on $[0,\infty)$; moreover, for all $T > 0$, non-uniqueness in law holds on $[0,T]$. 
\end{theorem}

\begin{remark}\label{Remark in transport case}
Theorems \ref{Theorem 2.4} and \ref{Theorem 2.5} are both concerned with the case of a transport noise in Stratonovich's interpretation; nonetheless, they  present interesting differences in terms of both results and proofs. Theorem \ref{Theorem 2.4} states that there exist some deterministic initial condition $\rho^{\text{in}} \in L^{p}(\mathbb{T}^{d})$ and a relatively smooth $u$ such that solution $\rho$ emanating from it violates the classical $L^{p}(\mathbb{T}^{d})$-inequality. On the other hand, Theorem \ref{Theorem 2.5} shows that we may prescribe zero initial data $\rho^{\text{in}}$ and construct $\rho$ that is non-zero regardless of any information about $u^{\text{in}} \in L^{p}(\mathbb{T}^{d})$ except that initially it is zero. The proof of Theorem \ref{Theorem 2.5} follows the idea from \cite{BMS21} and \cite{HZZ21a} and this approach also presented multiple difficulties, which we will describe within its proof (see Remarks \ref{Difficulty 1}-\ref{Difficulty 3}). We point out that the non-uniqueness is derived from the distinct values of \eqref{est 236}; this is in sharp contrast from analogous choice of energy $L^{2}(\mathbb{T}^{d})$-norm of the solution in the case of the Navier-Stokes equations, that corresponds to $L^{p}(\mathbb{T}^{d})$-norm of $\rho$ in the case of the transport-diffusion equation. We discovered this difference due to the structure of the transport-diffusion equation and believe that it is inevitable. 

Lastly, we point out that the proof of Theorem \ref{Theorem 2.5} also fails in the case of an additive noise because \eqref{est 236} implies non-uniqueness only by taking zero initial data and relying on the fact that $\rho \equiv 0$ is a solution regardless of $u$ in the case of transport noise, which is not valid in the case of additive noise. 
\end{remark}

Finally, as a corollary of our proofs of Theorems \ref{Theorem 2.2} and \ref{Theorem 2.4}, we are able to also prove non-uniqueness for deterministic transport-diffusion equation forced by non-zero external force $f \in L^{1}(0, T; L^{p}(\mathbb{T}^{d}))$ that is mean-zero. To the best of our knowledge this is new because all previous works \cite{MS18, MS19, MS20} ultimately proved non-uniqueness by taking advantage of the fact that starting from zero initial data $\rho^{\text{in}}$, $\rho \equiv 0$ is a solution regardless of $u$ to conclude the proof of non-uniqueness, and that breaks down once we add a non-zero external force.

First, for all $f \in L^{1}(0, T; L^{p}(\mathbb{T}^{d}))$ that is mean-zero for all $t \in [0,T]$, $u \in C([0,T]; L^{p'}(\mathbb{T}^{d}))$ such that $\nabla\cdot u = 0$ where $p \in (1,\infty)$, and $\rho^{\text{in}} \in L^{p}(\mathbb{T}^{d})$, by Lemma \ref{Lemma 8.1} there exists an analytically weak solution $\rho \in L^{\infty} (0,T; L^{p}(\mathbb{T}^{d}))$ to \eqref{transport} with diffusion and the external force 
\begin{subequations}\label{est 327}
\begin{align}
&\partial_{t} \rho(t,x) + \text{div} (u(t,x) \rho(t,x)) - \Delta \rho(t,x) = f(t,x), \hspace{3mm} t > 0, \label{est 329}\\
& \rho(0,x) = \rho^{\text{in}}(x), 
\end{align}
\end{subequations} 
and from \eqref{est 11} we see that it satisfies 
\begin{equation}
\lVert \rho(t) \rVert_{L_{x}^{p}} \leq \lVert \rho^{\text{in}} \rVert_{L^{p}} + \int_{0}^{t} \lVert f(s) \rVert_{L_{x}^{p}} ds \hspace{3mm} \forall \hspace{1mm} t \in [0,T]. 
\end{equation}   
\begin{corollary}\label{Corollary 2.6}
(Deterministic force) Suppose that $d \geq 3$. Let $T> 0, K > 1$, $p \in (1,\infty)$, $\tilde{p} \in [1,\infty)$ such that \eqref{est 8} holds, and $f \in L^{1}(0, T; L^{p}(\mathbb{T}^{d}))$ that is mean-zero for all $t \in [0,T]$. Then there exist $u \in C([0,T]; L^{p'} (\mathbb{T}^{d})) \cap C([0,T]; W^{1,\tilde{p}}(\mathbb{T}^{d}))$ that is divergence-free and $\rho \in C([0,T]; L^{p}(\mathbb{T}^{d}))$ that satisfy \eqref{est 327} analytically weakly and 
\begin{equation}\label{est 331}
\lVert \rho(T) \rVert_{L_{x}^{p}} > K [\lVert \rho^{\text{in}} \rVert_{L^{p}} + \int_{0}^{T} \lVert f (s) \rVert_{L_{x}^{p}} ds]. 
\end{equation} 
Consequently, for all $T > 0$, non-uniqueness holds for \eqref{est 327} on $[0,T]$.
\end{corollary}
We leave this proof in the Appendix A for completeness. 

\begin{remark}
Various extensions of our results may be possible. First, Theorems \ref{Theorem 2.1}-\ref{Theorem 2.5} may be extended to the cases of zero or arbitrarily strong diffusion by appropriate modifications described in \cite[Theorems 1.2 and 1.4]{MS20}, respectively. In particular, the case of zero diffusion may include the case $d = 2$ (see \cite[Remark on p. 1080]{MS20}). Improving from $u \in L_{x}^{\infty}$ to $u \in C_{x}$ in case $p = 1$ may also be possible (see \cite[Section 7.1]{MS20} and \cite{MS19}). The method in our work may be applied to \cite{CL21} to achieve its probabilistic analogue as well. 

To the best of our knowledge, convex integration technique has never been employed on stochastic transport equation. Via convex integration, we have proved the negative direction when forced by random noise of additive and linear multiplicative types similarly to previous works (e.g., \cite{HZZ19}); despite \cite{FGP10}, transport noise is no exception here. We note that the condition that the authors in \cite{FGP10} had on the vector field $u$ was stronger than ours. Our work also makes a contribution to the research direction of probabilistic convex integration by providing another equation, as well as another type of noise, specifically transport, on which we can apply such a technique. 
\end{remark}

In the following, we shall give some preliminaries and thereafter prove Theorems \ref{Theorem 2.2}-\ref{Theorem 2.3} and \ref{Theorem 2.5}.  The proof of Theorem \ref{Theorem 2.1} is similar to the deterministic case while that of Theorem \ref{Theorem 2.4} follows from similar computations in the proof of Theorem \ref{Theorem 2.2}; hence, along with the proof of Corollary \ref{Corollary 2.6}, they are left in the Appendix A. 

\section{Preliminaries}\label{Preliminaries}
\subsection{Notations and assumptions}\label{Subsection 3.1}
We write $A \lesssim_{a,b} B$ and $A \approx_{a,b} B$ to imply that $A \leq C_{a,b}^{1} B$ and $C_{a,b}^{1} B \leq A \leq C_{a,b}^{2}B$ for some constants $C_{a,b}^{j} = C^{j} (a,b)\geq 0$ for $j \in \{1,2\}$, respectively. We also write $A \overset{(\cdot)}{\lesssim}B$ to indicate that this inequality is due to an equation $(\cdot)$. We define $\mathbb{N} \triangleq \{ 1, 2, \hdots, \}$ while $\mathbb{N}_{0} \triangleq \{0\} \cup \mathbb{N}$. We write $\lVert f \rVert_{C_{t,x}^{N}} \triangleq \sum_{0\leq n + \lvert \alpha \rvert \leq N} \lVert \partial_{t}^{n} D^{\alpha} f \rVert_{L_{t,x}^{\infty}}$ and $\partial_{j} \triangleq \frac{\partial}{\partial x_{j}}$ for $j \in \{1,\hdots, d\}$. We denote a mathematical expectation with respect to (w.r.t.) any probability measure $P$ by $\mathbb{E}^{P}$. 
 
\subsection{Convex integration}\label{Subsection 3.2}
For any $f: \mathbb{T}^{d} \mapsto \mathbb{R}$ and $\lambda \in \mathbb{N}$, we denote its dilation  
\begin{equation}\label{est 66}
f_{\lambda}: \mathbb{T}^{d} \mapsto \mathbb{R}, \hspace{3mm} f_{\lambda}(x) \triangleq f(\lambda x) \hspace{1mm} \text{ so that }\hspace{1mm}  \lVert D^{k} f_{\lambda} \rVert_{L^{p}} = \lambda^{k} \lVert D^{k} f \rVert_{L^{p}} 
\end{equation} 
where $D^{k}$ represents some multi-index $\alpha$ such that $\lvert \alpha \rvert = k \in \mathbb{N}_{0}, p \in [1,\infty]$ (e.g., \cite[Equation (1.20)]{MS20}). 
\begin{lemma}
\rm{(\cite[Lemma 2.1]{MS18})} Let $\lambda \in \mathbb{N}_{0}$ and $f, g: \mathbb{T}^{d} \mapsto \mathbb{R}$ be smooth functions. Then, for every $p \in [1,\infty]$, there exists a constant $C_{p} \geq 0$ such that 
\begin{equation}\label{est 80}
\lvert \lVert f g_{\lambda} \rVert_{L^{p}} - \lVert f \rVert_{L^{p}} \lVert g \rVert_{L^{p}} \rvert \leq C_{p} \lambda^{-\frac{1}{p}} \lVert f \rVert_{C^{1}} \lVert g \rVert_{L^{p}}.
\end{equation} 
\end{lemma} 

\begin{lemma}
\rm{(\cite[Proposition 2]{BMS21})} Let $a \in C^{\infty} (\mathbb{T}^{d}), v \in C_{0}^{\infty} (\mathbb{T}^{d})$. Then, for any $p \in [1,\infty]$, there exists a constant $C_{p} \geq 0$ such that 
\begin{equation}\label{est 304}
\lvert \int_{\mathbb{T}^{d}} a v_{\lambda} dx \rvert \leq \lambda^{-1} C_{p} \lVert \nabla a \rVert_{L^{p}} \lVert v \rVert_{L^{p'}}. 
\end{equation} 
\end{lemma}

\begin{define}\label{Definition 3.1} 
\rm{(\cite[Definition on p. 1083]{MS20})} For any $f \in C^{\infty}( \mathbb{T}^{d})$, $k \in \mathbb{N}_{0}$, we define 
\begin{equation}\label{est 88}
\mathcal{D}^{k} f \triangleq 
\begin{cases}
\Delta^{\frac{k}{2}} f & \text{ if } k \text{ is even}, \\
\nabla \Delta^{\frac{k-1}{2}} f & \text{ if } k \text{ is odd}, 
\end{cases} 
\end{equation} 
with the convention that $\mathcal{D}^{0} = \Delta^{0} = Id$. For $k \in \mathbb{Z}_{< 0}$ the definition is identical under an additional hypothesis that $f \in C_{0}^{\infty} (\mathbb{T}^{d})$. We refer to $\mathcal{D}^{-1}$ as the anti-divergence operator. 
\end{define}

\begin{lemma}
\rm{ (\cite[p. 1084]{MS20})} $\mathcal{D}^{k}$ defined in Definition \ref{Definition 3.1} has the following properties. 
\begin{enumerate}
\item $\partial^{\alpha} \mathcal{D}^{k} f = \mathcal{D}^{k} \partial^{\alpha} f$ for all $k \in \mathbb{Z}$ and any multi-index $\alpha$. 
\item For any $k, n, m \in \mathbb{Z}$ and $f, g \in C_{0}^{\infty} (\mathbb{T}^{d})$, 
\begin{equation}
\int_{\mathbb{T}^{d}} \mathcal{D}^{k} f \cdot \mathcal{D}^{m+n} g dx = (-1)^{n} \int_{\mathbb{T}^{d}} \mathcal{D}^{k+n} f \cdot \mathcal{D}^{m} g dx. 
\end{equation} 
\item $\mathcal{D}^{k} u_{\lambda} = \lambda^{k} (\mathcal{D}^{k} u)_{\lambda}$ for any $k \in \mathbb{Z}$ and $\lambda \in \mathbb{N}_{0}$. 
\end{enumerate} 
\end{lemma} 
\begin{lemma}
\rm{ (\cite[Lemma 3.2]{MS20})} Let $p \in (1,\infty)$. Then there exists a constant $C_{d,p} \geq 0$ such that for any $f \in C_{0}^{\infty} (\mathbb{T}^{d})$, 
\begin{equation}
\lVert f \rVert_{W^{2,p}} \leq C_{d,p} \lVert \Delta f \rVert_{L^{p}}.
\end{equation} 
\end{lemma} 

\begin{lemma}
\rm{( \cite[Lemma 3.3]{MS20})} Let $p \in (1,\infty)$ and $k \in \mathbb{N}_{0}$. Then there exists a constant $C_{d,p,k} \geq 0$ such that for any $f \in C_{0}^{\infty} (\mathbb{T}^{d})$, 
\begin{equation}
\lVert \mathcal{D}^{-k} f \rVert_{W^{k,p}} \leq C_{d,p,k} \lVert f \rVert_{L^{p}}. 
\end{equation} 
\end{lemma}

\begin{lemma}
\rm{( \cite[Lemma 3.4]{MS20})} Let $p \in [1,\infty]$ and $k \in \mathbb{N}$. Then there exists a constant $C_{d,p,k} \geq 0$ such that for any $f \in C_{0}^{\infty} (\mathbb{T}^{d})$, 
\begin{equation}\label{est 121}
\lVert \mathcal{D}^{-k} f \rVert_{W^{k-1, p} } \leq C_{d,p,k} \lVert f \rVert_{L^{p}}. 
\end{equation} 
\end{lemma} 

\begin{define}\label{Definition 3.2}
\rm{( \cite[Definition on p. 1086]{MS20})}  (Bilinear anti-divergence operator) Let $N \in \mathbb{N}_{0}$. Define $\mathcal{R}_{N}: C^{\infty} (\mathbb{T}^{d}) \times C_{0}^{\infty} (\mathbb{T}^{d}) \mapsto C^{\infty} (\mathbb{T}^{d}; \mathbb{R}^{d})$ by 
\begin{equation}
\mathcal{R}_{N} (f,g) \triangleq \sum_{k=0}^{N-1} (-1)^{k} \mathcal{D}^{k} f \mathcal{D}^{-k-1} g + \mathcal{D}^{-1} \left( (-1)^{N} \mathcal{D}^{N} f \cdot \mathcal{D}^{-N} g - \fint_{\mathbb{T}^{d}} fg dx \right).
\end{equation} 
\end{define} 

\begin{lemma}
\rm{ (\cite[Lemma 3.5 and Remark on p. 1087]{MS20} )} 
\begin{enumerate}
\item Let $N \in \mathbb{N}_{0}, f \in C^{\infty} (\mathbb{T}^{d})$, and $g \in C_{0}^{\infty} (\mathbb{T}^{d})$. Then 
\begin{equation}\label{est 74}
\text{div} (\mathcal{R}_{N} (f,g)) = fg - \fint_{\mathbb{T}^{d}} fg dx.
\end{equation}  
\item Let $N \in \mathbb{N}_{0}, f \in C^{\infty} (\mathbb{T}^{d})$, and $g \in C_{0}^{\infty} (\mathbb{T}^{d})$. Then for all $j \in \{1,\hdots, d\}$, 
\begin{equation}\label{est 96}
\partial_{j} (\mathcal{R}_{N} (f,g)) = \mathcal{R}_{N} (\partial_{j} f, g) + \mathcal{R}_{N} (f, \partial_{j} g). 
\end{equation} 
\item Let $N \in \mathbb{N}_{0}$. Then, for any $p, r, s \in [1,\infty]$ such that $\frac{1}{p} = \frac{1}{r} + \frac{1}{s}$, there exists a constant $C_{d,p} \geq 0$ such that for any $f \in C^{\infty} (\mathbb{T}^{d})$ and $g \in C_{0}^{\infty} (\mathbb{T}^{d})$,  
\begin{equation}
\lVert \mathcal{R}_{N} (f,g) \rVert_{L^{p}} \leq \sum_{k=0}^{N-1} \lVert \mathcal{D}^{k} f \rVert_{L^{r}} \lVert \mathcal{D}^{-k-1} g \rVert_{L^{s}} + C_{d,p} \lVert \mathcal{D}^{N} f \rVert_{L^{r}} \lVert \mathcal{D}^{-N} g \rVert_{L^{s}}. 
\end{equation} 
\item Let $p \in [1,\infty]$ and $\lambda, N \in \mathbb{N}_{0}$. Then there exists a constant $C_{d,p,N} \geq 0$ such that for all $f \in C^{\infty} (\mathbb{T}^{d})$ and $g \in C_{0}^{\infty} (\mathbb{T}^{d})$, 
\begin{subequations}\label{est 283}
\begin{align}
\lVert \mathcal{R}_{N} (f,g_{\lambda}) \rVert_{L^{p}} \leq& C_{d,p,N} \lVert g \rVert_{L^{p}} \left( \sum_{k=0}^{N-1} \lambda^{-k-1} \lVert \mathcal{D}^{k} f \rVert_{L^{\infty}} + \lambda^{-N} \lVert \mathcal{D}^{N} f \rVert_{L^{\infty}} \right), \label{est 89} \\
\lVert \mathcal{R}_{N} (f, g_{\lambda}) \rVert_{L^{p}} \leq& C_{d,p,N} \lVert g \rVert_{L^{\infty}} \left( \sum_{k=0}^{N-1} \lambda^{-k-1} \lVert \mathcal{D}^{k} f \rVert_{L^{p}} + \lambda^{-N} \lVert \mathcal{D}^{N} f \rVert_{L^{p}} \right). \label{est 90}
\end{align}
\end{subequations}
\end{enumerate}
\end{lemma} 

The following are preliminaries on space-time Mikado densities and fields from \cite[Section 4.1]{MS20}. Let $e_{j}$ denote the $j$-th element in the standard basis of $\mathbb{R}^{d}$. For given $\zeta, v \in \mathbb{T}^{d}$, we consider the line on $\mathbb{T}^{d}: s \mapsto \zeta + s v \in \mathbb{T}^{d}$ for any $s \in \mathbb{R}$. 
\begin{lemma}\label{Lemma 3.8} 
\rm{ (Space-time Mikado lines \cite[Lemma 4.1]{MS20})} Let $d_{\mathbb{T}^{d}}$ denote the Euclidean distance on the torus. There exist $r > 0$ and $\zeta_{1}, \hdots, \zeta_{d} \in \mathbb{T}^{d}$ such that the lines 
\begin{equation}\label{est 193} 
\mathfrak{x}_{j}: \mathbb{R} \mapsto \mathbb{T}^{d}, \mathfrak{x}_{j}(s) \triangleq \zeta_{j} + s e_{j} 
\end{equation} 
satisfy 
\begin{equation}\label{est 194}
d_{\mathbb{T}^{d}} (\mathfrak{x}_{i}(s), \mathfrak{x}_{j}(s)) > 2r \hspace{4mm} \forall \hspace{1mm} s \in \mathbb{R}, \forall \hspace{1mm} i \neq j. 
\end{equation}  
\end{lemma}
For $r > 0$ from Lemma \ref{Lemma 3.8}, we let $\varrho$ be a smooth function on $\mathbb{R}^{d}$ such that 
\begin{equation}\label{est 56} 
supp\varrho \subset B(P, r) \subset (0,1)^{d} \hspace{2mm} \text{ where } P \triangleq (\frac{1}{2}, \hdots, \frac{1}{2}) \in (0,1)^{d} \text{ and } \int_{\mathbb{R}^{d}} \varrho^{2} dx = 1. 
\end{equation} 
For any $p \in (1,\infty)$, we define 
\begin{equation}\label{est 55}
a \triangleq \frac{d}{p} \hspace{2mm} \text{ and }\hspace{2mm}  b \triangleq \frac{d}{p'} \hspace{1mm} \text{ so that } a + b = d
\end{equation}
(recall \eqref{est 235}). Then we define the scaled functions 
\begin{equation}
\varrho_{\mu}(x) \triangleq \mu^{a} \varrho (\mu x), \hspace{2mm} \tilde{\varrho}_{\mu} (x) \triangleq \mu^{b} \varrho(\mu x), \hspace{2mm} \mu \geq 1. 
\end{equation} 
Concerning these scaled functions, we have the following result: 
\begin{lemma}\label{Lemma 3.9} 
\rm{ (\cite[Lemma 4.2 and p. 1088]{MS20})} For every $\mu \geq 1, k \in \mathbb{N}$, and $r \in [1,\infty]$, 
\begin{subequations}
\begin{align}
&\lVert D^{k} \varrho_{\mu} \rVert_{L^{r} (\mathbb{R}^{d})} = \mu^{a - \frac{d}{r} + k} \lVert D^{k} \varrho \rVert_{L^{r} (\mathbb{R}^{d})}, \hspace{2mm} \lVert D^{k} \tilde{\varrho}_{\mu} \rVert_{L^{r}(\mathbb{R}^{d})} = \mu^{b - \frac{d}{r} + k} \lVert D^{k} \varrho \rVert_{L^{r}(\mathbb{R}^{d})}, \\
& \int_{\mathbb{R}^{d}} \varrho_{\mu} \tilde{\varrho}_{\mu} dx = 1. \label{est 307}
\end{align}
\end{subequations} 
Consequently, $\lVert \varrho_{\mu} \rVert_{L^{p}(\mathbb{R}^{d})} = \lVert \varrho \rVert_{L^{p}(\mathbb{R}^{d})}, \lVert \tilde{\varrho}_{\mu} \rVert_{L^{p'}(\mathbb{R}^{d})} = \lVert \tilde{\varrho} \rVert_{L^{p'} (\mathbb{R}^{d})}$. Moreover, $supp \varrho_{\mu} = supp \tilde{\varrho}_{\mu}$ and both are contained in a ball with radius at most $r$ from Lemma \ref{Lemma 3.8}. 
\end{lemma} 

For any given $y \in \mathbb{T}^{d}$, we define $\tau_{y}: \mathbb{T}^{d} \mapsto \mathbb{T}^{d}$ by 
\begin{equation}\label{est 67}
\tau_{y}(x) \triangleq x - y  \hspace{1mm} \text{ so that } \hspace{1mm} \lVert D^{k} (g \circ \tau_{y}) \rVert_{L^{r}} = \lVert D^{k} g \rVert_{L^{r}}
\end{equation}
for every smooth function $g$ on $\mathbb{T}^{d}$, for all $k \in \mathbb{N}_{0},$ and all $r \in [1,\infty].$

\begin{lemma}
\rm{ (\cite[Lemma 4.3]{MS20})} There exist functions $\varrho_{\mu}^{j}: \mathbb{T}^{d} \mapsto \mathbb{R}, \tilde{\varrho}_{\mu}^{j}: \mathbb{T}^{d} \mapsto \mathbb{R}$ for $j \in \{1, \hdots, d \}$, such that for any $r \in [1,\infty]$, and any $k \in \mathbb{N}_{0}$ 
\begin{equation}\label{est 68}
\lVert D^{k} \varrho_{\mu}^{j} \rVert_{L^{r}} = \mu^{a - \frac{d}{r} + k} \lVert D^{k} \varrho \rVert_{L^{r}}, \hspace{3mm} \lVert D^{k} \tilde{\varrho}_{\mu}^{j} \rVert_{L^{r}} = \mu^{ b - \frac{d}{r} + k} \lVert D^{k} \varrho \rVert_{L^{r}}. 
\end{equation} 
Moreover, for any $i, j \in \{1, \hdots, d\}$ and $s \in \mathbb{R}$, 
\begin{equation}\label{est 113}
\fint_{\mathbb{T}^{d}} (\varrho_{\mu}^{i} \circ \tau_{se_{i}}) (\tilde{\varrho}_{\mu}^{i} \circ \tau_{se_{i}}) dx = 1 \hspace{1mm} \text{ and } \hspace{1mm} (\varrho_{\mu}^{i} \circ \tau_{se_{i}}) ( \tilde{\varrho}_{\mu}^{j} \circ \tau_{se_{j}}) = 0 \hspace{1mm} \forall \hspace{1mm} i \neq j. 
\end{equation} 
\end{lemma}
\begin{remark}
From the proof of \cite[Lemma 4.3]{MS20} we know that the precise forms of such $\varrho_{\mu}^{j}$ and $\tilde{\varrho}_{\mu}^{j}$ are 
\begin{equation}\label{est 192} 
\varrho_{\mu}^{j} \triangleq \varrho_{\mu} \circ \tau_{\zeta_{j}}, \hspace{3mm} \tilde{\varrho}_{\mu}^{j} \triangleq \tilde{\varrho}_{\mu} \circ \tau_{\zeta_{j}}
\end{equation} 
where $\zeta_{j}$ are those from Lemma \ref{Lemma 3.8}.  
\end{remark} 

Next, we fix a smooth function $\psi: \mathbb{T}^{d-1} \mapsto \mathbb{R}$ such that 
\begin{equation}\label{est 57} 
\fint_{\mathbb{T}^{d-1}} \psi dx = 0 \hspace{1mm} \text{ while } \hspace{1mm} \fint_{\mathbb{T}^{d-1}} \psi^{2} dx = 1 
\end{equation} 
and define for every $j \in \{1, \hdots, d \}$
\begin{equation}\label{est 72}
\psi^{j}: \mathbb{T}^{d} \mapsto \mathbb{R}, \hspace{3mm} \psi^{j} (x) \triangleq \psi^{j} (x_{1}, \hdots, x_{d}) \triangleq \psi(x_{1}, \hdots, x_{j-1}, x_{j+1}, \hdots, x_{d}) 
\end{equation} 
so that 
\begin{equation}\label{est 305}
\fint_{\mathbb{T}^{d}} \psi^{j} dx = 0 \hspace{1mm} \text{ while } \hspace{1mm} \fint_{\mathbb{T}^{d}} (\psi^{j})^{2} dx = 1. 
\end{equation} 
For the parameters 
\begin{subequations}\label{est 58} 
\begin{align}
&\text{ fast oscillation} = \lambda \in \mathbb{N}, \hspace{3mm} \text{ concentration}  = \mu \gg \lambda, \\
&\text{ phase speed} = \omega,  \hspace{12mm}\text{ very fast oscillation} =\nu \in \lambda \mathbb{N}, \nu \gg \lambda,
\end{align}
\end{subequations} 
we define for $j \in \{1, \hdots, d \}$, Mikado density, Mikado field, and quadratic corrector as 
\begin{subequations}\label{est 59}
\begin{align}
&\Theta_{\lambda, \mu, \omega, \nu}^{j} (t,x) \triangleq \varrho_{\mu}^{j} (\lambda (x - \omega t e_{j} )) \psi^{j} (\nu x) = (( \varrho_{\mu}^{j} )_{\lambda} \circ \tau_{\omega t e_{j}})(x) \psi_{\nu}^{j}(x), \\
&W_{\lambda, \mu, \omega, \nu}^{j} (t,x) \triangleq \tilde{\varrho}_{\mu}^{j} (\lambda (x - \omega t e_{j} )) \psi^{j} (\nu x) e_{j} = (( \tilde{\varrho}_{\mu}^{j})_{\lambda} \circ \tau_{\omega t e_{j}}) (x) \psi_{\nu}^{j} (x) e_{j}, \\
& Q_{\lambda, \mu, \omega, \nu}^{j} (t,x) \triangleq \omega^{-1} ( \varrho_{\mu}^{j} \tilde{\varrho}_{\mu}^{j}) ( \lambda (x - \omega t e_{j} )) (\psi^{j} (\nu x))^{2} = \omega^{-1} (( \varrho_{\mu}^{j} \tilde{\varrho}_{\mu}^{j})_{\lambda} \circ \tau_{\omega t e_{j}})(x) (\psi_{\nu}^{j}(x))^{2}, 
\end{align}
\end{subequations} 
to which we refer as $\Theta^{j}, W^{j}$, and $Q^{j}$ when no confusion arises, respectively. They satisfy 
\begin{equation}\label{est 126}
\partial_{t} \Theta_{\lambda, \mu, \nu, \omega}^{j} = - \lambda \omega (( \partial_{j}  \varrho_{\mu}^{j})_{\lambda} \circ \tau_{\omega t e_{j}}) \psi_{\nu}^{j} \hspace{2mm} \text{ and } \hspace{2mm} \text{ div} W_{\lambda, \mu, \omega, \nu}^{j} = \lambda (( \partial_{j} \tilde{\varrho}_{\mu}^{j})_{\lambda} \circ \tau_{\omega t e_{j}}) \psi_{\nu}^{j}.
\end{equation}  

\section{Proof of Theorem \ref{Theorem 2.2}}\label{Section 4}

\begin{proposition}\label{Proposition 4.1} 
Under the hypothesis of \eqref{est 31}, the solution $z$ to \eqref{stochastic heat} where $B$ is the $GG^{\ast}$-Wiener process satisfies for all $\delta \in (0, \frac{1}{2}), T > 0,$ and $l \in \mathbb{N}$, 
\begin{equation}
\mathbb{E}^{\mathbb{P}}[ \lVert z \rVert_{C_{T} \dot{H}_{x}^{\frac{d+ 2 + \varsigma}{2}}}^{l} + \lVert z \rVert_{C_{T}^{\frac{1}{2} - \delta} \dot{H}_{x}^{\frac{d+\varsigma}{2}}}^{l}] < \infty. 
\end{equation} 
\end{proposition} 

\begin{proof}[Proof of Proposition \ref{Proposition 4.1}]
This result is discussed in detail in \cite[Proposition 4.4]{Y21a} and follows from \cite[Proposition 3.6]{HZZ19} which in turn followed \cite[Proposition 34]{D13}. 
\end{proof}

For the Sobolev constant $C_{S} > 0$ such that $\lVert f \rVert_{L_{x}^{\infty}} \leq C_{S} \lVert f \rVert_{\dot{H}_{x}^{\frac{d+\varsigma}{2}}}$ for all $f \in \dot{H}^{\frac{d+\varsigma}{2}}(\mathbb{T}^{d})$ that is mean-zero and $\varpi \in (0, \frac{1}{4})$, we define 
\begin{equation}\label{est 54}
T_{L} \triangleq \inf \{ t \geq 0: C_{S} \lVert z(s) \rVert_{\dot{H}_{x}^{\frac{d+ 2 + \varsigma}{2}}} \geq L^{\frac{1}{4}} \}  
\wedge \inf\{t \geq 0: C_{S} \lVert z \rVert_{C_{t}^{\frac{1}{2} - 2 \varpi} \dot{H}_{x}^{\frac{d+\varsigma}{2}}} \geq L^{\frac{1}{2}} \} \wedge L. 
\end{equation} 
We see that $T_{L} > 0$ and $\lim_{L\to\infty} T_{L} = \infty$ $\mathbb{P}$-a.s. due to Proposition \ref{Proposition 4.1} and for all $t \in [0, T_{L}]$, 
\begin{equation}\label{est 32}
\lVert z(t) \rVert_{L_{x}^{\infty}} \leq L^{\frac{1}{4}}, \hspace{3mm} \lVert z(t) \rVert_{W_{x}^{1,\infty}} \leq L^{\frac{1}{4}}, \hspace{3mm} \lVert z \rVert_{C_{t}^{\frac{1}{2} - 2 \varpi} L_{x}^{\infty}} \leq L^{\frac{1}{2}}.
\end{equation} 
We let 
\begin{equation}\label{est 49}
M_{0}(t) \triangleq L^{4} e^{4Lt}. 
\end{equation} 
Theorem \ref{Theorem 2.2} essentially follows from this key proposition concerning the transport-diffusion-defect equation \eqref{est 29}.
\begin{proposition}\label{Proposition 4.2} 
There exists a constant $M > 0$ such that the following holds. Let $\varpi \in (0, \frac{1}{4})$, $p \in (1,\infty)$, $\tilde{p} \in [1,\infty)$ such that \eqref{est 8} holds. Then for any $\delta, \eta > 0$ and $(\mathcal{F}_{t})_{t\geq 0}$-adapted $(\theta_{0}, u_{0}, R_{0})$ that satisfies \eqref{est 29} such that for all $t \in [0, T_{L}]$ $\fint_{\mathbb{T}^{d}} \theta_{0}(t,x) dx = 0$, 
\begin{subequations}\label{est 240}
\begin{align}
&\theta_{0} \in C^{\infty} ([0,T_{L}] \times \mathbb{T}^{d}),  \hspace{3mm} u_{0} \in C^{\infty} ([0,T_{L}] \times \mathbb{T}^{d}), \\
&R_{0} \in C([0,T_{L}]; C^{1}(\mathbb{T}^{d})) \cap C^{\frac{1}{2} - 2 \varpi} ([0,T_{L}]; C(\mathbb{T}^{d})), 
\end{align}
\end{subequations}
and 
\begin{equation}\label{est 78}
\lVert R_{0}(t) \rVert_{L_{x}^{1}} \leq 2 \delta M_{0}(t), 
\end{equation}
there exists another $(\mathcal{F}_{t})_{t\geq 0}$-adapted $(\theta_{1}, u_{1}, R_{1})$ that satisfies \eqref{est 29} in same corresponding regularity class \eqref{est 240} such that for all $t \in [0, T_{L}]$ $\fint_{\mathbb{T}^{d}} \theta_{1}(t,x) dx = 0$ and 
\begin{subequations}\label{est 39}
\begin{align}
& \lVert (\theta_{1} -\theta_{0})(t) \rVert_{L_{x}^{p}} \leq M \eta ( 2 \delta M_{0}(t))^{\frac{1}{p}}, \label{est 35} \\ 
& \lVert (u_{1} - u_{0})(t) \rVert_{L_{x}^{p'}} \leq M \eta^{-1} (2\delta M_{0}(t))^{\frac{1}{p'}}, \label{est 36}\\
& \lVert (u_{1} - u_{0})(t) \rVert_{W_{x}^{1,\tilde{p}}} \leq \delta M_{0}(t), \label{est 37}\\
& \lVert R_{1}(t) \rVert_{L_{x}^{1}} \leq \delta M_{0}(t). \label{est 38}
\end{align}
\end{subequations}
Finally, if $(\theta_{0}, u_{0}, R_{0})(0,x)$ are deterministic, then so are $(\theta_{1}, u_{1}, R_{1})(0,x)$. 
\end{proposition} 

\begin{remark}\label{Remark 4.1}
Proposition \ref{Proposition 4.2} is a kind of a probabilistic analogue of \cite[Proposition 2.1]{MS20}. One key difference is that ``$\lVert R_{0}(t) \rVert_{L_{x}^{1}}^{\frac{1}{p}}$'' to bound $\lVert (\theta_{1} - \theta_{0})(t) \rVert_{L_{x}^{p}}$ and ``$\lVert R_{0}(t) \rVert_{L_{x}^{1}}^{\frac{1}{p'}}$'' to bound $\lVert (u_{1} - u_{0})(t) \rVert_{L_{x}^{p'}}$ in \cite[Equations (2.3a)-(2.3b)]{MS20} are unreasonable for us. Such bounds work well in \cite{MS20} because e.g., they construct $\theta_{1}(t)$ as a sum of $\theta_{0}(t)$ and certain perturbation (``$\vartheta(t,x)  + \vartheta_{c}(t) + q(t,x) + q_{c}(t)$'' on \cite[p. 1092]{MS20}) that completely vanishes when $R_{0}(t) \equiv 0$ and thus $\theta_{1}(t) - \theta_{0}(t) \equiv 0$ so that an inductive bound of ``$\lVert \theta_{1}(t) -\theta_{0}(t) \rVert_{L_{x}^{p}} \leq M \eta \lVert R_{0}(t) \rVert_{L_{x}^{1}}^{\frac{1}{p}}$'' is achievable. As we described in Remark \ref{Remark 2.3}, we will construct $\theta_{1}$ as a sum of $\theta_{0}$ and perturbations that does not vanish in general even if $R_{0}(t) \equiv 0$; thus we cannot bound $ \lVert (\theta_{1} -\theta_{0})(t) \rVert_{L_{x}^{p}}$ by $M\eta \lVert R_{0}(t) \rVert_{L^{1}}^{\frac{1}{p}}$ which may be zero. Analogous comment applies for our choice of $ \lVert (u_{1} - u_{0})(t) \rVert_{L_{x}^{p'}} \leq M \eta^{-1} (2\delta M_{0}(t))^{\frac{1}{p'}}$ in \eqref{est 36}.  This is precisely why we included a hypothesis \eqref{est 78} that is absent in \cite[Proposition 2.1]{MS20} to achieve bounds of $M \eta (2\delta M_{0}(t))^{\frac{1}{p}}$ and $M \eta^{-1} (2\delta M_{0}(t))^{\frac{1}{p'}}$ in \eqref{est 35}-\eqref{est 36}, respectively.  
\end{remark} 

Assuming Proposition \ref{Proposition 4.2}, we can prove Theorem \ref{Theorem 2.2} as follows.
\begin{proof}[Proof of Theorem \ref{Theorem 2.2} assuming Proposition \ref{Proposition 4.2}]
For any $L > 1$ we define 
\begin{equation}\label{est 179}
\theta_{0} (t,x) \triangleq M_{0}(t)^{\frac{1}{p}} (x_{d} - \frac{1}{2}), \hspace{3mm} u_{0} \equiv 0, \hspace{3mm} R_{0} \triangleq -\mathcal{D}^{-1} \partial_{t} \theta_{0}
\end{equation} 
so that $\theta_{0}$ has mean zero for all $t \geq 0$, $u_{0}$ is trivially divergence-free, and $\mathcal{D}^{-1}$ from Definition \ref{Definition 3.1} is well-defined because $\theta_{0}$ is mean-zero. Because $\Delta \theta_{0} = 0$, by construction $(\theta_{0}, u_{0}, R_{0})$ solves \eqref{est 29}; moreover, $\theta_{0} \in C_{t,x}^{\infty}, u_{0} \in C_{t,x}^{\infty}, R_{0} \in C_{t}C_{x}^{1} \cap C_{t}^{\frac{1}{2} - 2 \varpi}C_{x}$ and they are all $(\mathcal{F}_{t})_{t\geq 0}$-adapted. We can also readily compute 
\begin{equation}\label{est 40}
\lVert \theta_{0}(t) \rVert_{L^{p}} = \frac{M_{0} (t)^{\frac{1}{p}}}{2(p+1)^{\frac{1}{p}}}.
\end{equation} 
We note that a typical choice for such $\theta_{0}$ in the case of the 3D Navier-Stokes equations would be $\theta_{0}(t,x) = M_{0}(t)^{\frac{1}{2}} \sin(2\pi x_{3})$ for which it suffices to compute its $L^{2}(\mathbb{T}^{3})$-norm (see e.g., \cite[Proposition 4.7]{Y20a}); however, we need to compute its $L^{p}(\mathbb{T}^{d})$-norm for an arbitrary $p \in (1,\infty)$ while making sure that it is mean-zero and thus we chose $\theta_{0} = M_{0}(t)^{\frac{1}{p}} (x_{d} - \frac{1}{2})$ in \eqref{est 179} for simplicity. 

We set $\delta = \sup_{t \in [0,T]} \frac{ \lVert R_{0}(t) \rVert_{L_{x}^{1}}}{M_{0}(t)}$ so that $\lVert R_{0}(t) \rVert_{L_{x}^{1}} \leq 2 \delta M_{0}(t)$. We choose 
\begin{equation}\label{est 41}
\delta_{n} = \delta 2^{-(n-1)}, \hspace{3mm} n \in \mathbb{N}_{0}
\end{equation}
so that $\delta_{n+1} = \delta_{n} 2^{-1}$, $\delta_{0} = 2 \delta, \delta_{1} = \delta$, etc. Having fixed such $\delta_{n}$, we fix a sequence $\eta_{n} \in (1,\infty)$ for $n \in \mathbb{N}$ such that 
\begin{equation}\label{est 42} 
\delta_{n}^{\frac{1}{p}} \eta_{n} = \sigma \delta_{n}^{\frac{1}{2}}
\end{equation} 
for $\sigma > 0$ that satisfies 
\begin{equation}\label{est 47}
\sigma 4(p+1)^{\frac{1}{p}} M \sum_{n=0}^{\infty} \sqrt{\delta} 2^{- \frac{n-1}{2}} < 1. 
\end{equation} 
By repeated applications of Proposition \ref{Proposition 4.2} we obtain $(\theta_{n}, u_{n}, R_{n})_{n\in\mathbb{N}}$ that satisfies \eqref{est 29} and 
\begin{subequations}\label{est 180}
\begin{align}
& \lVert (\theta_{n+1} - \theta_{n})(t) \rVert_{L_{x}^{p}}  \overset{\eqref{est 35}}{\leq} M \eta_{n} ( 2 \delta_{n+1} M_{0}(t))^{\frac{1}{p}} \overset{\eqref{est 41}\eqref{est 42}}{=} M \sigma \delta_{n}^{\frac{1}{2}}M_{0}(t)^{\frac{1}{p}} , \label{est 43}\\
& \lVert (u_{n+1} - u_{n})(t) \rVert_{L_{x}^{p'}} \overset{\eqref{est 36}}{\leq} M \eta_{n}^{-1} ( 2 \delta_{n+1} M_{0}(t))^{\frac{1}{p'}} \overset{\eqref{est 41} \eqref{est 42}}{=} M \sigma^{-1} \delta_{n}^{\frac{1}{2}} M_{0}(t)^{\frac{1}{p'}}, \label{est 44} \\
& \lVert (u_{n+1} - u_{n})(t) \rVert_{W_{x}^{1,\tilde{p}}} \overset{\eqref{est 37}}{\leq} \delta_{n+1} M_{0}(t), \label{est 45}\\
& \lVert R_{n+1}(t) \rVert_{L_{x}^{1}} \overset{\eqref{est 38}}{\leq} \delta_{n+1} M_{0}(t). \label{est 46} 
\end{align}
\end{subequations}
Therefore,
\begin{subequations}  
\begin{align}
& \sum_{n=0}^{\infty} \lVert (\theta_{n+1} - \theta_{n})(t)\rVert_{L_{x}^{p}}  \overset{\eqref{est 43} \eqref{est 41}}{\leq} M M_{0}(t)^{\frac{1}{p}} \sigma \sum_{n=0}^{\infty} \sqrt{\delta} 2^{-\frac{(n-1)}{2}} < \infty, \\
& \sum_{n=0}^{\infty} \lVert (u_{n+1}- u_{n})(t) \rVert_{L_{x}^{p'}} \overset{\eqref{est 44} \eqref{est 41}}{\leq} M M_{0}(t)^{\frac{1}{p'}} \sigma^{-1} \sum_{n=0}^{\infty}  \sqrt{\delta} 2^{- \frac{ (n-1)}{2}} < \infty, \\
& \sum_{n=0}^{\infty} \lVert (u_{n+1} - u_{n})(t) \rVert_{W_{x}^{1, \tilde{p}}} \overset{\eqref{est 45} \eqref{est 41}}{\leq} M_{0}(t) \sum_{n=0}^{\infty} \delta 2^{-n} < \infty. 
\end{align}
\end{subequations} 
Thus, for all $t \in [0, T_{L}]$, $\{ \theta_{n}(t) \}_{n\in\mathbb{N}_{0}}$ and $\{ u_{n}(t) \}_{n\in\mathbb{N}_{0}}$ are Cauchy in $C([0, T_{L}]; L^{p} (\mathbb{T}^{d}))$ and $C([0, T_{L}]; L^{p'} (\mathbb{T}^{d})) \cap C([0,T_{L}]; W^{1, \tilde{p}}(\mathbb{T}^{d}))$, respectively. Therefore, there exist unique $\theta \in C([0, T_{L}]; L^{p}(\mathbb{T}^{d}))$ and $u \in C([0, T_{L}]; L^{p'} (\mathbb{T}^{d})) \cap C([0, T_{L}]; W^{1, \tilde{p}}(\mathbb{T}^{d}))$ such that 
\begin{equation}
\theta_{n} \to \theta \text{ in } C([0, T_{L}]; L^{p} (\mathbb{T}^{d})), \hspace{1mm} u_{n} \to u \text{ in } C([0, T_{L}]; L^{p'} (\mathbb{T}^{d})) \cap C([0, T_{L}]; W^{1, \tilde{p}}(\mathbb{T}^{d}))
\end{equation} 
and there exists a deterministic constant $C_{L} > 0$ such that 
\begin{equation}\label{est 347}
\max\{ \sup_{t \in [0, T_{L}]} \lVert \theta (t) \rVert_{L_{x}^{p}}, \sup_{t \in [0, T_{L}]} \lVert u(t) \rVert_{L_{x}^{p'}}, \sup_{t \in [0, T_{L}]} \lVert u(t) \rVert_{W_{x}^{1,\tilde{P}}} \} \leq C_{L}. 
\end{equation} 
Because $(\theta_{0}, u_{0}, R_{0})$ were all $(\mathcal{F}_{t})_{t\geq 0}$-adapted, so are $(\theta_{n}, u_{n}, R_{n})$ for all $n \in \mathbb{N}$ due to Proposition \ref{Proposition 4.2}; consequently, $(\theta, u)$ are both $(\mathcal{F}_{t})_{t\geq 0}$-adapted. Moreover, $\lVert R_{n}(t) \rVert_{L_{x}^{1}} \overset{\eqref{est 46}}{\leq} M_{0}(t) \delta_{n} \to 0$ as $n\to\infty$ and 
\begin{subequations}\label{est 181}
\begin{align}
& \lVert \theta_{n} u_{n} - \theta u \rVert_{C_{t}L_{x}^{1}} \leq \sup_{t \in [0, T_{L}]} \lVert \theta_{n}(t) \rVert_{L_{x}^{p}} \lVert (u_{n} - u)(t) \rVert_{L_{x}^{p'}} + \lVert (\theta_{n} - \theta)(t) \rVert_{L_{x}^{p}} \lVert u(t) \rVert_{L_{x}^{p'}} \to 0, \\
& \lVert u_{n} z - u z \rVert_{C_{t}L_{x}^{1}} \leq \sup_{t \in [0, T_{L}]} \lVert (u_{n} - u)(t) \rVert_{L_{x}^{p'}} L^{\frac{1}{4}} \to 0 
\end{align}
\end{subequations} 
as $n\to\infty$ due to \eqref{est 32}. Thus, $(\theta, u)$ satisfies \eqref{est 3} analytically weakly. Consequently, by defining $\rho = u + z$, we obtain $(\rho, u)$ that satisfies \eqref{stochastic transport} forced by the additive noise analytically weakly, i.e., \eqref{solution add}. The regularity claimed in \eqref{est 393}-\eqref{est 363} follow from \eqref{est 32} and \eqref{est 347}. Next, in order to prove \eqref{est 53} on $\{T_{L} \geq T \}$ for the fixed $T > 0$ from hypothesis, we fix such $(\rho,u)$, define $\rho^{\text{in}} = \rho \rvert_{t=0}$, and then take $L > 0$ larger if necessary so that 
\begin{subequations}
\begin{align}
&L^{\frac{4}{p}} e^{\frac{2L T}{p}} >( 3 L^{\frac{4}{p}} + L 4(p+1)^{\frac{1}{p}}), \label{est 50}\\
& (\lVert \rho^{\text{in}} \rVert_{L^{p}} + L) e^{\frac{2L T}{p}}  \nonumber\\
\geq& 
\begin{cases}
L^{\frac{1}{4}} + K [\lVert \rho^{\text{in}} \rVert_{L^{p}} + L^{\frac{1}{4}} ( \int_{0}^{T} \lVert u \rVert_{L_{x}^{p}} ds + 1)] & \text{ if } p \in (1,2), \\
L^{\frac{1}{4}} + K [\lVert \rho^{\text{in}} \rVert_{L^{2}} + \sqrt{T Tr(GG^{\ast} )}] & \text{ if } p = 2, \\
L^{\frac{1}{4}} + K e^{\frac{T}{p}} [ \lVert \rho^{\text{in}} \rVert_{L^{p}} + C(p, Tr((-\Delta)^{\frac{d}{2} + 2 \varsigma} GG^{\ast} ))^{\frac{1}{p}}] & \text{ if } p \in (2,\infty) \label{est 51}
\end{cases} 
\end{align}
\end{subequations}
where $C(p, Tr((-\Delta)^{\frac{d}{2} + 2 \varsigma}GG^{\ast})) $ was defined in \eqref{est 33}. Now for all $t \in [0, T_{L}]$, we can compute 
\begin{align}
\lVert (\theta - \theta_{0})(t) \rVert_{L_{x}^{p}} \overset{\eqref{est 41} \eqref{est 43}}{\leq} M_{0}(t)^{\frac{1}{p}} M \sigma  \sum_{n=0}^{\infty} \sqrt{\delta} 2^{- \frac{n-1}{2}} 
\overset{\eqref{est 47}}{\leq} \frac{M_{0}(t)^{\frac{1}{p}}}{4 (p+1)^{\frac{1}{p}}}. \label{est 48}
\end{align}
This leads us to an estimate of 
\begin{align}
( \lVert \theta(0) \rVert_{L_{x}^{p}} + L) e^{\frac{ 2L T}{p}} &\leq ( \lVert (\theta - \theta_{0}) (0) \rVert_{L_{x}^{p}} + \lVert \theta_{0} (0) \rVert_{L_{x}^{p}} + L) e^{\frac{2L T}{p}} \label{est 52} \\
\overset{\eqref{est 48} \eqref{est 40}}{\leq}& \left( \frac{ M_{0}(0)^{\frac{1}{p}}}{4 (p+1)^{\frac{1}{p}}} + \frac{ M_{0}(0)^{\frac{1}{p}}}{2(p+1)^{\frac{1}{p}}} + L\right) e^{\frac{ 2LT}{p}} \nonumber \\
\overset{\eqref{est 49} \eqref{est 50}}{<}& \frac{M_{0}(T)^{\frac{1}{p}}}{2(p+1)^{\frac{1}{p}}} - \frac{M_{0}(T)^{\frac{1}{p}}}{4(p+1)^{\frac{1}{p}}} 
\overset{\eqref{est 48} \eqref{est 40}}{\leq} \lVert \theta_{0}(T) \rVert_{L_{x}^{p}} - \lVert (\theta - \theta_{0})(T) \rVert_{L_{x}^{p}} \leq \lVert \theta(T) \rVert_{L_{x}^{p}}.  \nonumber 
\end{align}
Because $\rho = \theta + z$, this allows us to conclude via H$\ddot{\mathrm{o}}$lder's inequality and the fact that $z(0,x) \equiv 0$ from \eqref{stochastic heat} so that $\rho^{\text{in}}(x) = \theta(0,x)$, that on the set $\{T_{L} \geq T\}$ 
\begin{align}
\lVert \rho(T) \rVert_{L_{x}^{p}} \geq& \lVert \theta(T) \rVert_{L_{x}^{p}} - \lVert z(T) \rVert_{L_{x}^{p}} \overset{\eqref{est 52}}{>} ( \lVert \theta(0) \rVert_{L_{x}^{p}} + L) e^{\frac{2LT}{p}} - \lVert z(T) \rVert_{L_{x}^{\infty}} \nonumber \\
\overset{\eqref{stochastic heat}\eqref{est 32} }{\geq}& ( \lVert \rho^{\text{in}} \rVert_{L^{p}} + L) e^{\frac{2LT}{p}} - L^{\frac{1}{4}} \nonumber \\
\overset{\eqref{est 51}}{\geq}& 
\begin{cases}
K[ \lVert \rho^{\text{in}} \rVert_{L^{p}} + L^{\frac{1}{4}} (\int_{0}^{T} \lVert u \rVert_{L_{x}^{p}} ds + 1) ] & \text{ if } p \in (1,2), \\
K [\lVert \rho^{\text{in}} \rVert_{L^{2}} + \sqrt{T Tr (GG^{\ast})} & \text{ if } p = 2, \\
K e^{\frac{T}{p}} [ \lVert \rho^{\text{in}} \rVert_{L^{p}} + C(p, Tr((-\Delta)^{\frac{d}{2} + 2 \varsigma} GG^{\ast} ))^{\frac{1}{p}}] & \text{ if }  p \in (2,\infty).
\end{cases} \label{est 184}
\end{align}
This proves \eqref{est 53}. We can take $L > 0$ larger if necessary to achieve \eqref{est 140} due to $\lim_{L\to\infty} T_{L} = \infty$ $\mathbb{P}$-a.s. from \eqref{est 54}. Finally, because $(\theta_{0}, u_{0}, R_{0})(0,x)$ was deterministic, Proposition \ref{Proposition 4.2} implies that $(\theta_{n}, u_{n}, R_{n})(0,x)$ are deterministic for all $n \in \mathbb{N}$ and consequently so is $\rho^{\text{in}}$ because $z(0,x) \equiv 0$ due to \eqref{stochastic heat}. This completes the proof of Theorem \ref{Theorem 2.2}.  
\end{proof}

\begin{remark}\label{Remark 4.2}
In \cite[Theorem 1.2]{MS20}, the parameter $\eta$ in Proposition \ref{Proposition 4.2} was crucially utilized to deduce a freedom to choose $\sigma > 0$ in \eqref{est 42}-\eqref{est 47} and conclude that given any $\epsilon > 0$, any $\bar{\rho} \in C^{\infty} ([0,T] \times \mathbb{T}^{d})$ with zero mean and $\bar{u} \in C^{\infty} ([0,T] \times \mathbb{T}^{d})$ that is divergence-free, they can construct a solution via convex integration such that $\lVert \rho - \bar{\rho} \rVert_{C_{T}L_{x}^{p}} < \epsilon$ by taking $\sigma>0$ sufficiently small or $\lVert u(t) - \bar{u}(t) \rVert_{C_{T}L_{x}^{p'}} < \epsilon$ by taking $\sigma > 0$ sufficiently large. Because we are not pursuing such a result, one may wonder why we need the parameter $\eta$. First, $\eta$ leads to the freedom to choose $\sigma$ defined in \eqref{est 42}-\eqref{est 47}, namely $\sigma 4(p+1)^{\frac{1}{p}} M \sum_{n=0}^{\infty} \sqrt{\delta} 2^{- \frac{n-1}{2}} < 1$, and \eqref{est 47} is crucial to deduce the necessary estimate \eqref{est 48}. If we were to simplify the claim in Proposition \ref{Proposition 4.2} with $\eta = 1$ therein, we obtain 
\begin{align*}
\lVert (\theta_{1} -\theta_{0})(t) \rVert_{L_{x}^{p}} \leq M ( 2 \delta M_{0}(t))^{\frac{1}{p}}
\end{align*} 
instead of \eqref{est 35}; following the computations of \eqref{est 48} requires 
\begin{align*}
\lVert (\theta - \theta_{0})(t) \rVert_{L_{x}^{p}} \leq \sum_{n=0}^{\infty} \lVert (\theta_{n+1} - \theta_{n})(t) \rVert_{L_{x}^{p}} \leq \sum_{n=0}^{\infty} M (2\delta_{n+1} M_{0}(t))^{\frac{1}{p}} \leq \frac{M_{0}(t)^{\frac{1}{p}}}{4(p+1)^{\frac{1}{p}}}, 
\end{align*}
which in turn requires  $\delta  [\sum_{n=0}^{\infty} 2^{-\frac{n}{p}} 4M]^{p} 2(p+1) \leq 1$. On the other hand, we pointed out in Remark \ref{Remark 4.1} the convenience of the additional hypothesis \eqref{est 78}, and this implies the necessity of 
$\frac{\lVert R_{0}(t) \rVert_{L_{x}^{1}}}{2 M_{0}(t)} \leq \delta$ to satisfy the hypothesis \eqref{est 78} at initial step. Making sure that we can find such $\delta$ with the lower and the upper bounds is not trivial; one idea in this case is to take $L > 0$ sufficiently large so that $M_{0}(t) \gg 1$. Nonetheless, the same issue will arise in the proof of Theorem \ref{Theorem 2.4} assuming Proposition \ref{Proposition 7.3} in which the same strategy will not work due to the absence of stopping time therein. Therefore, for convenience and consistency, we chose to attain Propositions \ref{Proposition 4.2} and \ref{Proposition 7.3} with $\eta$ therein.  
\end{remark} 

The rest of the proof of Theorem \ref{Theorem 2.2} is devoted to the proof of Proposition \ref{Proposition 4.2}. 

\subsection{Proof of Proposition \ref{Proposition 4.2}}
We define a parameter 
\begin{equation}\label{est 95}
l \triangleq \lambda^{-\iota} \ll 1
\end{equation} 
for $\lambda \in\mathbb{N}$ and $\iota \in \mathbb{R}_{>0}$ to be chosen subsequently. Then we let 
\begin{equation}\label{est 242}
\{ \phi_{l}\}_{l > 0} \hspace{2mm} \text{ and } \hspace{2mm} \{\varphi_{l}\}_{l > 0} 
\end{equation} 
respectively be families of standard mollifiers on $\mathbb{R}^{d}$ and $\mathbb{R}$ with mass one where the latter is equipped with compact support on $(0,l]$. Then we extend $R_{0}$ to $t < 0$ with its value at $t = 0$ and mollify it with $\phi_{l}$ and $\varphi_{l}$ to obtain 
\begin{equation}\label{est 86}
R_{l} \triangleq R_{0} \ast_{x} \phi_{l} \ast_{t} \varphi_{l}.
\end{equation} 
\begin{remark}\label{Remark 4.3}
We note that in all of the previous works that employed Nash-type convex integration schemes to stochastic case mollified not only $R_{0}$ but analogues of $\theta_{0}, u_{0}$, and $z$ (e.g., \cite{HZZ19}). Mollifying only $R_{0}$ simplifies our proof significantly; indeed, otherwise our next step would be to write down the mollified equation of \eqref{est 29} which produces a commutator term from the nonlinear term. Not mollifying $\theta_{0}$ and $u_{0}$ will also be a crucial ingredient in proof of Proposition \ref{Proposition 6.1}, as we will describe in Remark \ref{Difficulty 2}.
\end{remark} 

For the fixed $p \in (1,\infty)$ from the hypothesis of Proposition \ref{Proposition 4.2}, we define $a \triangleq \frac{d}{p}$ and $b \triangleq \frac{d}{p'}$ so that $a +b = d$ as in \eqref{est 55}. We adhere to the setting of convex integration from Section \ref{Subsection 3.2}. In particular, we define $r$ from Lemma \ref{Lemma 3.8}, $\varrho$ from \eqref{est 56}, $\psi$ from \eqref{est 57}, the key parameters $\lambda, \mu, \omega, \nu$ from \eqref{est 58}, and $\Theta_{\lambda, \mu, \omega, \nu}^{j}$, $W_{\lambda, \mu, \omega, \nu}^{j}$, and $Q_{\lambda, \mu, \omega, \nu}^{j}$ for $j \in \{1, \hdots, d\}$ from \eqref{est 59}. Now we define 
\begin{equation}\label{est 60}
0 < \epsilon < \min \{ \frac{d}{\tilde{p}} - \frac{d}{p'} - 1, \frac{d}{p'} - 1 \}
\end{equation} 
where the positivity is guaranteed by the hypothesis \eqref{est 8}. We have the following result:
\begin{lemma}\label{Lemma 4.3}
\rm{ (\cite[Proposition 4.4]{MS20})} Define a constant 
\begin{equation}\label{est 69}
M \triangleq 2d \max_{k, k' \in \{0, 1\}} \{ \lVert D^{k} \varrho \rVert_{L^{\infty}} \lVert D^{k'} \psi \rVert_{L^{\infty}}, \lVert \varrho \rVert_{L^{\infty}}^{2} \lVert \psi \rVert_{L^{\infty}}^{2} \}. 
\end{equation} 
Then for all $j \in \{1, \hdots, d \}$ and $t \geq 0$, 
\begin{subequations}\label{est 269}
\begin{align}
& \lVert \Theta_{\lambda, \mu, \omega, \nu}^{j} (t) \rVert_{L_{x}^{p}} \leq \frac{M}{2d}, \hspace{2mm} \lVert W_{\lambda, \mu, \omega, \nu}^{j} (t) \rVert_{L_{x}^{p'}} \leq \frac{M}{2d}, \hspace{2mm} \lVert Q_{\lambda, \mu, \omega, \nu}^{j}(t)\rVert_{L_{x}^{p}} \leq \frac{M \mu^{b}}{\omega}, \label{est 61} \\
& \lVert \Theta_{\lambda, \mu, \omega, \nu}^{j}(t) \rVert_{L_{x}^{1}} \leq \frac{M}{\mu^{b}}, \hspace{2mm} \lVert W_{\lambda, \mu, \omega, \nu}^{j} (t) \rVert_{L_{x}^{1}} \leq \frac{M}{\mu^{a}}, \hspace{2mm} \lVert Q_{\lambda, \mu, \omega, \nu}^{j}(t) \rVert_{L_{x}^{1}} \leq \frac{M}{\omega}, \label{est 62}\\
& \lVert W_{\lambda, \mu, \omega, \nu}^{j}(t) \rVert_{C_{x}} \leq M \mu^{b}, \hspace{2mm} \lVert W_{\lambda, \mu,\omega, \nu}^{j}(t) \rVert_{W_{x}^{1, \tilde{p}}} \leq M \frac{ \lambda \mu + \nu}{\mu^{1+ \epsilon}}. \label{est 63} 
\end{align}
\end{subequations} 
Finally, for every $i \neq j$, 
\begin{equation}\label{est 64}
\Theta_{\lambda, \mu, \omega, \nu}^{i} W_{\lambda, \mu, \omega, \nu}^{j} = 0
\end{equation} 
and 
\begin{equation}\label{est 65} 
\partial_{t} Q_{\lambda, \mu, \omega, \nu}^{j} + \text{div} (\Theta_{\lambda, \mu, \omega, \nu}^{j} W_{\lambda,\mu,\omega,\nu}^{j}) = 0. 
\end{equation} 
\end{lemma} 

\begin{proof}[Proof of Lemma \ref{Lemma 4.3}]
This is essentially \cite[Proposition 4.4]{MS20}. The only difference is that because we have diffusion, we considered $\epsilon$ in \eqref{est 60} differently from \cite[Equation (4.11)]{MS20} following \cite[p. 1106]{MS20}. Therefore, the only claim that requires verification is the second inequality of \eqref{est 63} as it involves $\epsilon$. We prove it in the Appendix B for completeness. 
\end{proof}

We now start the definition of perturbations which will differ from \cite{MS20}. We denote by $R_{l}^{j}$ the $j$-th component of $R_{l}$ for $j \in \{1,\hdots, d\}$; i.e., 
\begin{equation}\label{est 104}
R_{l} (t,x) = (R_{0} \ast_{x} \phi_{l} \ast_{x} \varphi_{l})(t,x) = \sum_{j=1}^{d}R_{l}^{j}(t,x) e_{j}.   
\end{equation} 
We define 
\begin{subequations}\label{est 76}
\begin{align}
&\theta_{1}(t,x) \triangleq \theta_{0} (t,x) + \vartheta(t,x) + \vartheta_{c} (t) + q(t,x) + q_{c}(t), \\
&u_{1}(t,x) \triangleq  u_{0} (t,x) + w(t,x) + w_{c}(t,x), 
\end{align}
\end{subequations}
where 
\begin{subequations}\label{est 79}
\begin{align}
& \vartheta(t,x) \triangleq \eta \sum_{j=1}^{d} \chi_{j}(t,x) sgn (R_{l}^{j} (t,x)) \lvert R_{l}^{j} (t,x) \rvert^{\frac{1}{p}} \Theta_{\lambda, \mu, \omega, \nu}^{j} (t,x), \\
& w(t,x) \triangleq \eta^{-1} \sum_{j=1}^{d} \chi_{j} (t,x) \lvert R_{l}^{j} (t,x) \rvert^{\frac{1}{p'}} W_{\lambda, \mu, \omega, \nu}^{j} (t,x), \\
& q(t,x) \triangleq \sum_{j=1}^{d} \chi_{j}^{2}(t,x) R_{l}^{j} (t,x) Q_{\lambda, \mu, \omega, \nu}^{j}(t,x).
\end{align}
\end{subequations} 
Let us describe $\chi_{j}, \vartheta_{c}, q_{c},$ and $w_{c}$. First, $\chi_{j}$ is a cut-off function such that 
\begin{equation}\label{est 70}
\chi_{j}: \mathbb{R}_{\geq 0}\times \mathbb{T}^{d} \mapsto [0,1] \text{ and } \chi_{j} (t,x) = 
\begin{cases}
0 & \text{ if } \lvert R_{l}^{j} (t,x) \rvert \leq \frac{\delta}{4d} M_{0}(t), \\
1 & \text{ if } \lvert R_{l}^{j}(t,x) \rvert \geq \frac{\delta}{2d} M_{0}(t). 
\end{cases} 
\end{equation} 
Additionally, let us define 
\begin{equation}\label{est 77}
a_{j}(t,x) \triangleq  \eta \chi_{j}(t,x) sgn (R_{l}^{j} (t,x)) \lvert R_{l}^{j} (t,x)\rvert^{\frac{1}{p}}, \hspace{3mm}  b_{j}(t,x) \triangleq \eta^{-1} \chi_{j}(t,x)\lvert R_{l}^{j}(t,x) \rvert^{\frac{1}{p'}} 
\end{equation} 
so that 
\begin{equation}\label{est 101}
a_{j}(t,x) b_{j}(t,x)  = \chi_{j}^{2} (t,x) R_{l}^{j} (t,x)
\end{equation} 
and 
\begin{equation}\label{est 71}
\vartheta(t) = \sum_{j=1}^{d} a_{j}(t) \Theta^{j}(t), \hspace{3mm} w(t) = \sum_{j=1}^{d} b_{j}(t) W^{j}(t), \hspace{2mm} \text{ and } \hspace{2mm} q(t) = \sum_{j=1}^{d} a_{j}(t)b_{j}(t) Q^{j}(t). 
\end{equation} 
Moreover, $a_{j}, b_{j}$ satisfy due to \eqref{est 70} 
\begin{equation}\label{est 81}
\lVert a_{j}(t) \rVert_{L_{x}^{p}} \leq \eta \lVert R_{l}^{j}(t) \rVert_{L_{x}^{1}}^{\frac{1}{p}} \hspace{2mm} \text{ and } \hspace{2mm}  \lVert b_{j}(t) \rVert_{L_{x}^{p'}} \leq \eta^{-1}\lVert R_{l}^{j} (t) \rVert_{L_{x}^{1}}^{\frac{1}{p'}}.
\end{equation} 
We note that because $\theta_{0}$ is mean-zero by hypothesis, by defining 
\begin{equation}\label{est 85}
\vartheta_{c}(t) \triangleq - \fint_{\mathbb{T}^{d}} \vartheta(t,x) dx \hspace{1mm} \text{ and } \hspace{1mm} q_{c}(t) \triangleq - \fint_{\mathbb{T}^{d}} q(t,x) dx, 
\end{equation} 
we see that $\theta_{1}(t,x)$ in \eqref{est 76} is mean-zero for all $t$. Moreover, we can compute 
\begin{equation}\label{est 75}
\nabla\cdot w(t,x) \overset{\eqref{est 71}}{=} \nabla\cdot \left( \sum_{j=1}^{d} b_{j}(t,x) W^{j} (t,x) \right) \overset{\eqref{est 59} \eqref{est 72}}{=} \sum_{j=1}^{d} \nabla ( b_{j}(t,x) ( \tilde{ \varrho}_{\mu}^{j})_{\lambda} \circ \tau_{\omega t e_{j}}) \cdot \psi_{\nu}^{j}(x) e_{j} 
\end{equation} 
because $\nabla\cdot (\psi_{\nu}^{j} e_{j}) =0$ due to \eqref{est 72}. Thus, if we define 
\begin{equation}\label{est 73}
w_{c}(t,x) \triangleq - \sum_{j=1}^{d} \mathcal{R}_{N} (f_{j}(t,x), \psi_{\nu}^{j}(x)e_{j}), \hspace{3mm} f_{j}(t,x) \triangleq \nabla (b_{j}(t,x) (\tilde{\varrho}_{\mu}^{j})_{\lambda} \circ \tau_{\omega t e_{j}}(x)), 
\end{equation} 
which is well-defined by Definition \ref{Definition 3.2} because $\psi_{\nu}^{j}$ is mean-zero by \eqref{est 57}, then 
\begin{equation}\label{est 197} 
\nabla\cdot w_{c}(t,x) 
\overset{\eqref{est 74}\eqref{est 75}}{=} - \nabla\cdot w(t,x)
\end{equation}
and thus because $u_{0}$ is divergence-free by hypothesis, 
\begin{equation}\label{est 98}
\nabla\cdot u_{1} = 0 
\end{equation}
as desired; we note that $N \in \mathbb{N}$ in \eqref{est 73} will be chosen subsequently.

\begin{lemma}\label{Lemma 4.4} 
For all $j \in \{1,\hdots, d \}$ and $t \in [0, T_{L}]$, $a_{j}, b_{j}$ in \eqref{est 77} satisfy 
\begin{subequations}
\begin{align}
&\lVert a_{j}(t) \rVert_{L_{x}^{\infty}} \lesssim \eta \lVert R_{0} \rVert_{C_{t,x}}^{\frac{1}{p}}, \hspace{17mm} \lVert b_{j}(t) \rVert_{L_{x}^{\infty}} \lesssim \eta^{-1} \lVert R_{0} \rVert_{C_{t,x}}^{\frac{1}{p'}}, \label{est 82} \\
&\lVert a_{j}(t) \rVert_{C_{x}^{s}} \lesssim \eta l^{-(d+2) s} (\delta M_{0}(t))^{\frac{1}{p}},  \hspace{5mm} \lVert b_{j}(t) \rVert_{C_{x}^{s}} \lesssim \eta^{-1} l^{-(d+2) s} (\delta M_{0}(t))^{\frac{1}{p'}} \hspace{4mm} \forall \hspace{1mm} s \in \mathbb{N}, \label{est 83}\\
& \lVert \partial_{t} a_{j}(t) \rVert_{L_{x}^{\infty}} \lesssim \eta l^{-(d+2)} (\delta M_{0}(t))^{\frac{1}{p}}, \hspace{3mm}  \lVert \partial_{t} b_{j}(t) \rVert_{L_{x}^{\infty}} \lesssim \eta^{-1} l^{-(d+2)} (\delta M_{0}(t))^{\frac{1}{p'}}. \label{est 84} 
\end{align} 
\end{subequations} 
\end{lemma} 

\begin{proof}[Proof of Lemma \ref{Lemma 4.4}]
First, Young's inequality for convolution gives for all $t \in [0, T_{L}]$, 
\begin{align}
\lVert a_{j}(t) \rVert_{L_{x}^{\infty}} \overset{\eqref{est 77}\eqref{est 70}}{\leq}& \eta \lVert R_{l}^{j}(t) \rVert_{L_{x}^{\infty}}^{\frac{1}{p}} 
\lesssim \eta \lVert R_{0} \rVert_{C_{t,x}}^{\frac{1}{p}}.
\end{align}
Similarly, we can estimate for all $t \in [0, T_{L}]$, 
\begin{align*}
\lVert b_{j}(t) \rVert_{L_{x}^{\infty}} \overset{\eqref{est 77}\eqref{est 70}}{\leq}& \eta^{-1} \lVert R_{l}^{j}(t) \rVert_{L_{x}^{\infty}}^{\frac{1}{p'}}  \lesssim \eta^{-1} \lVert R_{0} \rVert_{C_{t,x}}^{\frac{1}{p'}}. 
\end{align*}
Next, we rely on $W^{r+ d + 1, 1}(\mathbb{T}^{d}) \hookrightarrow W^{r,\infty}(\mathbb{T}^{d})$ for all $r \in \mathbb{N}_{0}$, chain rule estimate in H$\ddot{\mathrm{o}}$lder space (e.g., \cite[Equation (130)]{BDIS15}) and the lower bound of $\lvert R_{l}^{j} (t,x) \rvert > \frac{\delta}{4d} M_{0}(t)$ in the support of $\chi_{j}$ due to \eqref{est 70} to estimate for all $s \in \mathbb{N}$, all $t \in [0, T_{L}]$, and all $\lambda \in\mathbb{N}$ sufficiently large 
\begin{align}
&\lVert a_{j} (t) \rVert_{C_{x}^{s}} \nonumber \\
\overset{ \eqref{est 70}}{\lesssim}& \eta [ \lVert R_{l}^{j} (t) \rVert_{W_{x}^{d+1, 1}}^{\frac{1}{p}} + [(\delta M_{0}(t))^{\frac{1}{p} -1} \lVert R_{l}^{j}(t) \rVert_{W_{x}^{s+d+1, 1}} + (M_{0}(t) \delta)^{\frac{1}{p} - s} \lVert R_{l}^{j} (t) \rVert_{W_{x}^{d+2, 1}}^{s} ]] \nonumber \\
\overset{\eqref{est 78}}{\lesssim}& \eta [ l^{-(d+1) \frac{1}{p}} (M_{0}(t) \delta)^{\frac{1}{p}} + [(\delta M_{0}(t))^{\frac{1}{p} -1} l^{-s-d-1} \delta M_{0}(t) + (M_{0}(t) \delta)^{\frac{1}{p} -s} l^{-(d+2) s} (\delta M_{0}(t))^{s} ]] \nonumber \\
\lesssim& \eta (\delta M_{0}(t))^{\frac{1}{p}} l^{-(d+2) s}. 
\end{align}
Similarly,  for all $s \in \mathbb{N},$ all $t \in [0, T_{L}]$, and all $\lambda \in\mathbb{N}$ sufficiently large 
\begin{align}
&\lVert b_{j} (t) \rVert_{C_{x}^{s}} \nonumber \\
\overset{ \eqref{est 70}}{\lesssim}&  \eta^{-1} [ \lVert R_{l}^{j} (t) \rVert_{W_{x}^{d+1, 1}}^{\frac{1}{p'}} + [(\delta M_{0}(t))^{\frac{1}{p'} -1} \lVert R_{l}^{j}(t) \rVert_{W_{x}^{s+d+1, 1}} + (M_{0}(t) \delta)^{\frac{1}{p'} - s} \lVert R_{l}^{j} (t) \rVert_{W_{x}^{d+2, 1}}^{s} ]]\nonumber \\
\overset{\eqref{est 78}}{\lesssim}& \eta^{-1} [ l^{-(d+1) \frac{1}{p'}} (M_{0}(t) \delta)^{\frac{1}{p'}} + [(\delta M_{0}(t))^{\frac{1}{p'} -1} l^{-s-d-1} \delta M_{0}(t) + (M_{0}(t) \delta)^{\frac{1}{p'} -s} l^{-(d+2) s} (\delta M_{0}(t))^{s} ]] \nonumber \\
\lesssim& \eta^{-1} (\delta M_{0}(t))^{\frac{1}{p'}} l^{-(d+2) s}. 
\end{align}
Finally, for all $t \in [0, T_{L}]$, and all $\lambda \in\mathbb{N}$ sufficiently large, 
\begin{align*}
 \lVert \partial_{t} a_{j} (t) \rVert_{L_{x}^{\infty}} \overset{\eqref{est 77}\eqref{est 70}}{\lesssim}&  \eta [ \lVert R_{l}^{j}(t) \rVert_{W_{x}^{d+1, 1}}^{\frac{1}{p}} + \lVert \partial_{t} R_{l}^{j}(t) \rVert_{W_{x}^{d+1, 1}} (\delta M_{0}(t))^{-1 + \frac{1}{p}} ] \\
\overset{\eqref{est 78}}{\lesssim}& \eta [ l^{- \frac{d+1}{p}} (\delta M_{0}(t))^{\frac{1}{p}} + l^{-(d+2)} \delta M_{0}(t) (\delta M_{0}(t))^{-1 + \frac{1}{p}}]
\lesssim \eta l^{-(d+2)} (\delta M_{0}(t))^{\frac{1}{p}} 
\end{align*}
and similarly 
\begin{align*}
 \lVert \partial_{t} b_{j} (t) \rVert_{L_{x}^{\infty}}& \overset{\eqref{est 77} \eqref{est 70}}{\lesssim}  \eta^{-1} [ \lVert R_{l}^{j}(t) \rVert_{W_{x}^{d+1, 1}}^{\frac{1}{p'}} + \lVert \partial_{t} R_{l}^{j}(t) \rVert_{W_{x}^{d+1, 1}} (\delta M_{0}(t))^{-1 + \frac{1}{p'}} ] \\
\overset{\eqref{est 78}}{\lesssim}& \eta^{-1} [ l^{- \frac{d+1}{p'}} (\delta M_{0}(t))^{\frac{1}{p'}} + l^{-(d+2)} \delta M_{0}(t) (\delta M_{0}(t))^{-1 + \frac{1}{p'}}] \lesssim \eta^{-1} l^{-(d+2)} (\delta M_{0}(t))^{\frac{1}{p'}}.
\end{align*}
\end{proof}

With Lemma \ref{Lemma 4.4} in hand, we can now start various necessary estimates. 
\begin{lemma}\label{Lemma 4.5}
There exist constants $C = C(p) \geq 0$ with which $\vartheta, q, w$ in \eqref{est 79}, $\vartheta_{c}$ and $q_{c}$ in \eqref{est 85} satisfy for all $t \in [0, T_{L}]$
\begin{subequations}\label{est 342}
\begin{align}
& \lVert \vartheta(t) \rVert_{L_{x}^{p}} \leq \frac{M\eta}{2} (2 \delta M_{0}(t))^{\frac{1}{p}} + \frac{C}{\lambda^{\frac{1}{p}}} \eta l^{-(d+2)} (\delta M_{0}(t))^{\frac{1}{p}}, \label{est 134}\\
& \lVert q(t) \rVert_{L_{x}^{p}} \leq C l^{-(d+1)} \delta M_{0}(t) \mu^{b} \omega^{-1}, \label{est 135}\\
& \lvert \vartheta_{c}(t) \rvert \leq C \eta \lVert R_{0} \rVert_{C_{t,x}}^{\frac{1}{p}} \mu^{-b} \hspace{2mm} \text{ and } \hspace{2mm}\lvert q_{c}(t) \rvert \leq C\lVert R_{0} \rVert_{C_{t,x}} \omega^{-1}, \label{est 151}\\
& \lVert w(t) \rVert_{L_{x}^{p'}} \leq \frac{M}{2\eta}(2 \delta M_{0}(t))^{\frac{1}{p'}} + \frac{C}{\lambda^{\frac{1}{p'}}} \eta^{-1} l^{-(d+2)} (\delta M_{0}(t))^{\frac{1}{p'}}, \label{est 153}\\
& \lVert w(t) \rVert_{W_{x}^{1, \tilde{p}}} \leq  C \eta^{-1} l^{-(d+2)} (\delta M_{0}(t))^{\frac{1}{p'}} \frac{ \lambda \mu + \nu}{\mu^{1+ \epsilon}}.\label{est 154} 
\end{align}
\end{subequations}
Furthermore, for any $k, h \in \mathbb{N}_{0}$ and $r \in [1,\infty]$, there exists a constant $C \geq 0$ with which for all $j \in \{1, \hdots, d\}$, $f_{j}$ in \eqref{est 73} satisfies for all $t \in [0, T_{L}]$
\begin{equation}\label{est 91}
\lVert \mathcal{D}^{k} D^{h} f_{j}(t) \rVert_{L_{x}^{r}} \leq C \eta^{-1} (\delta M_{0}(t))^{\frac{1}{p'}} l^{-(d+2) (k+ h + 1)} (\lambda \mu)^{k+ h + 1} \mu^{ b - \frac{d}{r}}
\end{equation} 
where $\mathcal{D}^{k}$ is defined in Definition \ref{Definition 3.1}. Consequently, there exist constants $C = C(p', \tilde{p}, N) \geq 0$ with which $w_{c}$ in \eqref{est 73} satisfies for all $t \in [0, T_{L}]$
\begin{subequations}
\begin{align}
& \lVert w_{c}(t) \rVert_{L_{x}^{p'}} \leq C\eta^{-1} (\delta M_{0}(t))^{\frac{1}{p'}} [\sum_{k=1}^{N} \left( \frac{ \lambda \mu l^{-(d+2)}}{\nu} \right)^{k} + \frac{ (\lambda \mu l^{-(d+2)})^{N+1}}{\nu^{N}} ], \label{est 136}\\
& \lVert w_{c}(t) \rVert_{W_{x}^{1, \tilde{p}}} \leq C\eta^{-1} \frac{ ( \delta M_{0}(t))^{\frac{1}{p'}} [ \lambda \mu l^{-(d+2)} + \nu]}{\mu^{1+ \epsilon}} [ \sum_{k=1}^{N} \left( \frac{ l^{-(d+2)} \lambda \mu}{\nu} \right)^{k} + \frac{ ( l^{-(d+2)} \lambda \mu)^{N+1}}{\nu^{N}} ].\label{est 93}
\end{align}
\end{subequations}
\end{lemma} 

\begin{proof}[Proof of Lemma \ref{Lemma 4.5}]
First, because $\Theta^{j}(t,x) = \varrho_{\mu}^{j} (\lambda (x- \omega t e_{j})) \psi^{j} (\nu x)$ due to \eqref{est 59} for $\nu \in \lambda \mathbb{N}$ due to \eqref{est 58}, \eqref{est 80} is applicable, allowing us to estimate 
\begin{align*}
\lVert \vartheta (t) \rVert_{L_{x}^{p}} \overset{\eqref{est 71}\eqref{est 80} \eqref{est 66}}{\leq}& \sum_{j=1}^{d} \lVert a_{j}(t) \rVert_{L_{x}^{p}} \lVert \Theta^{j}(t) \rVert_{L_{x}^{p}} + \frac{C}{\lambda^{\frac{1}{p}}} \lVert a_{j}(t) \rVert_{C_{x}^{1}} \lVert \Theta^{j} (t) \rVert_{L_{x}^{p}} \\
\overset{\eqref{est 61} \eqref{est 81} \eqref{est 83}}{\leq}& \sum_{j=1}^{d} \eta \lVert R_{l}^{j}(t) \rVert_{L_{x}^{1}}^{\frac{1}{p}} (\frac{M}{2d}) + \frac{C}{\lambda^{\frac{1}{p}}} \eta l^{-(d+2)} (\delta M_{0}(t))^{\frac{1}{p}} (\frac{M}{2d})  \\
\leq& \frac{M\eta}{2} (2\delta M_{0}(t))^{\frac{1}{p}} + \frac{C}{\lambda^{\frac{1}{p}}} \eta l^{-(d+2)} (\delta M_{0}(t))^{\frac{1}{p}}, 
\end{align*}
where in the last inequality we used the Young's inequality as follows. Because $M_{0}(t)$ is strictly increasing, the assumptions that $\lVert R_{0}(t) \rVert_{L_{x}^{1}} \leq 2 \delta M_{0}(t)$ for all $t \in [0, T_{L}]$ due to \eqref{est 78} and $\phi_{l}$ and $\varphi_{l}$ both have mass one imply 
\begin{equation}\label{est 87}
\lVert R_{l}(t) \rVert_{L_{x}^{1}} \overset{\eqref{est 86}}{\leq}  \sup_{s \in [0,t]} \lVert R_{0} (s,x) \rVert_{L_{x}^{1}} \lVert \varphi_{l} \rVert_{L_{t}^{1}} \overset{\eqref{est 78}}{\leq} 2 \delta M_{0}(t).
\end{equation}
Second, we can directly estimate for all $t \in [0, T_{L}]$ 
\begin{equation}
\lVert q(t) \rVert_{L_{x}^{p}} \overset{\eqref{est 79} \eqref{est 70}}{\leq} \sum_{j=1}^{d} \lVert R_{l}^{j}(t) \rVert_{L_{x}^{\infty}} \lVert Q^{j} (t) \rVert_{L_{x}^{p}} \overset{\eqref{est 61} \eqref{est 78}}{\leq} Cl^{-(d+1)} \delta M_{0}(t)\mu^{b} \omega^{-1}.
\end{equation}
Third, for all $t \in [0, T_{L}]$ we can compute using H$\ddot{\mathrm{o}}$lder's inequality 
\begin{equation}
\lvert \vartheta_{c}(t) \rvert \overset{\eqref{est 85}}{\leq} \lVert \vartheta(t) \rVert_{L_{x}^{1}} 
\overset{\eqref{est 71}}{\leq} \sum_{j=1}^{d} \lVert a_{j}(t) \rVert_{L_{x}^{\infty}} \lVert \Theta^{j}(t) \rVert_{L_{x}^{1}} 
\overset{\eqref{est 82} \eqref{est 62}}{\leq} C \eta \lVert R_{0} \rVert_{C_{t,x}}^{\frac{1}{p}}\mu^{-b}.
\end{equation} 
Similarly via \eqref{est 85}, \eqref{est 71}, \eqref{est 82}, and \eqref{est 62}, for all $t \in [0, T_{L}]$, by H$\ddot{\mathrm{o}}$lder's inequality 
\begin{align}
\lvert q_{c}(t) \rvert \leq \lVert q(t) \rVert_{L_{x}^{1}} \leq \sum_{j=1}^{d} \lVert a_{j}(t) \rVert_{L_{x}^{\infty}} \lVert b_{j}(t) \rVert_{L_{x}^{\infty}} \lVert Q^{j}(t) \rVert_{L_{x}^{1}} \leq C  \lVert R_{0} \rVert_{C_{t,x}}\omega^{-1}.
\end{align}
Fourth, similarly to \eqref{est 134}, because $W^{j}(t) = \tilde{\varrho}_{\mu}^{j} (\lambda (x - \omega t e_{j})) \psi^{j} (\nu x) e_{j}$ due to \eqref{est 59} for $\nu \in \lambda \mathbb{N}$ due to \eqref{est 58}, \eqref{est 80} is applicable, allowing us to estimate for all $t \in [0, T_{L}]$, 
\begin{align*}
\lVert w(t) \rVert_{L_{x}^{p'}} \overset{\eqref{est 71}\eqref{est 80} \eqref{est 66}}{\leq}& \sum_{j=1}^{d} \lVert b_{j}(t) \rVert_{L_{x}^{p'}} \lVert W^{j}(t) \rVert_{L_{x}^{p'}} + \frac{C}{\lambda^{\frac{1}{p'}}} \lVert b_{j}(t) \rVert_{C_{x}^{1}} \lVert W^{j}(t) \rVert_{L_{x}^{p'}} \\
\overset{\eqref{est 81} \eqref{est 61} \eqref{est 83}}{\lesssim}& \sum_{j=1}^{d} \eta^{-1} \lVert R_{l}^{j} (t) \rVert_{L_{x}^{1}}^{\frac{1}{p'}} (\frac{M}{2d}) + \frac{C}{\lambda^{\frac{1}{p'}}} \eta^{-1} l^{-(d+2)} (\delta M_{0}(t))^{\frac{1}{p'}} (\frac{M}{2d}) \\
\overset{\eqref{est 78}}{\leq}& \frac{M}{2\eta} (2 \delta M_{0}(t))^{\frac{1}{p'}} + \frac{C}{\lambda^{\frac{1}{p'}}} \eta^{-1} l^{-(d+2)} ( \delta M_{0}(t))^{\frac{1}{p'}}
\end{align*}
by Young's inequality for convolution similarly to \eqref{est 87}. Fifth, we compute for all $t \in [0, T_{L}]$ 
\begin{equation}\label{est 359}
\lVert w(t) \rVert_{W_{x}^{1,\tilde{p}}} \overset{\eqref{est 71}}{\lesssim} \sum_{j=1}^{d}  \lVert b_{j}(t) \rVert_{C_{x}^{1}} \lVert W^{j}(t) \rVert_{W_{x}^{1,\tilde{p}}} 
\overset{\eqref{est 83} \eqref{est 63} }{\leq}C \eta^{-1} l^{-(d+2)} (\delta M_{0}(t))^{\frac{1}{p'}} \frac{ \lambda \mu + \nu}{\mu^{1+ \epsilon}}.
\end{equation} 
Sixth, we compute for all $t \in [0, T_{L}]$ 
\begin{align}
\lVert \mathcal{D}^{k} D^{h} f_{j}(t) \rVert_{L_{x}^{r}} \overset{\eqref{est 88} \eqref{est 73}}{\lesssim}& \lVert b_{j}(t) \rVert_{C_{x}^{k+h+1}} \lVert (\tilde{\varrho}_{\mu}^{j})_{\lambda} \circ \tau_{\omega t e_{j}} \rVert_{W_{x}^{k+h+1, r}} \nonumber \\
\overset{\eqref{est 67} \eqref{est 66} \eqref{est 83}}{\lesssim}& [ \eta^{-1} l^{-(d+2) (k+h+1)} (\delta M_{0}(t))^{\frac{1}{p'}} ] \lambda^{k+h+1} \lVert \tilde{\varrho}_{\mu}^{j} \rVert_{W_{x}^{k+h+1, r}} \nonumber \\
\overset{\eqref{est 68}\eqref{est 56} }{=}& C \eta^{-1} (\delta M_{0}(t))^{\frac{1}{p'}} l^{-(d+2) (k+ h + 1)} (\lambda \mu)^{k+ h + 1} \mu^{ b - \frac{d}{r}}.
\end{align}
Seventh, we compute for all $t \in [0, T_{L}]$ 
\begin{align} 
\lVert w_{c}(t) \rVert_{L_{x}^{p'}} \overset{\eqref{est 73}\eqref{est 90}}{\leq}& \sum_{j=1}^{d} C_{d, p', N} \lVert \psi \rVert_{L^{\infty}} \left( \sum_{k=0}^{N-1} \nu^{-k-1} \lVert \mathcal{D}^{k} f_{j}(t) \rVert_{L_{x}^{p'}} + \nu^{-N} \lVert \mathcal{D}^{N} f_{j}(t) \rVert_{L_{x}^{p'}}\right) \nonumber \\
\overset{\eqref{est 91}}{\lesssim}&  \sum_{k=0}^{N-1} \nu^{-k-1}\eta^{-1} ( \delta M_{0}(t))^{\frac{1}{p'}} l^{-(d+2) (k+1)} (\lambda \mu)^{k+1} \mu^{b - \frac{d}{p'}} \nonumber\\
& \hspace{30mm} + \nu^{-N} \eta^{-1} (\delta M_{0}(t))^{\frac{1}{p'}} l^{-(d+2) (N+1)} (\lambda \mu)^{N+1} \mu^{b - \frac{d}{p'}}\nonumber \\
\overset{\eqref{est 55}}{\leq}& C\eta^{-1}(\delta M_{0}(t))^{\frac{1}{p'}} [ \sum_{k=1}^{N} \left( \frac{ \lambda \mu l^{-(d+2)}}{\nu} \right)^{k} + \frac{ (\lambda \mu l^{-(d+2)})^{N+1}}{\nu^{N}}].\label{est 92}
\end{align}
Eighth, by definition, $\lVert w_{c}(t) \rVert_{W_{x}^{1, \tilde{p}}} = \lVert w_{c}(t) \rVert_{L_{x}^{\tilde{p}}} + \lVert D w_{c}(t) \rVert_{L_{x}^{\tilde{p}}}$. First, 
\begin{equation}\label{est 94}
b - \frac{d}{\tilde{p}} \overset{\eqref{est 55}}{=} \frac{d}{p'} - \frac{d}{\tilde{p}} = - ( \frac{d}{\tilde{p}} - \frac{d}{p'} - 1) - 1 \overset{\eqref{est 60}}{<} - \epsilon - 1 
\end{equation} 
and $\mu \gg \lambda$ due to \eqref{est 58}  and therefore $\mu^{b - \frac{d}{\tilde{p}}} \leq \mu^{-\epsilon - 1}$. Therefore, we can repeat the computation in \eqref{est 92} with $p'$ replaced by $\tilde{p}$ and use the fact that $\lambda \mu l^{-(d+2)} + \nu \gg 1$ de to \eqref{est 58} and \eqref{est 95} to bound $\lVert w_{c}(t) \rVert_{L_{x}^{\tilde{p}}} $ by the r.h.s. of \eqref{est 93}. On the other hand, for all $t \in [0, T_{L}]$ we can compute using \eqref{est 96} 
\begin{align}
\lVert Dw_{c}(t) \rVert_{L_{x}^{\tilde{p}}}& 
\overset{\eqref{est 90}}{\lesssim} \sum_{j=1}^{d} \lVert \psi^{j} \rVert_{L^{\infty}} \left( \sum_{k=0}^{N-1} \nu^{-k-1} \lVert \mathcal{D}^{k} D f_{j}(t) \rVert_{L_{x}^{\tilde{p}}} + \nu^{-N} \lVert \mathcal{D}^{N} D f_{j} (t) \rVert_{L_{x}^{\tilde{p}}}\right) \nonumber \\
&+ \nu \lVert D \psi^{j} \rVert_{L^{\infty}} \left( \sum_{k=0}^{N-1}\nu^{-k-1} \lVert \mathcal{D}^{k} f_{j}(t)\rVert_{L_{x}^{\tilde{p}}} + \nu^{-N} \lVert \mathcal{D}^{N} f_{j}(t) \rVert_{L_{x}^{\tilde{p}}}\right) \nonumber \\
\overset{\eqref{est 91}\eqref{est 94}}{\leq}& C \eta^{-1}  \frac{ ( \delta M_{0}(t))^{\frac{1}{p'}} [ \lambda \mu l^{-(d+2)} + \nu]}{\mu^{1+ \epsilon}} [ \sum_{k=1}^{N} \left( \frac{ l^{-(d+2)} \lambda \mu}{\nu} \right)^{k} + \frac{ ( l^{-(d+2)} \lambda \mu)^{N+1}}{\nu^{N}} ].
\end{align}
\end{proof}
 
Next, we define the new defect $R_{1}$ using \eqref{est 29} as follows: 
\begin{align}
& - \text{div} R_{1} \label{est 99}\\
\overset{\eqref{est 29} \eqref{est 76}}{=}& - \text{div}R_{0} \nonumber \\
&+ \partial_{t} [ \vartheta + \vartheta_{c} + q + q_{c} ] + \text{div} ( \theta_{0} (w + w_{c} )) + \text{div} ( ( \vartheta + q) u_{0}) + \text{div} ((\vartheta + q) (w + w_{c})) \nonumber \\ 
&+ (\vartheta_{c} + q_{c})\underbrace{\text{div} (u_{0} + w + w_{c})}_{\overset{\eqref{est 76}}{=}\text{div} u_{1} \overset{\eqref{est 98}}{=}0} + \text{div} (z(w + w_{c} )) - \Delta (\vartheta + q) \nonumber \\
=& \underbrace{\partial_{t} (q+ q_{c}) + \text{div} (\vartheta w - R_{l})}_{\text{div} R^{\text{time,1}} + \text{div} R^{\text{quadr}} + \text{div} R^{\chi}} + \underbrace{\partial_{t} (\vartheta + \vartheta_{c}) + \text{div} (\theta_{0} w + \vartheta u_{0}) - \Delta (\vartheta + q)}_{\text{div} R^{\text{time,2}} + \text{div} R^{\text{lin}}} \nonumber \\
&+ \underbrace{\text{div} (q(u_{0} + w))}_{\text{div} R^{q}} + \underbrace{\text{div} ((\theta_{0} + \vartheta + q + z) w_{c})}_{\text{div} R^{\text{corr,1}}} +   \underbrace{\text{div} (zw)}_{\text{div} R^{\text{corr,2}}}  + \underbrace{\text{div} (R_{l} - R_{0})}_{\text{div} R^{\text{moll}}}\nonumber 
\end{align}
for $R^{\text{time,1}}, R^{\text{quadr}}, R^{\chi}, R^{\text{time,2}}, R^{\text{lin}}, R^{q}, R^{\text{corr,1}}, R^{\text{corr,2}}$, and $R^{\text{moll}}$ to be defined subsequently. Thus, we have defined 
\begin{equation}\label{est 150}
-R_{1} \triangleq R^{\text{time,1}} + R^{\text{quadr}} + R^{\chi} + R^{\text{time,2}} + R^{\text{lin}} + R^{q} + R^{\text{corr,1}} + R^{\text{corr,2}} + R^{\text{moll}}. 
\end{equation} 

\subsubsection{Estimates on $\partial_{t} (q+ q_{c}) + \text{div} (\vartheta w - R_{l}) = \text{div} R^{\text{time,1}} + \text{div} R^{\text{quadr}} + \text{div} R^{\chi}$ in \eqref{est 99}}\label{Section 4.1.1} 
We first observe that 
\begin{align}
\vartheta(t,x) w(t,x) \overset{\eqref{est 71}}{=}& (\sum_{j=1}^{d} a_{j}(t,x) \Theta^{j}(t,x))( \sum_{k=1}^{d} b_{k}(t,x) W^{k}(t,x)) \nonumber \\
\overset{\eqref{est 64}}{=}& \sum_{j=1}^{d} a_{j}(t,x) b_{j}(t,x) \Theta^{j}(t,x) W^{j}(t,x). \label{est 102}
\end{align}
Relying on \eqref{est 102} gives us 
\begin{equation}\label{est 106}
\text{div} (\vartheta w) = \sum_{j=1}^{d} a_{j}b_{j} \text{div} (\Theta^{j}W^{j}) + \nabla (a_{j}b_{j}) \cdot \Theta^{j}W^{j}.
\end{equation} 
On the other hand, by setting 
\begin{equation}\label{est 105}
R^{\chi} \triangleq - \sum_{j=1}^{d} (1- \chi_{j}^{2}) R_{l}^{j} e_{j}, 
\end{equation} 
we see that 
\begin{equation}\label{est 107}
-\text{div} R_{l} \overset{\eqref{est 104}}{=} - \text{div} [ \sum_{j=1}^{d} (1- \chi_{j}^{2}) R_{l}^{j} e_{j} + \chi_{j}^{2}R_{l}^{j} e_{j} ] \overset{\eqref{est 105}\eqref{est 101}}{=} \text{div} R^{\chi} - \sum_{j=1}^{d} \nabla (a_{j}b_{j}) \cdot e_{j}.
\end{equation}
These lead us to 
\begin{align}
\text{div} (\vartheta w - R_{l}) \overset{\eqref{est 106} \eqref{est 107}}{=}&\sum_{j=1}^{d} a_{j}b_{j} \text{div} (\Theta^{j} W^{j}) + \nabla (a_{j} b_{j}) \cdot \Theta^{j} W^{j} + \text{div} R^{\chi} - \sum_{j=1}^{d} \nabla (a_{j} b_{j}) \cdot e_{j} \nonumber\\
=& \sum_{j=1}^{d} a_{j}b_{j} \text{div} (\Theta^{j} W^{j}) \nonumber \\
&  + \nabla (a_{j} b_{j}) \cdot [\Theta^{j} W^{j} - e_{j} ] - \fint_{\mathbb{T}^{d}} \nabla (a_{j} b_{j}) \cdot [\Theta^{j} W^{j} - e_{j} ] dx \nonumber \\
&+ \fint_{\mathbb{T}^{d}} \nabla (a_{j} b_{j}) \cdot [\Theta^{j}W^{j} - e_{j} ] dx + \text{div} R^{\chi}. \label{est 108}
\end{align}
On the other hand, 
\begin{align}  
 \partial_{t} ( q+ q_{c}) \overset{\eqref{est 71}}{=}& \sum_{j=1}^{d} a_{j} b_{j} \partial_{t}Q^{j}  + \partial_{t} (a_{j} b_{j}) Q^{j} -  \fint_{\mathbb{T}^{d}} \partial_{t} (a_{j}b_{j}) Q^{j} dx \nonumber\\
&+ \fint_{\mathbb{T}^{d}} \partial_{t} (a_{j}b_{j}) Q^{j} dx + q_{c}'. \label{est 109}
\end{align}
Summing \eqref{est 108}-\eqref{est 109} gives us 
\begin{subequations}\label{est 112}
\begin{align}
& \partial_{t} (q+ q_{c}) + \text{div} (\vartheta w - R_{l}) \nonumber \\
\overset{\eqref{est 108} \eqref{est 109}}{=}&  \sum_{j=1}^{d} a_{j} b_{j} [\partial_{t}Q^{j} + \text{div} (\Theta^{j}W^{j})] \label{est 110}\\
&+ [\partial_{t} (a_{j}b_{j}) Q^{j} - \fint_{\mathbb{T}^{d}} \partial_{t} (a_{j}b_{j} ) Q^{j} dx] \\
&+ \nabla (a_{j}b_{j}) \cdot [\Theta^{j}W^{j} - e_{j} ] - \fint_{\mathbb{T}^{d}} \nabla (a_{j}b_{j}) \cdot [\Theta^{j} W^{j} - e_{j} ]dx  + \text{div} R^{\chi} \\
&+ \fint_{\mathbb{T}^{d}} \partial_{t} (a_{j}b_{j})Q^{j}dx + \fint_{\mathbb{T}^{d}} \nabla (a_{j}b_{j}) \cdot [\Theta^{j} W^{j} - e_{j} ] dx + q_{c}'.\label{est 111} 
\end{align}
\end{subequations}
Here, \eqref{est 110} vanishes due to \eqref{est 65} while \eqref{est 111} also vanishes due to 
\begin{align}
& \fint_{\mathbb{T}^{d}} \partial_{t} (a_{j}b_{j})Q^{j}dx + \fint_{\mathbb{T}^{d}} \nabla (a_{j}b_{j}) \cdot [\Theta^{j} W^{j} - e_{j} ] dx + q_{c}' \nonumber\\
\overset{\eqref{est 85}\eqref{est 71}}{=}& \fint_{\mathbb{T}^{d}} \nabla (a_{j}b_{j}) \cdot \Theta^{j}W^{j} dx - \fint_{\mathbb{T}^{d}} a_{j} b_{j} \partial_{t}Q^{j} dx 
\overset{\eqref{est 65}}{=}  0. \label{est 278}
\end{align}
Thus, we conclude from \eqref{est 112} that  
\begin{align}
 \partial_{t} (q+ q_{c}) + \text{div} (\vartheta w - R_{l}) 
=&  \sum_{j=1}^{d} [\partial_{t} (a_{j}b_{j}) Q^{j} - \fint_{\mathbb{T}^{d}} \partial_{t} (a_{j}b_{j} ) Q^{j} dx] \label{est 118}\\
+& \nabla (a_{j}b_{j}) \cdot [\Theta^{j}W^{j} - e_{j} ] - \fint_{\mathbb{T}^{d}} \nabla (a_{j}b_{j}) \cdot [\Theta^{j} W^{j} - e_{j} ]dx + \text{div} R^{\chi},\nonumber 
\end{align}
where we can further compute 
\begin{align}
 \nabla (a_{j} b_{j}) \cdot [\Theta^{j} W^{j} - e_{j} ]\overset{\eqref{est 59}}{=}& \nabla (a_{j} b_{j}) \cdot [ ( \varrho_{\mu}^{j} \tilde{\varrho}_{\mu}^{j})_{\lambda} \circ \tau_{\omega t e_{j}} (\psi_{\nu}^{j})^{2} - 1] e_{j} \label{est 117}\\
=& \partial_{j} ( a_{j} b_{j}) [ ( \varrho_{\mu}^{j} \tilde{\varrho}_{\mu}^{j})_{\lambda} \circ \tau_{\omega t e_{j}} (( \psi^{j})^{2} -1)_{\nu} + ( \varrho_{\mu}^{j} \tilde{\varrho}_{\mu}^{j} -1)_{\lambda} \circ \tau_{\omega t e_{j}}]. \nonumber
\end{align}
Thus, we define 
\begin{equation}\label{est 114}
R^{\text{time,1}} \triangleq \sum_{j=1}^{d} \mathcal{D}^{-1} ( \partial_{t} (a_{j} b_{j} )Q^{j} - \fint_{\mathbb{T}^{d}} \partial_{t} (a_{j} b_{j}) Q^{j} dx)
\end{equation} 
so that because $\mathcal{D}^{-1} = \nabla \Delta^{-1}$ according to \eqref{est 88} we obtain 
\begin{equation}\label{est 119}
\text{div} R^{\text{time,1}} \overset{\eqref{est 114}\eqref{est 88}}{=} \sum_{j=1}^{d} \partial_{t} (a_{j}b_{j}) Q^{j} - \fint_{\mathbb{T}^{d}} \partial_{t} (a_{j} b_{j} )Q^{j} dx. 
\end{equation} 
Additionally, we define 
\begin{equation}\label{est 115}
R^{\text{quadr}} \triangleq R^{\text{quadr,1}} + R^{\text{quadr,2}},
\end{equation} 
where 
\begin{subequations}\label{est 116}
\begin{align}
& R^{\text{quadr,1}} \triangleq \sum_{j=1}^{d} \mathcal{R}_{1} (\partial_{j} (a_{j} b_{j}) ( \varrho_{\mu}^{j} \tilde{\varrho}_{\mu}^{j})_{\lambda} \circ \tau_{\omega t e_{j}}, ((\psi^{j})^{2} -1)_{\nu}), \\
& R^{\text{quadr,2}} \triangleq \sum_{j=1}^{d} \mathcal{R}_{1} (\partial_{j} (a_{j} b_{j}), (\varrho_{\mu}^{j} \tilde{\varrho}_{\mu}^{j} -1)_{\lambda} \circ \tau_{\omega t e_{j}}); 
\end{align}
\end{subequations} 
both $R^{\text{quadr,1}}$ and $R^{\text{quadr,2}}$ are well-defined by Definition \ref{Definition 3.2} because $\fint_{\mathbb{T}^{d}} ((\psi^{j})^{2} -1)_{\nu} dx = 0$ due to \eqref{est 57} and $\fint_{\mathbb{T}^{d}} (\varrho_{\mu}^{j} \tilde{\varrho}_{\mu}^{j} - 1)_{\lambda} \circ \tau_{\omega t e_{j}} dx = 0$ due to \eqref{est 113}. We see that $R^{\text{quadr}}$ defined in \eqref{est 115}-\eqref{est 116} satisfies 
\begin{align}
& \text{div} R^{\text{quadr}} \overset{\eqref{est 115}\eqref{est 116}}{=} \text{div} ( \sum_{j=1}^{d} \mathcal{R}_{1} (\partial_{j} (a_{j} b_{j}) ( \varrho_{\mu}^{j} \tilde{\varrho}_{\mu}^{j})_{\lambda} \circ \tau_{\omega t e_{j}}, ((\psi^{j})^{2} -1)_{\nu})  \nonumber\\
& \hspace{35mm} + \mathcal{R}_{1} (\partial_{j} (a_{j} b_{j}), (\varrho_{\mu}^{j} \tilde{\varrho}_{\mu}^{j} -1)_{\lambda} \circ \tau_{\omega t e_{j}}) ) \nonumber \\
\overset{\eqref{est 74}\eqref{est 117}}{=}& \sum_{j=1}^{d} \nabla (a_{j} b_{j}) \cdot [\Theta^{j} W^{j} - e_{j}] - \fint_{\mathbb{T}^{d}} \nabla (a_{j} b_{j}) \cdot [\Theta^{j} W^{j} - e_{j} ]dx. \label{est 120}
\end{align}
Therefore, by applying \eqref{est 119} and \eqref{est 120} to \eqref{est 118} we conclude that $\partial_{t} (q+ q_{c}) + \text{div} (\vartheta w - R_{l}) = \text{div} R^{\text{time,1}} + \text{div} R^{\text{quadr}} + \text{div} R^{\chi}$ as claimed.

\begin{lemma}\label{Lemma 4.6}
There exist constants $C = C(d) \geq 0$ with which $R^{\chi}$ in \eqref{est 105}, $R^{\text{time,1}}$ in \eqref{est 114}, and $R^{\text{quadr}}$ in \eqref{est 115} satisfy for all $t \in [0, T_{L}]$
\begin{subequations}
\begin{align}
& \lVert R^{\chi}(t) \rVert_{L_{x}^{1}} \leq \frac{\delta M_{0}(t)}{2}, \label{est 157}\\
& \lVert R^{\text{time,1}}(t) \rVert_{L_{x}^{1}} \leq C\omega^{-1} l^{-(d+2)} \max\{ (\delta M_{0}(t))^{\frac{1}{p}} \lVert R_{0} \rVert_{C_{t,x}}^{\frac{1}{p'}}, (\delta M_{0}(t))^{\frac{1}{p'}} \lVert R_{0} \rVert_{C_{t,x}}^{\frac{1}{p}} \}, \label{est 155} \\
& \lVert R^{\text{quadr}}(t) \rVert_{L_{x}^{1}} \leq C \left( \frac{\lambda \mu}{\nu} + \frac{1}{\lambda} \right) l^{-(d+2) 2} \max\{ (\delta M_{0}(t))^{\frac{1}{p}} \lVert R_{0} \rVert_{C_{t,x}}^{\frac{1}{p'}}, (\delta M_{0}(t))^{\frac{1}{p'}} \lVert R_{0} \rVert_{C_{t,x}}^{\frac{1}{p}} \}.\label{est 156} 
\end{align}
\end{subequations}
\end{lemma} 

\begin{proof}[Proof of Lemma \ref{Lemma 4.6}]
First, we compute
\begin{equation}
\lVert R^{\chi} (t) \rVert_{L_{x}^{1}} \overset{\eqref{est 105}}{\leq} \sum_{j=1}^{d} \int_{supp (1- \chi_{j}^{2})(t)} \lvert R_{l}^{j} (t,x) \rvert dx  
\overset{\eqref{est 70}}{\leq} \frac{\delta M_{0}(t)}{2}.
\end{equation}
Second, we compute 
\begin{align*}
\lVert R^{\text{time,1}}(t) \rVert_{L_{x}^{1}} &\overset{\eqref{est 114}\eqref{est 121}}{\lesssim}\sum_{j=1}^{d} (\lVert \partial_{t} a_{j}(t) \rVert_{L_{x}^{\infty}} \lVert b_{j}(t) \rVert_{L_{x}^{\infty}} + \lVert a_{j}(t) \rVert_{L_{x}^{\infty}} \lVert \partial_{t} b_{j} (t) \rVert_{L_{x}^{\infty}}) \lVert Q^{j}(t) \rVert_{L_{x}^{1}}  \\
\overset{\eqref{est 82} \eqref{est 84} \eqref{est 62}}{\lesssim}&  [ \eta l^{-(d+2)} (\delta M_{0}(t))^{\frac{1}{p}} \eta^{-1} \lVert R_{0} \rVert_{C_{t,x}}^{\frac{1}{p'}}  + \eta \lVert R_{0} \rVert_{C_{t,x}}^{\frac{1}{p}} \eta^{-1} l^{-(d+2)} (\delta M_{0}(t))^{\frac{1}{p'}}] (\frac{M}{\omega}) \\
\leq& C \omega^{-1} l^{-(d+2)} \max\{ (\delta M_{0}(t))^{\frac{1}{p}} \lVert R_{0} \rVert_{C_{t,x}}^{\frac{1}{p'}}, (\delta M_{0}(t))^{\frac{1}{p'}} \lVert R_{0} \rVert_{C_{t,x}}^{\frac{1}{p}} \}.
\end{align*}
Third, we compute 
\begin{align}
& \lVert R^{\text{quadr,1}}(t) \rVert_{L_{x}^{1}} \nonumber \\
\overset{\eqref{est 116} \eqref{est 90}}{\lesssim}& \nu^{-1} \sum_{j=1}^{d} \lVert \partial_{j} (a_{j} b_{j}) (t)\rVert_{L_{x}^{\infty}} \lVert ( \varrho_{\mu}^{j} \tilde{\varrho}_{\mu}^{j})_{\lambda} \circ \tau_{\omega t e_{j}} \rVert_{L_{x}^{1}} + \lVert \partial_{j} (a_{j} b_{j})(t) \rVert_{C_{x}^{1}} \lVert (\varrho_{\mu}^{j} \tilde{\varrho}_{\mu}^{j})_{\lambda} \circ \tau_{\omega t e_{j}} \rVert_{W_{x}^{1,1}} \nonumber\\
\overset{\eqref{est 66}\eqref{est 67}}{\lesssim}& \nu^{-1} \sum_{j=1}^{d} (\lVert a_{j} (t) \rVert_{C_{x}^{2}} \lVert b_{j}(t) \rVert_{L_{x}^{\infty}} + \lVert a_{j}(t) \rVert_{L_{x}^{\infty}} \lVert b_{j}(t) \rVert_{C_{x}^{2}}) ( \lVert \varrho_{\mu}^{j} \tilde{\varrho}_{\mu}^{j} \rVert_{L_{x}^{1}} + \lambda \lVert \varrho_{\mu}^{j} \tilde{\varrho}_{\mu}^{j} \rVert_{W_{x}^{1,1}}) \nonumber\\
\overset{\eqref{est 82}\eqref{est 83}}{\lesssim}& \nu^{-1} ( \eta ( \delta M_{0}(t))^{\frac{1}{p}} l^{-(d+2) 2} \eta^{-1} \lVert R_{0} \rVert_{C_{t,x}}^{\frac{1}{p'}}  + \eta \lVert R_{0} \rVert_{C_{t,x}}^{\frac{1}{p}} \eta^{-1} (\delta M_{0}(t))^{\frac{1}{p'}} l^{-(d+2) 2})  \nonumber\\
& \times [ \lVert \varrho_{\mu}^{j} \rVert_{L^{2}} \lVert \tilde{\varrho}_{\mu}^{j} \rVert_{L^{2}} + \lambda( \lVert \varrho_{\mu}^{j} \rVert_{W^{1,2}} \lVert \tilde{\varrho}_{\mu}^{j} \rVert_{L^{2}} + \lVert \varrho_{\mu}^{j} \rVert_{L^{2}} \lVert \tilde{\varrho}_{\mu}^{j} \rVert_{W^{1,2}}]  \nonumber\\
\overset{\eqref{est 68} \eqref{est 55}}{\lesssim}& \left( \frac{\lambda \mu}{\nu} \right) l^{-(d+2) 2} \max\{ (\delta M_{0}(t))^{\frac{1}{p}} \lVert R_{0} \rVert_{C_{t,x}}^{\frac{1}{p'}}, (\delta M_{0}(t))^{\frac{1}{p'}} \lVert R_{0} \rVert_{C_{t,x}}^{\frac{1}{p}} \}.  \label{est 122}
\end{align}
On the other hand, 
\begin{align}
& \lVert R^{\text{quadr,2}}(t) \rVert_{L_{x}^{1}} \nonumber \\
\overset{\eqref{est 116}\eqref{est 89} }{\lesssim}& \sum_{j=1}^{d} \lVert ( \varrho_{\mu}^{j} \tilde{\varrho}_{\mu}^{j} -1) \circ \tau_{\omega t e_{j}} \rVert_{L_{x}^{1}} ( \lambda^{-1} \lVert \partial_{j} (a_{j} b_{j})(t) \rVert_{L_{x}^{\infty}} + \lambda^{-1} \lVert \mathcal{D} \partial_{j} (a_{j} b_{j})(t) \rVert_{L_{x}^{\infty}})  \nonumber\\
\overset{\eqref{est 67}}{\lesssim}& \sum_{j=1}^{d} \lVert \varrho_{\mu}^{j} \tilde{\varrho}_{\mu}^{j} - 1 \rVert_{L^{1}} \lambda^{-1} ( \lVert a_{j} (t) \rVert_{C_{x}^{2}} \lVert b_{j}(t) \rVert_{L_{x}^{\infty}} + \lVert a_{j}(t) \rVert_{L_{x}^{\infty}} \lVert b_{j}(t) \rVert_{C_{x}^{2}}) \nonumber\\
\overset{\eqref{est 82}\eqref{est 83}}{\lesssim}& \sum_{j=1}^{d} (\lVert \varrho_{\mu}^{j} \rVert_{L^{2}} \lVert \tilde{\varrho}_{\mu}^{j} \rVert_{L^{2}} + 1) \lambda^{-1}l^{-(d+2) 2} \max\{ (\delta M_{0}(t))^{\frac{1}{p}} \lVert R_{0} \rVert_{C_{t,x}}^{\frac{1}{p'}}, (\delta M_{0}(t))^{\frac{1}{p'}} \lVert R_{0} \rVert_{C_{t,x}}^{\frac{1}{p}} \}   \nonumber\\
\overset{\eqref{est 68}\eqref{est 55}}{\lesssim}& \lambda^{-1} l^{-(d+2) 2} \max\{ (\delta M_{0}(t))^{\frac{1}{p}} \lVert R_{0} \rVert_{C_{t,x}}^{\frac{1}{p'}}, (\delta M_{0}(t))^{\frac{1}{p'}} \lVert R_{0} \rVert_{C_{t,x}}^{\frac{1}{p}} \}.\label{est 123}
\end{align}
Therefore, we conclude \eqref{est 156} from \eqref{est 115}, \eqref{est 122}, and \eqref{est 123}. 
\end{proof}

\subsubsection{Estimates on $\partial_{t} (\vartheta + \vartheta_{c}) + \text{div} (\theta_{0} w + \vartheta u_{0}) - \Delta (\vartheta + q)= \text{div} R^{\text{time,2}} + \text{div} R^{\text{lin}}$ in \eqref{est 99}}\label{Section 4.1.2}
We compute 
\begin{align}
 \partial_{t} ( \vartheta + \vartheta_{c}) +& \text{div} (\theta_{0} w + \vartheta u_{0}) - \Delta (\vartheta + q) \overset{\eqref{est 71}}{=} \sum_{j=1}^{d} a_{j} \partial_{t} \Theta^{j} - \fint_{\mathbb{T}^{d}} a_{j} \partial_{t} \Theta^{j} dx \nonumber \\
&+ \partial_{t} a_{j} \Theta^{j} - \fint_{\mathbb{T}^{d}} \partial_{t} a_{j} \Theta^{j} dx + \text{div} (\theta_{0} w + \vartheta u_{0}) - \Delta (\vartheta + q) \label{est 127}
\end{align}
due to 
\begin{align*}
\sum_{j=1}^{d} \fint_{\mathbb{T}^{d}} a_{j} \partial_{t} \Theta^{j} dx + \fint_{\mathbb{T}^{d}} \partial_{t} a_{j} \Theta^{j}dx + \vartheta_{c}' 
\overset{\eqref{est 85}}{=} \partial_{t} \sum_{j=1}^{d} \fint_{\mathbb{T}^{d}} a_{j} \Theta^{j} dx - \fint_{\mathbb{T}^{d}} \vartheta' dx 
\overset{\eqref{est 71}}{=} 0.
\end{align*}
Thus, we define 
\begin{subequations}
\begin{align}
&R^{\text{lin}} \triangleq \sum_{j=1}^{d} \mathcal{D}^{-1} ( ( \partial_{t} a_{j}) \Theta^{j} - \fint_{\mathbb{T}^{d}} (\partial_{t} a_{j}) \Theta^{j} dx) + \theta_{0} w + \vartheta u_{0} - \nabla (\vartheta + q), \label{est 124}\\
& R^{\text{time,2}} \triangleq -\lambda \omega \sum_{j=1}^{d} \mathcal{R}_{N} (a_{j} (\partial_{j} \varrho_{\mu}^{j})_{\lambda} \circ \tau_{\omega t e_{j}}, \psi_{\nu}^{j}) \label{est 125}
\end{align}
\end{subequations}
where \eqref{est 125} is well-defined because $\fint_{\mathbb{T}^{d}} \psi_{\nu}^{j}(x) dx = 0$ due to \eqref{est 57}. It follows that 
\begin{subequations}
\begin{align}
\text{div} R^{\text{lin}} \overset{\eqref{est 124}\eqref{est 88}}{=}& \sum_{j=1}^{d} (\partial_{t} a_{j}) \Theta^{j} - \fint_{\mathbb{T}^{d}} (\partial_{t} a_{j}) \Theta^{j}dx + \text{div} (\theta_{0} w + \vartheta u_{0}) - \Delta (\vartheta + q) \label{est 128}, \\
\text{div} R^{\text{time,2}} \overset{\eqref{est 125}\eqref{est 74}}{=}& -\lambda \omega \sum_{j=1}^{d} [ a_{j} (\partial_{j} \varrho_{\mu}^{j})_{\lambda} \circ \tau_{\omega t e_{j}} \psi_{\nu}^{j} - \fint_{\mathbb{T}^{d}} a_{j} (\partial_{j} \varrho_{\mu}^{j})_{\lambda} \circ \tau_{\omega t e_{j}} \psi_{\nu}^{j} dx] \nonumber\\
\overset{\eqref{est 126}}{=}& \sum_{j=1}^{d} (a_{j} \partial_{t} \Theta^{j} - \fint_{\mathbb{T}^{d}} a_{j} \partial_{t} \Theta^{j} dx).\label{est 129}
\end{align}
\end{subequations}
Considering \eqref{est 128}-\eqref{est 129} in \eqref{est 127}, we conclude that $\partial_{t} (\vartheta + \vartheta_{c}) + \text{div} (\theta_{0} w + \vartheta u_{0}) - \Delta (\vartheta + q)= \text{div} R^{\text{time,2}} + \text{div} R^{\text{lin}}$. 

\begin{lemma}\label{Lemma 4.7} 
There exist constants $C \geq 0$ with which $R^{\text{lin}}$ in \eqref{est 124} and $R^{\text{time,2}}$ in \eqref{est 125} satisfy for all $t \in [0, T_{L}]$ 
\begin{subequations}
\begin{align}
&\lVert R^{\text{lin}}(t) \rVert_{L_{x}^{1}} \leq C( \mu^{-a} \lVert \theta_{0} \rVert_{C_{t,x}} \eta^{-1} \lVert R_{0} \rVert_{C_{t,x}}^{\frac{1}{p'}} + \mu^{-b}  \eta [ \lVert R_{0} \rVert_{C_{t,x}}^{\frac{1}{p}} \lVert u_{0} \rVert_{C_{t,x}} +  (\delta M_{0} (t))^{\frac{1}{p}} l^{-(d+2)}]\nonumber\\
& \hspace{20mm} + [\eta \lVert R_{0} \rVert_{C_{t,x}}^{\frac{1}{p}} \mu^{-b} + \lVert R_{0} \rVert_{C_{t,x}} \omega^{-1} ] [\lambda \mu + \nu] \nonumber \\
& \hspace{20mm} + l^{-(d+2)} \max\{ (\delta M_{0}(t))^{\frac{1}{p}} \lVert R_{0} \rVert_{C_{t,x}}^{\frac{1}{p'}}, (\delta M_{0}(t))^{\frac{1}{p'}} \lVert R_{0} \rVert_{C_{t,x}}^{\frac{1}{p}} \} \omega^{-1}), \label{est 159} \\
&\lVert R^{\text{time,2}}(t) \rVert_{L_{x}^{1}} \leq C \left(\frac{\omega}{\mu^{b}}\right) \eta (\delta M_{0}(t))^{\frac{1}{p}} \left(\sum_{k=1}^{N} \left( \frac{\lambda \mu}{\nu} \right)^{k} l^{-(d+2) (k-1)} + \frac{ (\lambda \mu)^{N+1}}{\nu^{N}} l^{-(d+2) N} \right).\label{est 158}
\end{align}
\end{subequations}
\end{lemma}

\begin{proof}[Proof of Lemma \ref{Lemma 4.7}]
First, we compute 
\begin{align}
& \lVert \sum_{j=1}^{d} \mathcal{D}^{-1} (( \partial_{t} a_{j})\Theta^{j} - \fint_{\mathbb{T}^{d}} (\partial_{t} a_{j}) \Theta^{j} dx )(t) \rVert_{L_{x}^{1}}\nonumber \\
\overset{\eqref{est 121}}{\lesssim}& \sum_{j=1}^{d} \lVert ( \partial_{t} a_{j} (t)) \Theta^{j}(t) - \fint_{\mathbb{T}^{d}} (\partial_{t} a_{j} (t)) \Theta^{j}(t) dx \rVert_{L_{x}^{1}} 
\overset{\eqref{est 84} \eqref{est 62}}{\lesssim} \eta l^{-(d+2)} (\delta M_{0}(t))^{\frac{1}{p}} \mu^{-b}. \label{est 130}
\end{align}
On the other hand, 
\begin{align}
 \lVert ( \theta_{0} w + \vartheta u_{0} )(t) \rVert_{L_{x}^{1}} 
\overset{\eqref{est 71}}{\leq}& \sum_{j=1}^{d} \lVert \theta_{0}(t) \rVert_{L_{x}^{\infty}} \lVert b_{j}(t) \rVert_{L_{x}^{\infty}} \lVert W^{j}(t) \rVert_{L_{x}^{1}} + \lVert a_{j}(t) \rVert_{L_{x}^{\infty}} \lVert \Theta^{j}(t)\rVert_{L_{x}^{1}} \lVert u_{0}(t) \rVert_{L_{x}^{\infty}}  \nonumber\\
\overset{\eqref{est 82} \eqref{est 62}}{\lesssim}&  \mu^{-a}   \lVert \theta_{0} \rVert_{C_{t,x}} \eta^{-1} \lVert R_{0} \rVert_{C_{t,x}}^{\frac{1}{p'}} + \mu^{-b}  \lVert u_{0} \rVert_{C_{t,x}} \eta \lVert R_{0} \rVert_{C_{t,x}}^{\frac{1}{p}}. \label{est 131}
\end{align}
Finally, we estimate   
\begin{align}
& \lVert \nabla (\vartheta + q)(t) \rVert_{L_{x}^{1}} \label{est 142}\\
&\overset{\eqref{est 71}}{\lesssim} \sum_{j=1}^{d} \lVert \nabla a_{j}(t) \rVert_{L_{x}^{\infty}} \lVert \Theta^{j}(t) \rVert_{L_{x}^{1}} + \lVert a_{j}(t) \rVert_{L_{x}^{\infty}} \lVert \nabla \Theta^{j}(t) \rVert_{L_{x}^{1}} \nonumber \\
& \hspace{5mm} + ( \lVert \nabla a_{j}(t) \rVert_{L_{x}^{\infty}} \lVert b_{j}(t) \rVert_{L_{x}^{\infty}} + \lVert a_{j}(t) \rVert_{L_{x}^{\infty}} \lVert \nabla b_{j}(t) \rVert_{L_{x}^{\infty}}) \lVert Q^{j}(t) \rVert_{L_{x}^{1}} + \lVert a_{j}(t) \rVert_{L_{x}^{\infty}} \lVert b_{j}(t) \rVert_{L_{x}^{\infty}} \lVert \nabla Q^{j}(t) \rVert_{L_{x}^{1}} \nonumber\\
&\overset{\eqref{est 82} \eqref{est 83} \eqref{est 62} \eqref{est 59}}{\lesssim} \sum_{j=1}^{d} \eta l^{-(d+2)} (\delta M_{0}(t))^{\frac{1}{p}} \mu^{-b} \nonumber\\
& \hspace{5mm}+ \eta \lVert R_{0} \rVert_{C_{t,x}}^{\frac{1}{p}}  [ \lVert \nabla ((\varrho_{\mu}^{j})_{\lambda} \circ \tau_{\omega t e_{j}}) \rVert_{L_{x}^{1}} \lVert \psi_{\nu}^{j} \rVert_{L^{\infty}} + \lVert (\varrho_{\mu}^{j})_{\lambda} \circ \tau_{\omega t e_{j}} \rVert_{L_{x}^{1}} \lVert \nabla \psi_{\nu}^{j} \rVert_{L^{\infty}} ]\nonumber\\
& \hspace{5mm} + l^{-(d+2)} \max\{ (\delta M_{0}(t))^{\frac{1}{p}} \lVert R_{0} \rVert_{C_{t,x}}^{\frac{1}{p'}}, (\delta M_{0}(t))^{\frac{1}{p'}} \lVert R_{0} \rVert_{C_{t,x}}^{\frac{1}{p}} \} \omega^{-1} \nonumber\\
& \hspace{5mm}+ \lVert R_{0} \rVert_{C_{t,x}} \omega^{-1} [ \lVert \nabla (( \varrho_{\mu}^{j} \tilde{\varrho}_{\mu}^{j})_{\lambda} \circ \tau_{\omega t e_{j}}) \rVert_{L_{x}^{1}} \lVert \psi_{\nu}^{j} \rVert_{L^{\infty}}^{2}  + \lVert ( \varrho_{\mu}^{j} \tilde{\varrho}_{\mu}^{j})_{\lambda} \circ \tau_{\omega t e_{j}} \rVert_{L_{x}^{1}} \lVert \nabla \psi_{\nu}^{j} \rVert_{L^{\infty}} \lVert \psi_{\nu}^{j} \rVert_{L^{\infty}}] \nonumber\\
& \overset{\eqref{est 66} \eqref{est 67}\eqref{est 68}}{\lesssim} \eta l^{-(d+2)} (\delta M_{0}(t))^{\frac{1}{p}} \mu^{-b} + \eta  \lVert R_{0} \rVert_{C_{t,x}}^{\frac{1}{p}} [ \lambda \mu^{a - d + 1} \lVert \nabla \varrho \rVert_{L_{x}^{1}} + \mu^{a-d} \lVert \varrho \rVert_{L_{x}^{1}} \nu]\nonumber\\
& \hspace{5mm}+ l^{-(d+2)} \max\{ (\delta M_{0}(t))^{\frac{1}{p}} \lVert R_{0} \rVert_{C_{t,x}}^{\frac{1}{p'}}, (\delta M_{0}(t))^{\frac{1}{p'}} \lVert R_{0} \rVert_{C_{t,x}}^{\frac{1}{p}} \}  \omega^{-1} \nonumber\\
& \hspace{5mm}+  \lVert R_{0} \rVert_{C_{t,x}} \omega^{-1} [ \lambda \lVert \nabla \varrho_{\mu}^{j} \rVert_{L^{2}} \lVert \tilde{\varrho}_{\mu}^{j} \rVert_{L^{2}} + \lambda \lVert \varrho_{\mu}^{j} \rVert_{L^{2}} \lVert \nabla \tilde{\varrho}_{\mu}^{j} \rVert_{L^{2}} + \lVert \varrho_{\mu}^{j} \rVert_{L^{2}} \lVert \tilde{\varrho}_{\mu}^{j} \rVert_{L^{2}} \nu] \nonumber\\
& \overset{\eqref{est 68} \eqref{est 55}}{\lesssim} \eta l^{-(d+2)} ( \delta M_{0}(t))^{\frac{1}{p}} \mu^{-b} + [ \eta \lVert R_{0} \rVert_{C_{t,x}}^{\frac{1}{p}} \mu^{-b} + \lVert R_{0} \rVert_{C_{t,x}} \omega^{-1} ][\lambda \mu + \nu] \nonumber\\
&\hspace{5mm} + l^{-(d+2)} \max\{ (\delta M_{0}(t))^{\frac{1}{p}} \lVert R_{0} \rVert_{C_{t,x}}^{\frac{1}{p'}}, (\delta M_{0}(t))^{\frac{1}{p'}} \lVert R_{0} \rVert_{C_{t,x}}^{\frac{1}{p}} \}\omega^{-1} . \nonumber
\end{align}
Thus, we conclude  \eqref{est 159} from \eqref{est 124}, \eqref{est 130}, \eqref{est 131}, and \eqref{est 142}. Second, we compute 
\begin{align}
& \lVert R^{\text{time,2}}(t) \rVert_{L_{x}^{1}} \nonumber \\
\overset{\eqref{est 125}\eqref{est 90}}{\lesssim}&  \lambda \omega \sum_{j=1}^{d} (\sum_{k=0}^{N-1} \nu^{-k-1} \lVert a_{j}(t) (\partial_{j} \varrho_{\mu}^{j})_{\lambda} \circ \tau_{\omega t e_{j}} \rVert_{W_{x}^{k,1}}  + \nu^{-N} \lVert a_{j}(t) (\partial_{j} \varrho_{\mu}^{j})_{\lambda} \circ \tau_{\omega t e_{j}} \rVert_{W_{x}^{N,1}})\nonumber \\
\overset{\eqref{est 67}\eqref{est 66} \eqref{est 83}}{\lesssim}& \lambda \omega \sum_{j=1}^{d} ( \sum_{k=0}^{N-1} \nu^{-k-1} \eta ( \delta M_{0}(t))^{\frac{1}{p}} l^{-(d+2) k} \lambda^{k}\lVert \partial_{j} \varrho_{\mu}^{j} \rVert_{W^{k,1}} \nonumber \\
& \hspace{30mm} + \nu^{-N} \eta (\delta M_{0}(t))^{\frac{1}{p}} l^{-(d+2) N} \lambda^{N} \lVert \partial_{j} \varrho_{\mu}^{j} \rVert_{W^{N,1}}) \nonumber \\
\overset{\eqref{est 68}\eqref{est 55}}{\leq}&C \left(\frac{\omega}{\mu^{b}}\right) \eta (\delta M_{0}(t))^{\frac{1}{p}} \left(\sum_{k=1}^{N} \left( \frac{\lambda \mu}{\nu} \right)^{k} l^{-(d+2) (k-1)} + \frac{ (\lambda \mu)^{N+1}}{\nu^{N}} l^{-(d+2) N} \right).
\end{align} 
\end{proof}

\subsubsection{Estimates on $\text{div} (q(u_{0} + w)) = \text{div} R^{q}$ in \eqref{est 99}}\label{Section 4.1.3}
We define
\begin{equation}\label{est 132}
R^{q} \triangleq q (u_{0} + w).
\end{equation} 

\begin{lemma}\label{Lemma 4.8}
There exists a constant $C \geq 0$ with which $R^{q}$ defined in \eqref{est 132} satisfies for all $t \in [0, T_{L}]$
\begin{equation}\label{est 160}
\lVert R^{q}(t) \rVert_{L_{x}^{1}} \leq C\omega^{-1} \lVert R_{0} \rVert_{C_{t,x}} (\lVert u_{0} \rVert_{C_{t,x}} + \eta^{-1} \lVert R_{0} \rVert_{C_{t,x}}^{\frac{1}{p'}} \mu^{b}). 
\end{equation} 
\end{lemma}

\begin{proof}[Proof of Lemma \ref{Lemma 4.8}]
We compute 
\begin{align}
\lVert R^{q} (t) \rVert_{L_{x}^{1}} \overset{\eqref{est 132}}{\leq}& \lVert q (t) \rVert_{L_{x}^{1}} (\lVert u_{0}(t) \rVert_{L_{x}^{\infty}} + \lVert w (t) \rVert_{L_{x}^{\infty}}) \nonumber \\
\overset{ \eqref{est 71}}{\leq}&  \sum_{j=1}^{d} \lVert a_{j}(t) \rVert_{L_{x}^{\infty}} \lVert b_{j}(t) \rVert_{L_{x}^{\infty}} \lVert Q^{j}(t) \rVert_{L_{x}^{1}} ( \lVert u_{0} \rVert_{C_{t,x}} + \sum_{k=1}^{d} \lVert b_{k}(t) \rVert_{L_{x}^{\infty}} \lVert W^{k}(t) \rVert_{L_{x}^{\infty}})\nonumber\\
\overset{\eqref{est 82} \eqref{est 62}\eqref{est 63}}{\leq}& C\omega^{-1} \lVert R_{0} \rVert_{C_{t,x}} (\lVert u_{0} \rVert_{C_{t,x}} + \eta^{-1} \lVert R_{0} \rVert_{C_{t,x}}^{\frac{1}{p'}} \mu^{b}). 
\end{align}
\end{proof}

\subsubsection{Estimates on $\text{div} ((\theta_{0} + \vartheta + q + z) w_{c}) = \text{div} R^{\text{corr,1}}$ in \eqref{est 99}}\label{Section 4.1.4}
We define 
\begin{equation}\label{est 133}
R^{\text{corr,1}} \triangleq (\theta_{0}+  \vartheta + q + z) w_{c}.
\end{equation} 

\begin{lemma}\label{Lemma 4.9}
There exists a constant $C \geq 0$ with which $R^{\text{corr,1}}$ defined in \eqref{est 133} satisfies for all $t \in [0, T_{L}]$
\begin{align}
\lVert R^{\text{corr,1}}(t) \rVert_{L_{x}^{1}} \leq& C \left(\lVert \theta_{0} \rVert_{C_{t}L_{x}^{p}} + \eta (\delta M_{0}(t))^{\frac{1}{p}}[1+  \lambda^{-\frac{1}{p}} l^{-(d+2)} ] + l^{-(d+1)} \delta M_{0}(t) \left(\frac{\mu^{b}}{\omega} \right) + L^{\frac{1}{4}}\right) \nonumber\\
& \hspace{20mm} \times  \eta^{-1}(\delta M_{0}(t))^{\frac{1}{p'}}  [ \sum_{k=1}^{N} \left( \frac{\lambda \mu l^{-(d+2)}}{\nu} \right)^{k} + \frac{ (\lambda \mu l^{-(d+2)} )^{N+1}}{\nu^{N}}]. \label{est 161} 
\end{align} 
\end{lemma}

\begin{proof}[Proof of Lemma \ref{Lemma 4.9}]
We compute for all $t \in [0, T_{L}]$  
\begin{align}
& \lVert R^{\text{corr,1}}(t) \rVert_{L_{x}^{1}} \nonumber \\
\overset{\eqref{est 133}}{\leq}& ( \lVert \theta_{0}(t) \rVert_{L_{x}^{p}} + \lVert \vartheta(t) \rVert_{L_{x}^{p}} + \lVert q(t) \rVert_{L_{x}^{p}} + \lVert z(t) \rVert_{L_{x}^{p}} ) \lVert w_{c}(t) \rVert_{L_{x}^{p'}} \nonumber \\
\overset{\eqref{est 32}\eqref{est 342} \eqref{est 136}}{\leq}& C \left(\lVert \theta_{0} \rVert_{C_{t}L_{x}^{p}} + \eta (\delta M_{0}(t))^{\frac{1}{p}}[1+  \lambda^{-\frac{1}{p}} l^{-(d+2)} ] + l^{-(d+1)} \delta M_{0}(t) \mu^{b} \omega^{-1} + L^{\frac{1}{4}}\right) \nonumber\\
& \hspace{5mm} \times \eta^{-1} (\delta M_{0}(t))^{\frac{1}{p'}}  [ \sum_{k=1}^{N} \left( \frac{\lambda \mu l^{-(d+2)}}{\nu} \right)^{k} + \frac{ (\lambda \mu l^{-(d+2)} )^{N+1}}{\nu^{N}}].
\end{align}
\end{proof}

\subsubsection{Estimates on $\text{div} (zw) = \text{div} R^{\text{corr,2}}$ in \eqref{est 99}}\label{Section 4.1.5}
We define 
\begin{equation}\label{est 138}
R^{\text{corr,2}}\triangleq zw.
\end{equation} 
\begin{lemma}\label{Lemma 4.10}
There exists a constant $C \geq 0$ with which $R^{\text{corr,2}}$ defined in \eqref{est 138} satisfies for all $t \in [0,T_{L}]$ 
\begin{equation}\label{est 163}
\lVert R^{\text{corr,2}}(t) \rVert_{L_{x}^{1}} \leq C\mu^{-1} L^{\frac{1}{4}} \eta^{-1} \lVert R_{0} \rVert_{C_{t,x}}^{\frac{1}{p'}}.
\end{equation} 
\end{lemma} 

\begin{proof}[Proof of Lemma \ref{Lemma 4.10}]
We compute for all $t \in [0, T_{L}]$ 
\begin{align}
\lVert R^{\text{corr,2}}(t) \rVert_{L_{x}^{1}} \overset{\eqref{est 32} \eqref{est 71}}{\lesssim}& L^{\frac{1}{4}} \sum_{j=1}^{d} \lVert b_{j}(t) \rVert_{L_{x}^{\infty}} \lVert W^{j}(t) \rVert_{L_{x}^{1}} \overset{\eqref{est 82} \eqref{est 62}}{\leq} C \mu^{-a} L^{\frac{1}{4}} \eta^{-1} \lVert R_{0} \rVert_{C_{t,x}}^{\frac{1}{p'}}.
\end{align}
\end{proof}

\subsubsection{Estimates on $\text{div} (R_{l} - R_{0}) = \text{div} R^{\text{moll}}$ in \eqref{est 99}}\label{Section 4.1.6}
We define 
\begin{equation}\label{est 100}
R^{\text{moll}} \triangleq R_{l} - R_{0}. 
\end{equation} 

\begin{lemma}\label{Lemma 4.11}
There exists a constant $C \geq 0$ with which $R^{\text{moll}}$ defined in \eqref{est 100} satisfies for all $t \in [0, T_{L}]$
\begin{equation}\label{est 139}
\lVert R^{\text{moll}}(t) \rVert_{L_{x}^{1}} \leq Cl^{\frac{1}{2} - 2 \varpi} ( \lVert R_{0} \rVert_{C_{t}C_{x}^{1}} + \lVert R_{0} \rVert_{C_{t}^{\frac{1}{2} - 2 \varpi}C_{x}}). 
\end{equation}
\end{lemma}

\begin{proof}[Proof of Lemma \ref{Lemma 4.11}] 
This follows from a standard property of mollifiers as $l \ll 1$ and $\varpi \in (0,\frac{1}{4})$ and taking $\lambda \in\mathbb{N}$ sufficiently large. 
\end{proof}

We now choose the parameters in the order of 
\begin{subequations}\label{est 148}
\begin{align}
& \mu = \lambda^{\alpha} \text{ such that } \alpha (\epsilon)  > 2 \epsilon^{-1}, \label{est 143} \\
& \nu = \lambda^{\gamma} \text{ for } \gamma (\alpha, \epsilon) \in \mathbb{N} \text{ such that } \alpha + 1 < \gamma < \alpha (1+ \epsilon), \label{est 144}\\
& \beta = \beta(b, \alpha, \gamma) \text{ such that } b \alpha < \beta < b \alpha + \gamma - (\alpha +1), \label{est 145}\\
& \omega  =  \lambda^{\beta}, \label{est 146}\\
& N(\alpha,\gamma) \in \mathbb{N} \text{ sufficiently large such that } \frac{N}{N-1} < \frac{\gamma}{1+ \alpha}. \label{est 147} 
\end{align}
\end{subequations}
Lastly, we choose $\iota$ in \eqref{est 95}. Before we do so, we observe that \eqref{est 60} and \eqref{est 55} imply that 
\begin{equation}\label{est 149}
\epsilon + 1 < b \text{ and consequently } \gamma < b \alpha \text{ which in turn implies } \gamma < \beta. 
\end{equation} 
With \eqref{est 148}-\eqref{est 149} in mind, we choose a positive real number $\iota$ such that 
\begin{align}
\iota < &\min \{ \frac{ \min\{\frac{1}{p}, \frac{1}{p'} \}}{d+2}, \frac{\beta - \alpha b}{2d+3}, \frac{1}{d+2} \left( \frac{\gamma N}{N+1} - 1 - \alpha \right), \label{est 152}\\
&\hspace{7mm}  \frac{ \alpha (1+ \epsilon)- \gamma}{d+2}, \frac{\gamma - 1 - \alpha}{3d+5}, \frac{1}{3d+5}, \frac{b \alpha + \gamma - (\beta + 1 + \alpha)}{(d+2) N}\}. \nonumber 
\end{align}
Now by construction $(\theta_{1}, u_{1}, R_{1})$ solves \eqref{est 29}. Moreover, due to the cut-offs $\chi_{j}$, $\theta_{1}$ and $u_{1}$ defined in \eqref{est 76} are in $C_{t,x}^{\infty}$. On the other hand, $R_{1}$ defined in \eqref{est 150} is in $C_{t}C_{x}^{1} \cap C_{t}^{\frac{1}{2} - 2 \delta} C_{x}$. Next, we need to verify \eqref{est 35}-\eqref{est 38}. We estimate 
\begin{align}
\lVert (\theta_{1} - \theta_{0})(t) \rVert_{L_{x}^{p}}
\overset{\eqref{est 76}}{\leq}& \lVert \vartheta(t) \rVert_{L_{x}^{p}} + \lvert \vartheta_{c}(t) \rvert + \lVert q(t) \rVert_{L_{x}^{p}}+ \lvert q_{c}(t) \rvert \label{est 208}  \\
\overset{\eqref{est 134} \eqref{est 135} \eqref{est 151}}{\leq}& \frac{M\eta}{2} (2 \delta M_{0}(t))^{\frac{1}{p}} + \frac{C}{\lambda^{\frac{1}{p}}} \eta l^{-(d+2)} (\delta M_{0}(t))^{\frac{1}{p}} \nonumber \\
&+  C \eta  \lVert R_{0} \rVert_{C_{t,x}}^{\frac{1}{p}}  \mu^{-b} + C l^{-(d+1)} \delta M_{0}(t)  \mu^{b} \omega^{-1}  + C\lVert R_{0} \rVert_{C_{t,x}} \omega^{-1} \nonumber \\
\overset{\eqref{est 95}\eqref{est 148}}{\leq}& \frac{M \eta }{2} ( 2 \delta M_{0}(t))^{\frac{1}{p}} + C[\lambda^{- \frac{1}{p} + (d+2) \iota} (\delta M_{0}(t))^{\frac{1}{p}} \nonumber\\
& \hspace{5mm} + \eta \lVert R_{0} \rVert_{C_{t,x}}^{\frac{1}{p}} \lambda^{-\alpha b} + \lambda^{(d+1) \iota} \delta M_{0}(t) \lambda^{\alpha b - \beta} + \lVert R_{0} \rVert_{C_{t,x}} \lambda^{-\beta}].\nonumber 
\end{align}
Now we use the fact that 
\begin{equation}\label{est 324}
\iota \overset{\eqref{est 152}}{<} \frac{ \min \{ \frac{1}{p}, \frac{1}{p'} \}}{d+2} \leq\frac{1}{p (d+2)}, \hspace{3mm} \iota \overset{\eqref{est 152}}{<} \frac{\beta - \alpha b}{2d+3} < \frac{\beta - \alpha b}{d+1}
\end{equation}
so that taking $\lambda \in\mathbb{N}$ sufficiently large gives us \eqref{est 35} as desired. Next, we estimate 
\begin{align}
& \lVert (u_{1} - u_{0})(t) \rVert_{L_{x}^{p'}} \overset{\eqref{est 76}}{\leq} \lVert w(t) \rVert_{L_{x}^{p'}} + \lVert w_{c}(t) \rVert_{L_{x}^{p'}} \label{est 209}\\
\overset{\eqref{est 136} \eqref{est 153}}{\leq} &  \frac{M}{2\eta}(2 \delta M_{0}(t))^{\frac{1}{p'}} + \frac{C}{\lambda^{\frac{1}{p'}}} \eta^{-1} l^{-(d+2)} (\delta M_{0}(t))^{\frac{1}{p'}} \nonumber\\
&+ C\eta^{-1} (\delta M_{0}(t))^{\frac{1}{p'}} [\sum_{k=1}^{N} \left( \frac{ \lambda \mu l^{-(d+2)}}{\nu} \right)^{k} + \frac{ (\lambda \mu l^{-(d+2)})^{N+1}}{\nu^{N}} ] \nonumber\\
\overset{\eqref{est 95} \eqref{est 148}}{\leq}& \frac{M}{2\eta} (2 \delta M_{0}(t))^{\frac{1}{p'}} \nonumber\\
&+ C\eta^{-1} (\delta M_{0}(t))^{\frac{1}{p'}}[ \lambda^{-\frac{1}{p'} + (d+2) \iota} +  ( \sum_{k=1}^{N} \lambda^{ [1 + \alpha - \gamma + (d+2) \iota]k} + \lambda^{-\gamma N} \lambda^{[1+ \alpha + (d+2) \iota] (N+1)} )]. \nonumber 
\end{align}
We now observe that 
\begin{align}
&\iota \overset{\eqref{est 152}}{<} \frac{ \min \{\frac{1}{p}, \frac{1}{p'} \}}{d+2} \leq \frac{1}{p' (d+2)}, \nonumber\\
&\iota \overset{\eqref{est 152}}{<} \frac{\gamma - 1 - \alpha}{3d+ 5} < \frac{\gamma - 1 - \alpha}{d+2}, \hspace{3mm}  \iota \overset{\eqref{est 152}}{<} \frac{1}{d+2} \left( \frac{\gamma N}{N+1} - 1 - \alpha \right) \label{est 325}
\end{align}
and therefore taking $\lambda  \in\mathbb{N}$ sufficiently large gives us \eqref{est 36}. Next, we estimate 
\begin{align}
 \lVert  (u_{1} - u_{0})(t) \rVert_{W_{x}^{1, \tilde{p}}}& \overset{\eqref{est 76}}{\leq} \lVert w(t) \rVert_{W_{x}^{1, \tilde{p}}} + \lVert w_{c}(t) \rVert_{W_{x}^{1, \tilde{p}}}  \label{est 210}\\
\overset{\eqref{est 154} \eqref{est 93}}{\lesssim} & C \eta^{-1} l^{-(d+2)} (\delta M_{0}(t))^{\frac{1}{p'}} \frac{ \lambda \mu + \nu}{\mu^{1+ \epsilon}} \nonumber\\
&+ \eta^{-1}  \frac{ ( \delta M_{0}(t))^{\frac{1}{p'}} [ \lambda \mu l^{-(d+2)} + \nu]}{\mu^{1+ \epsilon}} [ \sum_{k=1}^{N} \left( \frac{ l^{-(d+2)} \lambda \mu}{\nu} \right)^{k} + \frac{ ( l^{-(d+2)} \lambda \mu)^{N+1}}{\nu^{N}} ] \nonumber\\ 
\overset{\eqref{est 95} \eqref{est 148}}{\lesssim}& \eta^{-1} (\delta M_{0}(t))^{\frac{1}{p'}} \lambda^{(d+2) \iota} \frac{ \lambda^{1+ \alpha} + \lambda^{\gamma}}{\lambda^{\alpha (1+ \epsilon)}} [ \sum_{k=1}^{N} \lambda^{ [(d+2) \iota + 1 + \alpha -\gamma] k} + \lambda^{-\gamma N + [(d+2) \iota + 1+ \alpha] (N+1)}].\nonumber
\end{align}
We observe that due to \eqref{est 152}
\begin{align}
\iota < \frac{\alpha (1+ \epsilon) - \gamma}{d+2},  \hspace{2mm} \iota < \frac{\gamma - 1 - \alpha}{3d+5} < \frac{\gamma - 1 - \alpha}{d+2}, \hspace{2mm}   \iota < \frac{1}{d+2} \left( \frac{\gamma N}{N+1} - 1 - \alpha \right) \label{est 326}
\end{align}
so that taking $\lambda \in\mathbb{N}$ sufficiently large gives us \eqref{est 37}. Next, to prove \eqref{est 38}, we realize that $\lVert R^{\chi}(t) \rVert_{L_{x}^{1}} \leq \frac{\delta M_{0}(t)}{2}$ due to \eqref{est 157}; thus, according to \eqref{est 150}, it suffices to bound the $L_{x}^{1}$-norm of $R^{\text{time,1}} + R^{\text{quadr}} + R^{\text{time,2}} + R^{\text{lin}} + R^{q} + R^{\text{corr,1}} + R^{\text{corr,2}} + R^{\text{moll}}$ by $\frac{\delta M_{0}(t)}{2}$. We start with 
\begin{align}
 \lVert R^{\text{time,1}}(t) \rVert_{L_{x}^{1}} \label{est 211}
\overset{\eqref{est 155}}{\lesssim}&\omega^{-1} l^{-(d+2)} \max\{ ( \delta M_{0}(t))^{\frac{1}{p}} \lVert R_{0} \rVert_{C_{t,x}}^{\frac{1}{p'}}, ( \delta M_{0}(t))^{\frac{1}{p'}} \lVert R_{0} \rVert_{C_{t,x}}^{\frac{1}{p}} \} \nonumber\\
\overset{\eqref{est 146} \eqref{est 95}}{\lesssim}& \lambda^{-\beta + (d+2) \iota} \max\{ ( \delta M_{0}(t))^{\frac{1}{p}} \lVert R_{0} \rVert_{C_{t,x}}^{\frac{1}{p'}}, ( \delta M_{0}(t))^{\frac{1}{p'}} \lVert R_{0} \rVert_{C_{t,x}}^{\frac{1}{p}} \} \ll \delta M_{0}(t) 
\end{align}
for $\lambda \in\mathbb{N}$ sufficiently large due to 
\begin{align}
\iota \overset{\eqref{est 152}}{<} \frac{\beta - \alpha b}{2d+3} < \frac{\beta}{d+2}.  \label{est 353} 
\end{align}
Next, due to \eqref{est 143}, \eqref{est 144}, and \eqref{est 95}
\begin{align}
&\lVert R^{\text{quadr}}(t) \rVert_{L_{x}^{1}} \overset{\eqref{est 156}}{\lesssim}  \left( \frac{\lambda \mu}{\nu} + \frac{1}{\lambda} \right)  l^{-(d+2) 2} \max\{ ( \delta M_{0}(t))^{\frac{1}{p}} \lVert R_{0} \rVert_{C_{t,x}}^{\frac{1}{p'}}, ( \delta M_{0}(t))^{\frac{1}{p'}} \lVert R_{0} \rVert_{C_{t,x}}^{\frac{1}{p}} \} \nonumber \\
& \hspace{7mm}\approx ( \lambda^{1+ \alpha - \gamma} + \lambda^{-1}) \lambda^{(d+2) 2 \iota} \max\{ ( \delta M_{0}(t))^{\frac{1}{p}} \lVert R_{0} \rVert_{C_{t,x}}^{\frac{1}{p'}}, ( \delta M_{0}(t))^{\frac{1}{p'}} \lVert R_{0} \rVert_{C_{t,x}}^{\frac{1}{p}} \} \ll \delta M_{0}(t)  \label{est 164}
\end{align} 
for $\lambda \in\mathbb{N}$ sufficiently large due to 
\begin{align*}
 \iota \overset{\eqref{est 152}}{<} \frac{\gamma - 1 - \alpha}{3d+5} < \frac{\gamma - 1 - \alpha}{(d+2) 2}, \hspace{3mm}  \iota \overset{\eqref{est 152}}{<} \frac{1}{3d+5} < \frac{1}{(d+2) 2}.
\end{align*}
Next, 
\begin{align}
 \lVert &R^{\text{time,2}}(t) \rVert_{L_{x}^{1}} \overset{\eqref{est 158}}{\lesssim}  \left(\frac{\omega}{\mu^{b}}\right) \eta (\delta M_{0}(t))^{\frac{1}{p}} \left(\sum_{k=1}^{N} \left( \frac{\lambda \mu}{\nu} \right)^{k} l^{-(d+2) (k-1)} + \frac{ (\lambda \mu)^{N+1}}{\nu^{N}} l^{-(d+2) N} \right)  \label{est 165}\\
&\overset{\eqref{est 148} \eqref{est 95}}{\lesssim} \lambda^{\beta + 1 + \alpha - (b \alpha + \gamma) + (d+2) N \iota} \eta (\delta M_{0}(t))^{\frac{1}{p}}  [ \sum_{k=1}^{N} \lambda^{(1+ \alpha - \gamma) (k-1)} + \lambda^{(1+ \alpha) N - \gamma (N-1)}] \ll \delta M_{0}(t)\nonumber 
\end{align}
for $\lambda \in\mathbb{N}$ sufficiently large due to 
\begin{align*}
1 + \alpha - \gamma \overset{\eqref{est 144}}{<} 0, \hspace{3mm} (1+ \alpha) N - \gamma (N-1) \overset{\eqref{est 147}}{<} 0, \hspace{3mm} \iota \overset{\eqref{est 152}}{<} \frac{b \alpha + \gamma - (\beta + 1 + \alpha)}{(d+2) N}.
\end{align*}
Next,
\begin{align}
\lVert R^{\text{lin}}(t) \rVert_{L_{x}^{1}} 
\overset{\eqref{est 159}}{\lesssim}&  \mu^{-a}  \lVert \theta_{0} \rVert_{C_{t,x}} \eta^{-1} \lVert R_{0} \rVert_{C_{t,x}}^{\frac{1}{p'}}  + \mu^{-b} \eta [ \lVert R_{0} \rVert_{C_{t,x}}^{\frac{1}{p}} \lVert u_{0} \rVert_{C_{t,x}} + (\delta M_{0}(t))^{\frac{1}{p}} l^{-(d+2)}]\nonumber\\
& \hspace{5mm} + [\eta \lVert R_{0} \rVert_{C_{t,x}}^{\frac{1}{p}} \mu^{-b} + \lVert R_{0} \rVert_{C_{t,x}} \omega^{-1} ][\lambda \mu + \nu] \nonumber \\
& \hspace{5mm} + l^{-(d+2)}  \max\{ (\delta M_{0}(t))^{\frac{1}{p}} \lVert R_{0} \rVert_{C_{t,x}}^{\frac{1}{p'}}, (\delta M_{0}(t))^{\frac{1}{p'}} \lVert R_{0} \rVert_{C_{t,x}}^{\frac{1}{p}} \} \omega^{-1}\nonumber  \\
\overset{\eqref{est 148} \eqref{est 95} }{\lesssim}& \lambda^{- \alpha a} \lVert \theta_{0} \rVert_{C_{t,x}} \eta^{-1} \lVert R_{0} \rVert_{C_{t,x}}^{\frac{1}{p'}}  + \lambda^{-\alpha b}  \eta [ \lVert R_{0} \rVert_{C_{t,x}}^{\frac{1}{p}} \lVert u_{0} \rVert_{C_{t,x}} + (\delta M_{0}(t))^{\frac{1}{p}} \lambda^{(d+2) \iota} ] \nonumber \\
&+ [\eta \lVert R_{0} \rVert_{C_{t,x}}^{\frac{1}{p}} \lambda^{-\alpha b} + \lVert R_{0} \rVert_{C_{t,x}} \lambda^{-\beta} ][\lambda^{1+ \alpha} + \lambda^{\gamma}] \nonumber\\
&+ \lambda^{(d+2) \iota} \max\{ (\delta M_{0}(t))^{\frac{1}{p}} \lVert R_{0} \rVert_{C_{t,x}}^{\frac{1}{p'}}, (\delta M_{0}(t))^{\frac{1}{p'}} \lVert R_{0} \rVert_{C_{t,x}}^{\frac{1}{p}} \}  \lambda^{-\beta} \ll \delta   M_{0}(t) \label{est 166}
\end{align}
for $\lambda \in\mathbb{N}$ sufficiently large due to 
\begin{align}
& \iota \overset{\eqref{est 152}}{<} \frac{\beta - \alpha b}{2d+3} \overset{\eqref{est 149}}{<} \frac{\beta - \gamma}{2d+3}  \overset{\eqref{est 145}}{<} \frac{b \alpha - ( \alpha +1)}{2d+3} < \frac{\alpha b}{d+2}, \hspace{5mm}  \iota \overset{\eqref{est 152}}{<} \frac{\beta - \alpha b}{2d+3} < \frac{\beta}{d+2}, \nonumber\\
& \max\{ - \alpha b, - \beta \} + \max\{1+ \alpha, \gamma \}  \overset{\eqref{est 145} \eqref{est 144}}{\leq} - \alpha b + \gamma \overset{\eqref{est 149}}{<}  0.  \label{est 354} 
\end{align}
Next, 
\begin{align}
\lVert R^{q}(t) \rVert_{L_{x}^{1}} \overset{\eqref{est 160}}{\lesssim}&\omega^{-1} \lVert R_{0} \rVert_{C_{t,x}} (\lVert u_{0} \rVert_{C_{t,x}} + \eta^{-1} \lVert R_{0} \rVert_{C_{t,x}}^{\frac{1}{p'}} \mu^{b}) \nonumber \\
\overset{ \eqref{est 148} }{\approx}& \lambda^{-\beta} \lVert R_{0} \rVert_{C_{t,x}} (\lVert u_{0} \rVert_{C_{t,x}} + \eta^{-1} \lVert R_{0} \rVert_{C_{t,x}}^{\frac{1}{p'}} \lambda^{\alpha b}) \ll \delta M_{0}(t) \label{est 167}
\end{align}
for $\lambda \in\mathbb{N}$ sufficiently large due to $\alpha b < \beta$ from \eqref{est 145}. Next,
\begin{align}
\lVert R^{\text{corr,1}}(t) \rVert_{L_{x}^{1}}  
\overset{\eqref{est 161}}{\lesssim}& \left(\lVert \theta_{0} \rVert_{C_{t}L_{x}^{p}} + \eta (\delta M_{0}(t))^{\frac{1}{p}}[1+  \lambda^{-\frac{1}{p}} l^{-(d+2)} ] + l^{-(d+1)} \delta M_{0}(t) \left(\frac{\mu^{b}}{\omega} \right) + L^{\frac{1}{4}}\right) \nonumber\\
& \hspace{5mm} \times \eta^{-1}(\delta M_{0}(t))^{\frac{1}{p'}}  [ \sum_{k=1}^{N} \left( \frac{\lambda \mu l^{-(d+2)}}{\nu} \right)^{k} + \frac{ (\lambda \mu l^{-(d+2)} )^{N+1}}{\nu^{N}}] \label{est 168} \\
\overset{\eqref{est 95} \eqref{est 148} }{\lesssim}&( \lVert \theta_{0} \rVert_{C_{t}L_{x}^{p}} + \eta (\delta M_{0}(t))^{\frac{1}{p}} [ 1+ \lambda^{-\frac{1}{p}+ (d+2) \iota}] + \lambda^{(d+1) \iota + \alpha b - \beta} \delta M_{0}(t) + L^{\frac{1}{4}}) \nonumber \\
&  \times \eta^{-1} (\delta M_{0}(t))^{\frac{1}{p'}} [\sum_{k=1}^{N} \lambda^{(1+ \alpha - \gamma + (d+2) \iota )k} + \lambda^{- \gamma N + [1+ \alpha + (d+2) \iota](N+1)}] \ll \delta M_{0}(t)\nonumber
\end{align}
for $\lambda \in\mathbb{N}$ sufficiently large due to 
\begin{align*}
& \iota \overset{\eqref{est 152}}{<} \frac{ \min \{ \frac{1}{p}, \frac{1}{p'} \}}{d+2} \leq \frac{1}{p (d+2)}, \hspace{2mm} \iota \overset{\eqref{est 152}}{<} \frac{\beta - \alpha b}{2d+3} < \frac{\beta - \alpha b}{d+1},  \\
&\iota \overset{\eqref{est 152}}{<} \frac{\gamma - 1 - \alpha}{3d+5} < \frac{\gamma - 1 - \alpha}{d+2}, \hspace{2mm}  \iota \overset{\eqref{est 152}}{<} \frac{1}{d+2} \left( \frac{\gamma N}{N+1} - 1 - \alpha \right).
\end{align*}
Next, 
\begin{align}
\lVert R^{\text{corr, 2}}(t) \rVert_{L_{x}^{1}} \overset{\eqref{est 163}}{\lesssim}&  \mu^{-1} L^{\frac{1}{4}} \eta^{-1} \lVert R_{0} \rVert_{C_{t,x}}^{\frac{1}{p'}} 
\overset{\eqref{est 143}}{ \ll} \delta M_{0}(t)\label{est 170} 
\end{align}
for $\lambda \in\mathbb{N}$ sufficiently large. Finally, 
\begin{equation}\label{est 171}
\lVert R^{\text{moll}}(t) \rVert_{L_{x}^{1}} \overset{\eqref{est 139}}{\lesssim} l^{\frac{1}{2} - 2 \varpi} ( \lVert R_{0} \rVert_{C_{t}C_{x}^{1}} + \lVert R_{0} \rVert_{C_{t}^{\frac{1}{2} - 2 \varpi}C_{x}}) \overset{\eqref{est 95}}{\ll} \delta M_{0}(t) 
\end{equation}
for $\lambda \in\mathbb{N}$ sufficiently large as $\varpi \in (0, \frac{1}{4})$. Due to \eqref{est 211}, \eqref{est 164} - \eqref{est 166}, \eqref{est 167}-\eqref{est 171}, we conclude, along with $\lVert R^{\chi}(t) \rVert_{L_{x}^{1}} \leq \frac{\delta M_{0}(t)}{2}$ due to \eqref{est 157}, that \eqref{est 38} has been proven.

Finally, the proof that $(\theta_{1}, u_{1}, R_{1})$ are $(\mathcal{F}_{t})_{t\geq 0}$-adapted if $(\theta_{0}, u_{0}, R_{0})$ are $(\mathcal{F}_{t})_{t\geq 0}$-adapted, and that $(\theta_{1}, u_{1}, R_{1})(0,x)$ are deterministic if $(\theta_{0}, u_{0}, R_{0})(0,x)$ are deterministic, is very similar to the previous works (e.g., \cite{HZZ19}); in fact, it is simpler because we mollified only $R_{0}$, not $\theta_{0}$ or $u_{0}$. First, $z(t) = \int_{0}^{t} e^{(t-r) \Delta} dB(r)$ from \eqref{stochastic heat} is $(\mathcal{F}_{t})_{t\geq 0}$-adapted. Due to the compact support of $\varphi_{l}$ in $\mathbb{R}_{+}$, $R_{l}$ is $(\mathcal{F}_{t})_{t\geq 0}$-adapted. As $\Theta^{j}, Q^{j},$ and $W^{j}$ from \eqref{est 59} are deterministic, we see that $\vartheta, q$, and $w$ in \eqref{est 79}  are $(\mathcal{F}_{t})_{t\geq 0}$-adapted; consequently, so are $\vartheta_{c}$ and $q_{c}$ in \eqref{est 85}. As $M_{0}(t)$ from \eqref{est 49} is deterministic, $\chi_{j}$ from \eqref{est 70} and hence $a_{j}$ and $b_{j}$ from \eqref{est 77} are $(\mathcal{F}_{t})_{t\geq 0}$-adapted; because $\psi_{\nu}^{j}$ and $\tilde{\varrho}_{\mu}^{j}$ are deterministic, it follows that $w_{c}$ in \eqref{est 73} is $(\mathcal{F}_{t})_{t\geq 0}$-adapted. Therefore, $\theta_{1}$ and $u_{1}$ in \eqref{est 76} are $(\mathcal{F}_{t})_{t\geq 0}$-adapted. We see that $R^{\text{time,1}}$ in \eqref{est 114} is $(\mathcal{F}_{t})_{t\geq 0}$-adapted, $R^{\text{quadr,1}}$ and $R^{\text{quadr,2}}$ in \eqref{est 116} are both $(\mathcal{F}_{t})_{t\geq 0}$-adapted so that $R^{\text{quadr}}$ in \eqref{est 115} is $(\mathcal{F}_{t})_{t\geq 0}$-adapted. $R^{\chi}$ in \eqref{est 105} is $(\mathcal{F}_{t})_{t\geq 0}$-adapted because $R_{l}$ is $(\mathcal{F}_{t})_{t\geq 0}$-adapted. Similarly, $R^{\text{lin}}$ in \eqref{est 124}, $R^{\text{time,2}}$ in \eqref{est 125}, $R^{q}$ in \eqref{est 132}, $R^{\text{corr,1}}$ in \eqref{est 133}, $R^{\text{moll}}$ in \eqref{est 100}, and $R^{\text{corr,2}}$ in \eqref{est 138} are all $(\mathcal{F}_{t})_{t\geq 0}$ -adapted. Therefore, $R_{1}$ from \eqref{est 150} is also $(\mathcal{F}_{t})_{t\geq 0}$-adapted. Due to similarity, we omit the proof that $(\theta_{1}, u_{1}, R_{1})(0,x)$ are deterministic. This completes the proof of Proposition \ref{Proposition 4.2}. 

\section{Proof of Theorem \ref{Theorem 2.3}}\label{Section 5} 
We already have a convex integration solution $(\rho, u)$ for \eqref{stochastic transport}  forced by additive noise up to a stopping time $T_{L}$ due to Theorem \ref{Theorem 2.2}. To extend this convex integration solution to the interval $[0,T]$, we follow the argument given in \cite{HZZ21a}, and glue an appropriate weak solution of \eqref{stochastic transport} to this convex integration solution. Note that here, due to technical reasons aforementioned, we are gluing a convex integration solution with a weak solution as opposed to gluing two convex integration solutions as in proof of \cite[Theorem 1.1]{HZZ21a}. However, both ideas are similar in spirit and differs slightly in details. In what follows, our aim is to solve the equation \eqref{stochastic transport} with initial data $\rho(T_{L}) \in L^{p} (\mathbb{T}^{d})$  (due to \eqref{est 363}) and $u(t) \equiv 0$ for all $t \in (T_{L}, T]$. To that context, let $\hat{\rho}$ solve the following equation on $[0,T]$
\begin{equation}\label{est 364}
d \hat{\rho} (t) = \Delta \hat{\rho} (t) dt + d \hat{B}(t) \hspace{2mm} \text{ for } t > 0, \hspace{5mm} \hat{\rho} \rvert_{t=0} = 0, 
\end{equation} 
where $\hat{B}(t) \triangleq B(t+ T_{L}) - B(T_{L})$. Next, let $\tilde{\rho}(t) \triangleq \hat{\rho} (t) + e^{t\Delta} \rho(T_{L})$. Then we observe that $\tilde{\rho}$ solves 
\begin{equation}\label{est 365}
d \tilde{\rho} (t) = \Delta \tilde{\rho} (t) dt + d \hat{B}(t) \hspace{2mm} \text{ for } t > 0, \hspace{5mm} \tilde{\rho} \rvert_{t=0} = \rho(T_{L}) 
\end{equation} 
and is adapted to the filtration $(\hat{\mathcal{F}}_{t})_{t \geq 0}$ where $\hat{\mathcal{F}}_{t} \triangleq \sigma (\hat{B} (s), s \leq t ) \vee \sigma (\rho(T_{L}))$.  Now it follows that $(\bar{\rho}, \bar{u})$ defined by 
\begin{equation}\label{est 366} 
(\bar{\rho} (t), \bar{u}(t)) \triangleq  
\begin{cases}
(\rho(t), u(t)) & \text{ if } t < T_{L}, \\
(\tilde{\rho} (t- T_{L}), 0) & \text{ if } t \geq T_{L},
\end{cases} 
\end{equation} 
satisfies \eqref{stochastic transport} forced by additive noise. Moreover, following the argument presented in \cite[Proof of Theorem 1.1]{HZZ21a}, we conclude that $\bar{\rho}$ is an $(\mathcal{F}_{t})_{t\geq 0}$-adapted process; we notice that 
\begin{equation}
\bar{\rho} \in C([0,T]; L^{p} (\mathbb{T}^{d})) \hspace{3mm} \bar{u} \in L^{\infty} ([0,T]; L^{p'} (\mathbb{T}^{d})) \cap L^{\infty} ([0,T]; W^{1,\tilde{p}} (\mathbb{T}^{d})) \hspace{3mm} \mathbb{P}\text{-a.s.} 
\end{equation} 
and observe the loss of regularity in time for the vector field $u$. 

\section{Proof of Theorem \ref{Theorem 2.5}}\label{Section 6} 
As we mentioned, the proof of Theorem \ref{Theorem 2.4} follows from similar computations in the proof of Theorem \ref{Theorem 2.2} and thus is left to Appendix A; in this section we prove Theorem \ref{Theorem 2.5}. 

\begin{remark}\label{Difficulty 1}
As we mentioned in Remark \ref{Remark in transport case}, the proof of Theorem \ref{Theorem 2.5} follows the approach of \cite[Theorem 1.1 and Corollary 1.2]{HZZ21a} which in turn followed the proof of \cite[Theorem C]{BMS21}; however, its modification to the transport equation seems new, even in the deterministic case. In order to describe difficulty, let us informally recall some details from \cite{HZZ21a}, to which we refer for specific notations. On \cite[p. 41]{HZZ21a} the authors define the new velocity field $v_{q+1} = w_{q+1} + v_{l}$ where $w_{q+1} = \tilde{w}_{q+1}^{(p)} + \tilde{w}_{q+1}^{(c)} + \tilde{w}_{q+1}^{(t)}$ represents the perturbation and $v_{l}$ is $v_{q}$ that was mollified in space-time. They estimate for $\{\gamma_{q}\}_{q=0}^{\infty} \subset \mathbb{R}$ 
\begin{align}
 \left\lvert \lVert v_{q+1} \rVert_{L^{2}}^{2} - \lVert v_{q} \rVert_{L^{2}}^{2} - 3 \gamma_{q+1} \right\rvert =& \lvert \int_{\mathbb{T}^{3}} \lvert \tilde{w}_{q+1}^{(p)} \rvert^{2} + 2 \tilde{w}_{q+1}^{(p)} (\tilde{w}_{q+1}^{(c)} + \tilde{w}_{q+1}^{(t)}) + \lvert \tilde{w}_{q+1}^{(c)} + \tilde{w}_{q+1}^{(t)} \rvert^{2} \label{est 237}\\
&+ 2 \tilde{w}_{q+1}^{(p)} v_{l} + 2 (\tilde{w}_{q+1}^{(c)} + \tilde{w}_{q+1}^{(t)}) v_{l} dx + \lVert v_{l} \rVert_{L^{2}}^{2} - \lVert v_{q} \rVert_{L^{2}}^{2} - 3 \gamma_{q+1} \rvert  \nonumber 
\end{align}
(see \cite[Equation (5.44)]{HZZ21a}, also \cite[Equation (130)]{BMS21}) where $\lvert \tilde{w}_{q+1}^{(p)} \rvert^{2}$ represents the most difficult part of the nonlinear term, called oscillation term (see \cite[Equation (3.57)]{HZZ21a}). By defining 
\begin{equation}\label{est 239}
\rho \triangleq 2 \sqrt{ l^{2} + \lvert \mathring{R}_{l} \rvert^{2}} + \frac{\gamma_{q+1}}{(2\pi)^{3}}
\end{equation} 
(see \cite[p. 39]{HZZ21a}, also \cite[Equation (43)]{BMS21}) where $\mathring{R}_{l}$ is the mollified Reynolds  stress, for certain $t$, the authors in \cite{HZZ21a} were able to deduce 
\begin{align}
\lvert \tilde{w}_{q+1}^{(p)} \rvert^{2} - \frac{3\gamma_{q+1}}{(2\pi)^{3}} =& \text{Tr} [ - \mathring{R}_{l} + \sum_{\xi \in \Lambda} a_{(\xi)}^{2} \mathbb{P}_{\neq 0} W_{(\xi)} \otimes W_{(\xi)} + \rho \text{Id}] - \frac{3\gamma_{q+1}}{(2\pi)^{3}} \nonumber\\
=& 6 \sqrt{ l^{2} + \lvert \mathring{R}_{l} \rvert^{2}} + \sum_{\xi \in \Lambda} a_{(\xi)}^{2} \mathbb{P}_{\neq 0} \lvert W_{(\xi)} \rvert^{2} \label{est 238}
\end{align}
(see \cite[Equation (3.36) and p. 43]{HZZ21a}, also \cite[Equations (50), and (95)]{BMS21}) where $W_{(\xi)}$ represents intermittent jets, and the orthogonality of $W_{(\xi)} \otimes W_{(\xi')} \equiv 0$ for $\xi \neq \xi'$ and a geometric lemma \cite[Lemma B.1]{HZZ21a} were crucially used in \eqref{est 238}. 

Let us make three observations. First, only because it was $L_{x}^{2}$-norm, expansion in \eqref{est 237} was possible. Although we prefer to repeat the same argument with $\lVert \theta_{1} \rVert_{L_{x}^{p}}^{p} - \lVert \theta_{0} \rVert_{L_{x}^{p}}^{p}$ for an arbitrary $p \in (1,\infty)$, this seems to have no chance; moreover, we cannot consider $\lVert \theta_{1} \rVert_{L_{x}^{2}}^{2} - \lVert \theta_{0} \rVert_{L_{x}^{2}}^{2}$ in case $p \in (1,2)$. Second, it was crucial to utilize the special feature of intermittent jets such as orthogonality to handle the difficult oscillation term. These two observations lead us to the direction that we need to consider $\int_{\mathbb{T}^{d}} \theta_{1}(t,x) u_{1}(t, x+ B(t)) dx$ because, as we will see in \eqref{est 172}, the new nonlinear term will be $div (\theta_{1}(t,x) u_{1}(t, x+ B(t)))$, and the orthogonality of Mikado density $\Theta^{j}$ and Mikado field $W^{j}$, specifically \eqref{est 64}, was used indeed when handling the most difficult term $R^{\text{quadr}}$ in \eqref{est 102}. The third observation from \eqref{est 238} is that the geometric lemma produces a term ``$\rho Id$'' and by strategically including $\frac{\gamma_{q+1}}{(2\pi)^{3}}$ in \eqref{est 239}, the authors of \cite{HZZ21a} (and \cite{BMS21} similarly) were able to create a cancellation. Our situation is quite different; considering $\vartheta, w,$ and $q$ in \eqref{est 79} it is not clear at all how to somehow ``embed'' an analogous term to $\frac{\gamma_{q+1}}{(2\pi)^{3}}$ to make a cancellation. We were able to come up with a suitable alternative (see \eqref{est 255}, \eqref{est 244}, and Remark \ref{Difficulty 3}). 
\end{remark}

We first describe a key proposition Proposition \ref{Proposition 6.1} which is inspired by \cite[Proposition 5.1]{HZZ21a} and \cite[Proposition 16]{BMS21} concerning the following transport-diffusion-defect equation
\begin{equation}\label{est 172} 
\partial_{t} \theta (t,x) + \text{div} (u(t,x+ B(t)) \theta(t,x)) - \Delta \theta(t,x) = - \text{div}R(t,x), \hspace{3mm} \nabla\cdot u = 0. 
\end{equation} 
We prescribe an arbitrary initial values of both $\rho$ and $u$ by $\rho^{\text{in}} \in L^{p}(\mathbb{T}^{d})$ and $u^{\text{in}} \in L^{p'} (\mathbb{T}^{d})$ $\mathbb{P}$-a.s. which are independent of the given standard Brownian motion $B$ and let $(\mathcal{F}_{t})_{t\geq 0}$ be the augmented joint canonical filtration on $(\Omega, \mathcal{F})$ generated by $B, \rho^{\text{in}}$, and $u^{\text{in}}$ so that $\rho^{\text{in}}$ is $\mathcal{F}_{0}$-measurable. For such $\rho^{\text{in}}$ and $u^{\text{in}}$, we will construct $(\theta, u)$ that satisfies \eqref{est 172} such that $(\theta, u) \rvert_{t=0} = (\theta^{\text{in}}, u^{\text{in}})$ where $\theta^{\text{in}} = \rho^{\text{in}}$. We define $l \triangleq \lambda^{-\iota}$ identically to \eqref{est 95} where $\iota$ satisfies \eqref{est 152} and it will be taken smaller as needed.  

\begin{proposition}\label{Proposition 6.1}
There exists a constant $M > 0$ such that the following holds. Let $T> 0, \varpi \in (0, \frac{1}{4})$, $p \in (1,\infty)$, $\tilde{p} \in [1,\infty)$ such that \eqref{est 8} holds, $\upsilon \in (1, p)$, and $\theta^{\text{in}} \in L^{p}(\mathbb{T}^{d})$ and $u^{\text{in}} \in L^{p'} (\mathbb{T}^{d})$ $\mathbb{P}$-a.s. independently of the given standard Brownian motion $B$. Suppose that there exists a $(\mathcal{F}_{t})_{t\geq 0}$-adapted $(\theta_{0}, u_{0}, R_{0})$ that satisfies \eqref{est 172} such that $\theta_{0}(0,x) = \theta^{\text{in}}(x), u_{0}(0,x) = u^{\text{in}}(x)$, $\fint_{\mathbb{T}^{d}} \theta_{0}(t,x) dx = 0$ for all $t \in [0, T]$, 
\begin{equation}\label{est 241}
\theta_{0} \in C^{\infty} ([0,T] \times \mathbb{T}^{d}), \hspace{1mm} u_{0} \in C^{\frac{1}{2} - 2 \varpi} ([0,T]; C^{\infty} (\mathbb{T}^{d})), \hspace{1mm} R_{0} \in  C^{\frac{1}{2} - 2 \varpi} ([0,T]; C^{\infty} (\mathbb{T}^{d})). 
\end{equation} 
Choose any $\delta, \Sigma \in (0, 1]$ and $\Gamma > 0$ such that $\frac{\Gamma}{\delta} \leq \bar{C} < \infty$. Assume that 
\begin{equation}\label{est 251}
\lVert R_{0}(t) \rVert_{L_{x}^{1}} \leq 2 \delta \hspace{5mm} \forall \hspace{1mm} t \in [2\Sigma \wedge T, T]. 
\end{equation} 
Extend $R_{0}$ to $t < 0$ with its value at $t = 0$, and mollify it with $\phi_{l}$ and $\varphi_{l}$ from \eqref{est 242} to obtain $R_{l}$ identically to \eqref{est 86}, denote its $j$-th component by $R_{l}^{j}$ for $j \in \{1,\hdots, d\}$ identically to \eqref{est 104}, and then define cut-off functions
\begin{equation}\label{est 243}
\chi_{j}: [0, T] \times \mathbb{T}^{d} \mapsto [0,1] \hspace{1mm} \text{ such that } \hspace{1mm} \chi_{j} (t,x) = 
\begin{cases}
0 & \text{ if } \lvert R_{l}^{j} (t,x) + \Gamma \rvert \leq \frac{\delta}{4d}, \\
1 & \text{ if } \lvert R_{l}^{j}(t,x)  + \Gamma \rvert \geq \frac{\delta}{2d}, 
\end{cases} 
\end{equation} 
(cf. \eqref{est 70} and \eqref{est 322}). Then there exists another $(\mathcal{F}_{t})_{t\geq 0}$-adapted $(\theta_{1}, u_{1}, R_{1})$ that satisfies \eqref{est 172} in same corresponding regularity class \eqref{est 241} such that $\theta_{1}(0,x) = \theta^{\text{in}}(x), u_{1}(0,x) = u^{\text{in}}(x),$ $\fint_{\mathbb{T}^{d}} \theta_{1}(t,x) dx = 0$ for all $t \in [0, T]$, and 
\begin{equation}\label{est 246}
\lVert (\theta_{1} - \theta_{0})(t) \rVert_{L_{x}^{p}} \leq 
\begin{cases}
M  [  2 \delta + \Gamma]^{\frac{1}{p}} &\forall \hspace{1mm} t \in (4 \Sigma \wedge T, T], \\
M  ( \sup_{\tau \in [t- l, t]} \lVert R_{0} (\tau) \rVert_{L_{x}^{1}} + \Gamma)^{\frac{1}{p}} &\forall \hspace{1mm}  t \in (\frac{\Sigma}{2} \wedge T, 4 \Sigma \wedge T], \\
0 &\forall \hspace{1mm}  t \in [0, \frac{\Sigma}{2} \wedge T], 
\end{cases}
\end{equation} 
\begin{equation}\label{est 247} 
\lVert (\theta_{1} - \theta_{0})(t) \rVert_{L_{x}^{\upsilon}} \leq 
\begin{cases}
\delta & \forall \hspace{1mm} t \in [0, T], \\
0 & \forall \hspace{1mm} t \in [0, \frac{\Sigma}{2} \wedge T], 
\end{cases}
\end{equation} 
\begin{equation}\label{est 248}
\lVert (u_{1} - u_{0})(t) \rVert_{L_{x}^{p'}} \leq 
\begin{cases}
M  [ 2 \delta  + \Gamma]^{\frac{1}{p'}} &\forall \hspace{1mm}  t \in (4 \Sigma \wedge T, T], \\
M ( \sup_{\tau \in [t- l, t]} \lVert R_{0} (\tau) \rVert_{L_{x}^{1}} + \Gamma)^{\frac{1}{p'}} &\forall \hspace{1mm}  t \in (\frac{\Sigma}{2} \wedge T, 4 \Sigma \wedge T], \\
0 &\forall \hspace{1mm}  t \in [0, \frac{\Sigma}{2} \wedge T], 
\end{cases}
\end{equation} 
\begin{equation}\label{est 249}
\lVert (u_{1} -u_{0})(t) \rVert_{W_{x}^{1, \tilde{p}}} \leq 
\begin{cases}
\delta &\forall \hspace{1mm}  t \in [0, T], \\
0 &\forall \hspace{1mm}  t \in [0, \frac{\Sigma}{2} \wedge T], 
\end{cases}
\end{equation} 
\begin{equation}\label{est 250} 
\lVert R_{1}(t) \rVert_{L_{x}^{1}} \leq 
\begin{cases}
\delta &\forall \hspace{1mm}  t \in (\Sigma \wedge T, T], \\
\sup_{\tau \in [t-l,t]} \lVert R_{0} (\tau) \rVert_{L_{x}^{1}}  + \delta &\forall \hspace{1mm}  t \in (\frac{\Sigma}{2} \wedge T, \Sigma \wedge T], \\
\lVert R_{0}(t) \rVert_{L_{x}^{1}} &\forall \hspace{1mm}  t \in [0, \frac{\Sigma}{2} \wedge T], 
\end{cases}
\end{equation} 
and 
\begin{align}
&\lvert \int_{\mathbb{T}^{d}} \theta_{1}(t,x) u_{1}(t, x+ B(t))dx - \int_{\mathbb{T}^{d}} \theta_{0}(t,x) u_{0}(t,x +B(t)) dx \nonumber\\
& \hspace{25mm} - \sum_{j=1}^{d} \int_{\mathbb{T}^{d}} \chi_{j}^{2}(t,x) dx e_{j} \Gamma \rvert \leq  \delta 3d \hspace{3mm} \forall \hspace{1mm} t \in (4 \Sigma \wedge T, T]. \label{est 293}
\end{align} 
\end{proposition}

\begin{remark}\label{Remark 6.2}
To prove Theorem \ref{Theorem 2.5} we only need to rely on Proposition \ref{Proposition 6.1} for $\rho^{\text{in}} \equiv 0, u^{\text{in}} \equiv 0$ and the iteration argument in the proof of \cite[Theorem 1.1]{HZZ21a} is not needed because the case of transport noise does not require a stopping time $T_{L}$. Nonetheless, we proved such a slightly more general result in Proposition \ref{Proposition 6.1} allowing any $\theta^{\text{in}} \in L^{p}(\mathbb{T}^{d})$ and $u^{\text{in}} \in L^{p'}(\mathbb{T}^{d})$ in hope that it may in future lead to improvement of Theorem \ref{Theorem 2.2} preserving the continuity in time and $(\mathcal{F}_{t})_{t\geq 0}$-adaptedness of vector field $u$. We also mention that the estimate \eqref{est 247} in Proposition \ref{Proposition 6.1} is new, and we included it to guarantee that the initial data of each iteration of $\theta$ remains the same. 
\end{remark} 

\begin{proof}[Proof of Proposition \ref{Proposition 6.1}]
We adhere to some of the settings the proof of Proposition \ref{Proposition 4.2}: $a, b$ from \eqref{est 55}, $r$ from Lemma \ref{Lemma 3.8}, $\varrho$ from \eqref{est 56}, $\psi$ from \eqref{est 57}, $\lambda, \mu, \omega, \nu$ from \eqref{est 58} and more specifically \eqref{est 148}, $\Theta_{\lambda, \mu, \omega, \nu}^{j}$,$ W_{\lambda, \mu, \omega, \nu}^{j}$, $Q_{\lambda, \mu, \omega, \nu}^{j}$ for $j \in \{1,\hdots, d\}$ from \eqref{est 59}, and $\epsilon$ from \eqref{est 60} so that Lemma \ref{Lemma 4.3} remains applicable. We now define 
\begin{subequations}\label{est 245} 
\begin{align}
& \theta_{1}(t,x) \triangleq \theta_{0}(t,x) + \tilde{\vartheta} (t,x) + \tilde{\vartheta}_{c}(t) + \tilde{q} (t,x) + \tilde{q}_{c}(t), \label{est 245a}\\
& u_{1}(t,x) \triangleq u_{0} (t,x) + \tilde{w}(t,x - B(t)) + \tilde{w}_{c}(t,x- B(t)), \label{est 245b}
\end{align}
\end{subequations}
(cf. \eqref{est 76}) where
\begin{subequations}\label{est 255} 
\begin{align}
& \tilde{\vartheta}(t,x) \triangleq \tilde{\chi} (t) \vartheta(t,x), \hspace{1mm} \tilde{\vartheta}_{c}(t) \triangleq \tilde{\chi}(t) \vartheta_{c}(t), \hspace{1mm} \tilde{q}(t,x) \triangleq \tilde{\chi}^{2}(t) q(t,x), \hspace{1mm} \tilde{q}_{c}(t) \triangleq \tilde{\chi}^{2}(t) q_{c}(t), \\
& \tilde{w}(t,x) \triangleq \tilde{\chi}(t) w(t,x), \hspace{1mm} \tilde{w}_{c}(t,x) \triangleq \tilde{\chi}(t) w_{c}(t,x), \\
& \vartheta(t,x) \triangleq  \sum_{j=1}^{d} \chi_{j} (t,x) sgn (R_{l}^{j} (t,x) + \Gamma) \lvert R_{l}^{j}(t,x) + \Gamma \rvert^{\frac{1}{p}} \Theta_{\lambda, \mu, \omega, \nu}^{j} (t,x),\\
& w(t,x) \triangleq  \sum_{j=1}^{d} \chi_{j}(t,x) \lvert R_{l}^{j} (t,x) + \Gamma \rvert^{\frac{1}{p'}} W_{\lambda, \mu, w, \nu}^{j} (t,x), \\
&q(t,x) \triangleq \sum_{j=1}^{d} \chi_{j}^{2}(t,x) (R_{l}^{j} (t,x) + \Gamma) Q_{\lambda, \mu, w, \nu}^{j}(t,x), 
\end{align}
\end{subequations}  
(cf. \eqref{est 79}) with $\chi_{j}$ defined in \eqref{est 243} and  
\begin{equation}\label{est 257} 
\tilde{\chi}(t)  
\begin{cases}
=0 & t \leq \frac{\Sigma}{2}, \\
\in [0,1] & t \in (\frac{\Sigma}{2}, \Sigma), \\
=1 & t \geq \Sigma,  
\end{cases} 
\end{equation} 
is a smooth cut-off function. We note that $\tilde{\chi}^{2}$ in the definition of $\tilde{q}$ and $\tilde{q}_{c}$ are needed for a cancellation upon defining $R_{1}$, as we will see in \eqref{est 356}. Furthermore, we define $\vartheta_{c}$ and $q_{c}$ identically to \eqref{est 85} so that the mean-zero property of $\theta_{0}$ from hypothesis implies that of $\theta_{1}$ defined in \eqref{est 245}. Additionally, we define 
\begin{subequations}\label{est 244}
\begin{align}
& a_{j}(t,x) \triangleq  \chi_{j}(t,x) sgn (R_{l}^{j}(t,x) + \Gamma) \lvert R_{l}^{j}(t,x) + \Gamma \rvert^{\frac{1}{p}}, \\
& b_{j}(t,x) \triangleq  \chi_{j}(t,x) \lvert R_{l}^{j}(t,x) + \Gamma \rvert^{\frac{1}{p'}}, 
\end{align}
\end{subequations}
so that 
\begin{equation}\label{est 262}
a_{j}(t,x) b_{j}(t,x) = \chi_{j}^{2}(t,x) (R_{l}^{j}(t,x) + \Gamma)
\end{equation}
(cf. \eqref{est 77} and \eqref{est 101}) and 
\begin{equation}\label{est 258}
\vartheta(t) = \sum_{j=1}^{d} a_{j}(t) \Theta^{j}(t), \hspace{3mm} w(t) = \sum_{j=1}^{d} b_{j}(t) W^{j}(t), \hspace{3mm} q(t) = \sum_{j=1}^{d} a_{j}(t) b_{j}(t) Q^{j}(t). 
\end{equation} 
Moreover, similarly to \eqref{est 81}, $a_{j}, b_{j}$ satisfy due to \eqref{est 243} 
\begin{equation}\label{est 259}
\lVert a_{j}(t) \rVert_{L_{x}^{p}} \leq \lVert R_{l}^{j}(t) + \Gamma \rVert_{L_{x}^{1}}^{\frac{1}{p}} \hspace{3mm} \text{ and } \hspace{3mm} \lVert b_{j}(t) \rVert_{L_{x}^{p'}} \leq  \lVert R_{l}^{j}(t) + \Gamma \rVert_{L_{x}^{1}}^{\frac{1}{p'}}.
\end{equation}  
With this definition of $b_{j}$, the identity \eqref{est 75} remains valid so that if we define $f_{j}$ and $w_{c}$ identically to \eqref{est 73} with $b_{j}$ from \eqref{est 244}, then \eqref{est 197} remains valid so that $u_{1}$ defined in \eqref{est 245} satisfies \eqref{est 98}. 

\begin{remark}\label{Difficulty 2}
As we pointed out in Remark \ref{Remark 4.3}, the choice of not mollifying $\theta_{0}$ and $u_{0}$ is not only for the simplification of the proofs of Theorems \ref{Theorem 2.2}-\ref{Theorem 2.4} but necessary in the proof of Proposition \ref{Proposition 6.1}. Suppose that we mollified both $\theta_{0}$ and $u_{0}$ to obtain $\theta_{l} = \theta_{0}\ast_{x} \phi_{l}\ast_{t} \varphi_{l}$ and $u_{l} = u_{0} \ast_{x} \phi_{l} \ast_{t} \varphi_{l}$ so that instead of \eqref{est 245} we define $\theta_{1} = \theta_{l} + \tilde{\vartheta} + \tilde{\vartheta}_{c} + \tilde{q} + \tilde{q}_{c}$ and $u_{1} = u_{l} + \tilde{w} + \tilde{w}_{c}$. This will make it very difficult to prove $\lVert (\theta_{1} - \theta_{0})(t) \rVert_{L_{x}^{p}} = 0$ and $\lVert (u_{1} - u_{0})(t) \rVert_{L_{x}^{p'}} = 0$ over $[0, \frac{\Sigma}{2} \wedge T]$ in \eqref{est 246} and \eqref{est 248}. The authors in \cite{HZZ21a} are able to handle this issue as follows. They work on an additive case, assign $z(0,x)$ to be the initial data of the solution (see ``$z(0) = u_{0}$'' on \cite[p. 35]{HZZ21a}) so that the initial data of the solution within the actual convex integration have zero initial data (see ``$v_{q}(0) = 0$'' in \cite[Equation (5.1)]{HZZ21a}) which allows them to add an extra inductive hypothesis ``$\lVert v_{q}(t) \rVert_{L^{2}} = 0$ for all $t \in [0, \frac{\sigma_{q-1}}{2} \wedge T_{L}]$'' in \cite[Equation (5.5)]{HZZ21a} so that by taking advantage of $supp \varphi_{l} \subset [0,l]$, they are able to prove ``$\lVert v_{q+1} (t) - v_{q}(t) \rVert_{L^{2}} = 0$ for $t \in [0, \frac{\sigma_{q}}{2} \wedge T_{L}]$'' in \cite[Equation (5.9)]{HZZ21a}. As we are working in the case of transport noise rather than additive, we have no place to hide $\theta^{\text{in}}$ or $u^{\text{in}}$ such as ``$z(0)$'' in \cite{HZZ21a}. 
\end{remark}
\begin{remark}\label{Remark on difficulty of transport} 
At first sight, our choice of $u_{1}(t,x)$ in \eqref{est 245b} seems awkward and a more natural choice may be 
\begin{equation}\label{est 190}  
u_{1}(t,x) \triangleq u_{0} (t,x) + \tilde{w}(t,x) + \tilde{w}_{c}(t,x).
\end{equation} 
identically to \eqref{est 76}. To describe the problem with this choice simply, let us assume that $t \geq \Sigma$ so that $\tilde{\chi}(t) = 1$ by \eqref{est 257}  and hence $\tilde{\vartheta} = \vartheta$, $\tilde{\vartheta}_{c} = \vartheta_{c}$, $\tilde{q} = q$, $\tilde{q}_{c} = q$, $\tilde{w} = w,$ and $\tilde{w}_{c} = w_{c}$ due to \eqref{est 255} (see \eqref{est 357} for this case). Then, upon defining the new defect $R_{1}$, we obtain from \eqref{est 172}, \eqref{est 245a}, and \eqref{est 190}, 
\begin{align}
& - \text{div} R_{1}(t,x) \label{est 191}\\
=&\underbrace{ \partial_{t} (q(t,x) + q_{c}(t)) + \text{div} (\vartheta(t,x) w(t, x+ B(t)) - R_{l}(t,x))}_{(\text{div} R^{\text{time,1}} + \text{div} R^{\text{quadr}} + \text{div} R^{\chi} )(t,x)} \nonumber \\
&+ \underbrace{ \partial_{t} (\vartheta(t,x) + \vartheta_{c}(t)) + \text{div} ( \theta_{0}(t,x) w(t, x+ B(t) ) + \vartheta(t,x) u_{0}(t, x+ B(t))) - \Delta (\vartheta(t,x) + q(t,x))}_{( \text{div} R^{\text{time,2}} + \text{div}R^{\text{lin}})(t,x)} \nonumber \\
&+ \underbrace{\text{div} (q(t,x) (u_{0} (t, x+ B(t)) + w(t, x+ B(t)))}_{\text{div}R^{q}(t,x)} + \underbrace{ \text{div} ([\theta_{0}(t,x) + \vartheta(t,x) + q(t,x) ]w_{c}(t, x+ B(t) ))}_{\text{div} R^{\text{corr,1}}(t,x)} \nonumber\\
&+ \underbrace{\text{div} ( R_{l}(t,x) - R_{0} (t,x))}_{\text{div} R^{\text{moll}}(t,x)}.  \nonumber 
\end{align}
The main difficulty arises from the mismatch of variables in $\text{div} (\vartheta(t,x) w(t, x+ B(t)))$ in \eqref{est 191}. A glance at \eqref{est 102} shows that we need a certain cancellation due to orthogonality for this product; specifically, 
\begin{equation}
\vartheta(t,x) w(t, x +B(t)) \overset{\eqref{est 71}}{=} (\sum_{j=1}^{d} a_{j}(t,x) \Theta^{j}(t,x)) ( \sum_{k=1}^{d} b_{k}(t, x+ B(t)) W^{k}(t, x+B(t))) 
\end{equation}
where due to \eqref{est 59} 
\begin{equation}\label{est 195}
\Theta^{j}(t, x) W^{k} (t, x+B(t))= \varrho_{\mu}^{j} (\lambda (x - \omega t e_{j}) )\psi^{j}(\nu x)\tilde{\varrho}_{\mu}^{k} (\lambda ( x+ B(t) - \omega t e_{j})) \psi^{k} (\nu(x+ B(t))) e_{k} 
\end{equation} 
We used in \eqref{est 102} that fact that $\varrho_{\mu}^{j} (\lambda (x - \omega t e_{j}) ) \tilde{\varrho}_{\mu}^{k} (\lambda ( x - \omega t e_{j})) = 0$ unless $ j = k$. In detail, 
\begin{align*}
&\varrho_{\mu}^{j} (\lambda ( x - \omega t e_{j}) \overset{\eqref{est 192}}{=} ( \varrho_{\mu} \circ \tau_{\zeta_{j}}) (\lambda (x - \omega te_{j}))  \overset{\eqref{est 193}}{=} \varrho_{\mu} (\lambda x - \mathfrak{x}_{j} (\lambda \omega t)), \\
& \tilde{\varrho}_{\mu}^{k} (\lambda (x - \omega t e_{k} )) \overset{\eqref{est 192}}{=} (\tilde{\varrho}_{\mu} \circ \tau_{\zeta_{k}}) (\lambda (x - \omega t e_{k} )) \overset{\eqref{est 193}}{=} \tilde{\varrho}_{\mu} (\lambda x - \mathfrak{x}_{k} (\lambda \omega t)) 
\end{align*}
where 
\begin{align*}
d_{\mathbb{T}^{d}} (\lambda x - \mathfrak{x}_{j} (\lambda \omega t)), \lambda x - \mathfrak{x}_{k}(\lambda \omega t)) = d_{\mathbb{T}^{d}}(\mathfrak{x}_{j} (\lambda \omega t), \mathfrak{x}_{k} (\lambda \omega t)) \overset{\eqref{est 194}}{>} 2r  
\end{align*}
if $j \neq k$ and because $supp \varrho_{\mu} = supp \tilde{\varrho}_{\mu}$ and both are contained in a ball with radius at most $r$  due to Lemma \ref{Lemma 3.9}, we conclude that $\varrho_{\mu}^{j} (\lambda (x - \omega t e_{j}) ) \tilde{\varrho}_{\mu}^{k} (\lambda ( x - \omega t e_{j})) = 0$ unless $ j = k$. However, an identical computation in the case of \eqref{est 195} leads only to 
\begin{align*}
d_{\mathbb{T}^{d}} (\lambda x - \mathfrak{x}_{j} (\lambda \omega t)), \lambda x + \lambda B(t) - \mathfrak{x}_{k}(\lambda \omega t)) = d_{\mathbb{T}^{d}}(\mathfrak{x}_{j} (\lambda \omega t), - \lambda B(t) + \mathfrak{x}_{k} (\lambda \omega t)) 
\end{align*}
which is not deterministic and thus does not lead to the desired orthogonality. This lack of orthogonality, and lack of necessary cancellations, is quite significant. A naive attempt of adding $\Theta^{j}(t,x) W^{k} (t,x)$ to obtain the necessary orthogonality, and then subtracting and thereafter trying to handle an estimate of $\lVert \Theta^{j}(t,x) [W^{k} (t,x+ B(t)) - W^{k}(t,x) ] \rVert_{L_{x}^{1}}$ failed miserably. Therefore, the definition \eqref{est 190} really does not work and we need \eqref{est 245b}. Lastly, we point out that within \eqref{est 245b}, we need $\tilde{w}(t, x-B(t))$ to deduce the necessary orthogonality and $\tilde{w}_{c}(t, x-B(t))$ to secure the divergence-free property.  
\end{remark}
We first prove \eqref{est 246} and \eqref{est 248}. Let us work in case $t \in (4\Sigma \wedge T, T]$. Clearly if $T \leq 4 \Sigma$, then there is nothing to prove; thus, we assume $4 \Sigma < T$ so that $(4 \Sigma \wedge T, T] = (4\Sigma, T]$. 
\begin{lemma}\label{Lemma 6.2} 
For all $j \in \{1,\hdots, d \}$ and $t \in (4\Sigma, T]$, $a_{j}, b_{j}$ defined in \eqref{est 244} satisfy
\begin{subequations}
\begin{align}
&\lVert a_{j}(t) \rVert_{L_{x}^{\infty}} \lesssim l^{- \frac{d+1}{p}} \delta^{\frac{1}{p}}, \hspace{7mm}  \lVert b_{j}(t) \rVert_{L_{x}^{\infty}} \lesssim  l^{- \frac{d+1}{p'}} \delta^{\frac{1}{p'}},\label{est 252} \\
& \lVert a_{j}(t) \rVert_{C_{x}^{s}} \lesssim l^{-(d+2) s} \delta^{\frac{1}{p}}, \hspace{5mm} \lVert b_{j}(t) \rVert_{C_{x}^{s}} \lesssim l^{-(d+2) s} \delta^{\frac{1}{p'}} \hspace{3mm} \forall \hspace{1mm} s \in \mathbb{N}, \label{est 253} \\
& \lVert \partial_{t} a_{j}(t) \rVert_{L_{x}^{\infty}} \lesssim l^{-(d+2)} \delta^{\frac{1}{p}}, \hspace{3mm}  \lVert \partial_{t} b_{j}(t) \rVert_{L_{x}^{\infty}} \lesssim l^{-(d+2)} \delta^{\frac{1}{p'}}. \label{est 254} 
\end{align}
\end{subequations}
\end{lemma}

\begin{proof}[Proof of Lemma \ref{Lemma 6.2}]
Because $4 \Sigma < T$ by assumption, $2\Sigma < T$ so that $[2\Sigma \wedge T, T] = [2\Sigma, T]$. Hence, by taking $\lambda \in\mathbb{N}$ sufficiently large we can assure that $l \ll 2 \Sigma$ and hence for all $s \in supp \varphi_{l} \subset (0, l]$ from \eqref{est 242}, for any $t \in (4 \Sigma \wedge T, T]$, we can estimate $\sup_{s \in supp \varphi_{l}} \lVert R_{0}^{j} (t-s) \rVert_{L_{x}^{1}} \leq \sup_{\tau \in [2\Sigma \wedge T, T]} \lVert R_{0}^{j} (\tau) \rVert_{L_{x}^{1}}$ to which we can apply \eqref{est 251} to bound by $2\delta$. Thus, very similar computations to the proof of Lemma \ref{Lemma 4.4} using $W^{d+1,1}(\mathbb{T}^{d}) \hookrightarrow L^{\infty} (\mathbb{T}^{d})$ lead to \eqref{est 252}, e.g., 
\begin{align*}
\lVert a_{j}(t) \rVert_{L_{x}^{\infty}} \lesssim  ( \lVert R_{l}^{j}(t,x) \rVert_{L_{x}^{\infty}}^{\frac{1}{p}} + \Gamma^{\frac{1}{p}}) 
\overset{\eqref{est 251}}{\lesssim}  ( l^{-\frac{d+1}{p}} \delta^{\frac{1}{p}} + (\delta \bar{C})^{\frac{1}{p}})   
\lesssim  l^{-\frac{d+1}{p}} \delta^{\frac{1}{p}}  
\end{align*}
for $\lambda \in\mathbb{N}$ sufficiently large. For the estimates \eqref{est 253}-\eqref{est 254}, following the proof of Lemma \ref{Lemma 4.4} and keeping in mind the lower bound of $\lvert R_{l}^{j} (t,x) + \Gamma \rvert > \frac{\delta}{4d}$ in the support of $\chi_{j}$ due to \eqref{est 243} give us the desired results. 
\end{proof} 

Analogously to the proof of Theorem \ref{Theorem 2.2}, Lemma \ref{Lemma 6.2} leads to to the following result.
\begin{lemma}\label{Lemma 6.3}
There exist constants $C = C(p) \geq 0$ with which $\tilde{\vartheta}$, $\tilde{q}$, $\tilde{\vartheta}_{c}$, $\tilde{q}_{c}$, and $\tilde{w}$ in \eqref{est 255} satisfy for all $t \in (4\Sigma, T]$ 
\begin{subequations}\label{est 369}
\begin{align}
&\lVert \tilde{\vartheta}(t) \rVert_{L_{x}^{p}} \leq \frac{M}{2} [ 2 \delta  + \Gamma]^{\frac{1}{p}} + C \lambda^{-\frac{1}{p}}  l^{-d-2}\delta^{\frac{1}{p}},  \label{est 256}\\
& \lVert \tilde{q}(t) \rVert_{L_{x}^{p}} \leq C l^{-d-1} \delta \mu^{b} \omega^{-1}, \label{est 260} \\
&  \lvert \tilde{\vartheta}_{c} (t) \rvert \leq C  l^{-\frac{d+1}{p}} \delta^{\frac{1}{p}} \mu^{-b}, \hspace{5mm} \lvert \tilde{q}_{c}(t) \rvert \leq C l^{-d-1} \delta \omega^{-1}, \label{est 261} \\
& \lVert \tilde{w} (t) \rVert_{L_{x}^{p'}} \leq \frac{M}{2} [ 2 \delta + \Gamma]^{\frac{1}{p'}} +C\lambda^{-\frac{1}{p'}} \delta^{\frac{1}{p'}} l^{-d-2}, \label{est 263} \\
& \lVert \tilde{w} (t) \rVert_{W_{x}^{1,\tilde{p}}} \leq C \delta^{\frac{1}{p'}} l^{-d-2} \frac{\lambda \mu + \nu}{\mu^{1+ \epsilon}} \label{est 264} 
\end{align}
\end{subequations} 
(cf. \eqref{est 134}, \eqref{est 135}, \eqref{est 151}, \eqref{est 153}, and \eqref{est 154}). Furthermore, for any $k, h \in \mathbb{N}_{0}$ and $r \in [1,\infty]$, there exists a constant $C\geq 0$ with which for all $j \in 1, \hdots, d\}$, $f_{j}$ defined in \eqref{est 73} with $b_{j}$ from \eqref{est 244} satisfies for all $t \in (4\Sigma, T]$  
\begin{equation}\label{est 265}
\lVert \mathcal{D}^{k} D^{h} f_{j}(t) \rVert_{L_{x}^{r}} \leq C \delta^{\frac{1}{p'}} l^{-(d+2) (k+ h + 1)} (\lambda \mu)^{k+ h+1} \mu^{b - \frac{d}{r}}
\end{equation} 
(cf. \eqref{est 91}). Consequently, there exist constants $C \geq 0$ with which $\tilde{w}_{c}$ in \eqref{est 255} with $w_{c}$ in \eqref{est 73} and $b_{j}$ in \eqref{est 244} satisfies for all $t \in (4\Sigma, T]$
\begin{subequations}
\begin{align}
&\lVert \tilde{w}_{c}(t) \rVert_{L_{x}^{p'}} \leq C \delta^{\frac{1}{p'}} [\sum_{k=1}^{N} \left( \frac{\lambda \mu l^{-(d+2)}}{\nu} \right)^{k} + \frac{ (\lambda \mu l^{-(d+2)} )^{N+1}}{\nu^{N}} ],  \label{est 266}\\
& \lVert \tilde{w}_{c}(t) \rVert_{W_{x}^{1,\tilde{p}}} \leq C \frac{ \delta^{\frac{1}{p'}} [ \lambda \mu l^{-(d+2)} + \nu]}{\mu^{1+ \epsilon}} [ \sum_{k=1}^{N} \left( \frac{ l^{-(d+2)} \lambda \mu}{\nu} \right)^{k} + \frac{ ( l^{-(d+2)} \lambda \mu)^{N+1}}{\nu^{N}} ]\label{est 267}
\end{align}
\end{subequations}
(cf. \eqref{est 136} and \eqref{est 93}). 
\end{lemma} 

\begin{proof}[Proof of Lemma \ref{Lemma 6.3}]
Because $t \in (4 \Sigma \wedge T, T]$, we have $\tilde{\chi}(t) = 1$ by \eqref{est 257}.  Especially for \eqref{est 256} and \eqref{est 263}, by taking $\lambda \in\mathbb{N}$ sufficiently large so that $l \ll 2 \Sigma$, we can estimate for any $t \in (4\Sigma \wedge T, T]$ $\sup_{s \in supp \varphi_{l}} \lVert R_{0}^{j} (t-s) \rVert_{L_{x}^{1}} \leq \sup_{\tau \in [2\Sigma \wedge T, T]} \lVert R_{0}^{j} (\tau) \rVert_{L_{x}^{1}}$ and bound this by $2\delta$ due to \eqref{est 251}. Besides this issue, the proof is similar to that of Lemma \ref{Lemma 4.5} and thus we only sketch their computations: via \eqref{est 255}-\eqref{est 257}
\begin{align*}
\lVert \tilde{\vartheta} (t) \rVert_{L_{x}^{p}} \overset{\eqref{est 80}}{\leq} \sum_{j=1}^{d}& \lVert a_{j}(t) \rVert_{L_{x}^{p}} \lVert \Theta^{j} (t) \rVert_{L_{x}^{p}} + \frac{C_{p}}{\lambda^{\frac{1}{p}}} \lVert a_{j}(t) \rVert_{C_{x}^{1}} \lVert \Theta^{j}(t) \rVert_{L_{x}^{p}} \\
\overset{\eqref{est 259} \eqref{est 61} \eqref{est 253}\eqref{est 251}}{\leq}& \frac{M}{2} [ 2\delta + \Gamma]^{\frac{1}{p}} + C \lambda^{-\frac{1}{p}}  \delta^{\frac{1}{p}} l^{-d-2},
\end{align*} 
\begin{equation*}
\lVert \tilde{q}(t) \rVert_{L_{x}^{p}} \lesssim \sum_{j=1}^{d} ( \lVert R_{l}^{j}(t) \rVert_{L_{x}^{\infty}} + \Gamma) \lVert Q^{j} (t) \rVert_{L_{x}^{p}}  \overset{\eqref{est 61} \eqref{est 251}}{\lesssim} l^{-d-1} \delta  \mu^{b} \omega^{-1}, 
\end{equation*}
\begin{equation*}
\lvert \tilde{\vartheta}_{c}(t)\rvert \overset{\eqref{est 85} \eqref{est 258}}{\leq} \sum_{j=1}^{d} \lVert a_{j}(t) \rVert_{L_{x}^{\infty}} \lVert \Theta^{j} (t) \rVert_{L_{x}^{1}} \overset{\eqref{est 252}\eqref{est 62}}{\leq} C  l^{-\frac{d+1}{p}} \delta^{\frac{1}{p}} \mu^{-b}, 
\end{equation*}
\begin{equation*}
\lvert \tilde{q}_{c}(t) \rvert \overset{\eqref{est 85} \eqref{est 258} \eqref{est 262} \eqref{est 62}}{\lesssim} \sum_{j=1}^{d} ( \lVert R_{l}^{j}(t) \rVert_{L_{x}^{\infty}} + \Gamma) (\frac{M}{\omega}) \overset{\eqref{est 251}}{\lesssim} ( l^{-d-1} \delta + \delta \bar{C}) \omega^{-1} \leq C l^{-d-1} \delta \omega^{-1},
\end{equation*}
\begin{align*}
\lVert \tilde{w}(t) \rVert_{L_{x}^{p'}} \overset{\eqref{est 80}}{\leq}  \sum_{j=1}^{d}& \lVert b_{j}(t) \rVert_{L_{x}^{p'}} \lVert W^{j} (t) \rVert_{L_{x}^{p'}} + C_{p'} \lambda^{-\frac{1}{p'}} \lVert b_{j} (t) \rVert_{C_{x}^{1}} \lVert W^{j}(t)\rVert_{L_{x}^{p'}} \\
\overset{\eqref{est 259} \eqref{est 61} \eqref{est 253} \eqref{est 251}}{\leq}& \frac{ M}{2} [  2 \delta  + \Gamma]^{\frac{1}{p'}} + C \lambda^{-\frac{1}{p'}}  \delta^{\frac{1}{p'}} l^{-d-2}, 
\end{align*}
\begin{align*}
\lVert \tilde{w}(t) \rVert_{W_{x}^{1,\tilde{p}}} \overset{\eqref{est 258}}{\lesssim} \sum_{j=1}^{d} \lVert b_{j}(t) \rVert_{C_{x}^{1}} \lVert W^{j}(t) \rVert_{W_{x}^{1,\tilde{p}}} 
\overset{ \eqref{est 253} \eqref{est 63}}{\leq} C \delta^{\frac{1}{p'}} l^{-d-2} \frac{ \lambda \mu + \nu}{\mu^{1+ \epsilon}},
\end{align*}
\begin{align*}
\lVert \mathcal{D}^{k} D^{h} f_{j}(t) \rVert_{L_{x}^{r}} \overset{\eqref{est 88} \eqref{est 73}}{\lesssim} \lVert b_{j}(t)& \rVert_{C_{x}^{k+h+1}} \lVert ( \tilde{\varrho}_{\mu}^{j})_{\lambda} \circ \tau_{\omega t e_{j}} \rVert_{W_{x}^{k+h+1, r}} \\
\overset{\eqref{est 67} \eqref{est 66} \eqref{est 253} \eqref{est 68}}{\leq}& C \delta^{\frac{1}{p'}} l^{-(d+2) (k+h+1)} (\lambda \mu)^{k+h+1} \mu^{b - \frac{d}{r}}, 
\end{align*}
\begin{align*}
\lVert \tilde{w}_{c}(t) \rVert_{L_{x}^{p'}} \overset{\eqref{est 90}}{\lesssim}& \sum_{j=1}^{d} C_{d, p', N} \lVert \psi \rVert_{L_{x}^{\infty}} ( \sum_{k=0}^{N-1} \nu^{-k-1} \lVert \mathcal{D}^{k} f_{j}(t) \rVert_{L_{x}^{p'}} + \nu^{-N} \lVert \mathcal{D}^{N} f_{j} (t) \rVert_{L_{x}^{p'}})  \\
\overset{\eqref{est 265} \eqref{est 55}}{\leq}& C   \delta^{\frac{1}{p'}} [ \sum_{k=1}^{N} \left( \frac{\lambda \mu l^{-(d+2)}}{\nu} \right)^{k} + \frac{ (\lambda \mu l^{-(d+2)})^{N+1}}{\nu^{N}}], 
\end{align*} 
and
\begin{align*}
\lVert D\tilde{w}_{c}(t) \rVert_{L_{x}^{\tilde{p}}} \lesssim& \sum_{j=1}^{d} \lVert \mathcal{R}_{N} (Df_{j} (t), \psi_{\nu}^{j} e_{j}) \rVert_{L_{x}^{\tilde{p}}} + \lVert \mathcal{R}_{N} (f_{j}(t), D \psi_{\nu}^{j} e_{j}) \rVert_{L_{x}^{\tilde{p}}} \\
\overset{\eqref{est 66} \eqref{est 90}  \eqref{est 265}}{\leq}& C   \frac{ \delta^{\frac{1}{p'}} [ \lambda \mu l^{-(d+2)} + \nu]}{\mu^{1+ \epsilon}} [ \sum_{k=1}^{N} \left( \frac{ l^{-(d+2)} \lambda \mu}{\nu} \right)^{k} + \frac{ ( l^{-(d+2)} \lambda \mu)^{N+1}}{\nu^{N}} ].
\end{align*}
\end{proof} 

We are now ready to verify \eqref{est 246} and \eqref{est 248} on $(4 \Sigma \wedge T, T]$ by taking $\lambda \in\mathbb{N}$ sufficiently large and $\iota > 0$ sufficiently small: 
\begin{subequations}
\begin{align} 
&\lVert (\theta_{1}  - \theta_{0})(t) \rVert_{L_{x}^{p}} 
\overset{ \eqref{est 245}\eqref{est 369} \eqref{est 148}}{\leq} \frac{M}{2} [ 2\delta + \Gamma]^{\frac{1}{p}} \label{est 370} \\\
& \hspace{5mm} + C [\lambda^{-\frac{1}{p} + (d+2) \iota}  \delta^{\frac{1}{p}} +  \lambda^{(\frac{d+1}{p}) \iota - \alpha b} \delta^{\frac{1}{p}} + \lambda^{(d+1) \iota + \alpha b - \beta} \delta + \lambda^{(d+1) \iota - \beta} \delta] \overset{\eqref{est 324}}{\leq}  M  [2\delta + \Gamma]^{\frac{1}{p}}, \nonumber \\
&\lVert (u_{1} - u_{0})(t) \rVert_{L_{x}^{p'}} 
\overset{\eqref{est 245} \eqref{est 263} \eqref{est 266} \eqref{est 148}}{\leq} \frac{ M}{2} [2\delta  + \Gamma]^{\frac{1}{p'}} \\
&\hspace{5mm} + C[ \lambda^{-\frac{1}{p'} + (d+2) \iota}  \delta^{\frac{1}{p'}} +   \delta^{\frac{1}{p'}} [\sum_{k=1}^{N} \lambda^{[1+ \alpha - \gamma + (d+2) \iota] k} + \lambda^{[1+ \alpha + (d+2) \iota](N+1) - \gamma N}] \overset{\eqref{est 325}}{\leq} M  [ 2 \delta + \Gamma]^{\frac{1}{p'}}, \nonumber 
\end{align}
\end{subequations} 
where the last inequality in \eqref{est 370} also used 
\eqref{est 324} and 
\begin{align*}
\iota \overset{\eqref{est 354}}{<} \frac{\alpha b}{d+2} < \frac{p \alpha b}{d+1} \text{ and } \iota \overset{\eqref{est 354}}{<} \frac{\beta}{d+2} < \frac{\beta}{d+1}. 
\end{align*}
Next, concerning \eqref{est 246} and \eqref{est 248} in the case $t \in (\frac{\Sigma}{2} \wedge T, 4 \Sigma \wedge T]$, we cannot make use of \eqref{est 251}; nonetheless, due to $supp \varphi_{l} \subset [0,l]$, we can estimate $\lVert R_{l}^{j}(t) \rVert_{L_{x}^{1}} \leq \sup_{\tau \in [t-l, t]} \lVert R_{0}^{j}(\tau) \rVert_{L_{x}^{1}}$. Therefore, analogous computations in the case $t \in (4\Sigma \wedge T, T]$ give us \eqref{est 246} and \eqref{est 248} in case $t \in (\frac{\Sigma}{2}\wedge T, 4 \Sigma \wedge T]$. In short, we get $\frac{M}{2} (\sup_{\tau \in [t-l,t]} \lVert R_{0}(\tau) \rVert_{L_{x}^{1}} + \Gamma)^{\frac{1}{p}}$ from $\lVert \tilde{\vartheta}(t) \rVert_{L_{x}^{p}}$ and $\frac{M}{2}  (\sup_{\tau \in [t-l,t]} \lVert R_{0}(\tau) \rVert_{L_{x}^{1}} + \Gamma)^{\frac{1}{p'}}$ from $\lVert \tilde{w}(t) \rVert_{L_{x}^{p'}}$ similarly to \eqref{est 256} and \eqref{est 263} and $\lvert \tilde{\vartheta}_{c}(t) \rvert + \lVert \tilde{q} (t) \rVert_{L_{x}^{p}} + \lvert \tilde{q}_{c}(t) \rvert$ from $\lVert \theta_{1}(t) - \theta_{0}(t) \rVert_{L_{x}^{p}}$ and $\lVert \tilde{w}_{c}(t) \rVert_{L_{x}^{p'}}$ from $\lVert u_{1}(t) - u_{0}(t) \rVert_{L_{x}^{p'}}$ can be made small by taking $\lambda \in\mathbb{N}$ sufficiently large. 

Next, we consider $t \in [0, \frac{\Sigma}{2} \wedge T] \subset [0, \frac{\Sigma}{2}]$ in which $\tilde{\chi} \equiv 0$ by \eqref{est 257} so that $\theta_{1} \equiv \theta_{0}$ and $u_{1} \equiv u_{0}$ due to \eqref{est 245}. As we elaborated in Remark \ref{Difficulty 2}, the fact that we did not mollify $\theta_{0}$ or $u_{0}$ makes it easy for us here. Therefore, \eqref{est 246} and \eqref{est 248} for $t \in [0, \frac{\Sigma}{2} \wedge T]$ are proven.  

Next, we consider \eqref{est 247}. First, for $t \in [0,\frac{\Sigma}{2} \wedge T] \subset [0, \frac{\Sigma}{2}]$ in which $\tilde{\chi} \equiv 0$ due to \eqref{est 257}, we again have $\theta_{1} \equiv \theta_{0}$ so that the claim in \eqref{est 247} on this time interval is clear. Moreover, for all $t \in [0, T]$, we see from \eqref{est 245} that 
\begin{align*}
\lVert (\theta_{1} - \theta_{0})(t) \rVert_{L_{x}^{\upsilon}} = \lVert \tilde{\vartheta}(t,x) + \tilde{\vartheta}_{c}(t) + \tilde{q}(t,x) + \tilde{q}_{c}(t) \rVert_{L_{x}^{\upsilon}} \leq \lVert \tilde{\vartheta}(t) \rVert_{L_{x}^{\upsilon}} + \lvert \tilde{\vartheta}_{c}(t) \rvert + \lVert \tilde{q}(t) \rVert_{L_{x}^{p}} + \lvert \tilde{q}_{c}(t) \rvert 
\end{align*}
because $\upsilon \in (1,p)$, where we can make $\lvert \tilde{\vartheta}_{c}(t) \rvert + \lVert \tilde{q}(t) \rVert_{L_{x}^{p}} + \lvert \tilde{q}_{c}(t) \rvert \ll \delta$ for $\lambda \in\mathbb{N}$ sufficiently large. The only reason we could not estimate $\lVert \theta_{1}(t) - \theta_{0}(t) \rVert_{L_{x}^{p}}$ by a constant multiple of $\delta$ over $(\frac{\Sigma}{2} \wedge T, 4 \Sigma \wedge T]$ where \eqref{est 251} is not available was due to $\lVert \tilde{\vartheta}(t) \rVert_{L_{x}^{p}}$. In detail, we can readily deduce the following estimates for all $t \in [0, T]$: 
\begin{align*}
&\lVert \tilde{\vartheta} (t) \rVert_{L_{x}^{p}} \overset{\eqref{est 80}  \eqref{est 259} }{\leq}  \frac{M}{2} (\sup_{\tau \in [t-l,t]} \lVert R_{0}(\tau) \rVert_{L_{x}^{1}} + \Gamma)^{\frac{1}{p}} + C   \lambda^{-\frac{1}{p}} [ ( l^{-d-1} \sup_{\tau \in [t-l,t]} \lVert R_{0}(\tau) \rVert_{L_{x}^{1}} + \Gamma)^{\frac{1}{p}}  \nonumber\\
& \hspace{60mm} + \delta^{\frac{1}{p} -1} (l^{-d-2} \sup_{\tau \in [t-l,t]} \lVert R_{0}(\tau) \rVert_{L_{x}^{1}} + \Gamma) ],\\
& \lvert \tilde{\vartheta}_{c}(t) \rvert \overset{\eqref{est 258} \eqref{est 62}}{\lesssim}  (l^{-d-1} \sup_{\tau \in [t-l,t]} \lVert R_{0}(\tau) \rVert_{L_{x}^{1}} + \Gamma)^{\frac{1}{p}} \mu^{-b} \overset{\eqref{est 148} \eqref{est 354}}{\ll} \delta, \\
&\lVert \tilde{q}(t) \rVert_{L_{x}^{p}} \overset{\eqref{est 258} \eqref{est 61}}{\lesssim} (l^{-d-1} \sup_{\tau \in [t-l,t]} \lVert R_{0}(\tau) \rVert_{L_{x}^{1}} + \Gamma) \mu^{b}\omega^{-1} \overset{\eqref{est 148} \eqref{est 324}}{\ll} \delta, \\
&\lvert \tilde{q}_{c}(t) \rvert \overset{\eqref{est 262}\eqref{est 62} }{\lesssim} (l^{-d-1} \sup_{\tau \in [t-l,t]} \lVert R_{0}(\tau) \rVert_{L_{x}^{1}} + \Gamma) \omega^{-1} \overset{\eqref{est 148} \eqref{est 353}}{\ll} \delta.
\end{align*} 
However, because $\upsilon \in (1,p)$ we can interpolate and use \eqref{est 269} to deduce for all $t \in [0,T]$  
\begin{align*}
&\lVert \tilde{\vartheta}(t) \rVert_{L_{x}^{\upsilon}} \overset{\eqref{est 258} \eqref{est 80}}{\leq} \sum_{j=1}^{d} \lVert a_{j}(t) \rVert_{L_{x}^{p}} \lVert \Theta^{j} (t) \rVert_{L_{x}^{1}}^{\frac{p-\upsilon}{\upsilon(p-1)}} \lVert \Theta^{j}(t) \rVert_{L_{x}^{p}}^{\frac{\upsilon p - p}{\upsilon (p-1)}} + C_{\upsilon} \lambda^{-\frac{1}{\upsilon}} \lVert a_{j}(t) \rVert_{C_{x}^{1}} \lVert \Theta^{j}(t) \rVert_{L_{x}^{p}} \\
&\overset{\eqref{est 259} \eqref{est 269}}{\lesssim} \sum_{j=1}^{d}   \lVert R_{l}^{j}(t) + \Gamma \rVert_{L_{x}^{1}}^{\frac{1}{p}} (\frac{M}{\mu^{b}})^{\frac{p -\upsilon}{\upsilon (p-1)}} (\frac{M}{2d})^{\frac{\upsilon p-p}{\upsilon (p-1)}} + \frac{1}{\lambda^{\frac{1}{\upsilon}}} [(l^{-d-1} \sup_{\tau \in [t-l,t]} \lVert R_{0}^{j}(\tau) \rVert_{L_{x}^{1}} + \Gamma)^{\frac{1}{p}} \\
& \hspace{35mm} + \delta^{\frac{1}{p} -1} (l^{-d-2} \sup_{\tau \in [t-l,t]} \lVert R_{0}^{j}(\tau) \rVert_{L_{x}^{1}} + \Gamma)](\frac{M}{2d}) \ll \delta
\end{align*}
by taking $\lambda \in\mathbb{N}$ sufficiently large and $\iota < \frac{1}{\upsilon (d+2)} < \frac{p}{\upsilon (d+1)}$. Therefore, we have proven \eqref{est 247} in case $t \in [0,T]$ and hence \eqref{est 247} completely. 

Next, we work on \eqref{est 249}. Again, on $[0, \frac{\Sigma}{2} \wedge T] \subset [0, \frac{\Sigma}{2}]$, we have $\tilde{\chi} \equiv 0$ due to \eqref{est 257} so that $u_{1} \equiv u_{0}$ due to \eqref{est 245} and thus the claim is clear. For $t \in [0, T]$, we can just estimate $\lVert R_{l}^{j}(t) \rVert_{L_{x}^{1}} \leq \sup_{\tau \in [t-l, t]} \lVert R_{0}^{j}(\tau) \rVert_{L_{x}^{1}}$ and analogous computations to \eqref{est 210} and taking $\iota > 0$ sufficiently small while $\lambda \in\mathbb{N}$ sufficiently large gives the desired result. 

Next, we prove \eqref{est 250}. First, we consider $t \in (\Sigma \wedge T, T]$. If $T \leq \Sigma$, then there is nothing to prove. Thus, we assume that $\Sigma < T$; i.e., $t \in (\Sigma, T]$. Then by \eqref{est 257}, $\tilde{\chi}(t) = 1$. With that in mind, the new defect $R_{1}$ is defined by \eqref{est 172} and \eqref{est 245} as follows: 
\begin{subequations}\label{est 357}
\begin{align}
& - \text{div}R_{1}(t,x) \nonumber \\
=& \partial_{t} \theta_{0} (t,x) + \partial_{t} [ \vartheta(t,x) + \vartheta_{c}(t) + q(t,x) + q_{c}(t) ]  \label{est 200} \\
&+ \text{div} [\theta_{0}(t,x) u_{0}(t, x+ B(t))]\nonumber \\
&+ \text{div} ( \theta_{0} (t,x) ( w(t,x) + w_{c}(t,x)) + ( \vartheta(t,x) + q(t,x)) (u_{0} (t, x+ B(t)) + w(t,x) + w_{c}(t,x)) \nonumber \\
& \hspace{5mm} + (\vartheta_{c}(t) + q_{c}(t)) (u_{0} (t, x+ B(t))+ w(t,x) + w_{c}(t,x)))\nonumber \\
& - \Delta \theta_{0} (t,x) - \Delta (\vartheta(t,x) + q(t,x)) \nonumber \\
=& \underbrace{\partial_{t} (q(t,x) + q_{c}(t,x)) + \text{div} (\vartheta(t,x) w(t,x) - R_{l}(t,x))}_{(\text{div} R^{\text{time,1}} + \text{div}R^{\text{quadr}} + \text{div}R^{\chi})(t,x)} \label{est 202}\\
&+ \underbrace{\partial_{t} (\vartheta(t,x) + \vartheta_{c}(t)) + \text{div} (\theta_{0}(t,x) w(t,x) + \vartheta(t,x) u_{0} (t, x + B(t) )) - \Delta (\vartheta(t,x) + q(t,x))}_{(\text{div} R^{\text{time,2}} + \text{div}R^{\text{lin}})(t,x)} \nonumber \\
&+ \underbrace{ \text{div} (q(t,x) ( u_{0} (t, x+ B(t)) + w(t, x))}_{\text{div}R^{q}(t,x)} \nonumber \\
&+ \underbrace{ \text{div} ([\theta_{0}(t,x) + \vartheta(t,x) + q(t,x) ] w_{c} (t,x))}_{\text{div} R^{\text{corr}}(t,x)} + \underbrace{\text{div} (R_{l}(t,x) - R_{0}(t,x))}_{\text{div} R^{\text{moll}}(t,x)}. \nonumber 
\end{align} 
\end{subequations}
Thus, we defined 
\begin{equation}\label{est 214}
-R_{1} \triangleq R^{\text{time,1}} + R^{\text{quadr}} + R^{\chi} + R^{\text{time,2}} + R^{\text{lin}} + R^{q} + R^{\text{corr}} + R^{\text{moll}}. 
\end{equation} 
Due to our strategic definition of $u_{1}$ in \eqref{est 245} (recall Remark \ref{Remark on difficulty of transport}), the following computations have some similarity to those in the proof of Theorem \ref{Theorem 2.2}. Nonetheless, we changed the definitions of $\vartheta, w, $and $q$ from \eqref{est 79} to \eqref{est 255}  and hence we provide details.

First, similarly to Section \ref{Section 7.2.2} we work on the estimates on $\partial_{t} (q+ q_{c})(t,x) + \text{div} (\vartheta w - R_{l})(t,x) = (\text{div} R^{\text{time,1}} + \text{div} R^{\text{quadr}} + \text{div} R^{\chi})(t,x)$ of \eqref{est 357}. We observe that 
\begin{align*}
\vartheta(t,x)w(t,x) \overset{\eqref{est 258}}{=} (\sum_{j=1}^{d} a_{j}(t,x) \Theta^{j}(t,x)) (\sum_{k=1}^{d} b_{k}(t,x) W^{k}(t,x)) 
\overset{\eqref{est 64}}{=} \sum_{j=1}^{d} a_{j}(t)b_{j}(t) \Theta^{j}(t)W^{j}(t)
\end{align*}
so that 
\begin{equation}\label{est 271}
\text{div} (\vartheta w) = \sum_{j=1}^{d} a_{j}b_{j}\text{div} (\Theta^{j}W^{j}) + \nabla (a_{j}b_{j}) \cdot \Theta^{j}W^{j}. 
\end{equation}
We set, differently from \eqref{est 105}, 
\begin{equation}\label{est 270}
R^{\chi} \triangleq - \sum_{j=1}^{d} (1-\chi_{j}^{2}) (R_{l}^{j} + \Gamma) e_{j}
\end{equation} 
so that using that $\text{div} \sum_{j=1}^{d} \Gamma e_{j} = 0$, we can compute 
\begin{equation}\label{est 272}
-\text{div}R_{l} \overset{\eqref{est 104}}{=} -\text{div} \sum_{j=1}^{d}(R_{l}^{j} + \Gamma) e_{j}  \overset{\eqref{est 270}\eqref{est 262}}{=} \text{div}R^{\chi} - \sum_{j=1}^{d} \nabla (a_{j}b_{j}) \cdot e_{j}. 
\end{equation} 
This leads us to, identically to \eqref{est 108} 
\begin{align}
\text{div} (\vartheta w - R_{l}) \overset{\eqref{est 271} \eqref{est 272}}{=}& \sum_{j=1}^{d} a_{j}b_{j} \text{div} (\Theta^{j} W^{j}) \nonumber \\
&  + \nabla (a_{j} b_{j}) \cdot [\Theta^{j} W^{j} - e_{j} ] - \fint_{\mathbb{T}^{d}} \nabla (a_{j} b_{j}) \cdot [\Theta^{j} W^{j} - e_{j} ] dx \nonumber \\
&+ \fint_{\mathbb{T}^{d}} \nabla (a_{j} b_{j}) \cdot [\Theta^{j}W^{j} - e_{j} ] dx + \text{div} R^{\chi}. \label{est 274}
\end{align}
On the other hand, identical computation to \eqref{est 109} give us 
\begin{align}\label{est 273}
 \partial_{t} ( q+ q_{c}) \overset{\eqref{est 258}}{=}& \sum_{j=1}^{d} a_{j} b_{j} \partial_{t}Q^{j} + \partial_{t} (a_{j} b_{j}) Q^{j} -  \fint_{\mathbb{T}^{d}} \partial_{t} (a_{j}b_{j}) Q^{j} dx \nonumber\\
&+ \fint_{\mathbb{T}^{d}} \partial_{t} (a_{j}b_{j}) Q^{j} dx + q_{c}'.
\end{align}
Summing \eqref{est 274}-\eqref{est 273} gives us identically to \eqref{est 112} 
\begin{subequations}\label{est 275}
\begin{align}
& \partial_{t} (q+ q_{c}) + \text{div} (\vartheta w - R_{l}) \nonumber \\
\overset{\eqref{est 274} \eqref{est 273}}{=}&  \sum_{j=1}^{d} a_{j} b_{j} [\partial_{t}Q^{j} + \text{div} (\Theta^{j}W^{j})] \label{est 276}\\
&+ [\partial_{t} (a_{j}b_{j}) Q^{j} - \fint_{\mathbb{T}^{d}} \partial_{t} (a_{j}b_{j} ) Q^{j} dx] \\
&+ [\nabla (a_{j}b_{j}) \cdot [\Theta^{j}W^{j} - e_{j} ] - \fint_{\mathbb{T}^{d}} \nabla (a_{j}b_{j}) \cdot [\Theta^{j} W^{j} - e_{j} ]dx ] + \text{div} R^{\chi} \\
&+ \fint_{\mathbb{T}^{d}} \partial_{t} (a_{j}b_{j})Q^{j}dx + \fint_{\mathbb{T}^{d}} \nabla (a_{j}b_{j}) \cdot [\Theta^{j} W^{j} - e_{j} ] dx + q_{c}'\label{est 277} 
\end{align}
\end{subequations}
where \eqref{est 276} vanishes due to \eqref{est 65} and  \eqref{est 277} also vanishes identically to \eqref{est 278} using \eqref{est 85}, \eqref{est 258}, and \eqref{est 65}. Therefore, we are able to conclude the same identity as \eqref{est 118}, and \eqref{est 117} also continues to hold from \eqref{est 59}. Thus, we define $R^{\text{time,1}}$ identically to \eqref{est 114}, although with $a_{j}$ and $b_{j}$ from \eqref{est 244} rather than \eqref{est 77}, so that \eqref{est 119} continues to hold. Additionally, we define $R^{\text{quadr}}$ with $R^{\text{quadr,1}}$ and $R^{\text{quadr,2}}$ identically to \eqref{est 115}-\eqref{est 116} with $a_{j}, b_{j}$ replaced by those from \eqref{est 244} rather than \eqref{est 77} so that \eqref{est 120} continues to hold. Therefore, by applying \eqref{est 119} and \eqref{est 120} to \eqref{est 118}, all of which we just proved to remain valid, we conclude that $\partial_{t} (q+ q_{c}) + \text{div} (\vartheta w - R_{l}) = \text{div} R^{\text{time,1}} + \text{div} R^{\text{quadr}} + \text{div} R^{\chi}$ as claimed. Analogous computations to the proofs of Lemma \ref{Lemma 4.6} lead to the following results. Considering the case $t \in (\frac{\Sigma}{2} \wedge T, \Sigma \wedge T]$ we will work next, we prove this result for all $t \in [0,T]$. 
\begin{lemma}\label{Lemma 6.4}  
(Cf. Lemma \ref{Lemma 4.6}) $R^{\chi}$ defined in \eqref{est 270}, $R^{\text{time,1}}$ defined identically to \eqref{est 114} but with $a_{j}$ and $b_{j}$ from \eqref{est 244}, and $R^{\text{quadr}}$ with $R^{\text{quadr,1}}$ and $R^{\text{quadr,2}}$ defined identically to \eqref{est 115}-\eqref{est 116} with $a_{j}, b_{j}$ replaced by those from \eqref{est 244} satisfy for all $t \in [0,T]$ 
\begin{subequations}\label{est 358}
\begin{align}
&\lVert R^{\chi} (t) \rVert_{L_{x}^{1}} \overset{\eqref{est 270}}{\leq} \sum_{j=1}^{d} \int_{supp (1- \chi_{j}^{2})} \lvert R_{l}^{j}(t,x) + \Gamma \rvert dx 
\overset{\eqref{est 243}}{\leq} \frac{\delta}{2}, \label{est 280}  \\
& \lVert R^{\text{time,1}}(t) \rVert_{L_{x}^{1}} \ll \delta, \hspace{7mm} \lVert R^{\text{quadr}}(t) \rVert_{L_{x}^{1}} \ll \delta. \label{est 282} 
\end{align}
\end{subequations}
\end{lemma}  

\begin{proof}[Proof of Lemma \ref{Lemma 6.4}]
First, \eqref{est 280} is clear. Second, for the first claim in \eqref{est 282}, similarly to \eqref{est 155} we have for $\lambda \in\mathbb{N}$ sufficiently large 
\begin{align*}
\lVert R^{\text{time,1}}(t) \rVert_{L_{x}^{1}} \overset{\eqref{est 121}}{\lesssim}\sum_{j=1}^{d} ( \lVert \partial_{t} a_{j} (t) \rVert_{L_{x}^{\infty}} \lVert b_{j}(t) \rVert_{L_{x}^{\infty}} + \lVert a_{j}(t) \rVert_{L_{x}^{\infty}} \lVert \partial_{t} b_{j}(t) \rVert_{L_{x}^{\infty}}) \lVert Q^{j} (t) \rVert_{L_{x}^{1}} 
\overset{ \eqref{est 62}\eqref{est 152}}{\ll} \delta.
\end{align*}
Third, the second claim in \eqref{est 282} is a consequence of 
\begin{align*}
&  \lVert R^{\text{quadr,1}}(t) \rVert_{L_{x}^{1}} \overset{\eqref{est 283} \eqref{est 68}\eqref{est 55}}{\lesssim} \nu^{-1} \sum_{j=1}^{d} ( \lVert a_{j}(t) \rVert_{C_{x}^{2}} \lVert b_{j}(t) \rVert_{L_{x}^{\infty}} + \lVert a_{j}(t) \rVert_{L_{x}^{\infty}} \lVert b_{j}(t) \rVert_{C_{x}^{2}}) \lambda \mu 
\overset{\eqref{est 148}}{\ll} \delta, \\
&  \lVert R^{\text{quadr,2}}(t) \rVert_{L_{x}^{1}} \overset{\eqref{est 283} \eqref{est 67}  \eqref{est 68} \eqref{est 55}}{\lesssim} \lambda^{-1} \sum_{j=1}^{d} (\lVert a_{j}(t) \rVert_{C_{x}^{2}} \lVert b_{j}(t) \rVert_{L_{x}^{\infty}} + \lVert a_{j}(t) \rVert_{L_{x}^{\infty}} \lVert b_{j}(t) \rVert_{C_{x}^{2}})  \overset{\eqref{est 148}}{\ll} \delta,  
\end{align*} 
which can be verified by taking $\iota > 0$ sufficiently small and $\lambda \in\mathbb{N}$ sufficiently large. 
\end{proof}

Second, similarly to Section \ref{Section 4.1.2} we work on the estimates on $\partial_{t}(\vartheta(t,x) + \vartheta_{c}(t)) + \text{div}(\theta_{0}(t,x) w(t,x) + \vartheta(t,x) u_{0}(t, x+ B(t))) - \Delta (\vartheta(t,x) + q(t,x))= (\text{div} R^{\text{time,2}} + \text{div} R^{\text{lin}})(t,x)$ of \eqref{est 357}.   
We are able to deduce the analogous identity to \eqref{est 127} as follows: 
\begin{align}
&\partial_{t} ( \vartheta(t,x) + \vartheta_{c}(t)) + \text{div} (\theta_{0}(t,x) w(t,x) + \vartheta(t,x) u_{0}(t, x+ B(t))) - \Delta (\vartheta(t,x) + q(t,x)) \nonumber\\
=& \sum_{j=1}^{d} a_{j}(t,x) \partial_{t} \Theta^{j}(t,x) - \fint_{\mathbb{T}^{d}} a_{j}(t,x) \partial_{t} \Theta^{j}(t,x) dx + \partial_{t} a_{j}(t,x) \Theta^{j}(t,x) - \fint_{\mathbb{T}^{d}} \partial_{t} a_{j}(t,x) \Theta^{j}(t,x) dx\nonumber\\
& + \text{div} (\theta_{0}(t,x) w(t,x) + \vartheta(t,x) u_{0}(t, x+ B(t))) - \Delta (\vartheta(t,x) + q(t,x)). \label{est 284} 
\end{align}
Therefore, we only have to modify $R^{\text{lin}}$ in \eqref{est 124} slightly as 
\begin{align}
R^{\text{lin}}(t,x) \triangleq& \sum_{j=1}^{d} \mathcal{D}^{-1} ( ( \partial_{t} a_{j}(t,x)) \Theta^{j}(t,x) - \fint_{\mathbb{T}^{d}} (\partial_{t} a_{j}(t,x)) \Theta^{j}(t,x) dx) \nonumber\\
&+ \theta_{0}(t,x) w (t,x)+ \vartheta(t,x) u_{0}(t,x+B(t)) - \nabla (\vartheta(t,x) + q(t,x)) \label{est 212} 
\end{align}
and define $R^{\text{time,2}}$ identically to \eqref{est 125} although with $a_{j}$ and $b_{j}$ from \eqref{est 244} rather than \eqref{est 77}. These definitions allow \eqref{est 128} with $u_{0}(t,x)$ replaced by $u_{0}(t, x+ B(t))$ and \eqref{est 129} to continue to hold. By analogous computations to Lemma \ref{Lemma 4.7} we can prove the following estimates; considering the case $t \in ( \frac{\Sigma}{2} \wedge T, \Sigma \wedge T ]$ on which we will work next, we prove this result for all $t \in [0,T]$. 
\begin{lemma}\label{Lemma 6.5} 
(Cf. Lemma \ref{Lemma 4.7}) $R^{\text{lin}}$ defined identically to \eqref{est 212} and $R^{\text{time,2}}$ identically to \eqref{est 125} with $a_{j}$ and $b_{j}$ from \eqref{est 244} satisfy for all $t \in [0,T]$ 
\begin{equation}\label{est 285}
\lVert R^{\text{lin}}(t) \rVert_{L_{x}^{1}} \ll \delta \hspace{2mm} \text{ and } \hspace{2mm} \lVert R^{\text{time,2}}(t) \rVert_{L_{x}^{1}} \ll \delta.
\end{equation} 
\end{lemma} 

\begin{proof}[Proof of Lemma \ref{Lemma 6.5}]
The estimate on $\lVert R^{\text{lin}}(t) \rVert_{L_{x}^{1}}$ follows from 
\begin{align*}
& \lVert \sum_{j=1}^{d} \mathcal{D}^{-1} ( \partial_{t} a_{j} (t,x) \Theta^{j} (t,x) - \fint_{\mathbb{T}^{d}} (\partial_{t} a_{j}) \Theta^{j} (t,x) dx) \rVert_{L_{x}^{1}} \\
& \hspace{15mm} \overset{\eqref{est 121}}{\lesssim} \sum_{j=1}^{d}  \lVert \partial_{t} a_{j}(t) \rVert_{L_{x}^{\infty}} \lVert \Theta^{j} (t) \rVert_{L_{x}^{1}} \overset{\eqref{est 62} \eqref{est 148}}{\ll} \delta, \\
& \lVert \theta_{0} (t,x) w(t,x) + \vartheta(t,x) u_{0}(t, x+ B(t)) \rVert_{L_{x}^{1}}  \\
& \hspace{14mm} \overset{\eqref{est 258} \eqref{est 62}}{\lesssim}  \sum_{j=1}^{d} \lVert \theta_{0} (t) \rVert_{L_{x}^{\infty}} \lVert b_{j}(t) \rVert_{L_{x}^{\infty}} (\frac{M}{\mu^{a}}) + \lVert a_{j}(t) \rVert_{L_{x}^{\infty}} (\frac{M}{\mu^{b}}) \lVert u_{0}(t) \rVert_{L_{x}^{\infty}}\overset{\eqref{est 148}}{\ll}\delta,\\
& \lVert \nabla (\vartheta + q) (t) \rVert_{L_{x}^{1}} \overset{\eqref{est 258}}{=} \lVert \nabla (\sum_{j=1}^{d} a_{j}(t) \Theta^{j} (t) + \sum_{j=1}^{d} a_{j}(t) b_{j}(t) Q^{j}(t)) \rVert_{L_{x}^{1}} \overset{\eqref{est 62}  \eqref{est 68} \eqref{est 148}}{\ll} \delta, 
\end{align*} 
for $\iota > 0$ sufficiently small and $\lambda \in\mathbb{N}$ sufficiently large. Moreover, the estimate of $\lVert R^{\text{time,2}}(t)\rVert_{L_{x}^{1}}$ follows the proof of \eqref{est 158} and \eqref{est 165}. 
\end{proof} 

Third, similarly to Section \ref{Section 4.1.3} we work on the estimates on $\text{div} (q(t,x) (u_{0}(t, x+ B(t)) + w(t,x)) = \text{div} R^{q}(t,x)$ of \eqref{est 357}. We define $R^{q}$ similarly to \eqref{est 132} with the only difference being that we replace $u_{0}(t,x)$ therein by $u_{0}(t, x+B(t))$: 
\begin{equation}\label{est 213}
R^{q}(t,x) \triangleq q(t,x) (u_{0}(t, x+B(t)) + w(t,x)).
\end{equation} 
We can estimate for $\lambda \in\mathbb{N}$ sufficiently large and $\iota > 0$ sufficiently small 
\begin{equation}\label{est 286}
\lVert R^{q}(t) \rVert_{L_{x}^{1}} \lesssim \sum_{j=1}^{d} \lVert a_{j}(t) b_{j}(t) Q^{j}(t) \rVert_{L_{x}^{1}} (\lVert u_{0}(t) \rVert_{L_{x}^{\infty}} + \sum_{k=1}^{d} \lVert b_{k}(t) W^{k}(t) \rVert_{L_{x}^{\infty}} )  \ll \delta \hspace{3mm} \forall \hspace{1mm} t \in [0,T]
\end{equation} 
via \eqref{est 258}, \eqref{est 269}, and \eqref{est 148}; considering the case $t \in (\frac{\Sigma}{2} \wedge T, \Sigma \wedge T]$ on which we will work subsequently, we stated this result for all $t \in [0,T]$.

Fourth, similarly to Section \ref{Section 4.1.4} we work on the estimates on $\text{div} ([\theta_{0}(t,x) + \vartheta(t,x) + q(t,x)] w_{c}(t,x))= \text{div} R^{\text{corr}}(t,x)$ of \eqref{est 357}. We can define $R^{\text{corr}}$ identically to \eqref{est 133} but with no $z$, and we can estimate it for $\lambda \in\mathbb{N}$ sufficiently large and $\iota > 0$ sufficiently small 
\begin{equation}\label{est 287}
\lVert R^{\text{corr}}(t) \rVert_{L_{x}^{1}} \leq ( \lVert \theta_{0}(t) \rVert_{L_{x}^{p}} +\lVert \vartheta(t) \rVert_{L_{x}^{p}} + \lVert q(t) \rVert_{L_{x}^{p}}) \lVert w_{c}(t) \rVert_{L_{x}^{p'}} 
\overset{\eqref{est 148}}{\ll} \delta \hspace{3mm} \forall \hspace{1mm} t \in [0,T];
\end{equation} 
again, we stated this result over $[0,T]$ for the convenience when we work on the case $t \in (\frac{\Sigma}{2} \wedge T, \Sigma \wedge T]$ next. 

Fifth, similarly to Section \ref{Section 4.1.6} we work on the estimates of $\text{div} (R_{l}(t,x) - R_{0}(t,x)) = \text{div} R^{\text{moll}}(t,x)$ of \eqref{est 357}. We define $R^{\text{moll}}$ identically to \eqref{est 100} for which  the estimate from \eqref{est 139} applies here directly to prove that 
\begin{equation}\label{est 288}
\lVert R^{\text{moll}}(t) \rVert_{L_{x}^{1}} \ll \delta \hspace{3mm} \forall \hspace{1mm} t \in [0,T]; 
\end{equation} 
considering the case $t \in (\frac{\Sigma}{2} \wedge T, \Sigma \wedge T]$ on which we will work subsequently we stated this estimate for all $t \in [0,T]$.  Considering \eqref{est 280}, \eqref{est 282}, \eqref{est 285}, \eqref{est 286}, \eqref{est 287}, and \eqref{est 288} in \eqref{est 214}, we conclude \eqref{est 250} on $(\Sigma \wedge T, T]$: $\lVert R_{1}(t) \rVert_{L_{x}^{1}} \leq \delta$. 

Next, we prove \eqref{est 250} for $t \in (\frac{\Sigma}{2} \wedge T, \Sigma \wedge T]$. If $T \leq \frac{\Sigma}{2}$, then there is nothing to prove. Thus, we consider $\frac{\Sigma}{2} < T$ and thus $t \in (\frac{\Sigma}{2}, \Sigma \wedge T]$. We see from \eqref{est 172}, \eqref{est 245}, and \eqref{est 255} that 
\begin{align}\label{est 356} 
 - \text{div}R_{1}(t,x) =& \underbrace{ \tilde{\chi}(t)^{2} [\partial_{t} (q(t,x) + q_{c}(t)) + \text{div} (\vartheta(t,x) w(t,x) - R_{l}(t,x))]}_{(\text{div} R^{\text{time,1}} + \text{div}R^{\text{quadr}} + \text{div} R^{\chi} )(t,x)} \nonumber  \\
&+ \underbrace{\tilde{\chi}(t) [\partial_{t} (\vartheta(t,x) + \vartheta_{c}(t)) + \text{div} (\theta_{0}(t,x) w(t,x) + \vartheta(t,x) u_{0}(t, x+B(t)))}_{(\text{div}R^{\text{time,2}} + \text{div}R^{\text{lin}})(t,x)} \nonumber\\
& \hspace{37mm} \underbrace{- \Delta (\vartheta(t,x) + \tilde{\chi}(t) q(t,x))]}_{(\text{div} R^{\text{time,2}} + \text{div}R^{\text{lin}})(t,x) \text{ continued}} \nonumber \\
&+ \underbrace{\tilde{\chi}(t)^{2} \text{div} (q(t,x) (u_{0}(t, x+B(t)) + \tilde{\chi}(t) w(t,x)))}_{\text{div} R^{q}(t,x)} \nonumber \\
&+ \underbrace{\tilde{\chi}(t) \text{div} ([\theta_{0} (t,x) + \tilde{\chi}(t) \vartheta(t,x) + \tilde{\chi}(t)^{2} q(t,x) ] w_{c}(t,x))}_{\text{div} R^{\text{corr}}(t,x)} \nonumber \\
&+ (\tilde{\chi}(t)^{2} -1) \text{div} R_{l} (t,x) + \underbrace{\text{div} (R_{l} - R_{0})(t,x)}_{\text{div} R^{\text{moll}}(t,x)} \nonumber \\
&+ \underbrace{\partial_{t} \tilde{\chi}(t) (\vartheta(t,x) + \vartheta_{c}(t)) + \partial_{t} \tilde{\chi}(t)^{2} (q(t,x) + q_{c}(t))}_{\text{div} R^{\text{cutoff}}(t,x)}  
\end{align}
and therefore 
\begin{equation}\label{est 292}
-R_{1} \triangleq R^{\text{time,1}} + R^{\text{quadr}} + R^{\chi} + R^{\text{time,2}} + R^{\text{lin}} + R^{q} + R^{\text{corr}} + (\tilde{\chi}^{2} -1) R_{l} + R^{\text{moll}} + R^{\text{cutoff}}. 
\end{equation} 
First, concerning $\tilde{\chi}(t)^{2} [\partial_{t} (q(t,x) + q_{c}(t)) + \text{div} (\vartheta(t,x) w(t,x) - R_{l}(t,x))] = (\text{div} R^{\text{time,1}} + \text{div}R^{\text{quadr}} + \text{div} R^{\chi} )(t,x)$ in \eqref{est 356}, it suffices to take the same $R^{\chi}$ in \eqref{est 270}, $R^{\text{time,1}}$ from \eqref{est 114} with $a_{j}$ and $b_{j}$ in \eqref{est 244}, and $R^{\text{quadr}}$ with $R^{\text{quadr,1}}$ and $R^{\text{quadr,2}}$ identically to \eqref{est 115}-\eqref{est 116} with $a_{j}, b_{j}$ replaced by those from \eqref{est 244}, and multiply them all by $\tilde{\chi}(t)^{2}$. Because $\tilde{\chi}(t) \in [0,1]$ for all $t \in [0,T]$ by \eqref{est 257}, the estimates from \eqref{est 280} and \eqref{est 282} from Lemma \ref{Lemma 6.4} remain valid. 

Second, concerning $\tilde{\chi}(t) [\partial_{t} (\vartheta(t,x) + \vartheta_{c}(t)) + \text{div} (\theta_{0}(t,x) w(t,x) + \vartheta(t,x) u_{0}(t, x+B(t))) - \Delta (\vartheta(t,x) + \tilde{\chi}(t) q(t,x))] = (\text{div}R^{\text{time,2}} + \text{div}R^{\text{lin}})(t,x)$ in \eqref{est 356}, we can modify \eqref{est 212} and \eqref{est 125} as follows: 
\begin{subequations}\label{est 392} 
\begin{align}
R^{\text{lin}}(t,x) \triangleq& \tilde{\chi}(t) [ \sum_{j=1}^{d} \mathcal{D}^{-1} (\partial_{t} a_{j} (t,x)) \Theta^{j} (t,x) - \fint_{\mathbb{T}^{d}} (\partial_{t} a_{j} (t,x)) \Theta^{j} (t,x) dx \nonumber \\
& + \theta_{0}(t,x) w(t,x) + \vartheta(t,x) u_{0}(t, x+ B(t)) - \nabla (\vartheta(t,x) + \tilde{\chi}(t) q(t,x))], \\
R^{\text{time,2}} (t,x) \triangleq& \tilde{\chi}(t) [-\lambda \omega \sum_{j=1}^{d} \mathcal{R}_{N} (a_{j}(t,x) (\partial_{t} \varrho_{\mu}^{j})_{\lambda} \circ \tau_{\omega t e_{j}} (x), \psi_{\nu}^{j}(x))]. 
\end{align}
\end{subequations} 
Because $\tilde{\chi}(t) \in [0,1]$ for all $t \in [0,T]$ by \eqref{est 257}, it is clear that same estimates in \eqref{est 285} from Lemma \ref{Lemma 6.5} continue to hold for $R^{\text{lin}}$ and $R^{\text{time,2}}$ in \eqref{est 392}.

Third, concerning $\tilde{\chi}(t)^{2} \text{div} (q(t,x) (u_{0}(t, x+B(t)) + \tilde{\chi}(t) w(t,x))) = \text{div}R^{q}(t,x)$ in \eqref{est 356}, it suffices to define 
\begin{align*}
R^{q}(t,x) \triangleq \tilde{\chi}(t)^{2} q(t,x) (u_{0}(t, x+B(t)) + \tilde{\chi}(t) w(t,x))
\end{align*}
for which the same estimate in \eqref{est 286} clearly goes through considering that $\tilde{\chi}(t) \in [0,1]$ for all $t \in [0,T]$ by \eqref{est 257}. 

Fourth, concerning $\tilde{\chi}(t) \text{div} ([\theta_{0} (t,x) + \tilde{\chi}(t) \vartheta(t,x) + \tilde{\chi}(t)^{2} q(t,x) ] w_{c}(t,x)) = \text{div} R^{\text{corr}}(t,x)$ in \eqref{est 356}, we define 
\begin{equation}
R^{\text{corr}} (t,x) \triangleq \tilde{\chi}(t) [\theta_{0} (t,x) + \tilde{\chi}(t) \vartheta(t,x) + \tilde{\chi}(t)^{2} q(t,x) ] w_{c}(t,x)
\end{equation} 
for which the same estimate in \eqref{est 287} applies because $\tilde{\chi}(t) \in [0,1]$ for all $t \in [0,T]$ by \eqref{est 257}. 

Fifth, concerning $(\tilde{\chi}(t)^{2} -1) \text{div} R_{l}(t,x)$ in \eqref{est 356}, via \eqref{est 257} and Young's inequality for convolution we bound for all $t \in (\frac{\Sigma}{2}, \Sigma \wedge T]$ 
\begin{equation}\label{est 362}
\lVert (\tilde{\chi}(t)^{2} -1) R_{l}(t) \rVert_{L_{x}^{1}}  \leq  \sup_{\tau \in [t-l,t]} \lVert R_{0}(\tau) \rVert_{L_{x}^{1}}. 
\end{equation} 

Sixth, concerning $\text{div} (R_{l} - R_{0})(t,x) = \text{div} R^{\text{moll}}(t,x)$ in \eqref{est 356}, we define $R^{\text{moll}}$ identically to \eqref{est 100},  the same estimate in \eqref{est 288} applies. 

Finally, concerning $\partial_{t} \tilde{\chi}(t) (\vartheta(t,x) + \vartheta_{c}(t)) + \partial_{t} \tilde{\chi}(t)^{2} (q(t,x) + q_{c}(t)) = \text{div} R^{\text{cutoff}}(t,x)$ in \eqref{est 356}, we define 
\begin{equation}\label{est 291}
R^{\text{cutoff}}(t,x) \triangleq \mathcal{D}^{-1} (\partial_{t} \tilde{\chi}(t) (\vartheta(t,x) + \vartheta_{c}(t)) + \partial_{t} \tilde{\chi}(t)^{2} [q(t,x) + q_{c}(t)])
\end{equation} 
and separately estimate $L^{1}(\mathbb{T}^{d})$-norms of $\mathcal{D}^{-1} (\partial_{t} \tilde{\chi}(t) (\vartheta(t,x) + \vartheta_{c}(t))$ and $\mathcal{D}^{-1} (\partial_{t} \tilde{\chi}(t)^{2} [q(t,x) + q_{c}(t) ])$, both of which are well-defined because $\vartheta(t,x) + \vartheta_{c}(t)$ and $q(t,x) + q_{c}(t)$ are mean-zero due to \eqref{est 85}. First, for $\lambda \in\mathbb{N}$ sufficiently large  
\begin{align}
& \lVert \mathcal{D}^{-1} (\partial_{t} \tilde{\chi}(t) ( \vartheta(t,x) + \vartheta_{c}(t)) ) \rVert_{L_{x}^{1}}  \label{est 289}\\
\overset{\eqref{est 121}}{\lesssim}& \Sigma^{-1} \lVert \vartheta(t) \rVert_{L_{x}^{1}}  
\overset{\eqref{est 258} \eqref{est 269}\eqref{est 148}}{\lesssim}    \Sigma^{-1} \sum_{j=1}^{d}  ( \lambda^{(d+1) \iota} \sup_{\tau \in [t-l,t]} \lVert R_{0}^{j}(\tau) \rVert_{L_{x}^{1}} + \Gamma)^{\frac{1}{p}} \lambda^{-\alpha b}  \ll \delta \nonumber 
\end{align}
because $\iota < \frac{p \alpha b}{d+1}$ due to \eqref{est 354}. Second, for $\lambda \in\mathbb{N}$ sufficiently large 
\begin{align}
& \lVert \mathcal{D}^{-1} (\partial_{t} \tilde{\chi}(t)^{2} [q(t,x) + q_{c}(t) ]) \rVert_{L_{x}^{1}} \label{est 290}\\
\overset{\eqref{est 121}}{\lesssim}&  \Sigma^{-1} \lVert q(t) \rVert_{L_{x}^{1}} \overset{\eqref{est 258} \eqref{est 269} \eqref{est 148}}{\lesssim} \Sigma^{-1} \sum_{j=1}^{d} ( \lambda^{(d+1) \iota} \sup_{\tau \in [t- l, t]} \lVert R_{0}^{j}(\tau) \rVert_{L_{x}^{1}} + \Gamma) \lambda^{-\beta} \ll \delta \nonumber
\end{align}
because $\iota < \frac{\beta}{d+1}$ due to \eqref{est 354}. Considering \eqref{est 289} and \eqref{est 290} in \eqref{est 291} gives us 
\begin{equation}\label{est 360}
\lVert R^{\text{cutoff}}(t) \rVert_{L_{x}^{1}} \ll \delta. 
\end{equation} 
Applying these estimates \eqref{est 358}, \eqref{est 285}, \eqref{est 286}, \eqref{est 287}, \eqref{est 362}, \eqref{est 288}, and \eqref{est 360} to \eqref{est 292} allows us to conclude \eqref{est 250} for $t \in (\frac{\Sigma}{2} \wedge T, \Sigma \wedge T]$. 

Finally, on $[0, \frac{\Sigma}{2} \wedge T], t \leq \frac{\Sigma}{2}$ and thus $\tilde{\chi}(t) = 0$ by \eqref{est 257}. Therefore, $\theta_{1}\equiv \theta_{0}$ and $u_{1} \equiv u_{0}$ due to \eqref{est 245} so that $R_{1}\equiv R_{0}$; hence, \eqref{est 250} for this range of $t$ is trivially satisfied. 

Lastly, we prove \eqref{est 293}. We have $t \in (4\Sigma \wedge T, T]$. If $T \leq 4 \Sigma$, then there is nothing to prove; thus, we assume that $4\Sigma < T$ so that $t \in (4\Sigma, T]$. Therefore, $t > 4\Sigma$ so that $\tilde{\chi}(t) = 1$ by \eqref{est 257}. Moreover, as $4\Sigma < T$, we have $[2\Sigma \wedge T, T] = [2\Sigma, T]$. We also note that \eqref{est 251} is applicable here. Let us compute using \eqref{est 245}-\eqref{est 255} 
\begin{align} 
&\lvert \int_{\mathbb{T}^{d}} \theta_{1}(t,x) u_{1}(t, x+ B(t))dx - \int_{\mathbb{T}^{d}} \theta_{0}(t,x) u_{0}(t,x +B(t)) dx - \sum_{j=1}^{d} \int_{\mathbb{T}^{d}} \chi_{j}^{2}(t,x) dx \Gamma  e_{j}  \rvert  \nonumber \\
\leq& \mathfrak{A}(t) + \mathfrak{B}(t)+ \mathfrak{C}(t) + \mathfrak{D}(t)+ \mathfrak{E}(t) \label{est 294} 
\end{align}
where  
\begin{subequations}\label{est 388} 
\begin{align}
&\mathfrak{A}(t) \triangleq \vert \int_{\mathbb{T}^{d}} \vartheta(t,x) w(t,x) - \sum_{j=1}^{d} \chi_{j}^{2}(t,x) e_{j} \Gamma dx \rvert, \label{est 295}\\
&\mathfrak{B}(t) \triangleq \lvert \int_{\mathbb{T}^{d}} (\theta_{0} w)(t,x) + \vartheta(t,x) u_{0}(t, x+ B(t)) dx \rvert, \label{est 296}\\
&\mathfrak{C}(t) \triangleq \lvert \int_{\mathbb{T}^{d}} q(t,x) (u_{0}(t, x+B(t)) + w(t,x)) dx \rvert, \label{est 297} \\
&\mathfrak{D}(t) \triangleq \lvert \int_{\mathbb{T}^{d}} (\theta_{0} + \vartheta + q) (t,x) w_{c}(t,x) dx \rvert,  \label{est 298}\\
&\mathfrak{E}(t) \triangleq \lvert (\vartheta_{c} + q_{c})(t) \int_{\mathbb{T}^{d}} u_{0}(t, x+B(t)) + w(t,x) + w_{c}(t,x)) dx \rvert. \label{est 299} 
\end{align}
\end{subequations}
When defining a new defect $R_{1}$ (e.g., \eqref{est 356}), the last term $\mathfrak{E}(t)$ typically vanishes because $\text{div} ( (\vartheta_{c}(t) + q_{c}(t))(u_{0}(t, x+ B(t)) + w(t,x) + w_{c}(t,x)))=0$ but not here; hence, we need to estimate this term as well. Now the most difficult term in \eqref{est 388} is $\mathfrak{A}$. Importantly, we furthermore split it from \eqref{est 295} by relying on an orthogonality relation \eqref{est 64} as follows: 
\begin{align}
\mathfrak{A}(t) \overset{\eqref{est 258}\eqref{est 64}}{=}&  \lvert \int_{\mathbb{T}^{d}} \sum_{j=1}^{d} a_{j}(t,x) b_{j}(t,x) \Theta^{j}(t,x) W^{j}(t,x) - \chi_{j}^{2}(t,x) e_{j} \Gamma dx \rvert \nonumber \\
\overset{\eqref{est 262}}{=}& \lvert \int_{\mathbb{T}^{d}} \sum_{j=1}^{d} \chi_{j}^{2}(t,x) (R_{l}^{j}(t,x) + \Gamma) \Theta^{j}(t,x) W^{j}(t,x) - \sum_{j=1}^{d} \chi_{j}^{2}(t,x)  \Gamma e_{j} dx \rvert  \nonumber\\
=& \lvert \int_{\mathbb{T}^{d}} \sum_{j=1}^{d} \chi_{j}^{2}(t,x) R_{l}^{j}(t,x) [\Theta^{j}(t,x) W^{j}(t,x) - e_{j} ] + \chi_{j}^{2}(t,x) R_{l}^{j}(t,x) e_{j} \nonumber\\
&+ \chi_{j}^{2}(t,x)\Gamma [\Theta^{j}(t,x) W^{j}(t,x) - e_{j} ]dx \rvert 
\overset{\eqref{est 59}}{\leq}  \mathfrak{A}_{1}(t) + \mathfrak{A}_{2}(t) + \mathfrak{A}_{3}(t)   \label{est 300}
\end{align}
where 
\begin{subequations}\label{est 389} 
\begin{align}
& \mathfrak{A}_{1}(t) \triangleq \lvert \int_{\mathbb{T}^{d}} \sum_{j=1}^{d} \chi_{j}^{2}(t,x) R_{l}^{j}(t,x) [(\varrho_{\mu}^{j} \tilde{\varrho}_{\mu}^{j})_{\lambda} \circ \tau_{\omega t e_{j}} (\psi_{\nu}^{j})^{2} e_{j} - e_{j} ] dx \rvert, \label{est 301}\\
& \mathfrak{A}_{2}(t) \triangleq \lvert \int_{\mathbb{T}^{d}} \sum_{j=1}^{d} \chi_{j}^{2}(t,x) R_{l}^{j}(t,x) e_{j} dx \rvert, \label{est 302}\\
& \mathfrak{A}_{3}(t) \triangleq \lvert \int_{\mathbb{T}^{d}} \sum_{j=1}^{d} \chi_{j}^{2}(t,x) [(\varrho_{\mu}^{j} \tilde{\varrho}_{\mu}^{j})_{\lambda} \circ \tau_{\omega t e_{j}} (\psi_{\nu}^{j})^{2} e_{j} - e_{j} ] dx \Gamma \rvert.\label{est 303}
\end{align}
\end{subequations}

\begin{remark}\label{Difficulty 3}
Subtracting and adding terms in \eqref{est 389} is crucial. Indeed, estimating e.g. $\int_{\mathbb{T}^{d}} R_{l}^{j}(t,x) \Theta^{j}(t,x) W^{j}(t,x)dx$ by simply $\lVert R_{l}^{j}(t)\rVert_{L_{x}^{1}} \lVert \Theta^{j}(t)W^{j}(t) \rVert_{L_{x}^{\infty}}$ in hope to make use of \eqref{est 251} on $\lVert R_{l}^{j}(t) \rVert_{L_{x}^{1}}$ will not work because $\lVert \Theta^{j} W^{j} \rVert_{L_{x}^{\infty}} \overset{\eqref{est 59}}{=} \omega \lVert Q^{j} \rVert_{L_{x}^{\infty}}$ and this is too large (see e.g., $\lVert Q^{j}(t)\rVert_{L_{x}^{p}} \leq \frac{M \mu^{b}}{\omega}$ in \eqref{est 61}). Within \eqref{est 301} and \eqref{est 303} we will further split 
\begin{equation}\label{est 308}
(\varrho_{\mu}^{j} \tilde{\varrho}_{\mu}^{j})_{\lambda} \circ \tau_{\omega t e_{j}} (\psi_{\nu}^{j})^{2} - 1 =  ( \varrho_{\mu}^{j} \tilde{\varrho}_{\mu}^{j})_{\lambda} \circ \tau_{\omega t e_{j}} (( \psi^{j})^{2} -1)_{\nu} + ( \varrho_{\mu}^{j} \tilde{\varrho}_{\mu}^{j} -1)_{\lambda} \circ \tau_{\omega t e_{j}}
\end{equation} 
identically to the derivation of \eqref{est 117} and take advantage of the mean-zero property of $(\psi^{j})^{2} - 1$ and $(\varrho_{\mu}^{j} \tilde{\varrho}_{\mu}^{j} - 1) \circ \tau_{\lambda \omega t e_{j}}$ due to \eqref{est 305} and \eqref{est 113} so that we can employ \eqref{est 304}, which was never used in \cite{MS20}. On the other hand, we can handle $\mathfrak{A}_{2}$ in \eqref{est 302} by \eqref{est 251}. 
\end{remark}

Using \eqref{est 308} we further split 
\begin{equation}\label{est 311} 
\mathfrak{A}_{1}(t) \leq \mathfrak{A}_{1,1}(t) + \mathfrak{A}_{1,2}(t) 
\end{equation}
where 
\begin{subequations}
\begin{align}
& \mathfrak{A}_{1,1}(t) \triangleq  \lvert \int_{\mathbb{T}^{d}} \sum_{j=1}^{d} \chi_{j}^{2}(t,x) R_{l}^{j}(t,x) (\varrho_{\mu}^{j} \tilde{\varrho}_{\mu}^{j})_{\lambda} \circ \tau_{\omega t e_{j}} ((\psi^{j})^{2} -1)_{\nu} e_{j}dx \rvert,\\ 
& \mathfrak{A}_{1,2}(t) \triangleq \lvert \int_{\mathbb{T}^{d}} \sum_{j=1}^{d} \chi_{j}^{2}(t,x) R_{l}^{j}(t,x) (\varrho_{\mu}^{j} \tilde{\varrho}_{\mu}^{j} -1)_{\lambda} \circ \tau_{\omega t e_{j}} e_{j} dx\rvert \label{est 306}
\end{align}
\end{subequations}
similarly to \eqref{est 116}. Now we estimate using that $(\psi^{j})^{2} -1$ is mean-zero due to \eqref{est 305} and $W^{d+1, 1}(\mathbb{T}^{d}) \hookrightarrow L^{\infty} (\mathbb{T}^{d})$, for $\lambda \in\mathbb{N}$ sufficiently large 
\begin{align}
\mathfrak{A}_{1,1}(t) \overset{\eqref{est 304}}{\lesssim}& \sum_{j=1}^{d} \nu^{-1} \lVert \nabla (\chi_{j}^{2}(t,x) R_{l}^{j}(t,x) (\varrho_{\mu}^{j} \tilde{\varrho}_{\mu}^{j})_{\lambda} \circ \tau_{\omega t e_{j}}) \rVert_{L_{x}^{1}} \lVert (\psi^{j})^{2} -1 \rVert_{L_{x}^{\infty}} \nonumber\\
\overset{\eqref{est 66}\eqref{est 67}}{\lesssim}& \nu^{-1} \sum_{j=1}^{d} [\lVert R_{l}^{j}(t) \rVert_{W_{x}^{d+1,1}} \lVert \varrho_{\mu}^{j} \tilde{\varrho}_{\mu}^{j} \rVert_{L^{1}} + \lVert R_{l}^{j} (t) \rVert_{W_{x}^{d+2,1}} \lVert \varrho_{\mu}^{j}\tilde{\varrho}_{\mu}^{j} \rVert_{L^{1}} + \lVert R_{l}^{j}(t) \rVert_{W_{x}^{d+1,1}} \lambda \lVert \varrho_{\mu}^{j}\tilde{\varrho}_{\mu}^{j} \rVert_{W^{1,1}} ] \nonumber \\
\overset{\eqref{est 68}\eqref{est 55}\eqref{est 251}}{\lesssim}& \nu^{-1} \delta l^{-d-1} (l^{-1} + \lambda \mu) \overset{\eqref{est 148}}{\approx} \lambda^{-\gamma} \delta \lambda^{(d+1) \iota} (\lambda^{\iota} + \lambda^{1+ \alpha}) \ll \delta \label{est 312}
\end{align}
because $\iota \overset{\eqref{est 152}}{<} \frac{\gamma - 1 - \alpha}{3d+5} < \frac{\gamma-1-\alpha}{d+1}$. Next, we estimate $\mathfrak{A}_{1,2}$ from \eqref{est 306} using $\int_{\mathbb{T}^{d}} \varrho_{\mu}^{j} \tilde{\varrho}_{\mu}^{j} -1 dx = 0$ from \eqref{est 307} as follows:
\begin{align}
&\mathfrak{A}_{1,2}(t) \overset{\eqref{est 304}}{\lesssim} \lambda^{-1} \sum_{j=1}^{d} \lVert \nabla (\chi_{j}^{2}(t,x) R_{l}^{j}(t,x)) \rVert_{L_{x}^{\infty}} \lVert (\varrho_{\mu}^{j} \tilde{\varrho}_{\mu}^{j} -1) \circ \tau_{\lambda \omega t e_{j}} \rVert_{L_{x}^{1}}  \label{est 313}\\
\overset{\eqref{est 67}}{\lesssim}& \lambda^{-1} \sum_{j=1}^{d} \lVert R_{l}^{j}(t) \rVert_{W_{x}^{d+2,1}} (\lVert \varrho_{\mu}^{j} \rVert_{L^{2}} \lVert \tilde{\varrho}_{\mu}^{j} \rVert_{L^{2}} + 1) 
\overset{\eqref{est 68} \eqref{est 251}}{\lesssim}  \lambda^{-1} \lambda^{(d+2) \iota} \delta \ll \delta \nonumber
\end{align}
as $\iota \overset{\eqref{est 152}}{<} \frac{1}{3d+5} < \frac{1}{d+2}$. Second, we estimate 
\begin{equation}\label{est 315}
\mathfrak{A}_{2}(t) \overset{\eqref{est 302} \eqref{est 243}}{\leq} \sum_{j=1}^{d} \int_{\mathbb{T}^{d}} \lvert R_{l}^{j}(t,x) \rvert dx  \leq d \sup_{\tau \in [t-l,t]} \lVert R_{0}(\tau) \rVert_{L_{x}^{1}} \overset{\eqref{est 251}}{\leq} 2 d \delta.
\end{equation} 
Similarly to $\mathfrak{A}_{1}$, we make use of \eqref{est 308} and split $\mathfrak{A}_{3}$ from \eqref{est 303} by  
\begin{equation}\label{est 314}
\mathfrak{A}_{3}(t) \leq \mathfrak{A}_{3,1}(t) + \mathfrak{A}_{3,2}(t) 
\end{equation} 
where 
\begin{subequations}
\begin{align}
& \mathfrak{A}_{3,1}(t) \triangleq \lvert \int_{\mathbb{T}^{d}}\sum_{j=1}^{d}   \chi_{j}^{2} (t,x) \Gamma (\varrho_{\mu}^{j} \tilde{\varrho}_{\mu}^{j})_{\lambda} \circ \tau_{\omega t e_{j}} ( (\psi^{j})^{2} -1)_{\nu} e_{j} dx \rvert,\\
& \mathfrak{A}_{3,2}(t) \triangleq \lvert  \int_{\mathbb{T}^{d}}  \sum_{j=1}^{d} \chi_{j}^{2}(t,x) \Gamma (\varrho_{\mu}^{j} \tilde{\varrho}_{\mu}^{j} -1)_{\lambda} \circ \tau_{\omega t e_{j}} e_{j}dx \rvert 
\end{align}
\end{subequations} 
and estimate for $\lambda \in \mathbb{N}$ sufficiently large 
\begin{align}
&\mathfrak{A}_{3,1}(t) \overset{\eqref{est 304}}{\lesssim} \Gamma \sum_{j=1}^{d} \nu^{-1}\lVert \nabla ( \chi_{j}^{2}(t,x) (\varrho_{\mu}^{j}\tilde{\varrho}_{\mu}^{j})_{\lambda} \circ \tau_{\omega t e_{j}} ) \rVert_{L_{x}^{1}} \lVert (\psi^{j})^{2} -1 \rVert_{L_{x}^{\infty}} \label{est 309}\\
\lesssim& \Gamma \nu^{-1}\sum_{j=1}^{d} \lVert \varrho_{\mu}^{j} \rVert_{L^{2}} \lVert \tilde{\varrho}_{\mu}^{j} \rVert_{L^{2}} + \lambda( \lVert \varrho_{\mu}^{j} \rVert_{W^{1,2}} \lVert \tilde{\varrho}_{\mu}^{j} \rVert_{L^{2}} + \lVert \varrho_{\mu}^{j} \rVert_{L^{2}} \lVert \tilde{\varrho}_{\mu}^{j} \rVert_{W^{1,2}}) \overset{\eqref{est 68}}{\lesssim} \Gamma \lambda^{-\gamma}  [1+ \lambda^{1+ \alpha}]  \overset{\eqref{est 144}}{\ll} \delta  \nonumber 
\end{align}  
and 
\begin{align}
\mathfrak{A}_{3,2}(t) \overset{\eqref{est 304}}{\lesssim}& \Gamma \lambda^{-1} \sum_{j=1}^{d} \lVert \nabla \chi_{j}^{2}(t) \rVert_{L_{x}^{\infty}} \lVert (\varrho_{\mu}^{j} \tilde{\varrho}_{\mu}^{j} -1) \circ \tau_{\lambda \omega t e_{j}} \rVert_{L_{x}^{1}} \nonumber \\
\overset{\eqref{est 68}}{\lesssim}& \Gamma \lambda^{-1} \sum_{j=1}^{d} (\mu^{a- \frac{d}{2}} \lVert \varrho \rVert_{L^{2}} \mu^{b- \frac{d}{2}} \lVert \varrho \rVert_{L^{2}} + 1) \overset{\eqref{est 55}}{\approx} \Gamma \lambda^{-1} \ll \delta.\label{est 310}
\end{align}
Applying \eqref{est 312} and \eqref{est 313} to \eqref{est 311} for $\mathfrak{A}_{1}$, as well as \eqref{est 309} and \eqref{est 310} to \eqref{est 314} for $\mathfrak{A}_{3}$, and considering \eqref{est 315} imply that for all $t \in (4\Sigma, T]$ and $\lambda \in \mathbb{N}$ sufficiently large 
\begin{equation}\label{est 317}
\mathfrak{A}(t) \overset{\eqref{est 300}}{\leq} \mathfrak{A}_{1}(t) + \mathfrak{A}_{2}(t) + \mathfrak{A}_{3}(t) \leq \delta d (\frac{5}{2}).
\end{equation} 
Next, utilizing \eqref{est 252} we can estimate $\mathfrak{B}$ from \eqref{est 296} similarly to \eqref{est 131} by taking $\iota > 0$ sufficiently small so that 
\begin{equation}\label{est 390}
\iota < \min\{ \frac{\alpha a p'}{d+1}, \frac{\alpha b p}{d+1} \}
\end{equation} 
and $\lambda \in \mathbb{N}$ sufficiently large, for all $t \in (4\Sigma, T]$ 
\begin{align}
\mathfrak{B} (t) \overset{\eqref{est 258}}{\lesssim} \sum_{j=1}^{d}& \lVert \theta_{0}(t) \rVert_{L_{x}^{\infty}} \lVert b_{j}(t) \rVert_{L_{x}^{\infty}} \lVert W^{j}(t) \rVert_{L_{x}^{1}} + \lVert a_{j}(t) \rVert_{L_{x}^{\infty}} \lVert \Theta^{j}(t) \rVert_{L_{x}^{1}} \lVert u_{0}(t) \rVert_{L_{x}^{\infty}}  \nonumber \\
\overset{\eqref{est 252}\eqref{est 62}}{\lesssim}&  \lVert \theta_{0}(t) \rVert_{L_{x}^{\infty}}  \lambda^{ (\frac{d+1}{p'}) \iota} \delta^{\frac{1}{p'}} \lambda^{-\alpha a} +  \lambda^{ (\frac{d+1}{p}) \iota} \delta^{\frac{1}{p}} \lambda^{-\alpha b} \lVert u_{0} (t) \rVert_{L_{x}^{\infty}}  \overset{\eqref{est 145}}{\ll} \delta. \label{est 318}
\end{align}
Next, similarly to \eqref{est 160} we can estimate $\mathfrak{C}$ from \eqref{est 297} for all $t \in (4\Sigma, T]$ by taking $\lambda \in \mathbb{N}$ sufficiently large and $\iota < \frac{\beta - \alpha b}{(d+1)(1+ \frac{1}{p'})}$ (which is possible thanks to \eqref{est 145}) 
\begin{align}
\mathfrak{C}(t) \overset{\eqref{est 262} \eqref{est 258}}{\lesssim}& \sum_{j=1}^{d} (\lVert R_{l}^{j} (t) \rVert_{L_{x}^{\infty}} + \Gamma) \lVert Q^{j}(t) \rVert_{L_{x}^{1}}( \lVert u_{0}(t) \rVert_{L_{x}^{\infty}} + \sum_{k=1}^{d} \lVert b_{k}(t) \rVert_{L_{x}^{\infty}} \lVert W^{k}(t) \rVert_{L_{x}^{\infty}}) \label{est 319}\\
\overset{\eqref{est 251}}{\lesssim}& ( \lambda^{(d+1)\iota} \delta + \Gamma) \lambda^{-\beta} (\lVert u_{0}(t) \rVert_{L_{x}^{\infty}} +  \lambda^{(d+1) \iota \frac{1}{p'}} \delta^{\frac{1}{p'}} \lambda^{\alpha b}) \overset{\eqref{est 145}}{\ll} \delta.\nonumber
\end{align}
Next, similarly to \eqref{est 161} we can rely on the proofs of \eqref{est 256}-\eqref{est 260} and estimate $\mathfrak{D}$ from \eqref{est 298} for all $t \in (4\Sigma, T]$ 
\begin{align}
\mathfrak{D}(t) \overset{\eqref{est 256} \eqref{est 260} \eqref{est 266}}{\lesssim}&  ( \lVert \theta_{0}(t) \rVert_{L_{x}^{p}}+  M [ (2 \delta)^{\frac{1}{p}} + \Gamma^{\frac{1}{p}} ] +  \lambda^{-\frac{1}{p}} \delta^{\frac{1}{p}} l^{-d-2}+ l^{-d-1} \delta \mu^{b} \omega^{-1}) \nonumber   \\
&\times  \delta^{\frac{1}{p'}} [\sum_{k=1}^{N} \left( \frac{\lambda \mu l^{-d-2}}{\nu} \right)^{k} + \frac{ (\lambda \mu l^{-d-2} )^{N+1}}{\nu^{N}} ] \nonumber\\
\overset{\eqref{est 148} }{\approx}&  ( \lVert \theta_{0}(t) \rVert_{L_{x}^{p}}+   [ 2 \delta  + \Gamma]^{\frac{1}{p}}  + \lambda^{-\frac{1}{p} + (d+2) \iota}  \delta^{\frac{1}{p}}  + \lambda^{(d+1) \iota}\delta \lambda^{\alpha b - \beta}) \nonumber \\
& \times    \delta^{\frac{1}{p'}} [\sum_{k=1}^{N} \lambda^{( 1 + \alpha + (d+2) \iota - \gamma)k}+ \lambda^{(1+ \alpha + (d+2) \iota)(N+1) - \gamma N} ] \ll \delta\label{est 320}
\end{align}
due to \eqref{est 152} which guarantees 
\begin{align*}
\iota < \min \{ \frac{1}{p(d+2)}, \frac{\beta - \alpha b}{d+1}, \frac{\gamma - 1 - \alpha}{d+2}, \frac{1}{d+2} (\frac{\gamma N}{N+1} - 1 - \alpha ) \}.
\end{align*}
At last, we estimate $\mathfrak{E}$ from \eqref{est 299} as follows. Using \eqref{est 262}, \eqref{est 266}, \eqref{est 62}, \eqref{est 390}, and \eqref{est 149}-\eqref{est 152} that implies 
\begin{align*}
\iota < \min\{ \frac{p \alpha b}{d+1}, \frac{\gamma - 1 - \alpha}{d+2}, \frac{1}{d+2}[ \frac{\gamma N}{N+1} - 1 - \alpha] \}, 
\end{align*}
as well as $\lambda \in \mathbb{N}$ sufficiently large 
\begin{align}
\mathfrak{E}(t) \leq& ( \lVert \vartheta(t) \rVert_{L_{x}^{1}} + \lVert q(t) \rVert_{L_{x}^{1}}) ( \lVert u_{0}(t, x+B(t)) \rVert_{L_{x}^{1}} + \lVert w(t) \rVert_{L_{x}^{1}} + \lVert w_{c}(t) \rVert_{L_{x}^{1}})  \label{est 321}\\
\overset{\eqref{est 258}}{\lesssim}&( \lVert \sum_{j=1}^{d} (a_{j} \Theta^{j})(t) \rVert_{L_{x}^{1}} + \lVert \sum_{j=1}^{d} (a_{j}b_{j}Q^{j})(t) \rVert_{L_{x}^{1}}) ( \lVert u_{0}(t) \rVert_{L_{x}^{p'}} + \lVert \sum_{j=1}^{d} (b_{j} W^{j})(t) \rVert_{L_{x}^{1}} + \lVert w_{c}(t) \rVert_{L_{x}^{p'}})\nonumber\\
\overset{\eqref{est 252}}{\lesssim}& ( l^{- \frac{d+1}{p}} \delta^{\frac{1}{p}} (\frac{M}{\mu^{b}})  + (\sup_{\tau \in [t-l,t]} \lVert R_{0}^{j}(\tau) \rVert_{L_{x}^{\infty}} + \Gamma) (\frac{M}{\omega}))  \nonumber\\
&\hspace{5mm}  \times ( \lVert u_{0}(t) \rVert_{L_{x}^{p'}} +  l^{- \frac{d+1}{p'}} \delta^{\frac{1}{p'}}\frac{M}{\mu^{a}} +  \delta^{\frac{1}{p'}} [\sum_{k=1}^{N} \left( \frac{\lambda \mu l^{-d-2}}{\nu} \right)^{k} + \frac{ (\lambda \mu l^{-d-2} )^{N+1}}{\nu^{N}} ]) \nonumber\\
\overset{\eqref{est 148}}{\approx}&( \lambda^{(\frac{d+1}{p}) \iota - \alpha b} \delta^{\frac{1}{p}} + (\sup_{\tau \in [t-l,t]} \lVert R_{0}^{j}(\tau) \rVert_{L_{x}^{\infty}} + \Gamma) \lambda^{-\beta}) \nonumber\\
& \times ( \lVert u_{0}(t) \rVert_{L_{x}^{p'}} +  \lambda^{(\frac{d+1}{p'}) \iota - \alpha a} +  \delta^{\frac{1}{p'}} [ \sum_{k=1}^{N} (\lambda^{[1+\alpha+(d+2)\iota - \gamma] k} + \lambda^{[1+\alpha + (d+2)\iota ](N+1) - \gamma N}])\ll \delta. \nonumber
\end{align}
Considering \eqref{est 317}, \eqref{est 318}, \eqref{est 319}, \eqref{est 320}, and \eqref{est 321} into \eqref{est 294}, we finally conclude \eqref{est 293}:
\begin{align*}
&\lvert \int_{\mathbb{T}^{d}} \theta_{1}(t,x) u_{1}(t, x+ B(t))dx - \int_{\mathbb{T}^{d}} \theta_{0}(t,x) u_{0}(t,x +B(t)) dx - \sum_{j=1}^{d} \int_{\mathbb{T}^{d}} \chi_{j}^{2}(t,x) dx e_{j} \Gamma  \rvert  \\
\leq& (\mathfrak{A} + \mathfrak{B}+ \mathfrak{C} + \mathfrak{D}+ \mathfrak{E})(t) \leq \delta 3d.
\end{align*}
The facts that $(\theta_{1}, u_{1}, R_{1})$ are $(\mathcal{F}_{t})_{t\geq 0}$-adapted and has regularity listed in \eqref{est 241} can be shown identically to the proof of Proposition \ref{Proposition 4.2}. Finally, it is clear that $\tilde{\chi}(0) = 0$ from \eqref{est 257} implies 
\begin{equation}
\theta_{1}(0,x) \overset{\eqref{est 245}}{=} \theta_{0}(0,x) = \theta^{\text{in}}(x), \hspace{5mm} u_{1}(0,x) \overset{\eqref{est 245}}{=} u_{0}(0,x) = u^{\text{in}}(x).
\end{equation} 
\end{proof}

With Proposition \ref{Proposition 6.1} in hand, we are ready to prove Theorem \ref{Theorem 2.5}. 
\begin{proof}[Proof of Theorem \ref{Theorem 2.5}]
We take $\theta_{0} \equiv 0, u_{0} \equiv 0$ and $R_{0} \equiv 0$ so that the hypothesis of Proposition \ref{Proposition 6.1} with $\theta^{\text{in}} \equiv 0, u^{\text{in}} \equiv 0$ are all trivially satisfied. For the fixed $\xi \in (0,T)$ from the hypothesis of Theorem \ref{Theorem 2.5}, we can take e.g., 
\begin{equation}\label{est 343} 
\delta_{n} \triangleq 2^{-n}, \hspace{2mm} \Sigma_{n} \triangleq \frac{\xi}{4}2^{-n}, \hspace{2mm} { and } \hspace{2mm}  \Gamma_{n} \triangleq 
\begin{cases}
2 \delta_{n} & \text{ if } n \in \mathbb{N}\setminus \{3\},\\
K \gg 1 & \text{ if } n = 3. 
\end{cases}
\end{equation}  
Note that this choice satisfies ``$\frac{\Gamma}{\delta} \leq \bar{C} < \infty$'' from the hypothesis of Proposition \ref{Proposition 6.1}. Then clearly $\Sigma_{n} \searrow 0$ as $n\nearrow + \infty$ so that for any $t \in (0,T]$, there exists $N \in \mathbb{N}$ sufficiently large such that $N \geq 4$ and $4 \Sigma_{n} \leq t$ for all $n \geq N$ so that $t \in (4 \Sigma_{n} \wedge T, T]$. Due to \eqref{est 246} and \eqref{est 248} this implies 
\begin{align}
\sum_{n=0}^{\infty} \lVert (\theta_{n+1} - \theta_{n})(t) \rVert_{L_{x}^{p}} 
\overset{\eqref{est 343}}{\leq}&   \sum_{n=0}^{N} \lVert (\theta_{n+1} - \theta_{n})(t)  \rVert_{L_{x}^{p}} + 4^{\frac{1}{p}}M   \sum_{n=N+1}^{\infty}  2^{-\frac{n}{p}} < \infty\label{est 344}
\end{align}
and 
\begin{align}
\sum_{n=0}^{\infty} \lVert (u_{n+1} - u_{n})(t) \rVert_{L_{x}^{p'}} \overset{\eqref{est 343}}{\leq} 
\sum_{n=0}^{N} \lVert (u_{n+1} - u_{n})(t) \rVert_{L_{x}^{p'}} + 4^{\frac{1}{p'}}M \sum_{n=N+1}^{\infty} 2^{-\frac{n}{p'}} < \infty. \label{est 345} 
\end{align}
On the other hand, for any $t \in [0,T]$, we deduce immediately from \eqref{est 247} and \eqref{est 249}  that 
\begin{align}\label{est 346}
\sum_{n=0}^{\infty} \lVert (\theta_{n+1} - \theta_{n})(t) \rVert_{L_{x}^{\upsilon}} \overset{\eqref{est 343}}{\leq} \sum_{n=0}^{\infty}2^{-n} < \infty, \sum_{n=0}^{\infty} \lVert (u_{n+1} - u_{n})(t) \rVert_{W_{x}^{1,\tilde{p}}} \overset{\eqref{est 343}}{\leq} \sum_{n=0}^{\infty} 2^{-n} < \infty. 
\end{align} 
Therefore, $\{\theta_{n}(t)\}_{n\in\mathbb{N}_{0}}$ and $\{u_{n}(t)\}_{n\in\mathbb{N}_{0}}$ are Cauchy respectively in $C((0,T]; L^{p}(\mathbb{T}^{d})) \cap C([0,T]; L^{\upsilon}(\mathbb{T}^{d}))$ and $C((0,T]; L^{p'}(\mathbb{T}^{d})) \cap C([0,T]; W^{1,\tilde{p}}(\mathbb{T}^{d}))$ so that there exists unique $(\theta, u)$ in $[C((0,T]; L^{p}(\mathbb{T}^{d})) \cap C([0,T]; L^{\upsilon}(\mathbb{T}^{d}))] \times [C((0,T]; L^{p'}(\mathbb{T}^{d})) \cap C([0,T]; W^{1,\tilde{p}}(\mathbb{T}^{d}))]$ and a deterministic constant $C$ such that 
\begin{align}
\lVert \theta \rVert_{C((0,T]; L^{p}(\mathbb{T}^{d}))} + \lVert \theta \rVert_{C([0,T]; L^{\upsilon}(\mathbb{T}^{d}))} + \lVert u \rVert_{C((0,T]; L^{p'}(\mathbb{T}^{d}))} + \lVert u \rVert_{C([0,T]; W^{1,\tilde{p}}(\mathbb{T}^{d}))} \leq C < \infty. 
\end{align} 
Moreover, because $(\theta_{0}, u_{0},R_{0})$ are $(\mathcal{F}_{t})_{t\geq 0}$-adapted, so are $(\theta_{n}, u_{n}, R_{n})$ for all $n \in \mathbb{N}$ by Proposition \ref{Proposition 6.1}; consequently, so are $(\theta, u)$. Next, for all $t \in (0, T]$, because $\Sigma_{n} = \frac{\xi}{4} 2^{-n} \searrow 0$ as $n\nearrow \infty$, we see that there exists $N \in \mathbb{N}$ sufficiently large that $2\Sigma_{n} \wedge T \leq t$ for all $n \geq N$ and therefore $\lVert R_{n}(t) \rVert_{L_{x}^{1}} \leq 2 \delta_{n} = 2^{1-n} \searrow 0$ as $n\nearrow + \infty$ due to \eqref{est 251}. Moreover, from \eqref{est 247} and \eqref{est 249} we deduce that $(\theta, u)\rvert_{t=0} \equiv 0$ because $(\theta_{0}, u_{0}) \rvert_{t=0} \equiv 0$. We can now deduce that $(\rho, u)$ where $\rho(t,x) = \theta(t, x-B(t))$ from Section \ref{Section 2} satisfies \eqref{stochastic transport} forced by transport noise in Stratonovich's interpretation analytically weakly as follows; i.e., \eqref{solution transport}. Because $(\theta_{n}, u_{n}, R_{n})$ for all $n \in \mathbb{N}$ satisfy \eqref{est 172} strongly (recall \eqref{est 241}), we may take an arbitrary $\psi \in C_{c}^{\infty} (\mathbb{T}^{d})$ and deduce from \eqref{est 172}  
\begin{align*}
&\partial_{s} \theta_{n}(s,x) \psi(x+B(s)) + \text{div} (u_{n}(s, x+B(s)) \theta_{n}(s,x)) \psi(x+ B(s)) \\
& - \Delta \theta_{n}(s,x) \psi(x+B(s)) = -\text{div} R_{n}(s,x) \psi(x+B(s)). 
\end{align*}
We may integrate over $[0,t]\times \mathbb{T}^{d}$ for any $t \in [0,T]$, integrate by parts considering the smoothness of both $(\theta_{n}, u_{n}, R_{n})$ and $\psi$, pass the limit $n\to\infty$ considering the convergence results we have obtained already, apply It$\hat{\mathrm{o}}$-Wentzell-Kunita formula (e.g., \cite[Theorem 3.3.2 on p. 93]{K90}, also \cite{K82, K84}) on $\int \theta(t,x)\, \psi(x+y)\,dx$ which is smooth, and translate $x + B(t) = y$ to deduce the claim. At last, to prove non-uniqueness in law of \eqref{stochastic transport} forced by transport noise in Stratonovich's interpretation, we rely on \eqref{est 293}. First, we denote the $n$-th repetition of the cut-off function $\chi_{j}$ in \eqref{est 243} by 
\begin{equation}\label{est 349}
\chi_{n,j}: [0,T] \times \mathbb{T}^{d} \mapsto [0,1] \text{ such that } \chi_{n,j} (t,x) = 
\begin{cases}
0 & \text{ if } \lvert R_{n,l}^{j} (t,x) + \Gamma_{n} \rvert \leq \frac{\delta_{n}}{4d}, \\
1 & \text{ if } \lvert R_{n,l}^{j}(t,x)  + \Gamma_{n} \rvert \geq \frac{\delta_{n}}{2d}, 
\end{cases} 
\end{equation} 
where $R_{n,l}$ is $R_{n}$ that is mollified in space-time. It is easy to see that due to our choice of parameters from \eqref{est 343}, for any $t \in (4\Sigma_{n} \wedge T, T]$, $\sum_{j=1}^{d} \int_{\mathbb{T}^{d}} \chi_{n,j}^{2}(t,x) dx e_{j} \equiv 0$ is impossible. Now $\xi < T$ by hypothesis of Theorem \ref{Theorem 2.5} while $4 \Sigma_{0} \wedge T = \xi \wedge T$ by \eqref{est 343}; thus, $(4 \Sigma_{0} \wedge T, T] = (\xi, T]$. For such $t \in (4\Sigma_{0} \wedge T, T] = (\xi, T]$, using our choice of $(\theta_{0}, u_{0}) \equiv 0$, we can find a finite constant $\tilde{C}$ such that 
\begin{align}
& \lvert \int_{\mathbb{T}^{d}} u(t, x+B(t)) \theta(t,x) dx - \sum_{j=1}^{d} \int_{\mathbb{T}^{d}} \chi_{3,j}^{2}(t,x) dx e_{j} K \rvert \nonumber \\
\overset{\eqref{est 343} \eqref{est 349} }{\leq}& \lvert \sum_{n=0}^{\infty} \int_{\mathbb{T}^{d}} u_{n+1}(t, x+B(t)) \theta_{n+1}(t,x) - u_{n}(t,x+B(t)) \theta_{n}(t,x) - \sum_{j=1}^{d} \chi_{n,j}^{2}(t,x)  e_{j} \Gamma_{n} dx \rvert \nonumber \\
& \hspace{15mm} + \sum_{n\in\mathbb{N}_{0}: n \neq 3}d \Gamma_{n}  \overset{\eqref{est 293}\eqref{est 343} }{\leq} \sum_{n=0}^{\infty} \delta_{n} 3d +d \sum_{n\in\mathbb{N}_{0}: n\neq 3} 2(2^{-n}) = \tilde{C}. \label{est 351}
\end{align} 
We take two distinct $K, K' \gg 1$ such that 
\begin{equation}\label{est 350}
\sum_{j=1}^{d} \int_{\mathbb{T}^{d}} \chi_{3,j}(t,x)^{2} dx e_{j} \lvert K - K' \rvert > 2\tilde{C}. 
\end{equation}
For such $K, K'$, we can get a corresponding $(\theta^{K}, u^{K})$ and $(\theta^{K'}, u^{K'})$ so that for $t \in (4\Sigma_{0} \wedge T, T] = (\xi, T]$ 
\begin{align}
& \lvert \int_{\mathbb{T}^{d}} u^{K} (t, x+B(t)) \theta^{K} (t,x) dx - \int_{\mathbb{T}^{d}} u^{K'} (t, x+B(t)) \theta^{K'} (t,x) dx \rvert \nonumber\\
\geq&\sum_{j=1}^{d} \int_{\mathbb{T}^{d}} \chi_{3,j}^{2}(t,x) dx e_{j} \lvert K - K' \rvert \nonumber \\
&- \lvert \int_{\mathbb{T}^{d}} u^{K}(t, x+B(t)) \theta^{K} (t,x) dx - \left(\sum_{j=1}^{d} \int_{\mathbb{T}^{d}} \chi_{3,j}^{2}(t,x) dx e_{j} \right) K \rvert \nonumber\\
& - \lvert \int_{\mathbb{T}^{d}} u^{K'}(t, x+B(t)) \theta^{K'} (t,x) dx - \left(\sum_{j=1}^{d} \int_{\mathbb{T}^{d}} \chi_{3,j}^{2}(t,x) dx e_{j} \right) K' \rvert  \overset{\eqref{est 350} \eqref{est 351}}{>} 0. 
\end{align}
By a change of variable $y = x+B(t)$ and $\theta (t,x) = \rho(t, x+ B(t))$, we deduce that we have constructed at least two pairs $(\rho^{K},u^{K})$ and $(\rho^{K'}, u^{K'})$ that satisfy \eqref{stochastic transport} forced by transport noise in Stratonovich's interpretation analytically weakly, $\rho^{K}, u^{K}, \rho^{K'}, u^{K'}$ all vanish at $t =0$ and for all $t \in (4\Sigma_{0} \wedge T, T] = (\xi, T]$, 
\begin{align}
\int_{\mathbb{T}^{d}} u^{K} (t, y) \rho^{K}(t,y) dy \neq \int_{\mathbb{T}^{d}} u^{K'} (t, y) \rho^{K'} (t,y) dy \hspace{3mm} \mathbb{P}\text{-a.s.}\label{est 352}
\end{align}
I.e., \eqref{est 236} has been proven. At last, non-uniqueness in law follows immediately. As \eqref{est 352} implies that one of the integrals is non-zero for some $t_{0} \in (\xi, T]$, without loss of generality we assume $\int_{\mathbb{T}^{d}} u^{K} (t_{0}, y) \rho^{K}(t_{0},y) dy \neq 0$ $\mathbb{P}$-a.s. and hence $\rho^{K}(t_{0}) \not\equiv 0$ $\mathbb{P}$-a.s.; we let $P$ denote the law of $\rho^{K}$. Next, we take this $u^{K}$, construct an analytically weak solution $\tilde{\rho}$ to \eqref{stochastic transport} forced by transport noise in Stratonovich's interpretation via Galerkin approximation with same initial data $\tilde{\rho} \rvert_{t=0} \equiv 0$, namely $\tilde{\rho}\equiv 0$. Therefore, if we denote $\tilde{P}$ the law of $\tilde{\rho}$, then we see that $P$ and $\tilde{P}$ are distinct, allowing us to conclude the non-uniqueness in law.
\end{proof}

\section{Appendix A}\label{Appendix A} 
\subsection{Proof of Theorem \ref{Theorem 2.1}}
The purpose of this section is to prove Theorem \ref{Theorem 2.1}. In fact, it follows immediately from the following more general result.
\begin{theorem}\label{Theorem 7.1}  
Let $\epsilon > 0, T > 0$, and $\bar{\theta} \in C^{\infty} ([0,T] \times \mathbb{T}^{d})$ such that $\fint_{\mathbb{T}^{d}} \bar{\theta} (t,x) dx = 0$ for all $t \in [0,T]$ and $\bar{u} \in C^{\infty} ([0,T] \times \mathbb{T}^{d})$ such that $\nabla\cdot \bar{u} = 0$ on $[0,T] \times \mathbb{T}^{d}$. Set 
\begin{equation}\label{est 391}
E \triangleq \{ t \in [0,T]: \partial_{t} \bar{\theta} + \text{div} (\bar{u} \bar{\theta}) + \frac{1}{2} \bar{\theta} - \Delta \bar{\theta} = 0 \}. 
\end{equation} 
Let $p \in (1,\infty)$ and $\tilde{p} \in (1,\infty)$ such that \eqref{est 8} holds. Then there exist deterministic $\theta: [0,T] \times \mathbb{T}^{d} \mapsto \mathbb{R}$ and $u: [0,T] \times \mathbb{T}^{d} \mapsto \mathbb{R}^{d}$ such that 
\begin{enumerate}
\item $\theta \in C([0,T]; L^{p}(\mathbb{T}^{d}))$ and $u \in C([0,T]; L^{p'} (\mathbb{T}^{d})) \cap C([0,T];  W^{1,\tilde{p}}(\mathbb{T}^{d}))$, 
\item $(\theta, u)$ satisfies \eqref{est 4} analytically weakly,
\item $(\theta, u) (t,x) = (\bar{\theta}, \bar{u})(t,x)$ for all $(t,x) \in E \times \mathbb{T}^{d}$,  
\item $\lVert (\theta - \bar{\theta})(t) \rVert_{L_{x}^{p}} < \epsilon$ or $\lVert (u - \bar{u})(t) \rVert_{L_{x}^{p'}} < \epsilon$ for all $t \in [0,T]$. 
\end{enumerate}   
\end{theorem} 

\begin{proof}[Proof of Theorem \ref{Theorem 2.1} assuming Theorem \ref{Theorem 7.1}]
We take $\Upsilon \in C^{\infty} ([0,T])$ such that  
\begin{equation}\label{est 215}
\Upsilon(t) \in [0,1], \hspace{3mm} \Upsilon(t) = 
\begin{cases}
0 & \text{ on } [0, \frac{T}{3}], \\
e^{-\frac{t}{2}} & \text{ on } [\frac{2T}{3}, T].
\end{cases}
\end{equation} 
Define 
\begin{equation}\label{est 216} 
\bar{\theta}(x) \triangleq c (\frac{1}{2} - x_{d}) \hspace{2mm} \text{ for } \hspace{2mm} c \neq 0 
\end{equation} 
so that $\fint_{\mathbb{T}^{d}} \bar{\theta}(x) dx = 0$. Then $\Upsilon(t) \bar{\theta}(x) \in C^{\infty} ([0,T] \times \mathbb{T}^{d})$ and $\fint_{\mathbb{T}^{d}} \Upsilon(t) \bar{\theta} (x) dx = 0$ for all $t \in [0,T]$. We take $\bar{u} \equiv 0$ so that it is smooth and divergence-free. Now $(\Upsilon \bar{\theta}, \bar{u})$ satisfies 
\begin{equation}
\partial_{t} ( \Upsilon \bar{\theta}) + \text{div} ( \bar{u} \Upsilon \bar{\theta}) + \frac{1}{2} \Upsilon \bar{\theta} - \Delta ( \Upsilon \bar{\theta}) = 0
\end{equation}
on $[0, \frac{T}{3}] \cup [\frac{2T}{3}, T]$. Consequently, $[0, \frac{T}{3}] \cup [\frac{2T}{3}, T] \subset E$. By Theorem \ref{Theorem 7.1} we find functions $\theta \in C([0,T]; L^{p}(\mathbb{T}^{d}))$ and $u \in C([0,T]; L^{p'} (\mathbb{T}^{d})) \cap C([0,T]; W^{1,\tilde{p}}(\mathbb{T}^{d}))$ such that $(\theta, u)$  satisfies \eqref{est 4} analytically weakly and $(\theta, u) (t) = (\Upsilon \bar{\theta}, \bar{u})(t)$ for all $t \in E$. As a result, $\rho(t,x) = \theta(t,x) e^{B(t)}$ solves \eqref{solution multiplicative}, $\rho \in C_{t}L_{x}^{p}$ $\mathbb{P}$-a.s., $\rho(t) \rvert_{t=0}  = \Upsilon(0) \bar{\theta} e^{0} = 0$ due to \eqref{est 215}, and $\rho(t) = \theta(t) e^{B(t)}$ is $(\mathcal{F}_{t})_{t\geq 0}$-adapted. Therefore, given $B$ on $(\Omega, \mathcal{F}, \mathbb{P})$, for all $T > 0$ and $c \neq 0$, we constructed $\rho$ such that $\rho \rvert_{t=0} \equiv 0$ and $\rho \rvert_{t=T} = e^{B(T)} e^{-\frac{T}{2}} c(\frac{1}{2} - x_{d})$ so that $\mathbb{P} ( \{ \rho \rvert_{t=T} \equiv 0 \}) = 0$. Due to an arbitrariness of $c \neq 0$ in \eqref{est 216} we obtain infinitely many $(\theta, u)$ and therefore $(\rho, u)$ that satisfy these properties.  Because $(\tilde{\rho}, B)$ on $(\Omega, \mathcal{F}, \mathbb{P})$ where $\tilde{\rho} \equiv 0$ is another analytically weak solution with same initial distribution, we conclude that non-uniqueness in law holds on $[0,T]$.\end{proof}

The proof of Theorem \ref{Theorem 7.1} follows from the next proposition concerning the following damped transport-diffusion defect equation
\begin{equation}\label{est 217}
\partial_{t} \theta + \text{div} (u\theta) + \frac{1}{2} \theta - \Delta \theta = - \text{div} R, \hspace{3mm} \nabla \cdot u = 0. 
\end{equation} 
\begin{proposition}\label{Proposition 7.2} 
There exists a constant $M > 0$ such that the following holds. Let $p \in (1,\infty)$ and $\tilde{p} \in [1,\infty)$ such that \eqref{est 8} holds. Then, for any $\delta, \eta > 0$ and any smooth $(\theta_{0}, u_{0}, R_{0})$ that satisfies \eqref{est 217}, there is another smooth $(\theta_{1}, u_{1}, R_{1})$ that satisfies for al $t \in [0,T]$ \eqref{est 217} and  
\begin{subequations}\label{est 338}
\begin{align}
& \lVert (\theta_{1} - \theta_{0})(t) \rVert_{L_{x}^{p}} \leq M \eta \lVert R_{0}(t) \rVert_{L_{x}^{1}}^{\frac{1}{p}}, \\
&\lVert (u_{1} - u_{0})(t) \rVert_{L_{x}^{p'}} \leq M \eta^{-1} \lVert R_{0}(t) \rVert_{L_{x}^{1}}^{\frac{1}{p'}}, \\
& \lVert (u_{1} - u_{0})(t) \rVert_{W_{x}^{1,\tilde{p}}} \leq \delta, \\
& \lVert R_{1}(t) \rVert_{L_{x}^{1}} \leq \delta.  \label{est 223}
\end{align}
\end{subequations}
Furthermore, $R_{1}(t) \equiv 0$ and $(\theta_{1}, u_{1})(t) \equiv (\theta_{0}, u_{0})(t)$ for all $t \in [0,T]$ such that $R_{0}(t) \equiv 0$. 
\end{proposition}

\begin{proof}[Proof of Theorem \ref{Theorem 7.1} assuming Proposition \ref{Proposition 7.2}]
We take $(\bar{\theta}, \bar{u})$ in the hypothesis of Theorem \ref{Theorem 7.1} and define 
\begin{equation}\label{est 339}
(\theta_{0}, u_{0}) = (\bar{\theta}, \bar{u}), \hspace{3mm} R_{0}(t) \triangleq - \mathcal{D}^{-1}[\partial_{t} \bar{\theta}+ \text{div} (\bar{u} \bar{\theta}) + \frac{1}{2} \bar{\theta}](t) + \nabla \bar{\theta}(t)
\end{equation} 
where $R_{0}$ is well-defined by Definition \ref{Definition 3.1} because $\int_{\mathbb{T}^{d}} \frac{1}{2} \bar{\theta}(t) dx = 0$. The rest of this proof follows very closely the ``Proof of Theorem 1.2, assuming Proposition 2.1'' in \cite[Section 2]{MS20} and thus is omitted.
\end{proof} 

Therefore, the proof of Theorem \ref{Theorem 2.1} is complete once Proposition \ref{Proposition 7.2} is proven. In fact, we can adhere to the convex integration setting from Section \ref{Subsection 3.2} and Proposition \ref{Proposition 7.2} can be proven similarly to the proof of \cite[Proposition 2.1]{MS20} with only a few necessary modifications. We actually already addressed all such modifications. First, we need to choose $\epsilon$ in \eqref{est 60} rather than $\epsilon$ in \cite[Equation (4.11)]{MS20}. We prove in Appendix B that with such a choice of $\epsilon$, we can still retain Lemma \ref{Lemma 4.3}, which is \cite[Proposition 4.4]{MS20}.  Using the same observation in the proof of \eqref{est 93}, specifically \eqref{est 94}, we can also retain Lemma \cite[Lemma 4.12]{MS20}. Finally, concerning the definition of the new defect $R_{1}$, in the proofs of Theorems \ref{Theorem 2.2}, \ref{Theorem 2.4}, and \ref{Theorem 2.5}, we added the diffusion $-\Delta (\vartheta + q)$ within $R^{\text{lin}}$ (see e.g., \eqref{est 99} and \eqref{est 124}). Because the only difference from our current proof and that of \cite[Proposition 2.1]{MS20} is the diffusive and damping terms, let us define 
\begin{equation}\label{est 224} 
R^{\text{diff}} \triangleq \frac{1}{2} \mathcal{D}^{-1} [\vartheta + \vartheta_{c} + q + q_{c}] - \nabla [\vartheta + q]
\end{equation} 
which is well-defined as $\fint_{\mathbb{T}^{d}} \vartheta + \vartheta_{c} + q + q_{c} =0$ by definitions of $\vartheta_{c}$ and $q_{c}$ on \cite[p. 1093]{MS20} (see \eqref{est 85}) so that we have 
\begin{equation}
-R_{1} \triangleq R^{\text{time,1}} + R^{\text{quadr}} + R^{\chi} + R^{\text{time,2}} + R^{\text{lin}} + R^{q} + R^{\text{corr}} + R^{\text{diff}}
\end{equation}  
where $R^{\text{time,1}}, R^{\text{quadr}}, R^{\chi}, R^{\text{time,2}}, R^{\text{lin}}, R^{q}$, and $R^{\text{corr}}$ are defined identically to \cite{MS20}. Then the rest of the proof of \cite[Proposition 2.1]{MS20} completely goes through with same choice of parameters in \eqref{est 148} (\cite[Section 6.1]{MS20}) leaving us an only task to show that $\lVert R^{\text{diff}}(t) \rVert_{L_{x}^{1}} \ll \delta$ to conclude \eqref{est 223}. The estimate of the diffusive term $\nabla [\vartheta + q]$ is described on \cite[p. 1106]{MS20}; in fact, simplification of how we handled this term in \eqref{est 142} already shows that $\lVert \nabla (\vartheta + q) \rVert_{L_{x}^{1}} \ll \delta$ in our current case. The estimate of the damping term $\frac{1}{2} \mathcal{D}^{-1} [\vartheta + \vartheta_{c} + q + q_{c}]$ is even easier as follows: as $\vartheta + \vartheta_{c} + q + q_{c}$ is mean-zero, 
\begin{align}
&\lVert \mathcal{D}^{-1} [ (\vartheta + \vartheta_{c} + q + q_{c})(t) ] \rVert_{L_{x}^{1}} 
\overset{\eqref{est 121}}{\lesssim} \lVert (\vartheta + \vartheta_{c} + q + q_{c})(t) \rVert_{L_{x}^{1}} \lesssim  [ \lVert \vartheta(t) \rVert_{L_{x}^{1}} + \lVert q(t) \rVert_{L_{x}^{1}} ] \\
&\hspace{12mm} \lesssim \sum_{j=1}^{d} [ \lVert a_{j}(t) \rVert_{L_{x}^{\infty}} \lVert \Theta^{j} (t) \rVert_{L_{x}^{1}} + \lVert a_{j}(t) \rVert_{L^{\infty}} \lVert b^{j}(t) \rVert_{L_{x}^{\infty}} \lVert Q^{j} (t) \rVert_{L_{x}^{1}} ] \overset{\eqref{est 62}}{\lesssim} \frac{M}{\mu^{b}} + \frac{M}{\omega} \ll 1  \nonumber 
\end{align}
for $\lambda \in\mathbb{N}$ sufficiently large. This completes the proof of Theorem \ref{Theorem 2.1}.  

\subsection{Proof of Theorem \ref{Theorem 2.4}}
The following result is the key proposition, analogous to Proposition \ref{Proposition 4.2}.
\begin{proposition}\label{Proposition 7.3} 
There exists a constant $M > 0$ such the following holds. Let $p \in (1,\infty)$, $\tilde{p} \in [1,\infty)$ such that \eqref{est 8} holds, and $\varpi \in (0, \frac{1}{4})$. Then for any $\delta, \eta > 0$ and $(\mathcal{F}_{t})_{t\geq 0}$-adapted $(\theta_{0}, u_{0}, R_{0})$ in the regularity class \eqref{est 241} that satisfies \eqref{est 172} such that for all $t \in [0, T]$ $\fint_{\mathbb{T}^{d}} \theta_{0}(t,x) dx = 0$ and 
\begin{equation}\label{est 178}
\lVert R_{0}(t) \rVert_{L_{x}^{1}} \leq 2 \delta, 
\end{equation}
there exists another $(\mathcal{F}_{t})_{t\geq 0}$-adapted $(\theta_{1}, u_{1}, R_{1})$ that satisfies \eqref{est 172} in same corresponding regularity class \eqref{est 241} such that for all $t \in [0,T]$ $\fint_{\mathbb{T}^{d}} \theta_{1}(t,x) dx = 0$ and satisfies 
\begin{subequations}\label{est 173}
\begin{align}
& \lVert (\theta_{1} -\theta_{0})(t) \rVert_{L_{x}^{p}} \leq M \eta ( 2 \delta)^{\frac{1}{p}}, \label{est 174} \\ 
& \lVert (u_{1} - u_{0})(t) \rVert_{L_{x}^{p'}} \leq M \eta^{-1} (2\delta)^{\frac{1}{p'}}, \label{est 175}\\
& \lVert (u_{1} - u_{0})(t) \rVert_{W_{x}^{1,\tilde{p}}} \leq \delta, \label{est 176}\\
& \lVert R_{1}(t) \rVert_{L_{x}^{1}} \leq \delta. \label{est 177}
\end{align}
\end{subequations}
Finally, if $(\theta_{0}, u_{0}, R_{0})(0,x)$ are deterministic, then so are $(\theta_{1}, u_{1}, R_{1})(0,x)$. 
\end{proposition} 

\begin{proof}[Proof of Theorem \ref{Theorem 2.4} assuming Proposition \ref{Proposition 7.3}]
The following proof has much similarity with the proof of Theorem \ref{Theorem 2.2} assuming Proposition \ref{Proposition 4.2}; in fact, it's much simpler. For the  $K> 1$ fixed from hypothesis, we can define $\theta_{0}(t,x) \triangleq Kt (x_{d} - \frac{1}{2}), u_{0} \equiv 0$ and $R_{0} \triangleq -\mathcal{D}^{-1} \partial_{t} \theta_{0}$ similarly to \eqref{est 179}. By construction, $(\theta_{0}, u_{0}, R_{0})$ satisfies \eqref{est 172}, and $\theta_{0} \in C_{t,x}^{\infty}, u_{0} \in C_{t}^{\frac{1}{2} - 2 \varpi}C_{x}^{\infty}$, $R_{0} \in C_{t}^{\frac{1}{2} - 2 \varpi} C_{x}^{\infty}$. From \eqref{est 40} we see that this new definition of $\theta_{0}$ gives 
\begin{equation}\label{est 323}
\lVert \theta_{0}(t) \rVert_{L^{p}} = \frac{Kt}{2(p+1)^{\frac{1}{p}}}. 
\end{equation}
Similarly to the proof of Theorem \ref{Theorem 2.2} we set $\delta = \sup_{t \in [0,T]}  \lVert R_{0}(t) \rVert_{L_{x}^{1}}$ so that Proposition \ref{Proposition 7.3} is applicable. For the fixed $T > 0$ from hypothesis of Theorem \ref{Theorem 2.4}, we choose similarly to \eqref{est 41}-\eqref{est 42} $\delta_{n} = \delta 2^{-(n-1)}$ for $n \in \mathbb{N}_{0}$, $\eta_{n}$ so that $\delta_{n}^{\frac{1}{p}} \eta_{n} = \sigma \delta_{n}^{\frac{1}{2}}$ for $\sigma > 0$ such that 
\begin{equation}\label{est 169}
\sigma 4 (p+1)^{\frac{1}{p}} M \sum_{n=0}^{\infty} \sqrt{\delta} 2^{-\frac{n-1}{2}} < T
\end{equation}
similarly to \eqref{est 47}. We note that at first sight, the proof of Theorem \ref{Theorem 2.2} assuming Proposition \ref{Proposition 4.2} seems to rely heavily on $L$ (e.g., \eqref{est 50}, \eqref{est 51}, \eqref{est 48}, \eqref{est 52}) and hence seems to suggest difficulty in our current proof; this is somewhat compensated by the fact that we can choose this $\sigma$ freely and we included $T$ in its range. Now we apply Proposition \ref{Proposition 7.3} repeatedly, similarly to \eqref{est 180}, and deduce that there exist unique $\theta \in C([0,T]; L^{p}(\mathbb{T}^{d}))$, $u \in C([0,T]; L^{p'} (\mathbb{T}^{d})) \cap C([0,T]; W^{1, \tilde{p}}(\mathbb{T}^{d}))$ which are both $(\mathcal{F}_{t})_{t\geq 0}$-adapted, satisfies the same deterministic bound in \eqref{est 347} over $[0,T]$ such that 
\begin{subequations}\label{idea 2}
\begin{align}
&  \theta_{n} \to \theta \hspace{1mm} \text{ in } C([0,T]; L^{p} (\mathbb{T}^{d})), \hspace{1mm} u_{n} \to u \hspace{1mm} \text{ in } C([0,T]; L^{p'} (\mathbb{T}^{d})) \cap C([0,T]; W^{1,\tilde{p}}(\mathbb{T}^{d})), \\
& R_{n} \to 0 \hspace{1mm} \text{ in } C([0,T]; L^{1}(\mathbb{T}^{d})).
\end{align}
\end{subequations} 
Consequently, by similar computations to \eqref{est 181} and the proof of Theorem \ref{Theorem 2.5}, $(\theta, u)$ satisfies  \eqref{est 6} analytically weakly and hence $(\rho, u)$ satisfies \eqref{stochastic transport} forced by transport noise analytically weakly, i.e., \eqref{solution transport}. Now we fix such $\rho$ and $\rho^{\text{in}} = \rho \rvert_{t=0}$. Due to our choice of $\sigma$ in \eqref{est 169}, identically to \eqref{est 48} we can deduce  
\begin{equation}\label{est 182}
\lVert (\theta - \theta_{0})(t) \rVert_{L^{p}} \overset{\eqref{est 174}}{\leq} \sum_{n=0}^{\infty} M \eta_{n} ( 2 \delta_{n+1})^{\frac{1}{p}}  
\overset{\eqref{est 169}}{<} \frac{T}{4 (p+1)^{\frac{1}{p}}}.
\end{equation}
Then \eqref{est 355} follows immediately by using the fact that $\lVert \theta_{0}(0) \rVert_{L_{x}^{p}} = 0$ from \eqref{est 323}:
\begin{align*}
\lVert \rho(T) \rVert_{L_{x}^{p}} = \lVert \theta(T) \rVert_{L_{x}^{p}} \geq& \lVert \theta_{0}(T)\rVert_{L_{x}^{p}} - \lVert (\theta - \theta_{0})(T) \rVert_{L_{x}^{p}} \overset{\eqref{est 323} \eqref{est 182}}{>} \frac{K T}{2(p+1)^{\frac{1}{p}}} - \frac{T}{4(p+1)^{\frac{1}{p}}} \\
\overset{\eqref{est 182}}{>}& K\lVert \theta(0) - \theta_{0}(0) \rVert_{L_{x}^{p}} \overset{\eqref{est 323}}{=} K[\lVert \theta(0) - \theta_{0}(0) \rVert_{L_{x}^{p}} + \lVert \theta_{0}(0) \rVert_{L_{x}^{p}}] \geq K \lVert \rho^{\text{in}} \rVert_{L_{x}^{p}}.
\end{align*}
Identically to the proof of Theorem \ref{Theorem 2.2}, because $(\theta_{0}, u_{0}, R_{0})(0,x)$ were deterministic, Proposition \ref{Proposition 7.3} implies that $(\theta_{n}, u_{n}, R_{n})(0,x)$ for all $n \in \mathbb{N}$ are deterministic and consequently so is $\theta(0,x)  = \rho^{\text{in}}(x)$. Lastly, we deduce the non-uniqueness in law from \eqref{est 355}. Let $P$ denote the law of $\rho$ that we constructed. On the other hand, by Galerkin approximation we can construct a martingale solution $\tilde{P}$ such that $\mathbb{E}^{\tilde{P}} [ \lVert \xi(T) \rVert_{L^{p}}] \leq \lVert \xi^{\text{in}} \rVert_{L^{p}}$. In observation of \eqref{est 355}, we conclude that $P$ and $\tilde{P}$ are distinct. This completes the proof of Theorem \ref{Theorem 2.4}.  
\end{proof} 

The rest of the proof of Theorem \ref{Theorem 2.4} is devoted to the proof of Proposition \ref{Proposition 7.3}. Except several crucial modifications, we can follow the steps in the proof of Proposition \ref{Proposition 4.2}.  

\subsubsection{Proof of Proposition \ref{Proposition 7.3}}
We define $l \triangleq \lambda^{-\iota}$ identically to \eqref{est 95}. Given $(\mathcal{F}_{t})_{t\geq 0}$-adapted $(\theta_{0}, u_{0}, R_{0})$ in the regularity class \eqref{est 241} that satisfies \eqref{est 172}, we extend $R_{0}$ to $t < 0$ with its value at $t = 0$ and mollify it identically to \eqref{est 86}. For the fixed $p \in (1,\infty)$ from the hypothesis of Proposition \ref{Proposition 7.3}, we again define $a \triangleq \frac{d}{p}$ and $b \triangleq \frac{d}{p'}$ so that $a+ b = d$ as in \eqref{est 55}.  We continue with same convex integration settings identically to the proof of Theorem \ref{Theorem 2.2}: $r$ from Lemma \ref{Lemma 3.8}, $\varrho$ from \eqref{est 56}, $\psi$ from \eqref{est 57}, $\lambda, \mu, \omega, \nu$ from \eqref{est 58}, $\Theta_{\lambda, \mu, \omega, \nu}^{j}$, $W_{\lambda, \mu, \omega, \nu}^{j}$, and $Q_{\lambda, \mu, \omega, \nu}^{j}$ from \eqref{est 59}. We define $\epsilon$ to satisfy \eqref{est 60} again so that Lemma \ref{Lemma 4.3} remains valid for us. We write $R_{l}(t,x)$ in components again as in \eqref{est 104}. Next, we define $\theta_{1}$ identically to but $u_{1}$ differently from \eqref{est 76} as follows:
\begin{subequations}\label{est 186} 
\begin{align}
&\theta_{1}(t,x) \triangleq \theta_{0} (t,x) + \vartheta(t,x) + \vartheta_{c} (t) + q(t,x) + q_{c}(t), \label{est 187}\\
&u_{1}(t,x) \triangleq  u_{0} (t,x) + w(t,x-B(t)) + w_{c}(t,x-B(t)) \label{est 188}
\end{align}
\end{subequations}  
(recall Remark \ref{Remark on difficulty of transport}) where we identically define $\vartheta(t,x), w(t,x), q(t,x)$ in \eqref{est 79}, and $a_{j}, b_{j}$ in \eqref{est 77}, but with slightly modified $\chi_{j}$ from \eqref{est 70}: 
\begin{equation}\label{est 322}
\chi_{j}: \mathbb{R}_{\geq 0}\times \mathbb{T}^{d} \mapsto [0,1] \text{ and } \chi_{j} (t,x) = 
\begin{cases}
0 & \text{ if } \lvert R_{l}^{j} (t,x) \rvert \leq \frac{\delta}{4d}, \\
1 & \text{ if } \lvert R_{l}^{j}(t,x) \rvert \geq \frac{\delta}{2d}. 
\end{cases} 
\end{equation}
(recall \eqref{est 245}). We observe that the identities \eqref{est 101}-\eqref{est 71} and the estimate \eqref{est 81} all remain valid. We define $\vartheta_{c}$ and $q_{c}$ identically to \eqref{est 85} so that $\theta_{1}$ is mean-zero. Moreover, similarly to \eqref{est 75} we can compute 
\begin{equation}\label{est 199}
\nabla \cdot w(t, x- B(t)) \overset{\eqref{est 71}\eqref{est 59}\eqref{est 72}}{=} \sum_{j=1}^{d} \nabla (b_{j} (t, x - B(t) ) ( \tilde{\varrho}_{\mu}^{j})_{\lambda} \circ \tau_{B(t) + \omega t e_{j}} ) \cdot \psi_{\nu}^{j} (x - B(t)) e_{j} 
\end{equation} 
due to $\nabla \cdot \psi^{j} (\nu (x- B(t)) e_{j})  = 0$. We define $w_{c}(t,x)$ and $f_{j}(t,x)$ identically to \eqref{est 73} so that 
\begin{subequations}\label{est 198}
\begin{align}
&w_{c}(t,x- B(t)) = - \sum_{j=1}^{d} \mathcal{R}_{N} (f_{j}(t,x-B(t)), \psi_{\nu}^{j}(x-B(t))e_{j}), \\
& f_{j}(t,x-B(t)) = \nabla (b_{j}(t, x- B(t)) \tilde{\varrho}_{\mu}^{j} (\lambda (x- B(t) - \omega t e_{j} ))).
\end{align}
\end{subequations} 
Similarly to \eqref{est 197} we see from \eqref{est 198} that $\nabla \cdot w_{c}(x -B(t)) = - \nabla \cdot w(t, x-B(t))$. Therefore, we conclude from \eqref{est 188} that $\nabla\cdot u_{1} = 0$ as desired. Next, due to our choice of definitions thus far, it is clear from \eqref{est 178} and \eqref{est 322} that Lemmas \ref{Lemma 4.4}-\ref{Lemma 4.5} remain valid with ``$M_{0}(t)$'' therein replaced by ``1.'' For completeness, we list these estimates here: analogously to \eqref{est 82}, \eqref{est 83}, \eqref{est 84}, we have for all $j \in \{1,\hdots, d\}$ and $t \in [0,T]$ 
\begin{subequations}\label{est 361} 
\begin{align}
&\lVert a_{j}(t) \rVert_{L_{x}^{\infty}} \lesssim \eta \lVert R_{0} \rVert_{C_{t,x}}^{\frac{1}{p}}, \hspace{6mm} \lVert b_{j}(t) \rVert_{L_{x}^{\infty}} \lesssim \eta^{-1} \lVert R_{0} \rVert_{C_{t,x}}^{\frac{1}{p'}},  \label{new 1}\\
&\lVert a_{j}(t) \rVert_{C_{x}^{s}} \lesssim \eta l^{-(d+2) s} \delta^{\frac{1}{p}}, \hspace{4mm} \lVert b_{j}(t) \rVert_{C_{x}^{s}} \lesssim \eta^{-1} l^{-(d+2) s} \delta^{\frac{1}{p'}} \hspace{3mm} \forall \hspace{1mm} s \in \mathbb{N},  \label{new 3}\\
& \lVert \partial_{t} a_{j}(t) \rVert_{L_{x}^{\infty}} \lesssim \eta l^{-(d+2)} \delta^{\frac{1}{p}}, \hspace{2mm}  \lVert \partial_{t} b_{j}(t) \rVert_{L_{x}^{\infty}} \lesssim \eta^{-1} l^{-(d+2)} \delta^{\frac{1}{p'}}, \label{new 5}
\end{align}
\end{subequations}
which lead to, analogously to \eqref{est 134}, \eqref{est 135}, \eqref{est 151}, \eqref{est 153}, \eqref{est 154}, \eqref{est 91}, \eqref{est 136}, \eqref{est 93}:
\begin{subequations}
\begin{align}
& \lVert \vartheta(t) \rVert_{L_{x}^{p}} \leq \frac{M\eta}{2} (2 \delta)^{\frac{1}{p}} + \frac{C}{\lambda^{\frac{1}{p}}} \eta l^{-(d+2)} \delta^{\frac{1}{p}}, \label{new 7}\\
& \lVert q(t) \rVert_{L_{x}^{p}} \leq C l^{-(d+1)} \delta  \mu^{b} \omega^{-1}, \label{new 8} \\
& \lvert \vartheta_{c}(t) \rvert \leq C \eta \lVert R_{0} \rVert_{C_{t,x}}^{\frac{1}{p}} \mu^{-b} \hspace{2mm} \text{ and } \hspace{2mm}\lvert q_{c}(t) \rvert \leq C\lVert R_{0} \rVert_{C_{t,x}} \omega^{-1}, \label{new 9}\\
& \lVert w(t) \rVert_{L_{x}^{p'}} \leq \frac{M}{2\eta}(2 \delta)^{\frac{1}{p'}} + \frac{C}{\lambda^{\frac{1}{p'}}} \eta^{-1} l^{-(d+2)}\delta^{\frac{1}{p'}}, \label{new 10}\\
&\lVert w(t) \rVert_{W_{x}^{1, \tilde{p}}} \leq  C \eta^{-1} l^{-(d+2)} \delta^{\frac{1}{p'}} \frac{ \lambda \mu + \nu}{\mu^{1+ \epsilon}}, \label{new 11}\\
& \lVert \mathcal{D}^{k} D^{h} f_{j}(t) \rVert_{L_{x}^{r}} \leq C \eta^{-1} \delta^{\frac{1}{p'}} l^{-(d+2) (k+ h + 1)} (\lambda \mu)^{k+ h + 1} \mu^{ b - \frac{d}{r}} \hspace{3mm} \forall \hspace{1mm} k, h \in \mathbb{N}_{0}, r \in [1,\infty],\label{new 12} \\
& \lVert w_{c}(t) \rVert_{L_{x}^{p'}} \leq C\eta^{-1} \delta^{\frac{1}{p'}} [\sum_{k=1}^{N} \left( \frac{ \lambda \mu l^{-(d+2)}}{\nu} \right)^{k} + \frac{ (\lambda \mu l^{-(d+2)})^{N+1}}{\nu^{N}} ], \label{new 13}\\
&\lVert w_{c}(t) \rVert_{W_{x}^{1, \tilde{p}}} \leq C\eta^{-1} \frac{ \delta^{\frac{1}{p'}} [ \lambda \mu l^{-(d+2)} + \nu]}{\mu^{1+ \epsilon}} [ \sum_{k=1}^{N} \left( \frac{ l^{-(d+2)} \lambda \mu}{\nu} \right)^{k} + \frac{ ( l^{-(d+2)} \lambda \mu)^{N+1}}{\nu^{N}} ]. \label{new 14}
\end{align}
\end{subequations}
The next step is to define the new defect $R_{1}$; in fact, we observe that it is same as \eqref{est 200}-\eqref{est 202} and \eqref{est 214} due to \eqref{est 172} and \eqref{est 186}.

\subsubsection{Estimates on $\partial_{t} (q(t,x)+ q_{c}(t,x)) + \text{div} (\vartheta(t,x) w(t,x) - R_{l}(t,x)) = (\text{div} R^{\text{time,1}} + \text{div} R^{\text{quadr}} + \text{div} R^{\chi})(t,x)$ in \eqref{est 202}}\label{Section 7.2.2} 

The estimates here follow Section \ref{Section 4.1.1} completely. Specifically, the identities \eqref{est 102} and \eqref{est 106} go through identically. We define $R^{\chi}$ identically to \eqref{est 105}, although with $\chi_{j}$  in \eqref{est 322}. The identities \eqref{est 107}-\eqref{est 117} follow. We define $R^{\text{time,1}}$, $R^{\text{quadr,1}}, R^{\text{quadr,2}}$, and $R^{\text{quadr}}$ identically to \eqref{est 114}, \eqref{est 115}, \eqref{est 116} so that identities \eqref{est 119} and \eqref{est 120} hold. It follows that Lemmas \ref{Lemma 4.6} holds with ``$M_{0}(t)$'' therein replaced by ``1,'' specifically
\begin{subequations}
\begin{align}
& \lVert R^{\chi}(t) \rVert_{L_{x}^{1}} \leq \frac{\delta}{2}, \label{new 15}\\
&\lVert R^{\text{time,1}}(t) \rVert_{L_{x}^{1}} \leq C\omega^{-1} l^{-(d+2)} \max\{ \delta^{\frac{1}{p}} \lVert R_{0} \rVert_{C_{t,x}}^{\frac{1}{p'}}, \delta^{\frac{1}{p'}} \lVert R_{0} \rVert_{C_{t,x}}^{\frac{1}{p}} \}, \label{new 16}\\
& \lVert R^{\text{quadr}}(t) \rVert_{L_{x}^{1}} \leq C \left( \frac{\lambda \mu}{\nu} + \frac{1}{\lambda} \right) l^{-(d+2) 2} \max\{ \delta^{\frac{1}{p}} \lVert R_{0} \rVert_{C_{t,x}}^{\frac{1}{p'}}, \delta^{\frac{1}{p'}} \lVert R_{0} \rVert_{C_{t,x}}^{\frac{1}{p}} \}.\label{new 17}
\end{align}
\end{subequations}

\subsubsection{Estimates on $\partial_{t}(\vartheta(t,x) + \vartheta_{c}(t)) + \text{div}(\theta_{0}(t,x) w(t,x) + \vartheta(t,x) u_{0}(t, x+ B(t))) - \Delta (\vartheta(t,x) + q(t,x))= (\text{div} R^{\text{time,2}} + \text{div} R^{\text{lin}})(t,x)$ in \eqref{est 202}}\label{Section 7.2.3} 

The only difference here from Section \ref{Section 4.1.2} is $u_{0}(t, x+ B(t))$ rather than $u_{0}(t,x)$. The identity \eqref{est 127} continues to hold with $u_{0}(t,x)$ therein replaced by $u_{0}(t,x+B(t))$. We define $R^{\text{time,2}}$ identically to \eqref{est 125} and $R^{\text{lin}}$ identically to \eqref{est 212} except with $a_{j}$ from \eqref{est 77} instead of \eqref{est 244} and $w, \vartheta,$ and $q$ from \eqref{est 79} instead of \eqref{est 255}. This new definition allows \eqref{est 128} to continue to hold, only with $u_{0}(t,x)$ replaced by $u_{0}(t, x+B(t))$ while \eqref{est 129} also holds. With ``$M_{0}(t)$'' therein replaced by ``1,'' \eqref{est 159} continues to hold while \eqref{est 158} also remains valid; i.e., 
\begin{subequations}
\begin{align}
\lVert R^{\text{lin}}(t) \rVert_{L_{x}^{1}} &\leq C( \mu^{-a} \lVert \theta_{0} \rVert_{C_{t,x}} \eta^{-1} \lVert R_{0} \rVert_{C_{t,x}}^{\frac{1}{p'}} + \mu^{-b}  \eta [ \lVert R_{0} \rVert_{C_{t,x}}^{\frac{1}{p}} \lVert u_{0} \rVert_{C_{t,x}} +  \delta^{\frac{1}{p}} l^{-(d+2)}] \label{new 18}\\
& + [\eta \lVert R_{0} \rVert_{C_{t,x}}^{\frac{1}{p}} \mu^{-b} + \lVert R_{0} \rVert_{C_{t,x}} \omega^{-1} ] [\lambda \mu + \nu] + l^{-(d+2)} \max\{ \delta^{\frac{1}{p}} \lVert R_{0} \rVert_{C_{t,x}}^{\frac{1}{p'}}, \delta^{\frac{1}{p'}} \lVert R_{0} \rVert_{C_{t,x}}^{\frac{1}{p}} \} \omega^{-1}, \nonumber\\
\lVert R^{\text{time,2}}(t) &\rVert_{L_{x}^{1}} \leq C \left(\frac{\omega}{\mu^{b}}\right) \eta \delta^{\frac{1}{p}} \left(\sum_{k=1}^{N} \left( \frac{\lambda \mu}{\nu} \right)^{k} l^{-(d+2) (k-1)} + \frac{ (\lambda \mu)^{N+1}}{\nu^{N}} l^{-(d+2) N} \right). \label{new 19}
\end{align}
\end{subequations}

\subsubsection{Estimates on $\text{div} (q(t,x) (u_{0}(t, x+ B(t)) + w(t,x)) = \text{div} R^{q}(t,x)$ in \eqref{est 202}}\label{Section 7.2.4} 

Again, the only difference here from Section \ref{Section 4.1.3} is $u_{0}(t, x+B(t))$. We define $R^{q}$ identically to \eqref{est 213} except with $q(t,x)$ from \eqref{est 79} rather than \eqref{est 255}. Then Lemma \ref{Lemma 4.8} remains valid.

\subsubsection{Estimates on $\text{div} ([\theta_{0}(t,x) + \vartheta(t,x) + q(t,x)] w_{c}(t,x))= \text{div} R^{\text{corr}}(t,x)$ in \eqref{est 202}}\label{Section 7.2.5}
Our current situation is even simpler than Section \ref{Section 4.1.4} because we can define $R^{\text{corr}}$ identically to \eqref{est 133} but with no $z$. This implies that following the proof of Lemma \ref{Lemma 4.9} identically, we can bound $\lVert R^{\text{corr}}(t) \rVert_{L_{x}^{1}}$ by the r.h.s. of \eqref{est 161} without ``$L^{\frac{1}{4}}$'' and ``$M_{0}(t)$'' replaced by ``1,'' i.e., 
\begin{align}
\lVert R^{\text{corr}}(t) \rVert_{L_{x}^{1}} \leq& C \left(\lVert \theta_{0} \rVert_{C_{t}L_{x}^{p}} + \eta \delta^{\frac{1}{p}}[1+  \lambda^{-\frac{1}{p}} l^{-(d+2)} ] + l^{-(d+1)} \delta \mu^{b} \omega^{-1} \right) \nonumber\\
& \hspace{15mm} \times \eta^{-1} \delta^{\frac{1}{p'}}  [ \sum_{k=1}^{N} \left( \frac{\lambda \mu l^{-(d+2)}}{\nu} \right)^{k} + \frac{ (\lambda \mu l^{-(d+2)} )^{N+1}}{\nu^{N}}]. \label{new 20} 
\end{align} 

\subsubsection{Estimates on $\text{div}R^{\text{moll}}$ in \eqref{est 202}}\label{Section 7.2.6}

We define $R^{\text{moll}}(t,x)$ identically to \eqref{est 100} and the estimate from Lemma \ref{Lemma 4.11} directly applies for us here as well. 

Next, we choose parameters identically to \eqref{est 148} and $\iota$ sufficiently small to satisfy \eqref{est 152}. By construction, $(\theta_{1}, u_{1}, R_{1})$ solves \eqref{est 172}. Due to cut-offs $\chi_{j}$, $\theta_{1}$ and $u_{1}$ defined in \eqref{est 186} satisfy $\theta_{1} \in C_{t,x}^{\infty}, u_{1} \in C_{t}^{\frac{1}{2} - 2 \varpi}C_{x}^{\infty}$ while $R_{1} \in C_{t}^{\frac{1}{2} - 2 \varpi} C_{x}^{\infty}$ due to $w(t, x-B(t)) + w_{c}(t, x- B(t))$ within $u_{1}(t,x)$ and $u_{0}(t,x+B(t))$ within $R_{1}(t,x)$. We can now verify \eqref{est 174}-\eqref{est 177} very similarly to our proof of \eqref{est 39} with ``$M_{0}(t)$'' therein replaced by ``1.'' Lastly, the proof that $(\theta_{1}, u_{1}, R_{1})$ are $(\mathcal{F}_{t})_{t\geq 0}$-adapted and that $(\theta_{1}, u_{1}, R_{1})(0,x)$ are deterministic if $(\theta_{0}, u_{0}, R_{0})(0,x)$ are deterministic is similar to the proof of Theorem \ref{Theorem 2.2} and thus omitted. 

\subsection{Proof of Corollary \ref{Corollary 2.6}}
Following the appropriate modification of \cite[Proposition 2.1]{MS20} in the case with diffusion $-\Delta \rho$, we are able to deduce this key iteration scheme:
\begin{proposition}\label{Proposition 7.4} 
There exists a constant $M > 0$ such that the following holds. Let $p \in (1,\infty)$ and $\tilde{p} \in [1,\infty)$ such that \eqref{est 8} holds, $f\in L^{1}(0, T; L_{0}^{p}(\mathbb{T}^{d}))$. Then for any $\delta, \eta > 0$ and any smooth $(\rho_{0}, u_{0}, R_{0})$ that satisfies 
\begin{equation}
\partial_{t} \rho + \text{div} (u\rho) - \Delta \rho = f -\text{div} R, \hspace{3mm}  \nabla\cdot u = 0, \label{est 328}
\end{equation} 
there exists another smooth $(\rho_{1}, u_{1}, R_{1})$ that satisfies \eqref{est 328} and for all $t \in [0,T]$ 
\begin{subequations}\label{est 340}
\begin{align}
& \lVert (\rho_{1}- \rho_{0})(t) \rVert_{L_{x}^{p}} \leq M \eta \lVert R_{0}(t) \rVert_{L_{x}^{1}}^{\frac{1}{p}}, \hspace{3mm}  \lVert (u_{1} - u_{0})(t) \rVert_{L_{x}^{p'}} \leq M \eta^{-1} \lVert R_{0}(t) \rVert_{L_{x}^{1}}^{\frac{1}{p'}}, \\
& \lVert (u_{1} - u_{0})(t) \rVert_{W_{x}^{1,\tilde{p}}} \leq \delta, \hspace{17mm}  \lVert R_{1}(t) \rVert_{L_{x}^{1}} \leq \delta. 
\end{align}
\end{subequations} 
\end{proposition}

\begin{remark}
We may furthermore prove that if $R_{0}(t) \equiv 0$ for some $t \in [0,T]$, then $R_{1}(t) \equiv 0$ and $(\rho_{1}, u_{1})(t) \equiv (\rho_{0}, u_{0})(t)$ following \cite[Proposition 2.1]{MS20}; however, that result will not help in proving non-uniqueness for our case with non-zero force anyway. 
\end{remark} 

\begin{proof}[Proof of Proposition \ref{Proposition 7.4}]
The proof of Proposition \ref{Proposition 7.4} follows that of \cite[Proposition 2.1]{MS20} very similarly except that one needs to take the diffusion $-\Delta \rho$ into account as in \cite[Section 7.2]{MS20}. Indeed, although the external force $f$ is not smooth, it cancels out in the definition of $R_{1}$ because it appears in \eqref{est 328} for $(\rho_{0}, u_{0}, R_{0})$ and $(\rho_{1}, u_{1}, R_{1})$ (i.e., at the step of \eqref{est 99}). Thus, the same definition of $R_{1}$ in \cite[Equation p. 1096]{MS20} added by $\nabla(\vartheta + q)$ suffices for us. We omit details due to similarity to the proofs throughout this manuscript. 
\end{proof}

We now conclude the proof of Corollary \ref{Corollary 2.6} following those of Theorems \ref{Theorem 2.2} and \ref{Theorem 2.4}. 
\begin{proof}[Proof of Corollary \ref{Corollary 2.6}]
By hypothesis $K > 1$,  $T > 0$, and $f \in L^{1}(0, T; L_{0}^{p}(\mathbb{T}^{d}))$ are fixed. Thus, we can take $N \in \mathbb{N}$ sufficiently large so that 
\begin{equation}\label{est 333}
\frac{K^{N} T}{2(p+1)^{\frac{1}{p}}} > K [\frac{T}{2(p+1)^{\frac{1}{p}}} + \int_{0}^{T} \lVert f(s) \rVert_{L_{x}^{p}} ds]. 
\end{equation} 
Then we choose $\rho_{0}(t,x) \triangleq K^{N} t(x_{d} - \frac{1}{2})$, $u_{0} \equiv 0$, $R_{0} \triangleq \mathcal{D}^{-1} (f - \partial_{t} \rho_{0})$ which is well-defined because $f$ is mean-zero by hypothesis. Then by definition, $(\rho_{0}, u_{0}, R_{0})$ satisfy \eqref{est 328}, $\rho_{0}$ has mean-zero for all $t \geq 0$, $u_{0}$ is trivially divergence-free, and $(\rho_{0}, u_{0}, R_{0})$ is smooth. Moreover, 
\begin{equation}\label{est 332}
\lVert \rho_{0}(t) \rVert_{L_{x}^{p}} = \frac{K^{N} t}{2(p+1)^{\frac{1}{p}}}.
\end{equation} 
We set $\delta \triangleq \sup_{t\in [0,T]} \lVert R_{0}(t) \rVert_{L_{x}^{1}}$, $\delta_{n} = \delta 2^{-(n-1)}$ so that $\delta_{n+1} = \delta_{n} 2^{-1}$ identically to \eqref{est 41} and $\eta_{n} \subset (1,\infty)$ for $n \in \mathbb{N}$ such that $\delta_{n}^{\frac{1}{p}} \eta_{n} = \sigma \delta_{n}^{\frac{1}{2}}$ identically to \eqref{est 42} where $\sigma > 0$ and  
\begin{equation}\label{est 330}
\sigma 4(p+1)^{\frac{1}{p}} M \sum_{n=0}^{\infty} \sqrt{\delta} 2^{-\frac{n-1}{2}}< T.
\end{equation} 
Then using Proposition \ref{Proposition 7.4} iteratively, analogously to the proof of Theorem \ref{Theorem 2.2} we can find $\rho \in C([0,T]; L^{p}(\mathbb{T}^{d}))$ and $u \in C([0,T]; L^{p'} (\mathbb{T}^{d})) \cap C([0,T];  W^{1,\tilde{p}}(\mathbb{T}^{d}))$ that satisfies \eqref{est 329} analytically weakly. We fix this $\rho$ and define $\rho^{\text{in}} = \rho \rvert_{t=0}$. We can compute for all $t \in [0,T]$, 
\begin{equation}\label{est 334}
\lVert (\rho - \rho_{0})(t) \rVert_{L_{x}^{p}} \leq \sum_{n=0}^{\infty} \lVert (\rho_{n+1} - \rho_{n})(t) \rVert_{L_{x}^{p}}
\leq \sum_{n=0}^{\infty} M \sigma \delta_{n}^{\frac{1}{2}}  \overset{\eqref{est 330}}{<} \frac{T}{4(p+1)^{\frac{1}{p}}}. 
\end{equation} 
Finally, this leads to \eqref{est 331} as follows:
\begin{align*}
\lVert \rho(T) \rVert_{L_{x}^{p}} &\overset{\eqref{est 332}}{\geq} \frac{K^{N} T}{2(p+1)^{\frac{1}{p}}} - \lVert (\rho - \rho_{0})(T) \rVert_{L_{x}^{p}} \overset{\eqref{est 333}\eqref{est 334}}{\geq} K [ \frac{T}{4(p+1)^{\frac{1}{p}}} + \int_{0}^{T} \lVert f (s)\rVert_{L_{x}^{p}} ds] \\
\overset{\eqref{est 334}\eqref{est 332}}{>}& K[ \lVert (\rho - \rho_{0})(0) \rVert_{L_{x}^{p}} + \lVert \rho_{0}(0) \rVert_{L_{x}^{p}} + \int_{0}^{T} \lVert f(s) \rVert_{L_{x}^{p}} ds] \geq K [\lVert \rho^{\text{in}} \rVert_{L^{p}} + \int_{0}^{T} \lVert f(s) \rVert_{L_{x}^{p}} ds].
\end{align*}
\end{proof}

\section{Appendix B}\label{Appendix B}
\subsection{Further preliminaries}\label{Section 8.1} 
The following Lemma \ref{Lemma 8.1} is a slight generalization of \cite[Proposition II.1]{DL89} concerning existence of solution to the non-diffusive transport equation with an external force to the diffusive case. We note that the precise statement of \cite[Proposition II.1]{DL89} does not include an external force;  however, \cite[Remark on p. 514]{DL89} states that such an external force $f$ can be added if $f \in L_{t}^{1}L_{x}^{p}$. 
\begin{lemma}\label{Lemma 8.1}
Let $T > 0$, $p \in (1, \infty]$, and $\theta^{\text{in}} \in L^{p} (\mathbb{T}^{d})$. Assume that 
\begin{equation}\label{est 13}
u \in L^{1} (0, T; L^{p'} (\mathbb{T}^{d})), \hspace{2mm} \nabla\cdot u = 0, \hspace{2mm} c \in L^{1} (0, T; L^{\infty} (\mathbb{T}^{d})), \hspace{2mm} f \in L^{1} (0, T; L^{p} (\mathbb{T}^{d})).
\end{equation}  
Then there exists an analytically weak solution $\theta \in L^{\infty} (0, T; L^{p} (\mathbb{T}^{d}))$ to 
\begin{equation}\label{est 227}
\partial_{t} \theta + (u\cdot\nabla) \theta + c\theta  = \Delta \theta + f;
\end{equation} 
i.e., for all $\phi \in C([0,T] \times \mathbb{T}^{d})$ with compact support in $[0, T) \times \mathbb{T}^{d}$, 
\begin{align}
&- \int_{0}^{T} \int_{\mathbb{T}^{d}} \theta(t,x) \partial_{t} \phi(t,x) dx dt - \int_{\mathbb{T}^{d}} \theta^{\text{in}} (x) \phi(0,x) dx - \int_{0}^{T} \int_{\mathbb{T}^{d}}  (u(t,x)\cdot\nabla) \phi(t,x) \theta(t,x) dx dt \nonumber\\
& \hspace{10mm} + \int_{0}^{T} \int_{\mathbb{T}^{d}} c \phi(t,x) \theta (t,x) dx dt = \int_{0}^{T} \int_{\mathbb{T}^{d}} \theta(t,x) \Delta \phi(t,x)  + f(t,x) \phi(t,x) dx dt. \label{est 12}
\end{align} 
\end{lemma} 

\begin{proof}[Proof of Lemma \ref{Lemma 8.1}]
The proof follows that of \cite[Proposition II.1]{DL89} completely with only a care about the additional diffusive term in case $p \in (1,2]$, partially following the argument of \cite[Lemma 3.1]{CZ16}. We start with $\psi$ that is smooth with compact support, non-negative, and has mass one, define $\psi_{\epsilon} (x) \triangleq \epsilon^{-d} \psi(\frac{\cdot}{\epsilon})$ and mollify to obtain $u_{\epsilon} \triangleq u \ast_{x} \psi_{\epsilon}$, $c_{\epsilon} \triangleq c \ast_{x} \psi_{\epsilon}$, $\theta_{\epsilon}^{\text{in}} \triangleq \theta^{\text{in}} \ast_{x} \psi_{\epsilon}$, and $f_{\epsilon} \triangleq f \ast_{x} \psi_{\epsilon}$. Thus, we obtain a corresponding solution $\theta_{\epsilon} \in C_{t}C_{x}^{1}$ to 
\begin{equation}\label{est 9}
\partial_{t} \theta_{\epsilon} + \text{div} (u_{\epsilon} \theta_{\epsilon}) + c_{\epsilon} \theta_{\epsilon} - \Delta \theta_{\epsilon} = f_{\epsilon} \hspace{3mm} \forall \hspace{1mm} t > 0, \hspace{3mm} \theta_{\epsilon} \rvert_{t=0} = \theta_{\epsilon}^{\text{in}}. 
\end{equation} 
Now for all $p \in (1,\infty)$, as the function $r \mapsto r^{p}$ is $C^{1}$, we multiply \eqref{est 9} by $\lvert \theta_{\epsilon} \rvert^{p-2} \theta_{\epsilon}$, integrate over $\mathbb{T}^{d}$, use the fact that $\nabla \cdot u_{\epsilon} = 0$ to deduce 
\begin{equation}
\frac{1}{p} \partial_{t} \lVert \theta_{\epsilon} \rVert_{L_{x}^{p}}^{p} - \int_{\mathbb{T}^{d}} \Delta \theta_{\epsilon} \lvert \theta_{\epsilon} \rvert^{p-2} \theta_{\epsilon} dx = -  \int_{\mathbb{T}^{d}} c_{\epsilon} \lvert \theta_{\epsilon} \rvert^{p} dx + \int_{\mathbb{T}^{d}} f_{\epsilon} \lvert \theta_{\epsilon} \rvert^{p-2} \theta_{\epsilon} dx. 
\end{equation} 
If $p \in [2, \infty)$, then the function $r \mapsto r^{p-1}$ is $C^{1}$ and thus we can integrate by parts to deduce 
\begin{equation}\label{est 10}
\int_{\mathbb{T}^{d}} \Delta \theta_{\epsilon} \lvert \theta_{\epsilon} \rvert^{p-2} \theta_{\epsilon} dx = - (p-1) \int_{\mathbb{T}^{d}} \lvert \nabla \theta_{\epsilon} \rvert^{2} \lvert \theta_{\epsilon} \rvert^{p-2} dx. 
\end{equation}
On the other hand, if $p \in (1,2)$, then we can use  
\begin{align*}
\int_{\mathbb{T}^{d}} \Delta \theta_{\epsilon} ( \lvert \theta_{\epsilon} \rvert + \delta)^{p-2} \theta_{\epsilon} dx = - \int_{\mathbb{T}^{d}} (p-2) \lvert \nabla \theta_{\epsilon} \rvert^{2} \lvert \theta_{\epsilon} \rvert  ( \lvert \theta_{\epsilon} \rvert + \delta)^{p-3} + \lvert \nabla \theta_{\epsilon} \rvert^{2} ( \lvert \theta_{\epsilon} \rvert + \delta)^{p-2} dx, 
\end{align*}
to deduce \eqref{est 10} by monotone and dominated convergence theorems. Having obtained \eqref{est 10} for all $p \in (1,\infty)$, we can deduce in a standard manner 
\begin{equation}\label{est 11}
\lVert \theta_{\epsilon} (t) \rVert_{L_{x}^{p}} \leq \lVert \theta^{\text{in}} \rVert_{L^{p}} e^{\int_{0}^{t} \lVert c \rVert_{L_{x}^{\infty}} ds} + \int_{0}^{t} e^{\int_{s}^{t} \lVert c \rVert_{L_{x}^{\infty}} d\tau} \lVert f (s) \rVert_{L_{x}^{p}} ds; 
\end{equation}
the case $p = \infty$ follows by taking limit $p\nearrow \infty$ in \eqref{est 11}. Extracting a subsequence if necessary that converges to some $\theta$ and showing that it satisfies \eqref{est 12} is also standard; we refer to \cite[Proposition II.1]{DL89} for further details.  
\end{proof}

\subsection{Proof of second inequality of \eqref{est 63} in Lemma \ref{Lemma 4.3}}
We estimate 
\begin{align*}
\lVert W_{\lambda, \mu, \omega, \nu}^{j}(t) \rVert_{W_{x}^{1,\tilde{p}}}
\overset{\eqref{est 59}}{\leq}& \lVert ( \tilde{\varrho}_{\mu}^{j})_{\lambda} \circ \tau_{\omega t e_{j}} \rVert_{L_{x}^{\tilde{p}}} \lVert \psi_{\nu}^{j} \rVert_{L^{\infty}} \\
&+ \lVert D( ( \tilde{\varrho}_{\mu}^{j})_{\lambda} \circ \tau_{\omega t e_{j}}) \rVert_{L_{x}^{\tilde{p}}} \lVert \psi_{\nu}^{j} \rVert_{L^{\infty}} + \lVert (\tilde{\varrho}_{\mu}^{j})_{\lambda} \circ \tau_{\omega t e_{j}} \rVert_{L_{x}^{\tilde{p}}} \lVert D \psi_{\nu}^{j} \rVert_{L^{\infty}} \\
\overset{\eqref{est 66}\eqref{est 67} \eqref{est 68}}{=}& \mu^{ b - \frac{d}{\tilde{p}}} \lVert \varrho \rVert_{L^{\tilde{p}}} \lVert \psi \rVert_{L^{\infty}}+ \lambda \mu^{b - \frac{d}{\tilde{p}} + 1} \lVert D \varrho \rVert_{L^{\tilde{p}}} \lVert \psi \rVert_{L^{\infty}} + \nu \mu^{ b - \frac{d}{\tilde{p}}} \lVert \varrho \rVert_{L^{\tilde{p}}} \lVert D \psi \rVert_{L^{\infty}} \\
\overset{\eqref{est 55} \eqref{est 69}}{\leq}&   \mu^{\frac{d}{p'} - \frac{d}{\tilde{p}}}  \frac{M}{2d} + \lambda \mu^{\frac{d}{p'} - \frac{d}{\tilde{p}} + 1} \frac{M}{2d} + \nu\mu^{\frac{d}{p'} - \frac{d}{\tilde{p}}} \frac{M}{2d}.
\end{align*}
Now $\frac{d}{p'} - \frac{d}{\tilde{p}} < - 1 - \epsilon$ due to \eqref{est 94} and $\mu \gg \lambda$ due to \eqref{est 58} imply 
\begin{align*}
\lVert W_{\lambda, \mu, \omega, \nu}^{j}(t) \rVert_{W_{x}^{1,\tilde{p}}} \leq \mu^{-1 - \epsilon} (\frac{M}{2d}) + \lambda \mu^{-\epsilon} (\frac{M}{2d})  + \nu \mu^{-1 - \epsilon} (\frac{M}{2d}) \leq M \left( \frac{\lambda}{\mu^{\epsilon}} + \frac{\nu}{\mu^{1+ \epsilon}} \right) 
\end{align*}
for $\lambda \in \mathbb{N}$ sufficiently large which is the desired result.
 
\section*{Acknowledgements}
The second author would like to thank Prof. Carl Mueller for valuable discussions. The first author acknowledges the support of the Department of Atomic Energy,  Government of India, under project no.$12$-R$\&$D-TFR-$5.01$-$0520$, and DST-SERB SJF grant DST/SJF/MS/$2021$/$44$.

\end{document}